\documentclass[oneside]{book}

\usepackage{imakeidx}
\makeindex

\usepackage{amsmath,amssymb,amsthm,color,graphicx, fancybox}
\usepackage[titletoc]{appendix}
\usepackage{array}
\usepackage{booktabs}
\usepackage{multirow}

\usepackage{subfigure}
\usepackage{subfigmat}
\usepackage{float}

\usepackage{tcolorbox}
\usepackage{bclogo}
\usepackage[latin1]{inputenc}
\usepackage{mdframed}
\usepackage{amsmath}
\usepackage{amsfonts}
\usepackage{amssymb}
\usepackage{wasysym}
\usepackage{amsthm}
\usepackage{graphicx}
\newcommand{\B}{{\cal B}}

\newcommand{\vx}{\vec{x}}
\newcommand{\oPi}{\overline{\Pi}}

\newcommand{\veta}{\vec{\eta}}
\newcommand{\bm}{\bar{m}}

\newcommand{\bmu}{\bar{\mu}}
\newcommand{\bnu}{\bar{\nu}}

\newcommand{\uc}{\underline{c}}

\newcommand{\otheta}{\theta}
\newcommand{\utheta}{c}

\newcommand{\Xept}{\Xi^{\theta,+}}
\newcommand{\XicJ}{\Xi_{\cal J}}
\newcommand{\Xepto}{\Xi^{\theta}}
\newcommand{\wPzeta}{{\cal P}^{w,\zeta}}
\newcommand{\wSPzeta}{{\cal SP}^{w,\bar{\mu}}}
\newcommand{\sPzeta}{{{\cal P}^{\zeta}}}
\newcommand{\sSPzeta}{{\cal SP}^{\zeta}}
\newcommand{\wP}{{\cal P}^w}
\newcommand{\wSP}{{\cal SP}^w}
\newcommand{\cK}{\vec{K}}
\newcommand{\J}{{\cal J}}
\newcommand{\K}{{\cal K}}
\newcommand{\cI}{\mathfrak{D}}
\newcommand{\blambda}{\bar\lambda}
\newcommand{\Xep}{\Xi^{0,+}_\zeta}
\newcommand{\vZet}{\vec{\bf  Z}}
\newcommand{\vzet}{\vec{z}}
\newcommand{\vt}{\vec{t}}
\newcommand{\vy}{\vec{y}}
\newcommand{\vw}{\vec{w}}
\newcommand{\Xeeps}{\Xi^\eps_\zeta}
\newcommand{\Xeepsp}{\Xi^{\eps,+}_\zeta}
\newcommand{\OSP}{{\cal OSP}^N}

\newcommand{\OP}{{\cal OP}^N}

\newcommand{\vP}{{{\cal P}^{N}}}
\newcommand{\vlambda}{\vec{\lambda}}
\newcommand{\vSP}{{{\cal SP}^N}}
\newcommand{\vh}{\vec{h}}

\newcommand{\N}{\mathbb{N}}
 \newcommand{\Xee}{\underline{\Xi}^0_\zeta}
\newcommand{\Xeep}{\underline{\Xi}^{0,+}_\zeta}

\newcommand{\hatuXi}{\Sigma^{\theta,+}_\zeta}
\newcommand{\hatXi}{\Sigma^\theta_\zeta}

\newcommand{\Xipsi}{\Xi^\theta_\zeta}
\newcommand{\Xipsip}{\Xi^{\theta,+}_\zeta}
\newcommand{\vA}{\vec{A}}
\newcommand{\vzeta}{\bar{\zeta}}

\newcommand{\vM}{\vec{m}}

\newcommand{\I}{{\cal I}}

\newcommand{\D}{{{\mathbb M}_+}}
\newcommand{\Dt}{{\mathbb M}^{'}}
\newcommand{\SDt}{{\mathbb Q}}

\newcommand{\R}{\mathbb R}

\newcommand{\eps}{\epsilon}

\newcommand{\vq}{\vec{q}}
\newcommand{\vm}{\vec{m}}
\newcommand{\vpp}{\vec{p}}

\newcommand{\vtheta}{\vec{\theta}}

\newcommand{\xiphi}{\xi^\theta_\zeta}

\newcommand{\xiphip}{\xi^{\theta,+}_\zeta}

\newcommand{\cP}{\vec{\bf P}}
\newcommand{\cQ}{\vec{\bf Q}}
\newcommand{\cM}{\vec{\bf M}}

\newcommand{\vphi}{\vec{\phi}}
\newcommand{\vmu}{\vec{\mu}}
\newcommand{\vnu}{\vec{\nu}}

\newcommand{\vpsi}{\vec{\psi}}
\newcommand{\A}{{\cal A}}

\newcommand{\Ss}{S}
\newcommand{\vmS}{{\bf \Delta}}
\newcommand{\uvmS}{{\bf \underline{\Delta}}}
\newcommand{\E}{\mathbb{E}}

\newtheorem{theorem}{Theorem}[chapter]
\newtheorem{lemma}{Lemma}[chapter]
\newtheorem{prop}{Proposition}[chapter]
\newtheorem{cor}{Corollary}[section]

\newtheorem{example}{Example}[section]
\newtheorem{remark}{Remark}[section]
\newtheorem{defi}{Definition}[section]
\newtheorem{acknowledgment*}{Acknowledgment}
\newtheorem{assumption}{Assumption}[section]
\newtheorem{sassumption}{Standing Assumption}[section]
\newcommand{\be}{\begin{equation}}
\newcommand{\ee}{\end{equation}}

\title{\Huge \textbf{  Semi-discrete Optimal Transport } }
\author{\textsc{Gershon Wolansky\footnote{Department of Mathematics, Technion, Haifa 32000, Israel}
		\footnote{Acknowledgment: This book is partially supported by the the Israel Science Foundation grant  988/15}
}}
\begin{document}
\frontmatter
\maketitle








\section{Preface}

The theory of optimal transport was born
towards the end of the 18th century, its founding father being  Gaspard Monge \cite{Mo}.

Optimal transport Optimal transport theory has  connections with PDEs, kinetic theory, fluid dynamics, geometric inequalities, probability and many other mathematical fields  as well as in computer science and economics. As such, it  has attracted many  leading mathematicians in the last decades.

There are several very good textbooks and monographs on the subject. For the novice  we recommend, as an appetizer, the first book of C. Villani \cite{Vil1}, titled "Topics in optimal Transport". This   book describes, in a vivid way, most of what was known on this subject on  its publication date (2003). For a dynamical approach we recommend  the book of Ambrosio, Gigli and Savare \cite{amb1}, dealing with paths of probability  measures and the vector-field generating them.  This fits well with the thesis of Alessio Figalli on optimal transport and action minimizing measures \cite{Fig}.   The main treat is, undoubtedly,  the second  monster  book \cite{Vil2} of Villani published in 2008. This book   emphasizes the geometric point of view and contains  a lot more.  For the reader interested in application to economics we recommend  the recent book \cite{Gal} of A. Galichon, while for those interested in connections with probability theory and random processes we recommend the book of Rachev and Raschendorf \cite{Rec1}. As a desert we recommend  the recent book of F. Santambrogio \cite{SntA},  which provides an overview of the main landmarks from the point of view of applied mathematics, and includes also a description  of several up-to-date numerical methods.

In between these courses  the reader may browse through countless  number of review papers and monographs, written by leading experts in this fast  growing subject.

In the current book I suggest an off-road path to the subject. I tried to avoid  prior knowledge of analysis, PDE theory and functional analysis, as much as possible. Thus I  concentrate on discrete and semi-discrete cases, and always assume compactness for the underlying spaces.   However, some fundamental knowledge of measure theory and convexity is unavoidable. In order to make it as self-contained as possible I included   an appendix with some basic definitions and results. I believe that any graduate student in mathematics, as well as advanced undergraduate students, can read and understand this book. Some chapters  can also be of interest  for experts.

It is important to emphasize that this   book cannot replace any of the books mentioned above. For example the very relevant subject of elliptic  and parabolic PDE (the Monge-Amper and the Fokker-Plank equations, among others) is missing, along with regularity issues and many other subjects. It provides, however, an alternative way to the understanding of  some of the basic ideas behind optimal transport and its applications and, in addition, presents some extensions which cannot be found elsewhere. In particular, the subject of vector transport,  playing  a major role in part II of this book is, to the best of my knowledge,  new. The same can be said about some applications
discussed in chapter \ref{learningT} and Part III.

Starting with the the most fundamental, fully discrete problem  I attempted to place optimal transport as a particular  case of the celebrated stable
marriage problem. From there we proceed   to the partition problem, which can be formulated as a transport from a continuous space to a discrete one. Applications to  information theory and  game theory (cooperative and non-cooperative) are introduced as well.
 Finally, the general case of transport between two compact  measure spaces is introduced as a coupling between two semi-discrete transports.

\tableofcontents

\mainmatter
\section{How to read this book?}
The introduction (Chapter \ref{chintro}) provides an overview on the content of the book. 

Chapters \ref{ch1}, \ref{MtF} are independent of the rest of this book. 

Other than that, chapter \ref{S(M)P} is the core of this book, and it is a pre-requisite for the subsequent chapters. 

The readers who are mainly interested in the applications to economics and game theory may jump  from chapter  \ref{S(M)P} to  part \ref{part3}, starting from section \ref{indivalue} and taking Theorem \ref{old} for granted, and also section \ref{hedonic}. Some of these readers may find also an interest in Chapter \ref{optimstrong}, which, unfortunately,  is {\em not} independent of Chapter \ref{Wps},  

The reader interested in application to learning theory may skip from Chapter \ref{S(M)P}  to section \ref{1learn}, but it is recommended to read part \ref{partII} (or, at least go over   the definitions in Chapter \ref{Wps}) before  reading  section \ref{bottleN} on the information bottleneck
\index{information bottleneck}

It is also possible to read Part \ref{From Multipartitions to Multitransport} after Chapter \ref{S(M)P} which, except section \ref{symcong}, is independent of the rest.  
\section{Notations}\label{notations}
The following is a (non-exhaustive) list of  notations  used throughout the book. Other notations will be presented at the first time used. 
\begin{enumerate}
	\item $\R$ is the field of real numbers. $\R_+$ the real non-negative numbers, $\R_{++}$ the real positive numbers, $\R_-$ the real non positive and $\R_{--}$ the real negatives.
	\item For $x,y\in\R$, $\min(x,y):= x\wedge y$. $\max(x,y):= x\vee y$
	\item $\Delta^N(\gamma):= \{(x_1, \ldots x_N)\in \R^N_+ \ , \ \ \sum_{i=1}^Nx_i=\gamma\}$.
	\item $\underline{\Delta}^N(\gamma):= \{(x_1, \ldots x_N)\in \R^N_+ \ , \ \ \sum_{i=1}^N x_i\leq\gamma\}$.
	\item $\D(N,J):= \R_+^N\otimes\R_+^J$. It is  the set of $N\times J$ matrices of non-negative real numbers. Likewise,  $\Dt(N,J):=\R^N\otimes\R^J$- the set of $N\times J$ matrices of  real numbers.   
	\item
	For $\cM=\{m_{i,j}\}\in\D(N,J)$ and 
	$\cP=\{p_{i,j}\}\in \Dt(N,J)$, \\
	$\cP:\cM :=tr (\cM \cP^t)= \sum_{i=1}^N\sum_{j=1}^J p_{i,j}m_{i,j}$.
	\item $(X, {\cal B})$ is a compact measure space, and ${\cal B}$ the Borel $\sigma-$algebra  on $X$.
	\item  ${\cal M}(X)$ the set of Borel measures on $X$.    \index{Borel measure}
	 ${\cal M}_+\subset {\cal M}$ is the set of non-negative measures, and ${\cal M}_1$ the {\em probability measures}, namely $\mu(X)=1$. 
	$\sum_{j=1}^J\mu^{(j)}:=\mu$. 
\end{enumerate}
\chapter{Introduction}\label{chintro}

\section{The fully discrete case \protect\footnote{Part of this chapter was published by the author in \cite{wol2} }}

Imagine a set $\I_m$ composed  of  $N$ men and a set $\I_w$ composed of $N$ women.  
Your task is to form $N$ married pairs $\{ ii^{'}\}\subset \I_m\times \I_w$ out of these set, where each pair is composed of a single man $i\in\I_m$  and a single woman $i^{'}\in\I_w$, and make everybody happy.  This is the celebrated  stable marriage problem.
\par
What is the meaning of "making everybody happy"? There is, indeed, a very natural definition for it, starting from the definition of a {\em blocking pair}.

A  blocking pair is an unmarried couple (a man and a woman) who  prefers each other  over their assigned  spouses. The existence of  a blocking pair
will cause two couples to divorce and  
may start an avalanche destabilizing {\em all} the assigned matchings.  

The definition of a stable marriage (which, in our case, is a synonym to "happy marriage")  is
\begin{tcolorbox}
	\begin{center}{\em 	There are no blocking pairs \index{blocking pairs}.}\end{center}
\end{tcolorbox}

The main focus in this book is on the {\em transferable} model, which assumes a somewhat materialistic point of view.

\begin{tcolorbox}
	A married couple $ii^{'}\in \I_m\times \I_w$  can share a reward $\theta(i,i^{'})\geq 0$ (say, in US dollars). 
\end{tcolorbox}
Suppose now that you assigned man $i$ to woman $i^{'}$ and man $j\not=i$ to the woman $j^{'}\not= i^{'}$.  A necessary condition for a stable marriage is
\be\label{tworule}\theta(i, i^{'})+\theta(j, j^{'})\geq \theta(i, j^{'})+\theta(j, i^{'}) \ . \ee
Indeed, assume the couple $ii^{'}$ splits the reward between themselves, so that $i$ cuts $u_i=\alpha \theta(i, i^{'})$ dollars while   $i^{'}$ cuts $v_{i^{'}}=(1-\alpha)\theta(i, i^{'})$ dollars, where $\alpha\in(0,1)$. Likewise the couple $jj^{'}$ splits their reward according to the cuts  $u_j=\beta \theta(j, j^{'})$   and
$v_{j^{'}}=(1-\beta)\theta(j ,j^{'})$
where $\beta\in(0,1)$.
If $\theta( i, i^{'})+\theta(j, j^{'})< \theta( i, j^{'})+\theta(j, i^{'}) $ then $$ \theta(i,j^{'})+\theta(j, i^{'})> u_i+v_{i^{'}}+ u_j+v_{j^{'}}$$ so either $\theta(i, j^{'})> u_i+v_{j^{'}}$ or $\theta(j,i^{'})> u_j+v_{i^{'}}$ (or both). In any case at least one of the new pairs $ij^{'}$, $ji^{'}$  can share a reward bigger than the one they could get from their former matching, and thus improve their individual cuts.  Hence at least one of the pairs $ij^{'}$ or $ji^{'}$ is a blocking pair.


From the above argument we conclude that (\ref{tworule}) for any two matched pairs  is a necessary condition  for the  marriage to be stable. Is it also sufficient?

Suppose the pairs  $(i_1i_1^{'}), \ldots (i_ki_k^{'})$, $k\geq 2$  are matched. The sum of the rewards for these couples is
$\sum_{l=1}^k\theta(i_l, i^{'}_l)$.
Suppose they perform a  "chain deal" such that man $i_l$  marries woman $i^{'}_{l+1}$ for $1\leq l\leq k-1$, and the last  man $i_k$ marries the first woman $i^{'}_1$. The  net reward for the new matching is $\sum_{l=1}^{k-1}\theta(i_l,i^{'}_{l+1})+ \theta(i_k, i^{'}_1)$.
A similar argument implies  that a necessary condition for a stable marriage is that this new reward will not exceed the original net reward for these matching, that is
\begin{tcolorbox}For any choice of matched pairs $i_1i^{'}_1, \ldots i_ki^{'}_k$,
	\be\label{chainrule}\sum_{l=1}^k \theta(i_l,i^{'}_l)-\theta(i_l, i^{'}_{l+1})\geq 0\ee (where $i^{'}_{k+1}:= i^{'}_1$)
\end{tcolorbox}
Condition (\ref{chainrule}) generalizes (\ref{tworule}) to the case $k\geq 2$.  It is called {\em cyclical monotonicity} \index{cyclical monotonicity}. It is remarkable that cyclical monotonicity is, indeed,  equivalent to the stability of matching $\{ii^{'}\}$ (i.e to the absence of blocking pairs).

From cyclical monotonicity we can conclude directly an optimality characterization of stable matching. In fact, this is an equivalent definition of stable marriage in the transferable case:
\begin{tcolorbox}
	The marriage $\{ii^{'}\}$ is stable if and only if it maximizes the total reward among all possible $1-1$ matchings $i\in \I_m\rightarrow \tau(i)\in\I_w$, that is
	\be\label{maxnodual}\sum_{i=1}^N\theta(i,i^{'})\geq \sum_{i=1}^N\theta(i,\tau(i))\ee
\end{tcolorbox}

Another very important  notion for the marriage problem (and, in general, for any cooperative game) \index{cooperative game}  is the notion of feasibility set and core.

The feasibility set is the collection of men's cuts $u_i$ and women's cuts $v_{j^{'}}$ which satisfy the {\em feasibility condition}:
\be\label{feas}u_i+v_{j^{'}}
\geq\theta(i,j^{'})\ee for all  $ij^{'}\in\I_m\times \I_w$. The core of a given  matching $\{ii^{'}\}$ is composed of all such cuts $(u_1, \ldots u_N; v_1\ldots  v_N)$ in the  feasibility  set which  satisfies
the equality $u_i+v_{i^{'}}=\theta(i,i^{'})$  for any matched pair $ii^{'}$.
\begin{tcolorbox}
	The matching $\{ii^{'}\}\subset \{\I_m\times\I_w\}$ is stable if and only if the associated core  is not empty.\index{core}
\end{tcolorbox}
There is another, dual  optimality formulation for a stable matching via the feasibility set:
\begin{tcolorbox}
	The cuts $u^0_1, \ldots u^0_N; v^0_1\ldots v^0_N$ is a core if and only if it is a {\em minimizer} of
	the total cut
	$\sum_1^N u_i+v_i$ within the feasibility set (\ref{feas}):
	\be \label{mindual} \sum_1^N u^0_i+v^0_i\leq \sum_1^N u_i+v_i \ . \ee
	In particular if $u^0_1, \ldots v^0_N$ is such a minimizer  then for any man $i\in \I_m$ there exists {\em at least}   one woman $i^{'}\in \I_w$  and for    any woman $i^{'}\in\I_w$  there exists at least one man $i\in\I_m$ for which the equality $u^0_i+v^0_{i^{'}}=\theta_{ii^{'}}$ holds, and the matching $\{ii^{'}\}\subset \I_m\times \I_w$ is stable.
\end{tcolorbox}
Each of the  two dual optimality characterization (\ref{maxnodual}, \ref{mindual}) of stable matching guarantees that for any choice of the rewards $\{\theta(i, j)\}$, a stable matching always exists.

There are  other  ways to define a blocking pair\index{blocking pairs}. A natural way is the {\em non-transferable marriage}\index{non-transferable marriage}. In the non-transferable marriage game each man and woman have a preference list, by which he/she rates the women/men in the group. This is the celebrated marriage  problem of  Gale and Shapley \index{Gale-Shapley algorithm}(who won a Nobel price in economics in 2012).

We may quantify the Gale and Shapley game (after all, we live in a materialistic world). Assume a paring of man $i$ and woman $j^{'}$ will guarantee a cut $\theta_m(i,j^{'})$ to the man and $\theta_w(i,j^{'})$ to the woman. This will induce the preference list for both men and women: Indeed, the man $i$ will prefer the woman $i^{'}$ over $j^{'}$ if and only if $\theta_m(i,i^{'})>\theta_m(i,j^{'})$. Likewise,  the woman $i^{'}$ will prefer the man $i$ over $j$ if and only if $\theta_w(i,i^{'})>\theta_w(j,i^{'})$. A blocking pair for a matching $\{ii^{'}\}\subset \I_m\times \I_w$ is, then, a pair $ij^{'}$ such that $j^{'}\not= i^{'}$ and {\em both}
$$ \theta_m(i,j^{'})\geq \theta_m(i,i^{'}) \ \ \text{and} \ \ \theta_w(i,j^{'})\geq \theta_w(j,j^{'}) \  $$
are satisfied (were at least one of the inequalities is strong).
\par\noindent
GS (Gale-Shapley) stability for a set of rewards $\{\theta_m, \theta_w\}$ does not imply the stability of the transferable game where $\theta=\theta_m+\theta_w$ where each couple is permitted to share their individual rewards (and neither the opposite).
\begin{tcolorbox}
	
	A simple example ($N=2$):
	\begin{center}
		\begin{tabular}{ c c c }
			$\theta_m$& $w_1$& $w_2$\\
			$m_1$& 1 & 0 \\
			$m_2$ & 0 & 1
		\end{tabular}
		\ \ \ \ ; \ \ \
		\begin{tabular}{ c c c }
			$\theta_w$& $w_1$ & $w_2$\\
			$m_1$ & 1 & 5 \\
			$m_2$& 0 & 1
		\end{tabular}
	\end{center}
	The matching $\{ 11,  22\}$ is GS stable. Indeed  $\theta_m(1,1)=1> \theta_m(1,2)=0$
	while $\theta_m(2,2)=1>\theta_m(2,1)=0$, so both men are happy, and this is enough for GS stability, since   that  neither $\{ 12\}$ nor $\{21\}$ is a blocking pair. On the other hand, if the married pairs share their rewards
	$\theta(i,j^{'})=\theta_m(i,j^{'} )+\theta_w(i, j^{'})$ we get
	\begin{center}
		\begin{tabular}{ c c c }
			$\theta$& $w_1$& $w_2$\\
			$m_1$& 2 & 5 \\
			$m_2$ & 0 & 2
		\end{tabular}
	\end{center}	
	so
	$$ \theta(1,1)+\theta(2,2)=4< 5=\theta(1,2)+\theta(2,1) \ , $$
	thus $\{21, 12\}$ is the  stable marriage in the  transferable setting.
\end{tcolorbox}
On top of it, there exists a whole world of marriage games which contains the transferable and GS games as special cases.

There  is a deep theorem which guarantees the existence of a stable marriage for a wide class of partially transferable games, starting from the fully transferable, all the way  to  Gale-Shapley\index{Gale-Shapley algorithm}. The proof of this theorem   is much simpler in the transferable case (due to the optimality characterization) and the Gale Shapley case (due to the celebrated Gale-Shapley algorithm, which is described in Section \ref{smar}).  However, there is an essential difference between the transferable game  and all other cases. As far as we know:

\begin{tcolorbox}
	The transferable marriage game is the only one which is variational, i.e whose stable solutions are characterized by an optimality condition.
\end{tcolorbox}
A discussion on the marriage problem and some of its generalizations  is given in Chapter \ref{ch1}.
\section{Many to few: Partitions}
We may  extend the marriage paradigm to a setting of matching between two sets of different cardinality. Suppose $\I=\{1, \ldots N\}$  is a set representing experts (or sellers)    and $X$ is a much larger (possibly infinite) set representing, say,  the geographical space in which  the customers (or consumers) live.

We consider $(X, {\cal B}, \mu)$ as a measure space, equipped with a $\sigma-$algebra ${\cal B}$  \index{Borel $\sigma$ algebra}and a positive measure $\mu$. We shall always assume that $X$ is also a compact space and ${\cal B}$ a Borel. 
In the expert-customers interpretation $\mu(B)$ it is the number of customers living in $B\in {\cal B}$.

We also associate any  $i\in \I$ with a {\em capacity} $m_i>0$. This can be understood as the maximal possible number of customers the expert $i$ can serve.

A measurable matching $\tau:X\rightarrow \I$    can be represented by a partition $\vA=(A_0, A_1, \ldots A_N)$ where $A_i=\tau^{-1}(\{i\})\in{\cal B}$, $i\in\I\cup \{0\}$  are pairwise disjoint.
The set $A_i$, $i\in\I$  represents the geographical domain in $X$ served by the expert $i$, and  $\mu(A_i)$ represents the number of  customers served by $i$. The set $A_0$ represents a domain which is not served by any of the experts.
A feasible partition must satisfy the constraint\footnote{See section \ref{unbalanced} below for a discussion in the case of  inequality (\ref{strongcons}) vs. equality.}
\be\label{strongcons} \mu(A_i)\leq m_i\ , \ \ i\in\I  \ . \ee

Let us consider the generalization of the transferable marriage game in this context.
The utility of the assignment of $x\in X$ to $i\in\I$ is given by the function $\theta\in C(X\times \I)$. This function is assumed to be non-negative. We usually denote
$\theta(x,i):=\theta_i(x)$ for $i\in\I$, $x\in X$ and  $\theta_0(x)\equiv 0$ is the utility of non-consumer.
The optimal partition $A^0_1, \ldots A^0_N$ is the one which realizes the maximum \index{optimal partition}
\begin{tcolorbox}
	\be\label{sumitoNstrong*}	\sum_{i=1}^N\int_{A_i^0}\theta_i(x)d\mu\geq
	\sum_{i=1}^N\int_{A_i}\theta_i(x)d\mu\ee
	for any feasible subpartition $A_1, \ldots A_N$  verifying (\ref{strongcons}).
\end{tcolorbox}
The assumption $A_i\subset X$  seems to be too restrictive. Indeed, an expert can serve only part of the customers at a given location.
So, me may extend the notion of partition to a {\em weak partition}. \index{weak partition}A weak partition\index{weak partition}  is represented by $N$ non-negative measures $\mu_i$  on $(X, {\cal B})$  verifying the constraints
\be\label{weakconst} \mu_i\geq 0, \ \ \sum_i^N \mu_i\leq \mu \ \ , \ \  \mu_i(X)\leq m_i \ . \ee
Of course, any {\em strong} partition $A_1, \ldots A_N$ is a weak partition, where $\mu_i=\mu\lfloor A_i$ (the restriction of $\mu$ to $A_i$).

The general notion of  stable marriage in the fully discrete case $(\I_m, \I_w)$  can be generalized to stable  partition in the  {\em semi-discrete} case $(X, \I)$.

A natural generalization of (\ref{strongcons})  leads to a stable  weak partition $\vmu^0:=(\mu^0_1, \ldots \mu^0_N)$ obtained by maximizing the total utility

\begin{tcolorbox}
	\be\label{sumitoN*}	\sum_{i=1}^N\int_X\theta_i(x)d\mu^0_i\geq
	\sum_{i=1}^N\int_X\theta_i(x)d\mu_i\ee
	for any feasible subpartition verifying (\ref{weakconst}).
\end{tcolorbox}

As in the fully discrete setting of the marriage problem, we may consider other, non-transferable
partitions.  In particular, the Gale-Shapley\index{Gale-Shapley algorithm}  marriage game is generalized as follows:

Assume that $\I$ stands for a finite number of firms and $X$ the set of potential employees. Let $e(x,i)$ be the reward for $x$ if hired by $i$, and $f(x,i)$ the reward of firm $i$ employing $x$.  The condition for a strong,  stable partition \index{stable partition} $A_1, \ldots A_N$ under non-transferable assumption,   subjected to the capacity constraint $\mu(A_i)\leq m_i$  is
\vskip .3in\noindent
\begin{tcolorbox}
	Either  \ $e(x,i) \geq  e(x,j)$  \  \ for  \ \ $x\in A_i$, $j\in\I$  \ \ \text{or  \ there  \ exists} \ \ $y\in A_j$, \  $j\not= i$  \ \ \ \text{where} \ \  $f(y,j)> f(x,j)$  \ .
\end{tcolorbox}

\par
In Chapter \ref{MtF} we consider  the partition problem for both the completely transferable and non transferable cases.

In chapter \ref{S(M)P}, as well as in the rest of the book, we restrict ourselves to the fully transferable case.
There we lay the foundations of {\em duality theory} for optimal partitions. In the case of equality in (\ref{strongcons},\ref{weakconst})\footnote{cf. section \ref{unbalanced} below} and $\sum m_i=\mu(X)$, this  dual formulation takes the form of minimizing the    convex function
\be (p_1, \ldots p_N)\in\R^N \mapsto \Xi(p_1, \ldots p_N)+\sum_{i=1}^N p_im_i\in\R\ee
where
$$ \boxed{\Xi(p_1, \ldots p_N):= \int_X\max_{1\leq i\leq N}\left( \theta_i(x)-p_i\right)d\mu \ .} $$
In the agents-customers interpretation, the optimal $p_i$ stand for the equilibrium price charged by the agent $i$ for her service. The inequality
\be\label{minagentpineq}\boxed{\Xi(p_1, \ldots p_N)+\sum_{i=1}^N p_im_i\geq \sum_{i=1}^M \int_{A_i}\theta_i(x)d\mu  }\ee
plays a fundamental in part II.

\section{Optimal transport in a nutshell}
Both the transferable marriage  and  partition problems are special cases of the Monge problem in optimal transport.

The original formulation of the Monge problem is very intuitive. It  can be stated  as follows:

\begin{tcolorbox} 
	Given a pile of sand $X$ and a container $Y$, of the same volume,   what is the best plan of moving the sand from the pile in order to fill the container?
\end{tcolorbox}
What do we mean by "a plan"?  
\par\noindent
Let $\mu\in {\cal M}_+(X)$ be a  measure on $X$ signifying the distribution of sand. Let $\nu\in {\cal M}_+(Y)$ be a measure on $Y$ signifying the distribution of free space in the container. The {\em balanced condition}, representing statement "same volume" above, takes the form  
\be\label{bal}\mu(X)=\nu(Y) \ . \ee
A {\em strong} plan  is a mapping  $T:X\rightarrow Y$ which transport the measure $\mu$ to $\nu$, that is
\be\label{push} T_\#\mu=\nu \ \ \text{namely} \ \ \mu(T^{-1}(B))=\nu(B)\ee
for every  measurable set $B\subset Y$. 

The "best plan" is the one which 
minimizes the {\em average} distance
$$ \int_X |x-T(x)|\mu(dx)$$
among all other plans.
\par
The interest  of Monge  was mainly geometrical. In his only (known) paper  on this subject \cite{Mo}  he discovered some fundamental properties of the minimizer and connected  the notion of {\it transport rays} and {\it wavefronts} in optics to this geometrical problem.


In the generalized version of the Monge problem   the  distance function $(x,y)\mapsto d(x,y)$, $x,y\in X$ is replaced by a {\em cost of transportation}  $(x,y)\mapsto c(x,y)$, where $x\in X, y\in Y$. In particular, $X$ and $Y$ can be different domains.
The  Monge problem  takes the form of minimization problem

\be\label{mimMonge}c(\mu,\nu):= \min_T\int_X c(x,T(x)) \mu(dx) \  \ee
among all maps transporting the probability measure $\mu$ on $X$ to $\nu$ on $Y$ (i.e. $T_\#\mu=\nu$).

In the context of expert-customer (which we adopt throughout most of this book), it is more natural to replace the cost $c$ by the {\em utility} $\theta$ which we want to {\em maximize}. Evidently, one may switch from  $c(x,y)$ to $\theta(x,y)=-c(x,y)$ and from (\ref{mimMonge}) to
\be\label{mongeclass}\boxed{\theta(\mu,\nu):= \max_{T_\#\mu=\nu}\int_X \theta(x,T(x)) \mu(dx)} \ . \ee

After this pioneering publication of Monge, the problem fell asleep  for about 160 years, until  Kantorovich's paper  in 1941 \cite{Ka}.
Kantorovich fundamental observation was that this  problem is closely related to a relaxed problem on the set of two-points probability measures $\pi=\pi(dxdy)$
\begin{tcolorbox}
	\be\label{infmonge} \theta(\mu,\nu):= \max_{\pi\in\Pi(\mu,\nu)}\int_X\int_Y \theta(x,y)\pi(dxdy)\ee
\end{tcolorbox}
where $\Pi(\mu,\nu)$ is the set of "weak plans" composed of  point distributions on $X\times Y$ whose marginals are $\mu$ and $\nu$:
\be\label{Pidefmunu} \Pi(\mu,\nu):=\{ \pi\in {\cal M}_+(X\times Y); \ \int_X\pi(dxdy)=\nu(dy) \ , \ \int_Y\pi(dxdy)=\mu(dx) \ \} \ . \ee
The optimal measure $\pi(A\times B)$ represents the {\em probability} of transporting goods located in the measurable set $A\subset X$ to $B\subset Y$. The disintegration
\be\label{plianpi}\pi(A,B)=\int_A P_x(B)\mu(dx)\ee
reveals the {\it conditional probability}
$P_x$ of the transportation from $x\in X$ to $B\subset Y$. Thus, we can interpret Kantorovich's transport plan   as a {\em stochastic} transport. In contrast,   deterministic transport $T$ via Monge's paradigm is the special case where the conditional probability $P_x$ takes the form $P_x(dy)=\delta_{y-T(x)}$. 
\par

The transferable  marriage problem is a simplified version of  an optimal transport plan. Here we replaced the atoms  $x\in X$ and  $y\in Y$ by a finite, discrete sets of men $i\in \I_m$ and women $i^{'}\in\I_w$ of the same cardinality $N$. The measures $\mu,\nu$ are just the {\em uniform} discrete measures
$\mu(\{i\})= \nu(\{j^{'}\})=1$ for all
$i\in\I_m$ and $j^{'}\in \I_w$, while the utility $\theta(x,y)$ is now represented by $N\times N$ matrix $\theta(i,j^{'})$.  The Monge plan verifying  (\ref{infmonge}) takes now the form of
the assignment   given in terms of a permutation $i^{'}=\tau(i)$ \index{permutation}  which maximizes the total reward of matching \be\label{deap} \tau\Rightarrow\sum_{i=1}^N\theta(i,\tau(i)) \ . \ee

The Kantorovich program
replaces  the deterministic assignment by a probabilistic one:   $ \pi_i^{j^{'}}:= \pi(\{i\}, \{j^{'}\})\geq 0$ is the probability of assigning  $i$ to  $j^{'}$. The optimal solution  is then  reduced to the linear programming of maximizing \be\label{dsap}\sum_{i=1}^N\sum_{j^{'}=1}^N \pi_i^{j^{'}} \theta(i,j^{'})\ee 
over all stochastic $N\times N$ matrices $\{\pi_i^{j^{'}}\}$, i.e. these matrices which  satisfy the $3N$ linear constraints $$\sum_{i=1}^N \pi_i^{j^{'}}=\sum_{j^{'}=1}^N\pi_i^{j^{'}}=
1, \ \ \pi_i^{j^{'}}\geq 0$$
The  Birkhoff Theorem\footnote{See section \ref{birk}}  assures us that the optimal solution of this stochastic assignment problem (\ref{dsap}) is identical to the solution of the deterministic version (\ref{deap}). In particular, the optimal stochastic matrix $\{\pi_i^j\}$ is a {\em permutation matrix} $\delta_{\tau(i)-j}$ associated with the permutation $\tau$.\index{permutation}

Likewise,  the transferable partition {\em in the balanced case}   $\sum_{i\in\I}m_i= \mu(X)$  corresponds to a solution of the Kantorovich problem where  the target space $Y$is given  by the discrete space $\I$ of finite cardinality $N$. The measure $\nu$ is given by the capacities $m_i:=\nu(\{i\})$.  The utility $\theta(x,y)$ is represented by
$\theta_i(x)$ where $i\in\I$. A {\em strong} partition in $X$ corresponds to a transport (\ref{push}), where $A_i=T^{-1}(\{i\})$.
The optimal partition  (\ref{sumitoNstrong*}) corresponds to the solution of Monge problem(\ref{mongeclass}).

The weak optimal partition (\ref{sumitoN*})
is nothing but the Kantorovich relaxation (\ref{infmonge}) to the deterministic transport partition problem. Indeed, the set $\Pi(\mu,\nu)$ (\ref{Pidefmunu}) is now reduced to the the set of all weak partitions  \index{weak partition}  $ \pi(dx\times \{i\}):= \mu_i(dx)$ via weak partition
$$\Pi(\mu, \vec{m}):=\{\vmu:=(\mu_1, \ldots \mu_N), \ \mu_i(X)=m_i, \ \sum_1^N\mu_i=\mu \} \ . $$

As a particular example we may assume that $X=\{x_1, \ldots x_{N^*}\}$ is a discrete case as well. In that case we denote $\theta_i(x_j):= \theta(i,j)$,  $\mu(\{x_i\}) := m_i^*$. In the balanced case $\sum_{i=1}^N m_i=\sum_{i=1}^{N^*} m_i^*$
we get the optimal {\em weak} partition
$\mu_i^0(\{x_j\}):= (\pi_i^{0,1}, \ldots \pi_i^{0,N^*})$
as
$$\{\pi_i^{0.j}\}
=\arg\max_{\{m_i^j\}}  \sum_{i=1}^{N} \sum_{j=1}^{N^*} \pi_i^j\theta(i,j) $$
where $\{\pi_i^j\}$ verifying (\ref{weakconst}) {\em in the case of equality} \be\label{summ*n}\pi_i^j\geq 0 \ \ ; \ \ \sum_{j=1}^{N^*} \pi_i^j= m^*_i\ , \ \ \sum_{i=1}^{N} \pi_i^j= m_j \ \ \ ; \ \  (i,x_j)\in\I\times X \ . \ee
We may recover the fully discrete  transferable  marriage (\ref{deap}) in the particular case $N^*=N$ and $m_i=m^*_j=1$ for $1\leq i,j\leq N$ .

The Birkhoff Theorem \index{Birkhoff Theorem}  hints   that the case where the optimal partition $\vmu^0$ in (\ref{sumitoN*}) is a {\em strong \index{optimal partition} subpartition} $\mu^0_i=\mu\lfloor{A^0_i}$  is not so special, after all....
\subsection{Unbalanced transport}\label{unbalanced}
The  case of unbalanced transport  \index{unbalanced transport} $\mu(X)\not=\nu(Y)$ deserves a special attention. Note, in particular, that in  (\ref{strongcons}) we used the inequality $\mu(A_i)\leq m_i$. If the utilities $\theta_i$ are non-negative and if $\sum_1^N m_i\leq \mu(X)$  then  it is evident that the optimal partition   will satisfy the equality $\mu(A_i^0)=m_i$ (same for (\ref{weakconst}, \ref{sumitoN*})). This  presents no conceptual new case, since we can define $m_0:=\mu(X)-\sum_1^N m_i$ and $A_0:=X-\cup_1^NA_i$ constrained by $\mu(A_0)=m_0$, representing the non-consumers in the populations. This  reduces the problem to the case of equality $\sum_0^N m_i=\mu(X)$, where the utility of non-consuming  is $\theta_o\equiv 0$. 
In the dual formulation we may assign, in the case $m_0>0$,  the price $p_0=0$ for non consuming. The inequality  (\ref{minagentpineq}) will take, in this way, the same form as in (\ref{minagentpineq}) where  we integrate only on the {\em positive part} of $\theta_i-p_i$, i.e. $(\theta_i(x)-p_i)_+:= (\theta_i(x)-p_i)\vee 0$. Thus, $\Xi$ is replaced by  
$$\Xi^+(p_1, \ldots p_N):= \int_X\max_{1\leq i\leq N}\left( \theta_i(x)-p_i\right)_+d\mu \ . $$

In the same way we may adopt in the Monge problem (\ref{mongeclass}) the  case $\mu(X)>\nu(Y)$ by adding an auxiliary point $y_0$ to $Y$ and extend $\nu$ to $Y\cup\{y_0\}$ such that $\nu(\{y_0\})=\mu(X)-\nu(Y)$, together with $\theta(x,y_0)=0$ for any $x\in X$.

The case  $\mu(X)<\nu(Y)$ is treated similarly. We just add a virtual point $x_0$ to $X$, assign $\mu(\{x_0\})=\nu(Y)-\mu(X)$ and $\theta(x_0,y)=0$ for any $v\in Y$. 
In the semi-discrete case $\sum_1^N m_i>\mu(X)$ 
this changes (\ref{minagentpineq}) into 
$$\Xi(p_1, \ldots p_N)+\sum_{i=1}^N p_im_i-m_0\min_i p_i \geq \sum_{i=1}^M \int_{A_i}\theta_i(x)d\mu  $$
where $m_0:= \sum_1^N m_i-\mu(X)$ in that case. 
\section{Vector-valued Transport and Multipartitions}
A natural generalization of the optimal transport is {\em optimal vector-valued transport}. Here we replace the measures $\mu,\nu$ by {\em $\R_+^J$-valued} measures
$$\bar\mu:= (\mu^{(1)}, \ldots \mu^{(J)})\in {\cal M}_+^J(X), \ \ \bar\nu:= (\nu^{(1)}, \ldots \nu^{(J)})\in {\cal M}_+^J(Y)\ , $$
and we denote $\mu=|\bmu|:=\sum_1^J \mu^{(j)}$, $\nu=|\bnu|:=\sum_1^J\nu^{(j)}$. The
set $\Pi(\mu,\nu)$ (\ref{Pidefmunu}) is generalized into 
\be\label{Pidefmunuvec} \Pi(\bmu,\bnu):=\{ \pi\in {\cal M}_+(X\times Y); \  \int_X\frac{d\mu^{(j)}}{d\mu}(x)\pi(dxdy)=\nu^{(j)}(dy)\} \  \ .  \ee 
where $d\mu_i/d\mu$, $d\nu_i/ d\nu$  stands for the Radon-Nikodym derivative. \index{Radon-Nikodym derivative}
\par
In general the set $\bar\Pi(\bmu,\bnu)$ can be an empty one.  If  $\Pi(\bmu, \bnu)\not=\emptyset$ then we say that
$\bmu$ {\em dominates} $\bnu$. This is an order relation (in particular transitive), denoted by
\be\label{succ} \bmu \succ \bnu \ . \ee
The generalization of the Kantorovich problem (\ref{infmonge}) takes the form
\begin{tcolorbox}
$$\theta(\bar\mu,\bar\nu):= \max_{\pi\in\Pi(\bar\mu,\bar\nu)}\int_X\int_Y \theta(x,y)\pi(dxdy) \ . $$
	$\theta(\bmu,\bnu)=\infty$  if $\bmu\not \succ\bnu$.
\end{tcolorbox}
Several recent publications deal with a notion of vector valued (or even matrix-valued) optimal transport. See, in particular \cite{CTT} as well as related works \cite{CGP,JNG,TGT,CGGT,NGT,CGT}. There is, however, a fundamental difference between our notion  of vector transport and those publications, since (\ref{Pidefmunuvec}) implies a {\em single} transport plan for all components of the vector.

\par
A possible  motivation for studying such a transport concerns some application to learning theory. A vector-valued measure  $\bmu:= (\mu^{(1)}, \ldots \mu^{(J)})$ on a set $X$ is interpreted as a distribution of a classifier  \index{classifier} of a label $j\in \{1, \ldots J\}$ given a sample $x$ in some feature space $X$. The object of learning is to model this classifier by a simpler one on a finite sample space $\I$, while preserving as much as possible the information stored in the given classifier. This subject is discussed in chapter \ref{learningT}.

In part II we  consider an implementation of $\R_+^J$-valued transport to multi-partitions.
Here we replace the space $Y$ with the discrete space $\I=\{1, \ldots N\}$, and the $\R^J$-valued measure $\bar\nu$ is represented by an $N\times J$ matrix $\cM:=\{m_i^{(j)}\}$, where $m_i^{(j)}$ stands for $\nu^{(j)}(\{i\})$.

A multi partition of $X$ subjected to  $\cM$ is a partition of $X$ into mutually disjoint  measurable sets $A_1, \ldots A_N\subset X$  satisfying
\be\label{multistrongconst}\mu^{(j)}(A_i)= m^{(j)}_i, \ \ \text{for} \  i= 1, \ldots N, \ \ j=1\ldots J \ ,  \ \ \cup A_i= X \ .  \ee

Similarly, a {\em weak multi partitions} stands for  $N$ non-negative  measures $\mu_1, \ldots \mu_N$ verifying \be\label{summui=mubar}\sum_1^N\mu_i= \mu:=\sum_{j=1}^J\mu^{(j)} \ . \ee The {\em induced weak  partition}   $\bar{\mu}_i:= (\mu^{(1)}_i, \ldots \mu^{(J)}_i)$, $i=1, \ldots N$ is defined by
$$ \mu_i^{(j)}(dx):= \frac{d\mu_i}{d\mu}(x)\mu^{(j)}(dx) \ \text{such that} \ .  $$
Such a weak partition \index{weak partition}  is assumed to satisfy
\be\label{multiweakconst}  \mu_i^{(j)}(X)= m_i^{(j)}  \ , \ \ i=1\ldots N, \ j=1\ldots J \ .  \ee
An optimal multi partition $\vmu^0:=(\mu^0_1, \ldots \mu^0_N)$ is a natural generalization of (\ref{weakconst}): It is  the one which  maximizes
\begin{tcolorbox}
	\be\label{multisumitoN*}	
	\theta(\bar\mu; \cM):=\sum_{i=1}^N\int_X\theta_i(x)d\mu^0_i\equiv \max_{\vmu}	\sum_{i=1}^N\int_X\theta_i(x)d\mu_i\ee
	among all  weak partitions\index{weak partition}$\vmu=(\mu_1, \ldots \mu_N)$ verifying 	 (\ref{multiweakconst}) for the assigned $\bar\mu$, $\cM:=\{m_i^{(j)}\}$.
\end{tcolorbox}
At the first step, we should ask ourselves if such a weak multi partition exists at all. By (\ref{summui=mubar}) we can see that a necessary condition for this is the component-wise balance 
$\sum_{i=1}^N m_i^{(j)}=\mu^{(j)}(X)$ for $1\leq j\leq J$ .
In general, however, this is not
a sufficient condition.
If  a weak partition \index{weak partition}verifying (\ref{multiweakconst}) exists for  a pair $\bar{\mu}, \cM$, we say that $\bar\mu$ {\em dominates} $\cM$ and denote it by $\bar\mu\succ \cM$.
The set of all $N\times J$ matrices $\cM$ satisfying $\bar\mu\succ\cM$ is denoted by $\vmS_N(\bar\mu)$. We denote $\mu\succ_N\nu$ if $\Delta_N(\bmu)\supset\Delta_N(\bnu)$. The connection with (\ref{succ}) is: 
\begin{tcolorbox} 	 
	{\bf Theorem}: \ \ $\bmu\succ\bar\nu$ if and only if $\bmu\succ_N\bar\nu$  for any $N\in\N$. 	
\end{tcolorbox}     

The feasibility condition for (\ref{multisumitoN*}), namely the condition $\vmS_N(\bar\mu)\not=\emptyset$ and the characterization of $\vmS_N(\bar\mu)$ in general
is addressed in chapter \ref{Wps}.
The function
$$ \Xi^0(\cP;\bmu):= \int_X\max_{1\leq i\leq N}\vpp_i\cdot \frac{d\bar{\mu}}{d\mu} d\mu \ $$
plays a central  rule.  Here  $\vpp\in\R^J$ and $\cP:=(\vpp_1, \ldots \vpp_N)\in\R^{N\times J}$ is $N\times J$ matrix.

The main result of this chapter is the following:

\begin{tcolorbox}
	The set $\vmS_N(\bar\mu)$ is a closed and convex.
	
	$\cM\in \vmS_N(\bar\mu)$  (i.e. $\bar{\mu}\succ\cM$) if and only if one of the following equivalent conditions holds:
	\begin{itemize}
		\item
		$$ \Xi^0(\cP;\bmu)- \cP:\cM\geq 0$$\index{ $\Xi^0$}
		where $\cM:=(\bm_1, \ldots \bm_N)$,  $\bm_i:= (m_i^{(1)}, \ldots m_i^{(J)})$ and $\cP:\cM:=\sum_{i=1}^N \vpp_i\cdot \bm_i$. 
		\item For any convex  function $f:\R^N\rightarrow \R$
		$$ \int_X f\left(\frac{d\bar{\mu}}{d\mu}\right) d\mu \geq \sum_{i=1}^N|\bar{m}_i| f\left(\frac{\bar{m}_i}{|\bar{m}_i|}\right) \  $$
		where  $|\bar{m}_i|=\sum_{j=1}^J m_i^{(j)}$.
	\end{itemize}
\end{tcolorbox}

The existence of {\em strong}  partitions
verifying (\ref{multistrongconst}) 
is discussed in Chapter \ref{optimstrong}. In particular we obtain
\begin{tcolorbox}
	If $\bar\mu \succ \cM$ and
	\be\label{pdmudmu=0} \mu\left( x\in X; \vpp\cdot d\bmu/d\mu=0\right)=0 \ \text{for any} \  \vpp\in\R^J, \  \vpp\not= 0 \ , \ee
	then there exists a strong partition $A_1, \ldots A_N$ verifying (\ref{multistrongconst}) corresponding to $\cM$.
	
	Moreover, if there exists $\cP_0:=(\vpp_1^0, \ldots \vpp_N^0)\not=0$ which verifies $$\Xi_0(\cP_0;\bmu)- \cP_0:\cM= 0$$  and  satisfies $\vpp_i^0\not= \vpp_j^0$ for $i\not= j$ then the strong partition is unique.  \index{strong  (deterministic) partition}
	
	More generally, if $\I$ is decomposed into $k$ disjoint subsets $\I_1, \ldots \I_k$ and $\vpp_i^0\not=\vpp_j^0$ if $i\in\I_m$, $j\in \I_n$ and $n\not=m$ then there exists a unique $k-$partition $\A_1, \ldots \A_k$ of $X$ such that any partition  verifying (\ref{multistrongconst}) corresponding to $\cM$
	satisfies
	$$ A_i\subset \A_m \ \ \text{if and only if} \ \ i\in\I_m \ . $$
\end{tcolorbox}

In Chapter \ref{chomp} we consider the optimization problem for multi partitions.  The function
$$ \Xi_\theta(\cP;\bmu):= \int_X\max_{1\leq i\leq N}\left[\theta_i(x)+p_i\cdot \frac{d\bar{\mu}}{d\mu} \right]_+d\mu \ $$
plays a central rule for the optimization.
One of the main results of this chapter are
\begin{tcolorbox}
	If $\bar\mu\succ\cM$  then the optimal transport (\ref{multisumitoN*}) is given by
	\be\label{thetamyP}\theta(\bar\mu, \cM)=\inf_{\cP} \Xi_\theta(\cP;\bmu)- \cP:\cM \ . \ee
 Moreover, if (\ref{pdmudmu=0}) holds then 
	there is a {\em strong partition} which verifies (\ref{multisumitoN*}). 
\end{tcolorbox}



\section{Cooperative and non-cooperative partitions}\index{cooperative game} 
In Part \ref{part3} we return to the scalar transport case $J=1$ and discuss partitions under both cooperation and competition of the agents. Taking advantage on the uniqueness result for partition obtained in Chapter \ref{scunique} we define, in Chapter \ref{sectionBtM}, the {\em individual value} \index{individual (surplus) value} $V_i$ of an agent $i$ as the surplus value  \index{individual (surplus) value} she creates for her customers:
$$ V_i:=\int_{A_i} \theta_i(x)d\mu$$
where $A_i$ is the set of the customers of $i$ under the  optimality condition. We address the following question:
\begin{tcolorbox}
	What is the effect of increase of the utility $\theta_i(x)$ of agent $i$ on its individual value $V_i$, {\em assuming}  the utilities of the other agents, as well as the capacities  $m_k=\mu(A_k)$ are preserved for all agents?
\end{tcolorbox}
The answer to this question is somewhat surprising. It turns out that the individual value may {\em decrease} in that case. In Theorem \ref{new0}-\ref{new1} we establish sharp quantitative estimates of the change of the individual value.

In Chapter \ref{chinprofit} we deal with different possibilities of sharing the individual value $V_i$ \index{individual (surplus) value} produced by the agent $i$ with her customers $A_i$. The most natural strategy is "flat price", where the agent $i$ charge a constant  price $p_i$ from all her customers, so her profit is  $P_i:= p_i\mu(A_i)$. Since $\mu(A_i)$ is determined by the prices $p_1, \ldots p_N$ imposed by all other agents, we obtain a competitive game
where each agent wishes to maximize her profit. This leads us naturally  to  the notion of Nash equilibrium. We also discuss other strategies, such as commission, \index{commission} where the agent $i$ charges a certain portion $q_i\theta_i$   where $q_i\in (0,1)$,  hence  $P_i= q_i V_i$.

Motivated by these results we ask the natural question regarding cooperation of agents: Suppose a subgroup of agents $\J\subset \I:=\{1, \ldots N\}$ decide to form a coalition (cartel), such that the utility of this coalition is the maximum of utilities of its agents: $\theta_\J(x):=\max_{j\in \J}\theta_j(x)$, and the capacity is the sum of the capacities $m_\J:= \sum_{i\in \J}m_i$. The stability of the  {\em grand coalition} \index{grand coalition} $\theta_\I=\max_{j\in \I}\theta_j$ and $m_\I=\mu(X)$ is addressed in Chapter \ref{Cooperative partitions}. This leads us to discuss {\em cooperative games} \index{cooperative game}  for transferable utilities. In some special cases we establish the stability of the grand coalition $\J=\I$.
\part{Stable marriage and optimal partitions}\label{partI} \index{optimal partition}
\chapter{The stable marriage problem}\label{ch1}

	
 \small{{ \it Obviously, marriage is not a synonym for morality. But stable marriages and families do encourage moral behavior} (Gary Bauer)}
\section{Marriage without sharing}\label{smar}
Consider two sets of $N$ elements each: A set of men ($\I_m$)   and  women ($\I_w$).
Each  man in $i\in\I_m$ lists the women  according to his own preference: For any $j_1, j_2\in\I_w$
\be\label{order1}   j_1 \succ_i j_2 \ \ \text{iff} \ i \ \text{prefers} \ \ j_1 \ \ \text{over} \ \ j_2 \ . \ee
Likewise, each woman  $j\in\I_w$ lists the men in $\I_m$ according to her preference: For any $i_1, i_2\in\I_m$
\be\label{order2} i_1 \succ_j i_2 \ \ \text{iff} \ j \ \text{prefers} \ \ i_1 \ \ \text{over} \ \ i_2 \ . \ee
Here $\succ_i, \succ_j$ are complete order relations, namely:\\
\begin{enumerate}
	\item Any $i\in\I_m$ and $j_1\not=j_2\in\I_w$ either $j_1\succ_i j_2$ or $j_2\succ_i j_1$ (but not both),
	\item $j_1\succ_i j_2$, $j_2\succ_i j_3$ implies $j_1\succ_i j_3$ for any distinct triple $j_1,j_2,j_3\in\I_w$.
	\item Likewise for $\succ_j$ where $j\in\I_w$.
\end{enumerate}

A matching $\tau$ is a bijection $i\leftrightarrow \tau(i)$:  Any  man $i\in\I_m$ marries  a single woman $\tau(i)\in\I_w$, and any woman $j\in\I_w$ is married to a single man $\tau^{-1}(j)\in\I_m$.
\par
A blocking pair \index{blocking pairs}  $(i,j)\in\I_m\times \I_w$ is defined as follows:
\begin{itemize}
	\item $j$ and $i$ are not married ($j\not=\tau(i)$).
	\item   $i$ prefers $j$ over his mate $\tau(i)$: \ \ $j\succ_i\tau(i)$
	\item $j$ prefer $i$ over her  mate $\tau^{-1}(j)$: \ \ $i\succ_j\tau^{-1}(j)$ \ .
\end{itemize}
\begin{defi}\label{NTstabmer}
	A marriage $\tau$ is called  stable  if and only if there are no blocking pairs.
\end{defi}
This is a very natural (although somewhat conservative) definition of stability, as the existence of a blocking pair will break two married couples and may disturb the happiness of the rest.

\par
The question of existence of a stable marriage is not trivial. It follows from a celebrated, constructive algorithm due to Gale and Shapley \cite{GSH}, which we describe below:
\par
\subsection{Gale-shapley algorithm}\label{GSalgomer}\index{Gale-Shapley algorithm} 
{\small {\it Freedom's just another word for nothin' left to lose}, Jenis Joplin} 
\begin{enumerate}
	\item At the first stage, each man $i\in\I_m$ proposes to the woman $j\in\I_w$ at the top of his list.
	At the end of this stage, some women got proposals (possibly more than one), other women may not get any proposal.
	\item At the second stage, each woman who got more than one proposal, bind the man whose proposal is most preferable according to her list (who is now engaged). She releases all the other men who proposed. At the end of this stage, the men's set $\I_m$ is composed of two parts: engaged and released.
	\item At the next stage each {\em released} man makes a proposal to the {\em next} woman in his preference list (whenever she is engaged or not).
	\item Back to stage 2.
\end{enumerate}
It is easy to verify that this process must end at a finite number of steps. At the end of this process all women and men are engaged. This is a stable matching!
\par
Of course, we could reverse the role of men and women in this algorithm. In both cases we get a stable matching. The algorithm we indicated is the one which is {\em best} from the men's point of view. Of course, the reversed case is best for the women. In fact (see e.g.\cite{Mo,Irv})  
\begin{theorem}
	For any stable matching $\tau$ the rank of the woman $\tau(i)$ according to man $i$ is {\em at most} the rank of the woman matched to $i$ by the
	above, men proposing algorithm.
\end{theorem}
\begin{example}. \\
	
	\begin{tcolorbox}
		{\bf men preference}
		\begin{itemize}
\item 	$w_1 \succ_1 w_2\succ_1 w_3$
\item 	$w_3\succ_1 w_2\succ_1 w_1$
\item 	$w_1 \succ_1 w_2\succ_1 w_3$
\end{itemize}	
	{\bf women preference}
\begin{itemize}
	\item 	$m_2 \succ_1 m_3\succ_1 m_1$
	\item 	$m_1\succ_1 m_2\succ_1 m_3$
	\item 	$m_1 \succ_1 m_2\succ_1 m_3$
\end{itemize}
		
		Men propose: $\begin{pmatrix} 
		m_1  \\
		w_1 
		\end{pmatrix}
		\begin{pmatrix} 
		m_2  \\
		w_3
		\end{pmatrix}
		\begin{pmatrix} 
		m_3  \\
		w_1 
		\end{pmatrix}
		\Rightarrow
		\begin{pmatrix} 
		m_1  \\
		w_2 
		\end{pmatrix}
		\begin{pmatrix} 
		m_2  \\
		w_3
		\end{pmatrix}
		\begin{pmatrix} 
		m_3  \\
		w_1 
		\end{pmatrix}
		$
		\\
		Women propose: 
		$\begin{pmatrix} 
		m_2  \\
		w_1 
		\end{pmatrix}
		\begin{pmatrix} 
		m_1 \\
		w_2
		\end{pmatrix}
		\begin{pmatrix} 
		m_1  \\
		w_3 
		\end{pmatrix}
		\Rightarrow
		\begin{pmatrix} 
		m_2  \\
		w_1
		\end{pmatrix}
		\begin{pmatrix} 
		m_1  \\
		w_2
		\end{pmatrix}
		\begin{pmatrix} 
		m_2  \\
		w_3 
		\end{pmatrix}
		\\
		\Rightarrow
		\begin{pmatrix} 
		m_2  \\
		w_1
		\end{pmatrix}
		\begin{pmatrix} 
		m_1  \\
		w_2
		\end{pmatrix}
		\begin{pmatrix} 
		m_3  \\
		w_3 
		\end{pmatrix}
		$
		
	\end{tcolorbox}
\end{example}
In particular we obtain
\begin{theorem}\label{GaleShap}A stable matching always exists.
\end{theorem}


\section{Where money comes in...}\label{bribing}
Assume that  we can guarantee a  "cut"  $u_i$  for each married man $i$, and a cut $v_j$ for  each married woman $j$ (both in, say,  US dollars). In order to define a stable marriage we have to impose some conditions which will guarantee that no man or woman can increase his or her cut by marrying a different partner. For this let us define, for each pair $(i,j)$, a {\em bargaining  set} $F(i,j)\subset \R^2$ which contains all possible cuts $(u_i,v_j)$ for a matching of man $i$ with woman $j$.
\begin{assumption}\label{genass}  \ .
	\begin{description}
		\item{i)}
		For each $i\in \I_m$ and $j\in \I_w$, $F(i,j)$ are closed sets in $\R^2$. Let   $F_0(i,j)$ the interior of $F(i,j)$.
		\item{ii)} $F(i,j)$ is monotone in the following sense:
		If $(u ,v)\in F(i,j)$ then $(u^{'}, v^{'})\in F(i,j)$ whenever $u^{'}\leq u$ and $v^{'}\leq v$.
		\item{iii)} There exist $C_1, C_2\in\R$ such that $$ \left\{(u,v);  \max(u,v)\leq C_2   \right\}\subset F(i,j)\subset \left\{(u,v); u+v\leq C_1  \right\} \ $$
		for any $i\in\I_m, j\in \I_w$.
	\end{description}
\end{assumption}
\par
The meaning of the feasibility set is as follows:
\begin{tcolorbox}
	Any married couple  $(i,j)\in\I_m\times \I_w$ can guarantee the cut $u$ for $i$ and $v$ for $j$, provided $(u,v)\in F(i,j)$.
\end{tcolorbox}

\begin{defi}\label{deffea} \ .
	A matching $\tau:\I_m\rightarrow\I_w$ is stable iff there exists a vector $(u_1, \ldots u_N, v_1\ldots v_N)\in\R^{2N}$ such that $(u_i,v_j)\in \R^2-F_0(i,j)$ for any $(i,j)\in \I_m\times\I_w$, and $(u_i, v_{\tau(i)})\in F(i, \tau(i))$ for any $i\in\I_w$.
\end{defi}
We now demonstrate that  Definition \ref{deffea} is a generalization of  stable marriage in the  {\em non-transferable case}, as described in section \ref{smar} above. For this we quantify
the  preference list  introduced in (\ref{order1}, \ref{order2}). Assume that a man $i\in\I_m$ will gain the cut  $\theta_m(i,j)$ if he marries the woman $j\in\I_w$. So, the vector $\theta_m(i,1), \ldots \theta_m(i,N)$ is a numeration of (\ref{order1}). In particular, $j_1\succ_ij_2$ iff
$\theta_m(i,j_1)>\theta_m(i, j_2)$.
\par
Likewise, we associate a cut $\theta_w(i,j)$ for a woman $j\in\I_w$ marrying a man $i\in\I_m$, such that
$i_1\succ_ji_2$ iff
$\theta_m(i,j_1)>\theta_m(i, j_2)$.
\par
Define the feasibility sets
\be\label{Fblockp} F(i,j):= \{u\leq \theta_m(i,j) ; \ \ v\leq \theta_w(i,j) \}\  \ee
see Fig [\ref{bargsets}-a]. 
Suppose now $\tau$ is a stable matching according to Definition \ref{deffea}. Let $(u_1, \ldots u_N, v_1, \ldots v_N)$ as given in Definition \ref{deffea}.  We obtain that for any man $i$,  $(u_i, v_{\tau(i)})\in F(i, \tau(i))$ which, by  (\ref{Fblockp}) is equivalent to   $u_i\leq \theta_m(i, \tau(i))$ and $v_{\tau(i)}\leq \theta_w(i, \tau(i))$.
Likewise, for any woman $j$, $(u_{\tau^{-1}(j)}, v_j)\in F(\tau^{-1}(j),j)$ which, by  (\ref{Fblockp}) is equivalent to
$u_{\tau^{-1}(j)}\leq \theta_m(\tau^{-1}(j), j)$ and $v_j\leq \theta_w(\tau^{-1}(j), j)$.

If  $j\not=\tau(i)$ then, by definition again, $(u_i, v_j)\not\in F_0(i,j)$ which means, by  (\ref{Fblockp}), that either $u_i\geq \theta_m(i,j)$ and/or $v_j\geq \theta_w(i,j)$. Hence
either $\theta_m(i, \tau(i))\geq \theta_m(i,j)$ and/or $\theta_w(\tau^{-1}(j),j)\geq \theta_w(i,j)$.
According to our interpretation it means  that either man $i$ prefers woman $\tau(i)$ over $j$, or woman $j$ prefers man $\tau^{-1}(j)$ over  $i$. That is, $(i,j)$  is {\em not} a blocking pair.  \index{blocking pairs}
\begin{figure} 
	\begin{subfigmatrix}{4}
		\subfigure[Non-transferable]{\includegraphics[width=0.44\textwidth]{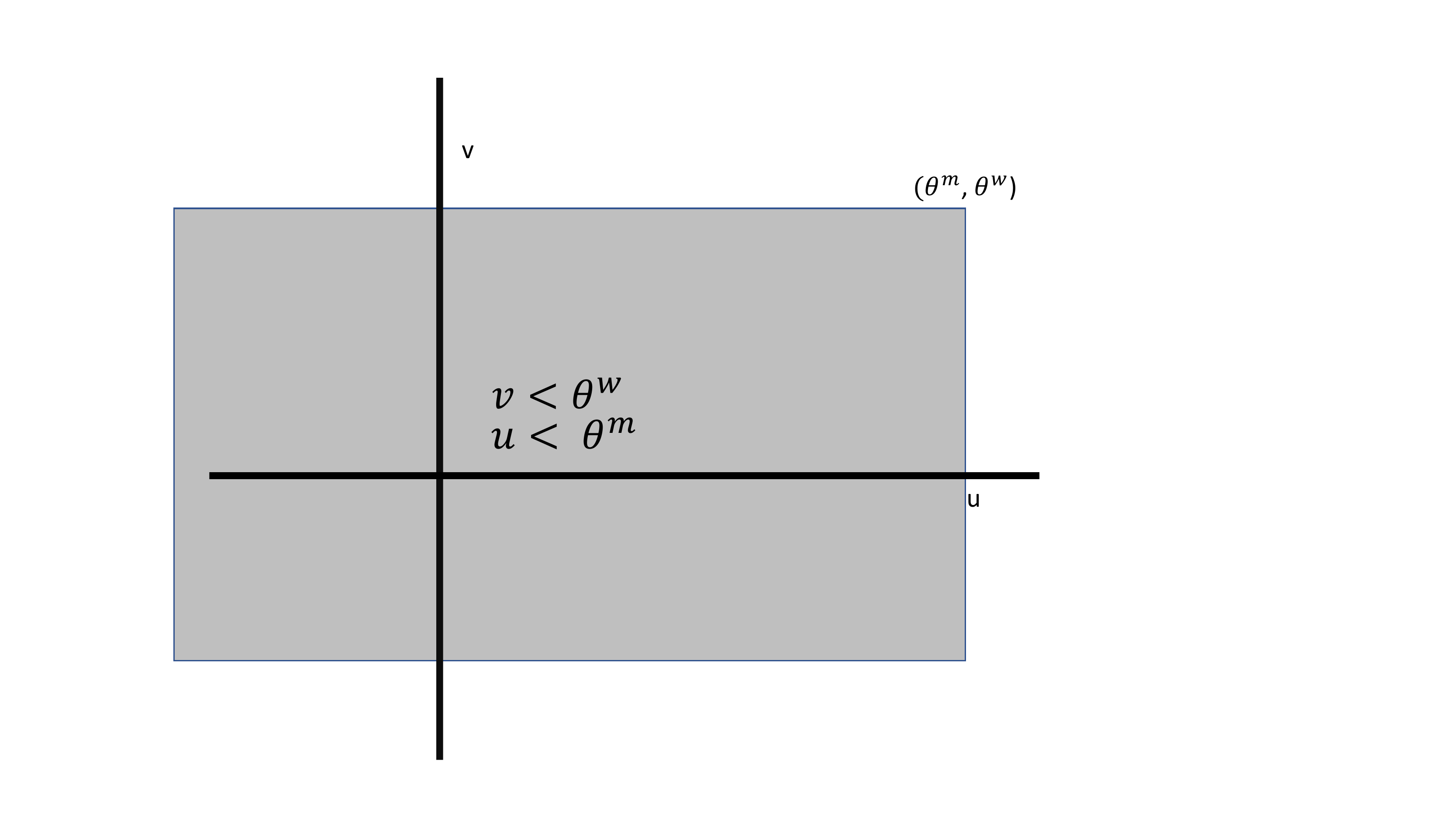}}
		\subfigure[Transferable]{\includegraphics[width=0.44\textwidth]{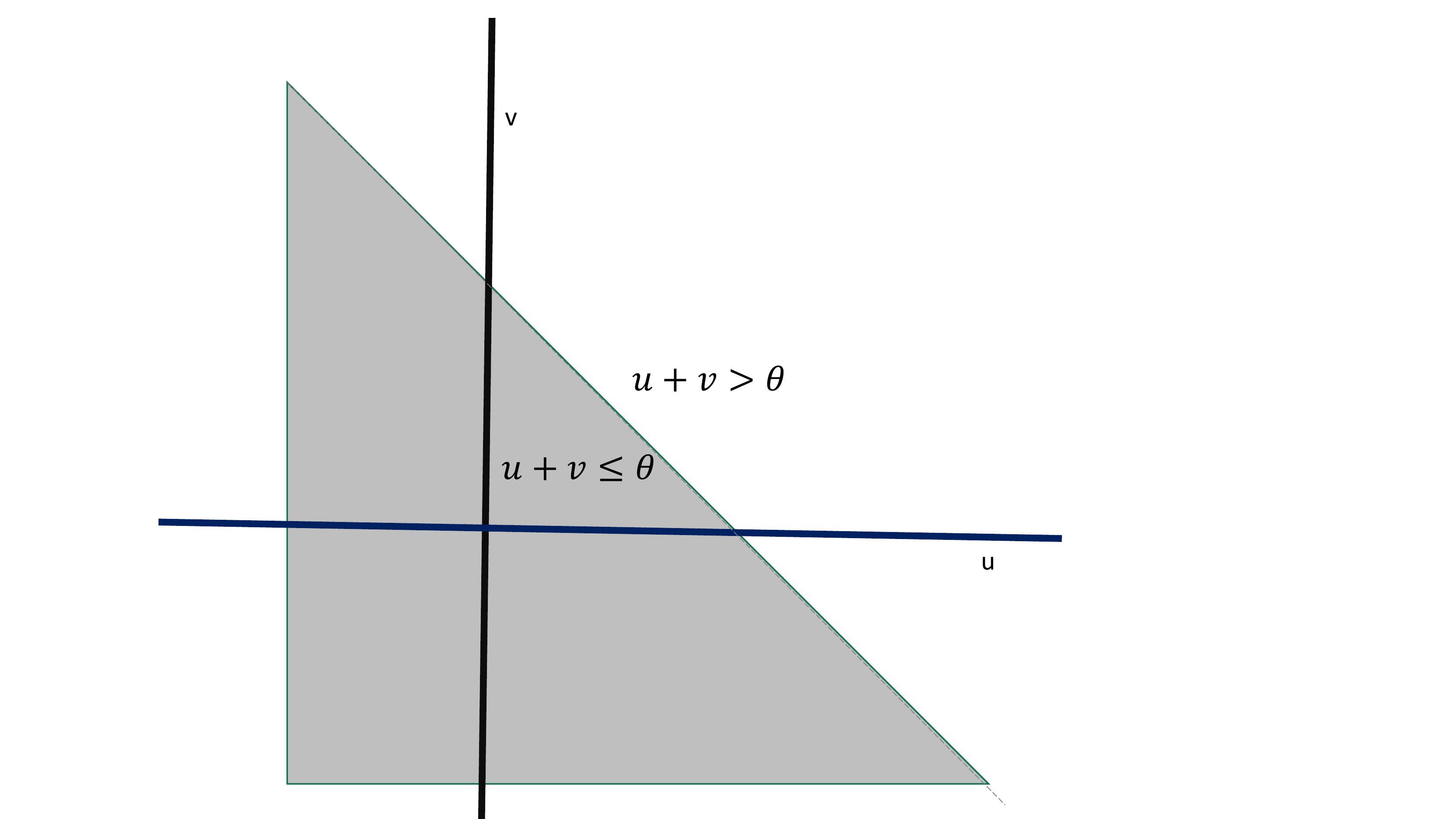}}
		
		\subfigure[Case 1]{\includegraphics[width=0.44\textwidth]{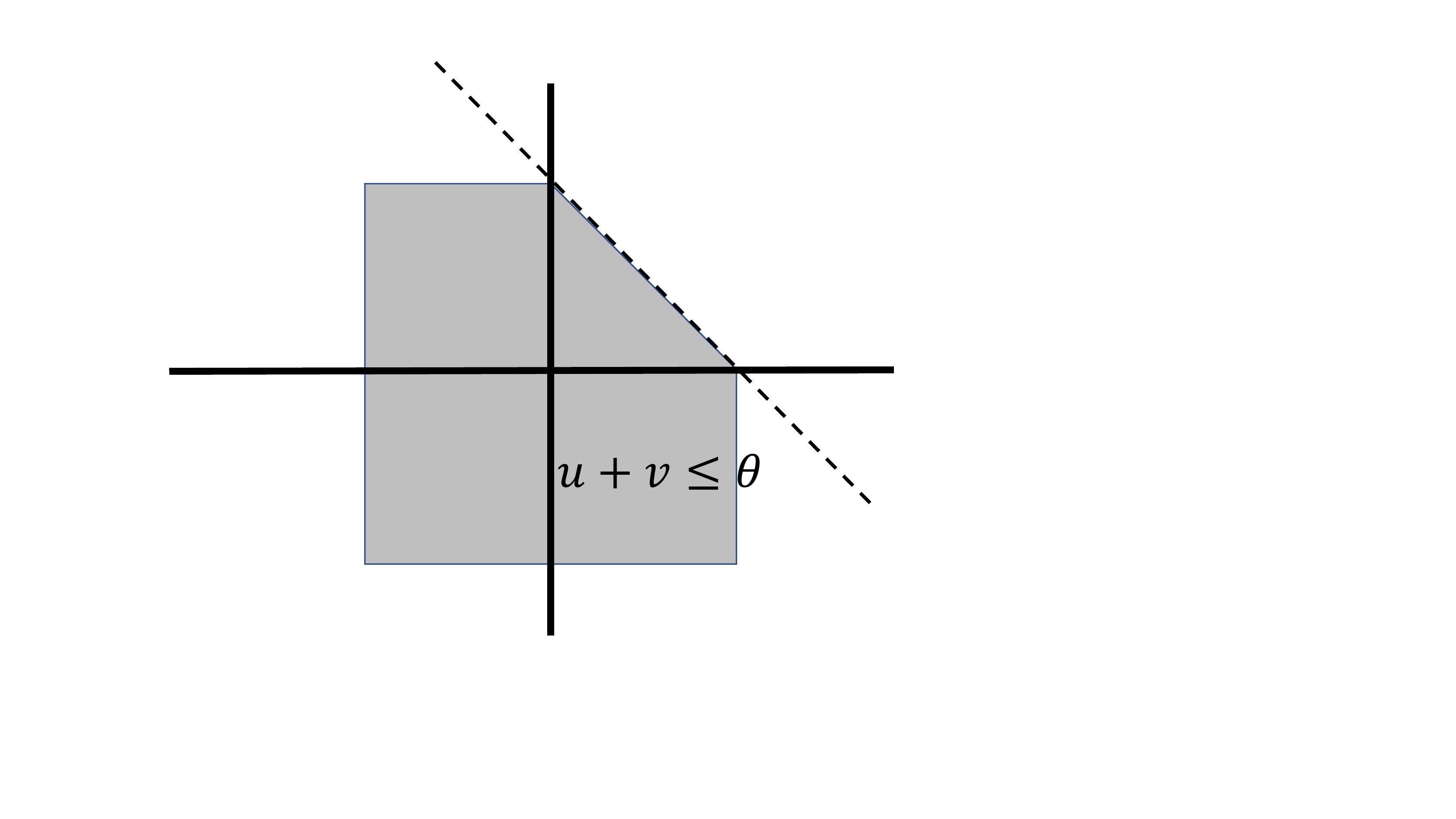}}
		\subfigure[Case 2]{\includegraphics[width=0.44\textwidth]{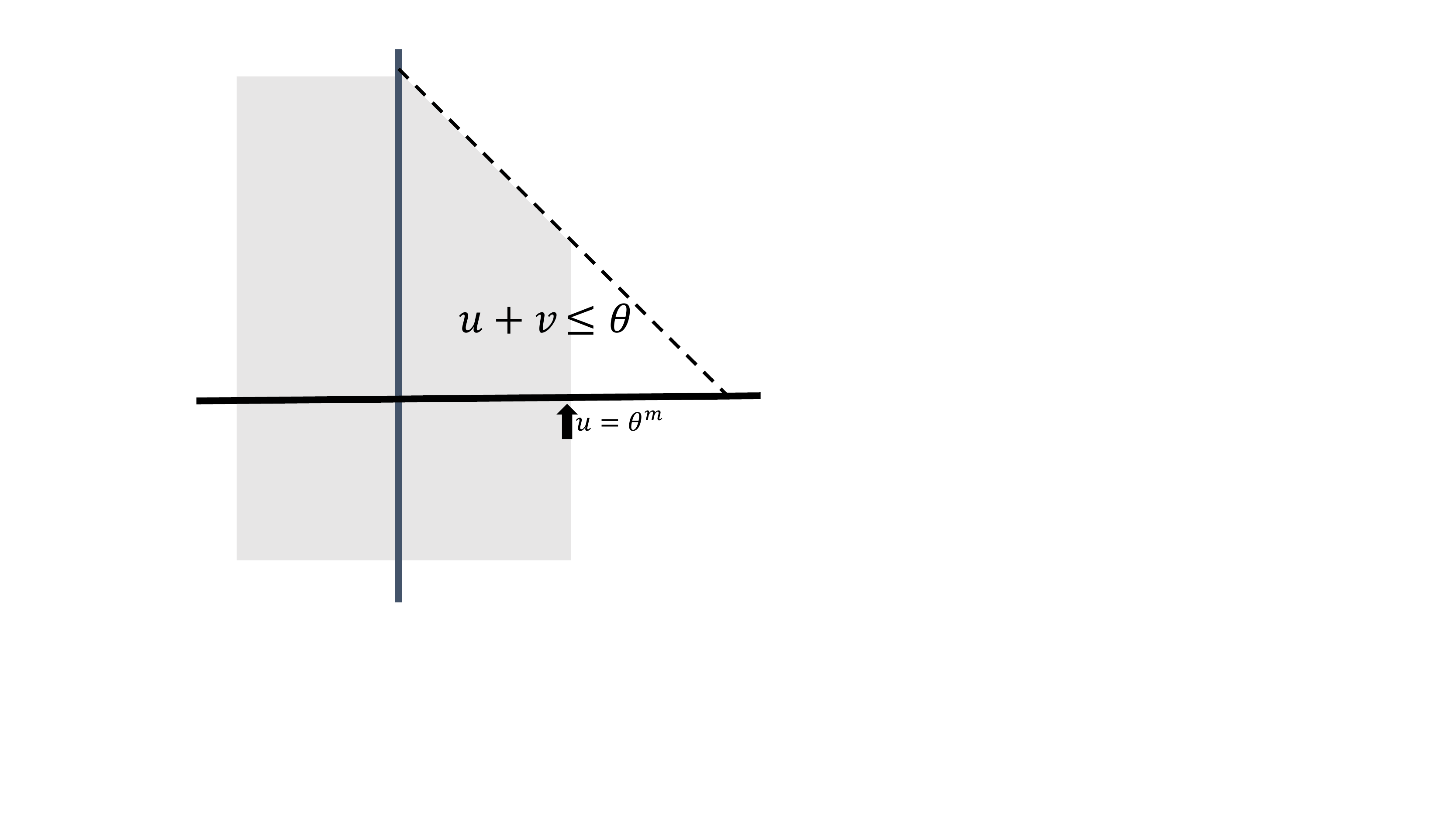}}
		\label{fig1}
	\end{subfigmatrix}
	\caption{pairwise bargaining  sets}\label{bargsets} 
\end{figure}
\section{Marriage  under sharing}
In the case we allow sharing ({\em transferable utility}) we assume that each married couple may share their individual cuts. Thus, if $\theta_m(i,j), \theta_w(i,j)$ are  as defined in the non-transferable case above,  man $i$ can transfer  a sum $w$ to  woman $j$ (in order to prevent a gender  bias we assume that $w$ can  be negative as well). Thus, the man's cut from this marriage is $\theta_m(i,j)-w$, while the woman's cut is $\theta_w(i,j)+w$. Since we do not prescribe $w$, the feasibility set for a pair $(i,j)$ takes the form
\be\label{stfultrans} F(i,j):= \{(u,v): u+v\leq \theta(i,j)\}\ee
where \
$$\theta(i,j):= \theta_m(i,j) + \theta_w(i,j)  \ , $$
c.f. Fig \ref{bargsets}-b. 
The definition of a stable marriage in the transferable case is implied from Definition \ref{deffea} in this special case:
\begin{defi}\label{transmer}
	A matching $\tau$ is stable iff there exists $(u_1, \ldots u_N, v_1, \ldots v_N)\in\R^{2N}$ such that
	$u_i+v_{\tau(i)}=\theta(i, \tau(i))$ for any $i$, and $u_i+v_j\geq \theta(i,j)$ for any $i,j$.
\end{defi}
It turns out that there are several equivalent definitions of stable marriages in the sense of Definition \ref{transmer}. Here we introduces three of these 
\begin{theorem}\label{matchingtrans}  $\tau$ is a stable marriage   in the sense of Definition \ref{transmer} iff one of the following equivalent conditions  is satisfied:
	\begin{description}
		\item{i)} Optimality: There exists $(u^0_1, \ldots v^0_N)\in\R^{2N}$ satisfying $$u^0_i+v^0_{\tau(i)}=\theta(i, \tau(i))$$ for any $i\in\I_m$ which minimizes $\sum_{i\in\I} (u_i+v_i)$ over the set
		$$W:=\{(u_1, \ldots v_N)\in\R^{2N}; \ u_i+v_j\geq \theta(i,j) \ \forall (i,j)\in\I_m\times\I_w\} \ . $$
		\item{ii)} Efficiency: \index{efficiency} $\tau$ maximizes $\sum_{i=1}^N \theta(i, \sigma(i))$ on the set of all matchings $\sigma:\I_m\rightarrow \I_w$.
		\item{iii)} Cyclic monotonicity: For any chain $i_1, \ldots i_k\in \{1, \ldots N\}$, the inequality
		\be\label{cyctau}\sum_{j=1}^k \left(\theta(i_j, \tau(i_j))-\theta(i_{j}, \tau(i_{j+1})\right)\geq 0\ee
		holds, where $i_{k+1}=i_1$.
	\end{description}
\end{theorem}
In particular
\begin{cor}\label{exctm}A stable matching according to Definition \ref{transmer} always exists.
\end{cor}
For the proof of Theorem \ref{matchingtrans} see section \ref{dmonge}. In fact, the reader may, at this point, skip  sections \ref{gencase}-\ref{fakeP}  and chapter \ref{MtF} 
as the rest of the book is independent of these. 

\section{General case}\label{gencase}
In the general case of Assumption \ref{genass}, the existence of a stable matching follows from the following Theorem:
\begin{theorem}\label{cornoemp}
	Let $V\subset \R^{2N}$ defined as follows:
	$(u_1, \ldots u_N, v_1, \ldots v_N)\in V$ $$ \Leftrightarrow \exists\ \ \text{an injection} \  \tau:\I_m\rightarrow\I_w \ \ \text{such that} \ \ (u_i, v_{\tau(i)})\in F(i, \tau(i)) \ \ \forall \ i\in \I_m  \ . $$
	Then there exists $(u_1, \ldots u_N, v_1, \ldots v_N)\in V$  such that
	\be\label{mercore} (u_i,v_j)\in \R^2- F_0(i,j) \  \ee
	for any $(i,j)\in\I_m\times \I_w$.
\end{theorem}
The set of vectors in $V$ satisfying (\ref{mercore}) is called {\em the core}. Definition \ref{deffea} can now be recognized as the non-emptiness of the core, which is equivalent to the existence of a stable matching.\index{core}
\par
Theorem \ref{cornoemp} is, in fact, a special case the celebrated Theorem of Scarf   \cite{Sc}  for  cooperative games\index{cooperative game} , tailored to the marriage scenario.  As we saw, it can be applied to the fully non-transferable case (\ref{Fblockp}), as well as to the fully transferable case (\ref{stfultrans}).
\par
There are other, sensible models of  {\em partial transfers} which fit into the formalism of Definition \ref{deffea} and Theorem \ref{cornoemp}.  Let us consider several examples:
\begin{enumerate}
	\item {\em Transferable marriages restricted to non-negative cuts} : In the transferable case the feasibility sets may contain negative cuts for the man $u$ or for the woman $v$ (even though not for both, if it is assumed $\theta(i,j)>0$). To avoid the undesired stable marriages were one of the partners get a negative cut we may replace the feasibility set(\ref{stfultrans}) by
	$$ F(i,j):=\{(u,v)\in\R^2; u+v\leq \theta(i,j) \ ,  u\leq \theta(i,j) \ , v\leq \theta(i,j)\} \ , $$
	see Fig [\ref{bargsets}-c].
	It can be easily verified that if $(u_1, \ldots v_N)\in V$ contains negative components, then $([u_1]_+, \ldots [v_N]_+)$, obtained by replacing the negative components by $0$, is in $V$ as well. Thus, the core of this game contains vectors in $V$ of non-negative elements.\index{core}
	\item In the transferable case (\ref{stfultrans}) we allowed both men and women to transfer money to their partner. Indeed, we assumed that the man's $i$ cut is $\theta_m(i,j)-w$ and the woman's $j$ cut is  $\theta_w(i,j)+w$, where $w\in\R$. Suppose we wish to allow only transfer between men to women, so we insists on $w\geq0$.\footnote{Of course we could make the opposite assumption $w\leq 0$. We leave the reader to change this example according to his view on political correctness...}
	In that case we choose (Fig \ref{bargsets}-d)
	$$ F(i,j):=\{ (u,v)\in\R^2; \ u+v\leq \theta(i,j); \ \ u\leq \theta_m(i,j) \}  \ . $$
	\item
	Let us assume that the transfer $w$ from man $i$ to woman $j$ is taxed, and the tax depends on $i,j$. Thus, if man $i$ transfers $w>0$ to a woman $j$  he reduces his cut by $w$, but the woman cut is increased by an amount $\beta_{i,j}w$, were $\beta_{i,j}\in [0,1]$. Here  $1-\beta_{i,j}$ is the tax implied for this transfer. It follows that
	$$ u_i\leq \theta_m(i,j)-w \ \ \ ; \ \ \ v_j\leq \theta_w(i,j)+\beta_{i,j}w \ , \ \ w\geq 0 $$
	Hence
	$$ F(i,j):=\{(u,v)\in\R^2; \ u_i + \beta^{-1}_{i,j}v_j\leq  \theta_\beta(i,j) , \ \ u_i\leq \theta_m(i,j)\} \ ,  $$
	where $\theta_\beta(i,j):= \theta_m(i,j) + \beta^{-1}_{i,j}\theta_w(i,j)$. This is demonstrated by  Fig \ref{bargsets}-d where the dashed line is tilted. 
\end{enumerate}


\section{Stability by fake promises}\label{fakeP}

We now describe a different notion of stability. Suppose a man can make a promise to a married woman (which is not his wife), and v.v. The principle behind it is that each of them does not intend to honor his/her own promise, but, nevertheless, believes that the other parti will honor her/his promise.
It is also based on some collaboration between the set of betraying couples.
\par
For simplicity of presentation we assume  that the matching $\tau$ is given by "the identity" $\tau(i)=i$, where $i\in\I_m$ represent a man, and $i=\tau(i)\in\I_w$ represents that matched woman. Evidently, we can always assume this by ordering  the list of men (or women) in a different way.

Let us repeat the definition of stability in the context of non-transferable matching (Definition \ref{NTstabmer}).
For this, we recall the definition of a blocking pair $(i,j)$:  \index{blocking pairs} \index{blocking pairs}
$$\theta_m(i, j)>\theta_m(i, i) \ \ \text{and} \ \ \theta_w(i,j)>\theta_w(j,j) \ , $$
which we rewrite as
\be\label{intro00} 
\Delta^{(0)}(i,j):=\min\{ \theta_m(i, j)-\theta_m(i, i), \theta_w(i,j)-\theta_w(j,j)\}>0 \ . \ee
\par
Assume that a man $i\in\I_m$ can offer  some bribe $b$   to any  other women $j$ he might  be  interested in (except his own wife, so $j\not=i$).  His cut for marrying $j$ is now
$\theta_m(i,j)-b$. The cut of the woman $j$ should have been $\theta_w(i,j)+b$. However, the happy woman should pay some tax for accepting this bribe. Let $q\in[0,1]$ be  the fraction of the bribe she can get (after paying her tax). Her supposed cut for marrying $i$ is just $\theta_w(i,j)+qb$. Woman $j$ will believe and accept offer from man $i$ if two conditions are satisfied: the offer should be both
\begin{enumerate}
	\item {\em Competitive},   namely
	$\theta_w(i,j)+qb\geq \theta_w(j,j)$.
	\item {\em Trusted}, if woman $j$  believes that  man $i$ is motivated. This implies $\theta_m(i,j)-b\geq \theta_m(i,i)$.
\end{enumerate}
The two conditions above can be satisfied, and  the offer is {\em acceptable}, if
\be\label{profitman}  q(\theta_m(i,j)-\theta_m(i,i))+\theta_w(i,j)-\theta_w(j, j)>0  \  . \ee
Symmetrically, man $i$ will accept an offer from a woman $j\not=i$ if
\be\label{profitwoman}  q(\theta_w(i,j)-\theta_w(i,i))+\theta_m(i,j)-\theta_m(j, j)> 0  \  . \ee
Let us define the {\em utility} of the exchange $i\leftrightarrow j$:

\be\label{introq}  \Delta^{(q)}(i,j):=\min\left\{ \begin{array}{c}
	q(\theta_m(i,j)-\theta_m(i,i))+\theta_w(i,j)-\theta_w(j, j) \\
	q(\theta_w(i,j)-\theta_w(j, j))+\theta_m(i,j)-\theta_m(i,i)
\end{array}\right\} \ ,
\ee
so, a {\it blocking-$q$} pair $(i,j)$ is defined by the condition that the utility of exchange is positive for both parties:
\be\label{Deltaqpositive}\Delta^{(q)}(i,j)>0 \ . \ee

\par
Evidently, if $q=0$ there is no point of bribing, so a blocking pair corresponding to (\ref{Deltaqpositive}) is equivalent to condition (\ref{intro00}) for the  non-transferable case, as expected. For  the other extreme case ($q=1$) where the bribe is not penalized, the expected profit of both $i,j$ is the same, and equals
\be\label{Delta1} \Delta^{(1)}(i,j)= \theta_m(i,j)-\theta_m(i,i)+\theta_w(i,j)-\theta_w(j, j) \ . \ee

We now consider an additional parameter $p\in[0,1]$ and define the real valued function on $\R$:
\be\label{pospdef} x \mapsto [x]_p:= [x]_+ - p[x]_- \  \ee
Note that $[x]_p=x$  for any $p$ if $x\geq 0$, while $[x]_1=x$ for any real $x$.

\begin{defi}\label{defshapleyq} Let \ $0\leq p,q\leq 1$. The matching $\tau(i)= i$  is $(p,q)-$stable if for any $k\in\N$ and  $i_1, i_2, \ldots i_k\in \{1, \ldots N\}$
	$$ \sum_{l=1}^k\left[ \Delta^{(q)}(i_l, i_{l+1})\right]_p  \leq 0 \  \text{where} \ \ i_{k+1}=i_1 \  $$
	where $i_{k+1}:= i_1$.
\end{defi}
What does it mean? Within the chain of pairs exchange
\begin{quote}
	$(i_1,i_1)\rightarrow (i_1,i_2),  \ldots  (i_{k-1},i_{k-1})\rightarrow (i_{k-1}, i_k), (i_k,i_k)\rightarrow (i_k, i_1)$
\end{quote}
each of the  pair exchange  $(i_l,i_l)\rightarrow (i_l,i_{l+1})$ yields a utility   $\Delta^{(q)}(i_l, i_{l+1})$ for  the new pair.
The lucky new pairs in this chain of couples exchange are those  who  makes a positive utility.
The unfortunate new pairs are those whose utility is non-positive.
The lucky pairs, whose interest is to activate this chain, are ready to compensate the unfortunate ones by contributing some of their gained utility.
The chain will be activated (and the original marriages will break down) if the mutual contribution of the fortunate pairs is enough to cover
{\it at least} the $p-$ part of the mutually   loss of utility
of the unfortunate pairs. This is the condition
$$ \sum_{\Delta^{(q)}(i_l, i_{l+1})>0}  \Delta^{(q)}(i_l, i_{l+1}) + p\sum_{\Delta^{(q)}(i_l, i_{l+1})<0} \Delta^{(q)}(i_l, i_{l+1})
\equiv  \sum_{l=1}^k\left[ \Delta^{(q)}(i_l, i_{l+1})\right]_p > 0 \   \ . $$
\vskip .3in
\begin{center}\begin{Ovalbox} {\it
			Definition \ref{defshapleyq}   grantees that {\em no such chain is activated}.
}\end{Ovalbox}\end{center}

In order to  practice  this definition, lets look at the extreme cases:
\begin{itemize}
	\item
	$p=0, q=0$. In particular, there is   {\it no bribing}: A $(0,0)-$stable marriage is precisely the stability in the non-transferable case  introduced in Section \ref{smar}.
	\par\noindent
	\item $p=q=1$. Definition \ref{defshapleyq} implies stability if and only if
	\be\label{Delta1pq1} \sum_{l=1}^k\Delta^{(1)}(i_l, i_{l+1}) \leq 0  \ee
	for any $k-$chain and any $k\in\N$.
	Let $\theta(i,j):= \theta_m(i,j)+\theta_w(i,j)$.
	Then, (\ref{introq}) implies that (\ref{Delta1pq1}) is satisfied if and only if
	\be\label{cm}\sum_{l=1}^k \theta(i_l, i_{l+1})-\theta(i_l,i_l)  \leq 0 \  \text{where} \ \ i_{k+1}=i_1 \ , \ee
	(check it!). \par
\end{itemize}
By point (iii) of Theorem \ref{matchingtrans} and Corollary \ref{exctm} we obtain the (not really surprising) result
\begin{cor}
	A matching is (1,1) stable iff it is stable in the  completely transferable case (\ref{stfultrans}). In particular, there always exists a $(1,1)-$stable matching.
\end{cor}

The observation (\ref{intro00}) and the definition $[x]_0:= [x]_+$ imply, together with Theorem \ref{GaleShap},
\begin{cor}
	A matching is (0,0) stable iff it is stable in the  non-transferable case (\ref{Fblockp}). In particular, there always exists a $(0,0)-$stable matching.
\end{cor}
We now point out the following observation
\begin{theorem}\label{comparepq}
	If $\tau$ is $(p,q)-$stable, then $\tau$ is also $(p^{'}, q^{'})-$stable for $p^{'}\geq p$ and $q^{'}\leq q$.\par
\end{theorem}

The proof of this Theorem follows from the definitions (\ref{introq}, \ref{pospdef}) and the following
\begin{lemma}\label{lemonotone}
	For any , $i\not=j$ and  $1\geq q>q^{'}\geq 0$,
	$$(1+q)^{-1}\Delta^{(q)}(i,j)>  (1+q^{'})^{-1}\Delta^{(q^{'})}(i,j) .  $$
\end{lemma}
\begin{proof}
	For $a,b\in\R$ and $r\in[0,1]$ define
	$$\Delta_r(a,b):= \frac{1}{2}(a+b)-\frac{r}{2}|a-b| \ . $$
	Observe that $\Delta_1(a,b)\equiv \min(a,b)$. In addition, $r \mapsto \Delta_r(a,b)$ is monotone not increasing in $r$.
	A straightforward calculation yields
	$$ \min(qa+b, qb+a)
	= \Delta_1(qa+b, qb+a)
	=(q+1)\Delta_{\frac{1-q}{1+q}}(a,b) \ , $$
	and the Lemma follows from the above observation, upon inserting $a=\theta_m(i,j)-\theta_m(i,i)$ and $b=\theta_w(i,j)-\theta_w(j, j)$.
\end{proof}
What can be said about the existence of s $(p,q)-$ stable matching in the general case? Unfortunately, we can prove now only a negative result:
\begin{prop}\label{propimpos}
	For any $1\geq q>p\geq 0$, a stable marriage does not exist unconditionally.
\end{prop}
\begin{proof}
	We only need to present a counter-example. So, let $N=2$. To show that the matching $\tau(1)=1, \tau(2)=2$ is not stable we have to show
	\be\label{t1}\left[ \Delta^{(q)}(1,2)\right]_p+ \left[ \Delta^{(q)}(2,1)\right]_p>0\ee
	while, to show that $\tau(1)=2, \tau(2)=1$ is not stable we have to show
	\be\label{t2} \left[ \Delta^{(q)}(1,1)\right]_p+ \left[ \Delta^{(q)}(2,2)\right]_p>0 \ . \ee
	By definition (\ref{introq}) and Lemma \ref{lemonotone}
	$$ \Delta^{(q)}(1,2)= (q+1)\Delta_r\left( \theta_m(1,2)-\theta_m(1,1), \theta_w(1,2)-\theta_w(2,2)\right)$$
	$$ \Delta^{(q)}(2,1)= (q+1)\Delta_r\left( \theta_m(2,1)-\theta_m(2,2), \theta_w(2,1)-\theta_w(1,1)\right)$$
	where $r=\frac{1-q}{1+q}$. To obtain $\Delta^{(q)}(1,1), \Delta^{(q)}(2,2)$ we just have to exchange man $1$ with man $2$, so
	$$ \Delta^{(q)}(2,2)= (q+1)\Delta_r\left( \theta_m(2,2)-\theta_m(2,1), \theta_w(2,2)-\theta_w(1,2)\right)$$
	$$ \Delta^{(q)}(1,1)= (q+1)\Delta_r\left( \theta_m(1,1)-\theta_m(1,2), \theta_w(1,1)-\theta_w(2,1)\right) \ . $$
	All in all, we only have 4 parameters to play with:
	$$ a_1:= \theta_m(1,2)-\theta_m(1,1), \ \ a_2=\theta_w(1,2)-\theta_w(2,2) \ , $$
	$$ b_1=\theta_m(2,1)-\theta_m(2,2), \ b_2=\theta_w(2,1)-\theta_w(1,1) \ ,  $$
	so the two conditions to be verified are
	$$ [\Delta_r(a_1, a_2)]_p+ [\Delta_r(b_1, b_2)]_p>0 \ \ ; \ \ [\Delta_r(-a_1, -b_2)]_p+ [\Delta_r( -b_1, -a_2)]_p >0 \ . $$
	Let us insert $a_1=a_2:=a>0$. $b_1=b_2:=-b$ where $b>0$. So
	$$[\Delta_r(a_1,a_1)]_p=a, \ \ \ [\Delta_r(b_1, b_2)]_p=-pb \ , $$
	while $\Delta_r(-a_1, -b_2)=\Delta_r(-b_1, -a_2)=\frac{b-a}{2}-\frac{r}{2}(a+b)$. In particular, the condition
	$\frac{a}{b}< \frac{1-r}{1+r}$ implies $[\Delta_r(-a_1, -b_2)]_p=[\Delta_r( -b_1, -a_2)]_p>0$ which verifies (\ref{t2}). On the other hand, if $a-pb>0$ then (\ref{t1}) is verified. Both conditions can be verified if $\frac{1-r}{1+r}>p$. Recalling $q=\frac{1-r}{1+r}$ we obtain the result.
\end{proof}
Based on Theorem \ref{comparepq} and Proposition \ref{propimpos} we propose:
\begin{tcolorbox}
	{\bf Conjecture}:  \ {\em There always exists a stable $(p,q)-$ marriage iff $q\leq p$.}
\end{tcolorbox}
\begin{figure} 
	\centering
	\includegraphics[height=6.cm, width=10.cm]{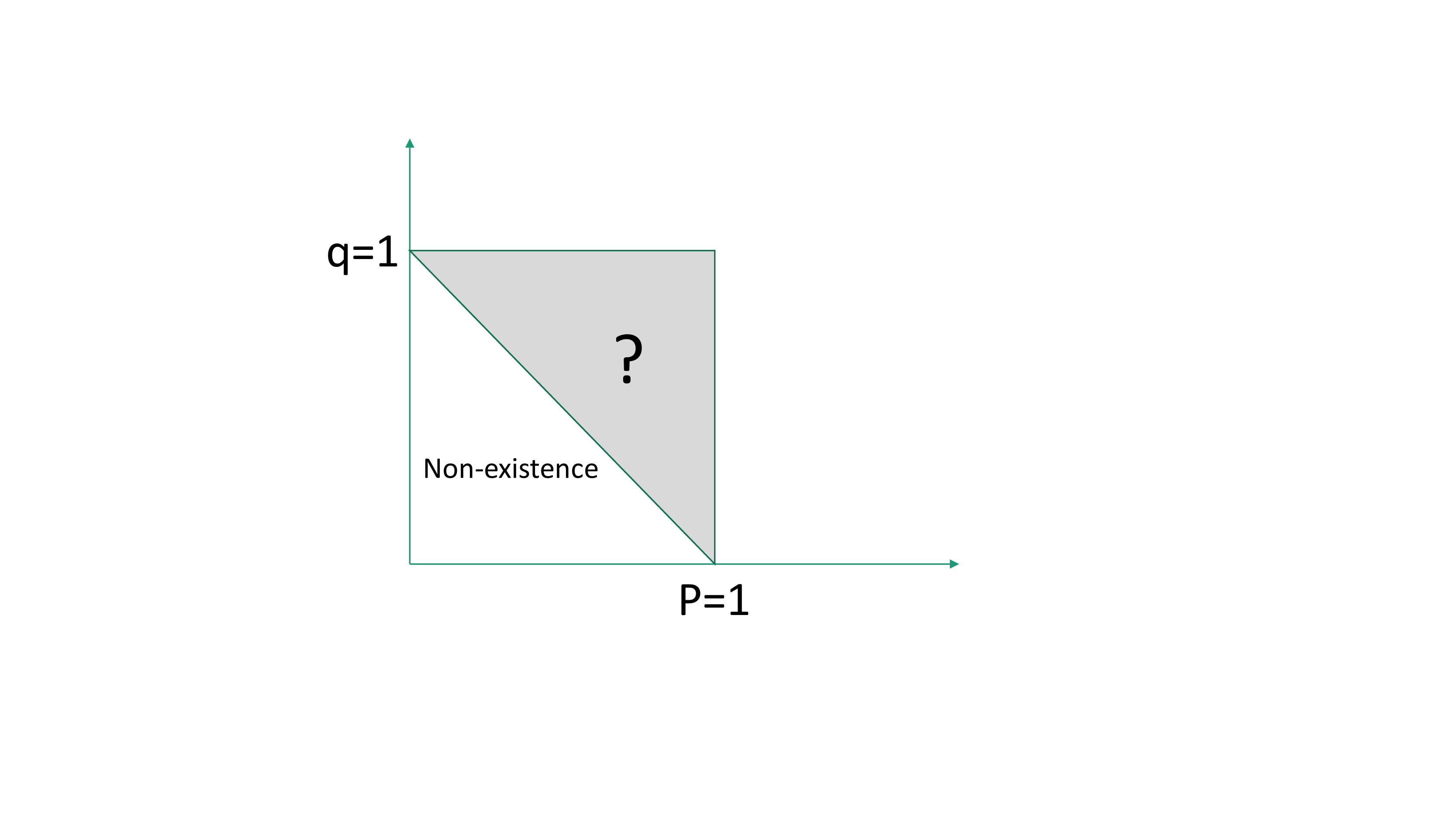}\\
	\caption{Conjecture: Is there an unconditional existence of stable marriages in the gray area? }
\end{figure}

\section{The discrete Monge problem}\label{dmonge}
In Theorem \ref{matchingtrans}
now encountered, for the first time, the Monge problem in its discrete setting:
\par
Let $\{\theta(i,j)\}$ be an $N\times N$ matrix of rewards. The reward of a given bijection $\tau: \I_m\leftrightarrow \I_w$ is defined as
\be\label{mongectau} \theta(\tau):= \sum_{i=1}^N \theta(i,\tau(i)) \ . \ee
\begin{defi}\label{defMonge}
	A bijection $\tau$ is a  Monge solution with respect to $\{c\}$  if it maximizes $\tau \mapsto \theta(\tau)$ among  all bijections.
\end{defi}
Theorem \ref{matchingtrans} claims, in particular,  that $\tau$ is a Monge solution iff it is  a stable marriage  with respect to transferable utility (\ref{stfultrans}). To show it we first establish the equivalence between  Monge solutions (ii) to  {\em cyclically monotone} matching, as defined in part (iii) of this Theorem.

Again we may assume, with no limitation of generality, that  $\tau(i)=i$ is a Monge solution, namely
$$ \sum_{i=1}^N \theta(i,i)\geq \sum_{i=1}^N \theta(i, \sigma(i))$$ for any other matching $\sigma$.
Given a $k-$chain $\{i_1, \ldots i_k\}$, consider the associated  cyclic permutation \index{permutation}
$\sigma(i_1)=i_2, \ldots \sigma(i_{k-1})=i_k, \sigma(i_{k})=i_1$.  Then $\theta(\sigma\circ\tau)\leq \theta(\tau)$ by definition. On the other hand, $\theta(\tau)-\theta(\sigma\circ\tau)$ is precisely the left side of (\ref{cyctau})
$$\sum_{j=1}^k \theta(i_j, i_j))-\theta(i_{j}, i_{j+1})\geq 0 \ . $$
In the opposite direction:
let
\be\label{pidef} -u^0_i:=\inf_{k-chains, k\in\N}\left( \sum_{l=1}^{k-1} \theta(i_l, i_l)-\theta(i_{l+1}, i_{l})\right)
+ \theta(i_k, i_k)-\theta(i, i_k) \ . \ee
Let $\alpha>-u^0_i$ and consider a $k-$chain realizing
\be\label{w1} \alpha > \left( \sum_{l=1}^{k-1} \theta(i_l, i_l)-\theta(i_{l+1}, i_l)\right)
+ \theta(i_k, i_k)-\theta(i, i_k) \ \ .  \ee
 
By cyclic monotonicity, $ \sum_{l=1}^{k} \theta(i_l, i_l)-\theta(i_{l+1}, i_l)\geq 0$. Since $i_{k+1}=i_1$,
$$ \sum_{l=1}^{k-1} \theta(i_l, i_l)-\theta(i_{l+1}, i_l)\geq   \theta(i_1,i_k)-\theta(i_k, i_k)  \ , $$
so (\ref{w1}) implies
$$ \alpha> \theta(i_1,i_k) - \theta(i, i_k)\geq 0 \ , $$
in particular $u^0_i<\infty$.

Hence, for any $j\in\I_m$
\begin{multline} \alpha+\theta(i,i)-\theta(j,i) > \left( \sum_{l=1}^{k-1} \theta(i_l, i_l)-\theta(i_{l+1}, i_l)\right) \\
+ \theta(i_k, i_k)-\theta(i,i_k)+\theta(i,i)-\theta(j,i) \geq -u^0_j\  \end{multline}
where the last inequality follows by the substitution of  the $k+1-$cycle $i, i_1\ldots   i_k$  (where$i_{k+1}=i$) in  (\ref{pidef}).
Since $\alpha$ is any number bigger than $-u^0_i$ it follows
\be\label{ui-cii}-u^0_i+\theta(i,i)-\theta(j,i) \geq -u^0_j\ ,\ee

To prove that the Monge solution is stable, we define
$v^0_j:= \theta(j,j)-u^0_j$ so
\be\label{ui+vi} u^0_j+v^0_j=\theta(j,j) \ . \ee
Then (\ref{ui-cii}) implies (after interchanging $i$ and $j$)
\be\label{ui+vj} u^0_i+v^0_j= u^0_i+\theta(j,j)-u^0_j\geq u^0_i-u^0_i+\theta(i,j)=\theta(i,j) \  \ee
for any $i,j$. Thus, (\ref{ui+vi},\ref{ui+vj})  establish that $\tau(i)=i$ is a stable marriage via Definition \ref{transmer}.
\par
Finally, to establish the equivalence of the optimality condition (i) in Theorem \ref{matchingtrans} to condition (ii)  (Monge solution), we note that for any $(u_1, \ldots v_N)\in W$,
$\sum_{i=1}^N u_i+v_i\geq \sum_{i=1}^N \theta(i,i)$, while $(u^0_1, \ldots v^0_N)$ calculated above is in $W$ and satisfy the equality.
\chapter{Many to few: Stable partitions}\label{MtF}
{\small{\it  The employer generally gets the employees he deserves} (J. Paul Getty)}
\vskip .3in
\section{A non-transferable Partition problem}\label{MtF1}
We now abandon  the gender approach of chapter \ref{ch1}. Instead of the  men-women groups $\I_m, \I_w$, let us  consider  a set $\I$ of  $I\in \mathbb{N}$ agents (firms) and set of consumers (employees)  $X$. We {\it do not} assume, as in Chapter \ref{ch1}, that the two sets are of equal cardinality. In fact, we take the cardinality of $X$ to be  much larger than that of $\I$. It can also be (and in general is) an infinite set.
\vskip .2in\noindent
Let us start from the ordinal viewpoint: 
We  equip $X$ with a sigma-algebra $\B\subset 2^X$  \index{Borel $\sigma$ algebra} such that $X\in \B$ as well as, for any $x\in X$,  $\{x\}\in \B$,   and an atomless,  positive measure  $\mu$:
\be\label{sigmaalgebra}(X,\B,\mu) \ \ ; \ \ \mu:\B\rightarrow \R_+\cup\{0\} \ \   \ . \ee

\par\noindent
In addition, we consider the structure of preference list generalizing  (\ref{order1}, \ref{order2}):
Each firm $i\in\I$ orders the potential employees $X$ according to a strict preference list. Let
$\succ_i$ be a strict, measurable order relation on $X$. That is,
\begin{defi}\label{qwe} .
	\begin{description}
		\item {i)} non-symmetric: For any $x\not=y$ either $x\succ_i y$ or $y\succ_i x$ (but not both).
		\item{ii)} Transitive: $x\succ_i y$, $y\succ_i z$ implies $x\succ_i z$ for any distinct triple $x,y,z\in X$.
		\item{iii)} $\forall y\in X$, $i\in\I$,   $A_i(y):=\{x\in X; \ \ x\succ_i y\}\in \B $.
		\item {iv)} If $x_1\succ_i x_2$ then
		$\mu\left( y; (y\succ_i x_2) \cap (x_1\succ_i y)\right)>0$.
	\end{description}
\end{defi}
In addition, for any $x\in X$  we also assume the existence of order relation  $\succ_x$ on $\I$ such that
\begin{defi}\label{qweX} .
	\begin{description}
		\item {i)} non-symmetric: For any $i\not=j$ either $i\succ_x j$ or $j\succ_x i$ (but not both).
		\item{ii)} Transitive: $i\succ_x j$, $j\succ_x k$ implies $i\succ_x k$ for any distinct triple $i,j,k\in\I$.
		\item{iii)}  $\forall i\not= j\in\I$,  $\{x\in X; \ \ i\succ_x j\}\in B $.
	\end{description}
\end{defi}

Thus
\par

\begin{tcolorbox}
	{\em The firm $i$ prefers to hire $x\in X$ over $y\in X$ iff
		$x\succ_i y$. Likewise, a candidate $x\in X$ prefers firm $i$ over $j$ as a employer  iff $i\succ_x j$.}
\end{tcolorbox}

What is the extension of a bijection $\tau: \I_m\leftrightarrow \I_w$ to that case? Since the cardinality of $X$ is larger than that of $\I$, there are no such bijections. We replace the bijection $\tau$ by a {\it  measurable mapping} $\tau:X\rightarrow \I$.  \par

We can think about such a surjection as a \underline{\em partition}
$$ \vA:=\{  A_i\in  \B,  \ i\in\I, \ \ \ A_i\cap A_j=\emptyset \ \ \text{if} \ i\not= j, \ \ \cup_{i\in\I}A_i= X \ , \}$$
where $A_i:=\tau^{-1}(i)$. We also consider cases where $\tau$ is not a surjection, so there are unemployed people $A_0$ and $\cup_{i\in\I}A_i\subset X$.

Another assumption we make is that the {\em capacity} of the firms can be limited. That is, for any firm $i\in\I$, the number of its employees are not larger than  some  $m_i>0$: $\mu(A_i)\leq m_i$.

Note that we do not impose any condition on the capacities $m_i$ (except positivity). In particular, $\sum_{i\in\I} m_i$ can be either smaller, equal or bigger than $\mu(X)$. Evidently, if $\sum_{i\in\I} m_i<\mu(X)$ then there is an  unemployed set of positive measure.

Let us define a "fictitious firm" $\{0\}$ which contains all unfortunate candidates which are not accepted by any firm. The order relation $(X, \succ_x)$ is extended to $\I\cup\{0\}$ as $i\succ_x 0$ for any $i\in\I$  and any $x\in X$ (i.e we assume that anybody prefers an employment by  {\em any} firm over unemployment).

\begin{defi}\label{defGSMpar}
	Let $\vM:=(m_1, \ldots m_N)\in \R^{N}_+$. Let $\vA:= (A_1, \ldots A_N)$ be a subpartition and $A_0:= X-\cup_{i\in\I}A_i$.
	
	Such a  sub-partition $\vA$ is called an $\vM-$subpartition   if
	$\mu(A_i)\leq m_i$ for any $i\in\I$, and $\mu(A_0)= 0\vee\{\mu(X)-\sum_{i\in\I} m_i\}$.
\end{defi}
\begin{defi}\label{defGSMpar1}
	A subpartition is called  stable if,
	for any $i\not= j$, $i,j\in\I\cup\{0\}$ and  any $x\in A_{i}$, either
	$i\succ_x j$ or $ y\succ_{j} x$ for any $y\in A_{j}$.
\end{defi}

\begin{theorem}\label{Shap2}
	For any $\vM\in\R_+^{N}$ there exists
	a stable $\vM-$subpartition.
\end{theorem}
The proof of this Theorem, outlined in section \ref{GSA} below,  is a constructive one. It is based on a generalization of the   Gale-Shapley algorithm, \index{Gale-Shapley algorithm}  described in section \ref{GSalgomer}.

 
For describing this algorithm we need few more definitions: For any $i\in\I$ and $A\in \B$, the set $C_i^{(1)}(A)\in \B$ is the set of all candidates in $A$ whose $i$ is the first choice:
$$ C_i^{(1)}(A):= \{ x\in B; \forall j\not= i, \ i\succ_x j\} \ . $$
By recursion we define $C_i^{(k)}(A)$ to be the set of employees in $A$ such that $i$ is their $k-$choice:
$C_i^{(k)}(A):=$
$$\{ x\in A; \exists \I_{k-1}\subset \I; i\not \in\I_{k-1}; |\I_{k-1}|=k-1; $$
$$\ \forall j\in\I_{k-1},  j\succ_x i; \ \forall j\in\I-(\I_{k-1}\cup\{i\}), \ i\succ_x j\} \ . $$
By definition, $C_i^{(k)}(A)\in\B$ for any $i\in\I$, $k=1, \ldots N$ and $A\in \B$.

\subsection{The Gale-Shapley algorithm for partitions}\label{GSA}
At the beginning of each step $k$ there is a subset $X_{k-1}\subset X$ of free candidates.
At the beginning of the first step all candidates are free so $X_0:=X$.

At the first stage, each $x\in X_0$ applies to the firm at the top of his list. So, at the end of this stage,  each firm $i$ gets an employment request from $C_i^{(1)}(X_0)$ (which, incidentally, can be empty).

At the second part of the first stage, each firm evaluates the number of  requests she got. If   $\mu(C_i^{(1)}(X_0))<m_i$ she keeps  all candidates and we define $A_i^{(1)}:=C_i^{(1)}(X_0)$. Otherwise, she ejects all less favorable candidates until she fill her quota $m_i$: Let
$$ A_i^{(1)}:= \cup_{y\in X}\left\{  A_i(y)\cap C_i^{(1)}(X_0); \ \mu(A_i(y)\cap C_i^{(1)}(X_0))\leq m_i\right\} \ . $$
where $A_i(y)$ as in Definition \ref{qwe}-(iii).

Note that $A_i^{(1)}\in \B$. Indeed, let
$\alpha(y):=\mu(A_i(y)\cap C_i^{(1)}(X_0))$ and
$$\underline{m}_i:= \sup_{y\in X}\{ \alpha(y); \ \alpha(y)\leq m_i \} \ . $$
Then there exists a sequence $y_n\in X$ such that $\alpha(y_n)$ is monotone non-decreasing and $\lim\alpha(y_n)=\underline{m}_i$. We obtain that
$$  A_i^{(1)}\equiv  \cup_{n}\left\{  A_i(y_n)\cap C_i^{(1)}(X_0)\right\}$$
so $A_i^{(1)}\in \B$ since $A_i(y_n)$ and $C_i^{(1)}(X_0)$ are both in $\B$.

The set of candidates who where rejected at the end of the first step is the set of free candidates
$$ X_1:=X-\cup_{i\in\I}A_i^{(1)} \ . $$

At the $k+1$ stage we consider the set of free candidates  $X_k$ as the set who where rejected at the end of the $k$ stage. Each employee in $X_k$ was rejected $n$ times, for some $1\leq n\leq k$. So each $x\in X_k$ who was rejected $n$ times, proposes to the firm $i$ if $i$ is the next ($n+1$) firm on its priority list, that is, if $x\in C_i^{(n+1)}(X)$.  Note that for any such person there exists  a chain $1\leq l_1<l_2<\ldots <l_n=k$ such that
$$x\in \cap_{j\leq n} X_{l_j}-(\cup_{1\leq q<k; q\not=i_j; 1\leq j\leq n} X_q):=X_{l_1, \ldots l_n} \ . $$

So, the firm $i$  obtains, at the end of the first part of the $k+1$ step, the candidate $\hat{A}_i^{(k+1)}$ who is composed of her previous candidates $A_i^{(k)}$, and the new candidates. Thus
$$ \hat{A}_i^{(k+1)}:= A_i^{(k)}\cup_{n\leq k}\cup_{\{l_1, \ldots  l_n\}; 1\leq l_1< \ldots l_n=k} C_i^{(n+1)}\left(X_{l_1, \ldots l_n}\right) \ . $$

At the second step of the $k+1$ stage, the firm $i$ evaluates again its candidates set $\hat{A}_i^{(k+1)}$. If
$\mu\left(\hat{A}_i^{(k+1)}\right)\leq m_i$ then
$A_i^{(k+1)}=\hat{A}_i^{(k+1)}$. Otherwise she rejects all less favorable candidates to obtain
$$ A_i^{(k+1)}:= \cap_{y\in X}\left\{  A_i(y)\cap \hat{A}_i{(k+1)}; \ \mu(A_i(y)\cap \hat{A}_i^{(k+1)} )\geq m_i\right\} \ . $$
Note that $ A_i^{(k)}\in \B$ for $k\geq 1$ by the same argument which implies $A_i^{(1)}\in \B$.
\begin{proof} of Theorem \ref{Shap2}: \par\noindent
	Each candidate applies at most once to any of the firms. Candidates who applied, after a finite number of steps, to all $\I$ firms will be rejected at all the next steps. Let us call $A_0$  the set of all these candidates.
	
	So, for any $x\not\in A_0$ there exists $i\in\I$ such that $x\in A_i^{(k)}$ for all $k$ large enough. Define
	$$ A_i:= \liminf_{k\rightarrow\infty} A_i^{(k)} \ . $$
	It follows that $\mu(A_i)\leq \liminf_{k\rightarrow\infty} \mu(A_i^{(k)})\leq m_i$. If $\mu(A_0)>\mu(X)- \sum_{i\in\I} m_i$ then $\mu(A_i)<m_i$ for some $i\in\I$. This, on the other hand, implies $A_0=\emptyset$ by the algorithm, and a contradiction. Hence $\vA$ is an $\vM-$subpartition.
	
 	Next, assume $x\in A_i$ and $j\succ_x i$. By the algorithm, $x$ had applied to $j$ at some step, and were rejected by $j$ at some later step (otherwise he  belongs to $A_j^{(k)}$ for all $k$ large enough, hence $x\in A_j$). It follows that  $i\succ_x j$. This completes the conditions of Definition \ref{defGSMpar}.
\end{proof}
It can be shown that the stable partition obtained by the algorithm described in Section \ref{GSA} is the {\em best one  for the candidates}. In fact the following can be obtained:
\begin{theorem}
	If $\tilde{A}_1, \ldots \tilde{A}_N$ is another $\vm-$stable partition for the order relations $\succ_x, \succ_i$, and if $x\in \tilde{A}_i$ for some $i\in\I\cup\{0\}$, then either $x\in A_i$ or $x\in A_j$ for some $i\succ_x j$. Here $A_1,\ldots A_N$ is the $\vm-$partition obtained in Section  \ref{GSA}.
\end{theorem}
The algorithm described in Section \ref{GSA} can be reversed. If, at each step, the {\em firms} propose to their favorable candidates (instead of the other way around), the algorithm will converge to an $\vm-$stable partition as well. The last algorithm will be the best {\it from the point of view of the firms}.
\section{Transferable utilities}
As we did in Chapter \ref{ch1}, it may be possible to quantify the utilities of firms and candidates,
and then  allow a transfer of money between a firm and her employees, as well as between different firms and employees.

We may generalize Definition \ref{genass} and define the {\em feasibility sets}
$$ (i,x)\in \I\times X \Rightarrow F(i,x)\in 2^{\R^2}$$
where $v$ is the utility of $x$, $u$ the utility of $i$, and $(u,v)\in F(i,x)$ iff $x$ is employed by $i$.
At this stage we only assume that $F(i,x)$ are closed, monotone  sets in $\R^2$ and denote $F_0(i,x)$ the interior of $F(i,x)$.  Recall that $F(i,x)$ monotone means
$$ (u,v)\in F(i,x) \ \ \text{and} \ \ u^{'}\leq u, \  v^{'}\leq v \ \ \text{implies} \ \ (u^{'}, v^{'})\in F(i,x) \ . $$
If we allow an unemployment,
we extend the definition of $F$ to
$$F(0,x):=\R^2_-$$
for any $x\in X$. In that case, however, we must insist that $F(i,x)\cap \R^2_{++}\not=\emptyset$   for any $x\in X$ and $i\not=0$.
\par
The definition of a stable partition is a direct generalization of Definition  \ref{deffea}: 
\begin{defi}
	A partition $A_0, \ldots A_N$ of $X$ is stable iff there exists a function $v=v(x):X\rightarrow \R$  and a vector $(u_1, \ldots u_N)\in \R^N$ such that $x\in A_i$ iff
	\begin{enumerate}
		\item $ (u_i, v(x))\in F(i,x)$,
		\item $(u_j, v(x))\in \R^2-F_0(j,x)$ for any $j\not= i$.
	\end{enumerate}
\end{defi}

The existence of stable  $\vM-$partition (recall Definition \ref{defGSMpar}) in this general setting is beyond the scope of
this book. In non-transferable case we may generalize
the matrices $\theta_m, \theta_w$ of  section  \ref{bribing} and define a pair of  functions $$\vpsi:\I\times X\rightarrow \R_+ \ , \ \ \vphi:\I\times X\rightarrow \R_+  \ \ \text{where}$$
\begin{quote}
	$\psi_i(x):=\psi(i,x)$ is the utility of the firm $i$ for hiring  $x$. Likewise,
	$\phi_i(x):=\phi(i,x)$ is the utility of candidate  $x$ if hired by the firm   $i$.
\end{quote}

The order relation $i\succ_x j$ is now replaced by $\phi_i(x)> \phi_j(x)$, and $x\succ_i y$ by $\psi_i(x)>\psi_i(y)$.  However, the cases $\phi_i(x)= \phi_j(y)$ and $\psi_i(x)=\psi_i(y)$ violate condition (i) in Definitions \ref{qwe}, \ref{qweX}.  For  $\succ_x, \succ_i$ to be consistent with these Definitions we  omit from the set $X$ all points for which there is an equality of $\psi_i(x)=\psi_j(x)$ or $\phi_i(x)=\phi_i(y)$. Let
$$\Delta_1(X):= \{x\in X; \exists y\not=x, i\in\I, \phi_i(x)=\phi_i(y)   \}  \ , $$
$$\Delta_2(X):= \{x\in X; \exists i\not=j\in\I, \psi_i(x)=\psi_j(x)   \} \ , $$
and define $X_0:= X-(\Delta_1(X)\cup \Delta_2(X))$.  Then
\be\label{succiXnum} \forall x,y\in X_0, i,j\in\I \ ;  i\succ_x j \ \  \text{iff} \ \ \phi_i(x)> \phi_j(x), \ \ x\succ_i y \
\ \text{iff} \  \psi_i(x)= \psi_j(y)  \ . \ee
As in section \ref{MtF1}  we consider the "null firm" $\{0\}$ and $\phi_0(x)=0$ for all $x\in X$, while $\phi_i(x)>0$ for any $x\in X, i\in\I$.
Under the above definition, the non-transferable partition model is obtained under the following definition of feasibility sets:
\be\label{Fdefpartnt} F(i,x):= \left\{ (u,v); \ \ u\leq \psi_i(x) \ v\leq \phi_i(x)\right\} \ \ , \ \ i\in \I, \ \ x\in X_0\ee
where $\phi_i, \psi_i$ are assumed to be {\em strictly positive,  measurable} functions on $X_0$.
The existence of a stable $\vM-$partition under  (\ref{Fdefpartnt}) is, then, guaranteed by Theorem \ref{Shap2}.

The case where  firms and employees share their  utilities is a generalization of  (\ref{stfultrans}):

\begin{tcolorbox}
	
	\be\label{DefFixtrans} F(i,x):=\left\{ (u,v); u+v\leq \theta_i(x)\right\}\ee
	where
	$$\theta_i(x):=\phi_i(x)+\psi_i(x) \ . $$
\end{tcolorbox} 

{\it The existence of stable   $\vM-$partitions in the transferable case (\ref{DefFixtrans}), and its generalization, is the main topic of this book!}

We may also attempt to generalize the notion of $q-$blocking pairs with respect to a partition $\vA\in {\cal P}^N$. In analogy to (\ref{introq}), $(x,y)$ is a blocking pair \index{blocking pairs} if $x\in A_i, y\in A_j$ and $\Delta^{(q)}(x,y)>0$ where
\be\label{introqX} \Delta^{(q)}(x,y):= \min\left\{ \begin{array}{c}
	q(\psi_j(x)-\psi_i(x))+\phi_j(x)-\phi_j(y) \\
	q(\phi_j(x)-\phi_j(y))+\psi_j(x)-\psi_i(x)
\end{array}\right\} \  \ .
\ee

Definition \ref{defshapleyq} is generalized as follows:
\begin{defi}\label{defshapleyqX} Given a partition $\vA$, a $k-$chain is a sequence
	$$x_{i_1}\ldots x_{i_k} \  \text{where} \  x_{i_l}\in A_{i_l}\ \text{for any} \  1\leq l\leq k, \  \text{and} \  x_{i_k}=x_{i_1} \ , $$
	(in particular, $i_k=i_1$).
	
	A partition $\vA$ in $X_0$ is $(p,q)-$stable if for any $k\in\N$, any $k-$chain
	$$ \sum_{l=1}^k\left[ \Delta^{(q)}( x_{i_l}, x_{i_{l+1}})\right]_p  \leq 0 \  , $$
	where $[\cdot]_p$ as defined in (\ref{pospdef}).
\end{defi}
What does it mean? Again let us assume first $q=0$ (no bribing) and $p=0$ (no sharing). Then

\begin{quote}
	A partition $\vA$ is $(0,0)-$unstable iff there exist  $x\in A_i$, $y\in A_j$, $i\not=j$ for which $\Delta^{(0)}(x,y)>0$. This implies    that  $\psi_j(x)>\psi_i(x)$ and, in addition,
	$\phi_j(x)>\phi_j(y)$. Surly $x$ will prefer the agent $j$ over his assigned agent $i$, and the agent $j$ will prefer $x$  over one of his assigned customer  $y$ as well. So, $j$  will kick  $y$ out  and $x$ will join $j$ instead, for the benefit of both $x$ and $j$.
\end{quote}
In particular,
\begin{quote}
	Any stable $(0,0)-$subpartition is a stable subpartition in the sense of Definition \ref{defGSMpar1}.
\end{quote}

What about the other extreme case $p=q=1$? It implies (using $i_{k+1}=i_1$)
\begin{multline}
\sum_{l=1}^k \Delta^{(1)}_\tau(x_{i_l}, x_{i_{l+1}})\equiv
\sum_{l=1}^k \psi_{i_{l}}(x_{i_{l+1}})-\psi_{i_{l+1}}(x_{i_{l+1}})+\phi_{i_l}(x_{i_{l+1}})-\phi_{i_{l}}(x_{i_{l}})\equiv \\
\equiv
\sum_{l=1}^k \theta_{i_{l}}(x_{i_{l+1}})-\theta_{i_{l+1}}(x_{i_{l+1}})\leq 0
\end{multline}
where $\theta_i$ as defined in (\ref{DefFixtrans}).

Let us define $\vA$ to be $\theta-$cyclic monotone iff for any $k\in\N$, and $k-$chain $(i_1, \ldots i_k)$ in $\I$ and any $x_{i_1}\in A_{i_l}$, $1\leq l\leq k$,
$$\sum_{l=1}^k \theta_{i_{l}}(x_{i_{l+1}})-\theta_{i_{l}}(x_{i_{l}})\leq 0 \ . $$

\begin{tcolorbox}
	In the complete  cooperative economy, were a firm $i$ and an employee $x$ share their utilities $\theta_i(x)=\psi_i(x)+\phi_i(x)$, a partition $\vA$ is $(1,1)-$stable iff it is cyclical monotone\index{cyclical monotonicity}. It means that
	\begin{quote}
		Not {\em all} members of  any given chain of replacements $x_{i_1}\rightarrow x_{i_2}, x_{i_2}\rightarrow x_{i_3}\ldots  x_{i_k}\rightarrow x_{i_{k+1}}\equiv x_{i_1}$, $x_{i_l}\in A_{i_l}$ will  gain utility, even if
		the other member   are ready to {\it share} their benefits (and losses)  among themselves.
	\end{quote}
\end{tcolorbox}
The connection between a stable partition in the $(1,1)$ sense and the (\ref{DefFixtrans}) sense is not evident. In the next chapter we  discuss this subject in some details.
\chapter{Monge partitions}\label{S(M)P}
{\small{\it The purpose of a business is to create a customer. } (Peter Drucker)} 
\vskip.3in
We pose some structure on $X$ and the utility functions $\theta_i$.
\begin{sassumption}\label{sassump1}
	\begin{description} .
		\item{i)} $X$ is a compact topological space. 
		\item{ii)} The $N$ utility functions $\theta_1, \ldots \theta_N:X\rightarrow \R$ are   continuous.
	\end{description}
\end{sassumption}
We find it convenient to change the interpretation of candidates/firms of Chapter \ref{MtF} as follows: The set $X$ is the set of customers (or consumers), and the set $\I$ is the set of agents (or experts).
The function $\theta_i:X\rightarrow\R$  represents the "utility" of agent $i$, namely,  $\theta_i(x)$  it is the {\em the surplus of the coupling  of  $x$ to $i$}. 
\index{$\OSP$} 
\begin{defi}\label{defiopenpartition}
	An {\em Open $N$ Subpartition}  of $X$ is a collection of $N$ disjoint open subsets of $X$. We denote the collection of all such subpartitions   by
	$$\OSP:= \left\{ \vec{A}=(A_1, \ldots , A_N), \ \ A_i \ \text{is an open subset of} \ \ X\ , \ \ A_i\cap A_j=\emptyset \ \text{if} \ i\not= j\right\} \ . $$
	For any $\vA\in\OSP$ we denote $A_0:= X-\cup_{i\in\I} A_i$.
	\par
\end{defi} 
\begin{defi}\label{cycPw}
	An  open subpartition $\vA$  is  stable iff it is cyclically monotone with respect to $A_0, A_1, \ldots , A_N$, i.e, for any $k\in\N$ and any $k-$chain
	$x_{i_1},  \ldots x_{i_k}$,  $i_l\in\I\cup\{0\}$ where  $x_{i_l}$ is an interior point of $A_{i_l}$, $1\leq l\leq k$, 
	\be\label{iecycPw}\sum_{l=1}^k \theta_{i_{l}}(x_{i_{l}})-\theta_{i_{l+1}}(x_{i_{l}})\geq 0 \ . \ee
	Here   $\theta_{i_{k+1}}:=\theta_{i_i}$ and $\theta_0\equiv 0$.
	
	Note that, since $A_i$ are open sets for $i\in\I$, then the condition "$x_{i_l}$ is an interior point of $A_{i_l}$" simply means $x_{i_l}\in A_{i_l}$ if $i_l\not=0$. If $A_0$ has a null interior then we only consider chains in $\I$.
\end{defi}
\section{Capacities}

Here we assume that the agents have  a limited capacity.  This symbolizes the total number of consumers each agent  can serve.  For this we define an additional structure on the set $X$:

\begin{sassumption}\label{sassump1*}
	$\B$ is the Borel $\sigma$-algebra  \index{Borel $\sigma$ algebra} corresponding to the assumed topology of $X$.
	$\mu\in {\cal M}_+(X)$ is a given positive, {\em regular  and atomless Borel  measure}  on $(X, \B)$, and $X= supp(\mu)$.
\end{sassumption}

Let us recall that if $\mu$ is  {\em regular} positive Borel measure   \index{Borel measure} on $X$, then for any $A\in \B$, $A\not=\emptyset$ and any $\eps>0$ there exists an open $U\supset A$ and a compact $K\subset A$ such that
$$ \mu(U)\leq \mu(A)+\eps, \ \ \ \ \mu(K)\geq \mu(A)-\eps \ . $$
An {\em atom} of $\mu$  is a point $x\in X$ for which $\mu(\{x\})>0$. An {\em atomless} measure contains no atoms.

Recall that $supp(\mu)$ is a closed set, obtain as the intersection of all compact  sets $K\subseteq X$ for which $\mu(K)=\mu(X)$.

The measure $\mu$ represents the distribution of the consumers: for $A\in\B$, $\mu(A)$ stands for the number of consumers in $A$ (not necessarily an integer).
The meaning of a limited capacity $m_i>0$ for an agent $i$ is
$\mu(A_i)\leq m_i$.
\par
The set of open subpartitions subjected to a given capacity $\vM:=(m_1, \ldots m_N)$, $m_i\geq 0$ is denoted by  
\be\label{KOSPVM} \OSP_{\{\leq\vM\}}= \{ \vA:=(A_1, \ldots A_N)\in \OSP \ ; \  \mu(A_i)\leq m_i   \}\ \ . \ee
More generally:
For any closed set $K\subset \R^N_+$,
\be\label{KOSPM} \OSP_{K}:= \{ \vA:=(A_1, \ldots A_N)\in \OSP \ ; \ \left( \mu(A_1) \ldots, \mu(A_N)\right) \in K   \}\ \ . \ee
In particular, if 
\be\label{K_m}K=K_{\vM}:=\{  \vM\geq \vec{s}\geq \vec{0}\}\ee
then $\OSP_{K_{\vM}}$ is reduced to $\OSP_{\{\leq\vM\}}$.


We distinguish three  cases: $\vM$ is
\be\label{OS} \text{Over Saturated  (OS) if} \ \   \sum_{i\in\I} m_i> \mu(X) \ ,\ee
which means that the supply   of the experts surpass the demand of the consumers.
\be\label{S}\text{Saturated  (S) if} \ \   \sum_{i\in\I} m_i= \mu(X) \ , \ee
which means that the supply of the experts and the demand of consumers are balanced, and
\be\label{US} \text{Under-Saturated (US) if} \ \   \sum_{i\in\I} m_i< \mu(X) \ ,\ee \index{under saturated (S)}
which means that the demand of the consumers surpass the supply of the experts. \index{over saturated (OS)}
\par
If $\vM$ is either  S or  US  we denote $\OSP_K$ where $K:=\{\vM\}$ by $\OSP_{\{\vM\}}$, i.e:
\be\label{KOSPM=} \OSP_{\{\vM\}}:= \{ \vA:=(A_1, \ldots A_N)\in \OSP \ ; \ \mu(A_i)=m_i  \}\ \ . \ee

\section{First Paradigm: The big brother} \label{secbigbrother}


\paragraph{ The big brother}:
The "big brother" splits the consumers $X$ between the experts in order to maximize the total surplus, taking into account the capacity constraints. If $\vM$ is either US or S,  \index{saturated}then \index{over saturated (OS)} \index{under saturated (US)}
\be\label{defSigma}{\Sigma^\theta}(\vM):= \sup_{\vA\in \OSP_{\{\vM\}}} \theta(\vA)\  \ee
where
\be\label{defcmu} \theta(\vA):=\sum_{i\in\I}\int_{A_i}\theta_i(x)d\mu \ . \ee 
is the total profit conditioned on the partition $\{A_i\}$.  Note that, by this definition, ${\Sigma^\theta}(\vM) =-\infty$ if $\vM$ is OS.\footnote{The supremum over a null set is always $-\infty$.}
\par
More generally, for any closed   $K\subset\R_+^N$,
\be\label{defSigmaK}{\Sigma^\theta}(K):= \sup_{\vA\in \OSP_K} \theta(\vA)=\sup_{\vM\in K}{\Sigma^\theta}(\vM)\ \ .  \ee
\begin{remark} \label{ex100}  Let $K=K_{\vM}$ (\ref{K_m}). 
	If $\vM$ is S or US and, in addition, the utilities $\theta_i$ are all non-negative on $X$ then the maximizer $\vA$ of (\ref{defSigmaK}) is also a maximizer of (\ref{defSigma}), i.e. it  satisfies $\mu(A_i)=m_i$ for any $i\in\I$ (so $\vA\in\OSP_{\{\vM\}} $).
\end{remark}
What is the relation between maximizers of (\ref{defSigma}) and stable subpartitions (in the sense of Definition \ref{cycPw})?
\begin{prop}\label{equiA}
	If $\vA\in \OSP_{\{\vM\}}$ is a maximizer in (\ref{defSigma}) then it is a stable open subpartition.
\end{prop}
\begin{proof} 
	Let $\vA$ be a maximizer of (\ref{defSigma}).  If $\vA$ is not stable then by Definition \ref{cycPw} there exists a chain $(x_{i_l}, A_{i_l})$ such that
	$$\sum_{l=1}^k \theta_{i_{l}}(x_{i_{l}})-\theta_{i_{l+1}}(x_{i_{l}})< 0 \ . $$
	Since $x_i$ are interior points of $A_i$ by assumption and $\mu$ is regular there exists $\eps>0$ and open neighborhoods  $x_i\ni U_i\subset A_i$ such that $\mu(U_i)=\mu(\bar{U}_i)=\eps$ for any $i\in\I$ (here $\bar{U}$ is the closure of $U$).   Since $\vtheta$   are continuous functions   we can choose $\eps$ sufficiently small such that, for some $\delta>0$,
	$$\sum_{l=1}^k \theta_{i_{l}}(\tilde{x}_{i_{l}})-\theta_{i_{l+1}}(\tilde{x}_{i_{l}})< -\delta \  $$
	for any sequence $\tilde{x}_{i}\in \bar{U}_{i}$, $i\in\I$ (again we set $\tilde{x}_{i_{k+1}}=\tilde{x}_{i_1}$).  In particular
	\be\label{sumBi}\sum_{l=1}^k \int_{U_{i_{l}}}\left[\theta_{i_{l}}-\theta_{i_{l+1}}\right]d\mu < -\eps\delta \  \ee
	Define $B_{i_l}:=(U_{i_{l-1}}\cup A_{i_l})-\bar{U}_{i_{l}}$ for $l=1, \ldots k$ (recall $i_{0}=i_k$), and $B_j=A_j$ if $j\not\in \{i_1, \ldots i_k\}$.  By definition $\mu(B_{i})=m_i$ for any $i\in\I$ so $\vec{B}:=(B_1, \ldots B_N)\in\ \OSP_{\{\vM\}}$. By (\ref{sumBi}) we obtain
	$$\theta(\vec{B}):= \sum_{i\in\I} \int_{B_i}\theta_id\mu \leq \theta(\vA)-\eps\delta $$
	contradicting the maximality of $\theta(\vA)$ on $\OSP_{\vM}$.
\end{proof} 

\section{Second paradigm: Free market}\label{seconfparfreemarket}
Suppose there is no big brother. The market is free, and each consumer may choose his favorite expert to maximize his own  utility.   Each expert determines the price she collects for consulting a consumer. Let $p_i\in\R$ the price requested by expert $i$, $\vpp:= (p_1, \ldots p_N)\in\R^N$.
\begin{remark}\label{bonus}A price $p_i$ can be either positive, negative or zero. In the second case  $-p_i$ is a "bonus".
\end{remark}
The utility of a consumer $x$ choosing the expert $i$ is, therefore, $\theta_i(x)-p_i$, if it is positive. If $\theta_i(x)-p_i\leq 0$ then the consumer will avoid the expert $i$, so he pays nothing and get nothing form expert $i$. The net income of consumer $x$ choosing expert $i$ is, therefore, $[\theta_i(x)-p_i]_+:=(\theta_i(x)-p_i)\vee 0$. Since any consumer wishes to maximize his income we obtain the income of any consumer $x\in X$ by
\be\label{xi++}\xi^+(\vpp,x):= \max_{ i\in \I}[\theta_i(x)-p_i]_+ \ . \ee
\par 
The set of consumers who give up counseling by any of the experts is
\be\label{A+0} A^+_0(\vpp)=\{x\in X;  \theta_i(x)-p_i< 0 \ \ \text{for any} \ i\in\I\} \  \ee 
while the  set of consumers who prefer expert  $i$ is, then
\be\label{A+i} A^+_i(\vpp):= \{x\in X; \theta_i(x)-p_i\geq  \theta_j(x)-p_j \ \ \forall j\in\I\}-A^+_0(\vpp) \ .  \  \ee
Let
$$ \vA(\vpp):= \left( A^+_1(\vpp), \ldots, A^+_N(\vpp)\right) \ . $$
Note that the sets $A^+_i(\vpp)$ are  not necessarily disjoint (for $i\in\I$) nor open. So $\vA(\vpp)\not\in\OSP$, in general. We denote:
\be\label{competativeP}\vA:=(A_1, \ldots, A_N)\subseteq \vA(\vpp) \Leftrightarrow  \ A_i\subseteq A^+_i(\vpp)\  \ \ \text{for any} \ \ i\in\I\cup\{0\} \ , \ee
where $A_0:=X-\sum_{i\in\I}A_i$.
\begin{defi}\label{equipi}
	The vector  $\vpp:= \{ p_{1} , \ldots p_N\}\in \R^N$ is  an  {\em equilibrium price} vector  with respect to $\vM$ if there exists
	$\vA\in \OSP_{\{\vM\}}$ such that $\vA\subseteq \vA(\vpp)$. \index{$\vA$}
	
	Conversely, if $\vA\in \OSP_{\{\vM\}}$ and $\vpp\in \R^N$ satisfies (\ref{competativeP}), then $\vA$ is a {\em competitive} $\vM-$subpartition with respect to $\vpp$.
\end{defi}
An easy consequence is:
\begin{prop}\label{equiP}

	If  $\vpp\in\R^N$ is an equilibrium  price vector with respect to $\vM$, then  the corresponding  subpartition in $\OSP_{\{\vM\}}$ is stable.
\end{prop}
\begin{proof}
	Let  $(i_1, \ldots i_k)$ is a $k-$chain in $\I\cup\{0\}$, and $x_{i_j}\in A_{i_j}\subseteq A^+_{i_j}(\vpp)$. Let $\vpp=(p_{1}, \ldots p_{k})$ be an equilibrium vector and set $p_0=0$.
	We may assume $i_j\not=i_{j-1}$. Then by definition of $A_i^+(\vpp)$ (\ref{A+0},\ref{A+i}), $$\theta_{i_j}(x_{i_{j+1}})-p_{i_{j}}\leq [\theta_{i_j}(x_{i_{j+1}})-p_{i_{j}}]_+\leq [\theta_{i_{j+1}}(x_{i_{j+1}})-p_{i_{j+1}}]_+= \theta_{i_{j+1}}(x_{i_{j+1}})-p_{i_{j+1}}$$
	while, (recall $\theta_0\equiv 0$),if $i_j=0$,  $\theta_{i_j}(x_{i_{j+1}})-p_{i_{j}}= \theta_{i_{j}}(x_{i_j})-p_{i_{j}}=0$.
	Then
	$$ \sum_{j=1}^k \theta_{i_{j+1}}(x_{i_{j+1}})-\theta_{i_{j}}(x_{i_{j+1}}) \equiv \sum_{j=1}^k [\theta_{i_{j+1}}(x_{i_{j+1}})-p_{i_{j+1}}]-[\theta_{i_{j}}(x_{i_{j+1}})-p_{i_{j}}]\geq 0 \   $$
	hence  (\ref{iecycPw}).
	
\end{proof}

\section{The Big brother meets the Free market}\label{BBMFM}

Suppose the  price vector is $\vpp\in\R^N$. The profit of client $x$ is  $\xi^+(\vpp,x)$ (\ref{xi++}).   The {\it overall   profit} of the clients population
is
\be\label{Xi+def} \Xept(\vpp):= \int_X\xi^+(\vpp,x)\mu(dx) \ . \ee \index{$\Xept$} 
Given the capacity vector $\vM$, \index{capacity vector}suppose that the clients are grouped into a {\it feasible partition} $\vA\in \OSP_K$ where, e.g., $K:=\{\vM\}$, where $\vM$ is either S or US.  \index{under saturated (US)}\index{saturated}
The total profit of the client's population  is $\theta(\vA)$ a defined in (\ref{defcmu}).

Can we compare $\theta(\vA)$ to $\Xept(\vpp)$? The first result we state is that there, is, indeed, such a comparison.
\begin{prop}\label{vmuinwsPP}
	For any given  $\vA\in\OSP_{\{\vM\}}$ and $\vpp\in\R^N$,
	\be\label{cleqxi+vcircMpPE}\theta(\vA)\leq \Xept(\vpp)+\vpp\cdot\vM \ . \ee
\end{prop}
\begin{proof}
	By definition of $\xi^+(\vpp,x)$ (\ref{xi++}),
	\be\label{sec11} \theta_i(x)\leq \xi^+(\vpp,x)+ p_i \ . \ee
	Integrate (\ref{sec11}) with respect to $\mu$ over $X$ and sum over $\I$ to obtain
	\begin{multline}\label{sec12}\theta(\vA) \leq \sum_{i\in\I}\int_{A_i} \left[\xi^+(\vpp,x)+ p_i\right]\mu(dx) \\
	= \int_{\cup_{i\in\I}A_i} \xi^+(\vpp,x)d\mu + \sum_{i\in\I}p_i\mu(A_i) \leq \Xept(\vpp)+\vpp\cdot\vM
	\end{multline}
	where we used $\cup_{i\in\I}A_i\subset X$, (\ref{Xi+def})    and $\vA\in\OSP_{\{\vM\}}$.
\end{proof}
It follows that an equality in (\ref{cleqxi+vcircMpPE}) at $\vA=\vA_0$, $\vpp=\vpp_0$ implies that $\vA_0$ is a maximizer of $\theta$ in $\OSP_{\{\vM\}}$ and $\vpp_0$ is a minimizer of $\vpp \mapsto\Xept(\vpp)+
\vpp\cdot\vM$ in $\R^N$. Moreover
\begin{prop}\label{equalityinprop}
	There is an equality in (\ref{cleqxi+vcircMpPE}) at $(\vA,\vpp)=(\vA_0, \vpp_0)$ if and only if  $\vpp_0$ is an equilibrium price vector with respect to $\vM$.
\end{prop}
\begin{proof}
	If there is an   equality in (\ref{cleqxi+vcircMpPE})  then
	the inequalities in (\ref{sec12}) turn into equalities as well. In particular
	\be\label{4567}\theta(\vA_0)\equiv \sum_{i\in\I}\int_{A_{0,i}}\theta_id\mu_i =\sum_{i\in\I}\int_{A_{0,i}} \left[\xi^+(\vpp_0,x)+ p_{0,i}\right]\mu(dx) \ . \ee
	But
	\be\label{thetaleqxi+pi}\theta_i(x)\leq \xi^+(\vpp_0,x)+ p_{0,i} \ \ \text{for any} \ x\in X \ \ \text{and} \ \theta_i(x)= \xi^+(\vpp_0,x)+ p_{0,i} \ \text{iff} \ x\in A_i^+(\vpp_0)\ \ee
	by definition.    Hence $A_{0,i}\subseteq A^+_i(\vpp_0)$.  In particular, $\vpp_0$ is an equilibrium price vector corresponding to the subpartition $\vA_0\in \OSP_{\{\vM\}}$.
	\par
	Conversely, suppose $\vpp_0$ is an equilibrium price vector with respect to $\vM$. Let $\vA_0\in\OSP_{\{\vM\}}$  be the corresponding open subpartition. Then
	$\xi^+(\vpp_0,x)+ p_{0,i}= \theta_i(x)$ for any $x\in A_{0,i}$, and (\ref{4567}) follows.  Since $\mu(A_{0,i})=m_i$ and $\xi^+(\vpp_0,x)=0$ on $A_{0,0}\subset A^+_0(\vpp_0)$ we obtain that the second inequality in (\ref{sec12}) is an equality as well.
	
\end{proof}

Given a convex $K\subset\R^N_+$,   the {\em support function} of $K$ is $H_K:\R^N\rightarrow \R$ given by
\be\label{defpcircM} H_K(\vpp):=\max_{\vM\in K}\vpp\cdot\vM \ . \ee
In particular, if $K=K_m$ (\ref{K_m}) then 
\be\label{Kdef} H_K(\vpp)\equiv \sum_{i\in\I}m_i[p_i]_+ \ .\ee
Proposition \ref{vmuinwsPP} and (\ref{defpcircM}) imply
\begin{prop}\label{vmuinwsPPH}
	For any given $K\subset\R^N_+$,  $\vA\in\OSP_K$ and $\vpp\in\R^N$,
	\be\label{cleqxi+vcircMpPEK}\theta(\vA)\leq \Xept(\vpp)+H_K(\vpp) \ . \ee
\end{prop}
In addition:
\begin{prop}\label{prhe}
	If there is an equality (\ref{cleqxi+vcircMpPEK}) at $(\vA, \vpp)$ then   $\vpp$ is an equilibrium price vector with respect to some   $\vM_0\in K$ verifying $\vpp\cdot\vM_0=H_K(\vpp)$, while  $\vA$   is a maximizer of $\theta$ in $\OSP_{\{\vM_0\}}$. If, in addition, $\theta_i$ are non-negative and  $K$ given by (\ref{K_m}) then 
	\begin{description}
		\item{i)} If $\vM\in \R^N_{++}$ is either saturated or under-saturated,  then  $\vpp\in\R^N_{+}$.\index{saturated}
		\item{ii)} If  $\vM$ is over-saturated and  if $p_i> 0$ then $\mu(A_{i})\equiv m_{0,i}=m_i$  while if $p_{i}<0$ then $\mu(A_{i})\equiv m_{0,i}=0$. In particular, if $0<m_{0,i}<m_i$ then $p_i=0$.
		\item{iii)} In any of the above cases, if $\vpp\in\R^N$ satisfies the equality in (\ref{cleqxi+vcircMpPEK}) for some $\vA$,  then $[\vpp]_+:= ([p_1]_+, \ldots [p_N]_+ )\in \R^N_+$ and $\vA$ satisfies the equality in (\ref{cleqxi+vcircMpPEK}) as well.
	\end{description}
\end{prop}
\begin{quote} 
	Let us linger a little bit about the meaning of (i,ii). In the (under)saturated cases the market is in favor of the agents. In that case no agent will offer a bonus (see Remark \ref{bonus}) at equilibrium. In the over-saturated case the market is in favor of the consumers, so some agents may be tempted to  offer bonus to attract clients. However, these unfortunate agents will have no clients ($\mu(A_i)=0$)! If an agent $i$ requests  a positive price $p_i>0$ at equilibrium, it means that he is fully booked ($\mu(A_{i})=m_i$). All other agents neither offer  a bonus nor charge a price for their service ($p_{i}=0$). Finally, if the  unfortunate agent $i$ offer a bonus ($p_i<0$) and nevertheless  get no clients, she can obtain the same by  giving his service for free $p_i=0$ (since she gets no profit anyway).
\end{quote}
\begin{proof}
	Since now $\vA\in\OSP_K$ we obtain, as in (\ref{sec12}),    \begin{multline}\label{sec123}\theta(\vA) \leq \sum_{i\in\I}\int_{A_i} \left[\xi^+(\vpp,x)+ p_i\right]\mu(dx) \\
	= \int_{\cup_{i\in\I}A_{i}} \xi^+(\vpp,x)\mu(dx) + \sum_{i\in\I}p_{i}\mu(A_{i}) \leq \Xept(\vpp)+H_K(\vpp) \ ,
	\end{multline}
	where we used $\cup_{i\in\I}A_i\subset X$ (hence $\Xept(\vpp)\geq \int_{\cup_{i\in\I}A_{i}} \xi^+(\vpp,x)\mu(dx)$) and $m_{0,i}:=\mu(A_i)\leq m_i$ (hence $H_K(\vpp)\geq \sum_{i\in\I}p_{i}\mu(A_{i})$).
	Under the assumption $\theta(\vA) = \Xept(\vpp)+H_K(\vpp)$ we obtain both
	\be\label{am22}\sum_{i\in\I}p_{i}\mu(A_{i}) =H_K(\vpp) \ ,  \ \Xept(\vpp)=  \int_{\cup_{i\in\I}A_{i}} \xi^+(\vpp,x)\mu(dx) \ . \ee
	In particular $\vpp\cdot\vM_0=H_K(\vpp)$. Moreover,  (\ref{Kdef}) implies that $p_{i}\geq 0$ if $M_{0,i}>0$.
	If $\vM$ is (under)saturated \index{saturated}\index{under saturated (US)}  then $\mu(A_{i})=m_i= m_{i,0}$ by Remark \ref{ex100}. Thus, $\vM\in\R^N_{++}$ implies  $p_i\geq 0$ for any $i\in\I$.
	\par
	To prove the last part (iii) note that $H_K([\vpp]_+)=H_K(\vpp)$ by (\ref{Kdef}), while $\Xept([\vpp]_+)\leq \Xee(\vpp)$ by definition. Hence, the right side of (\ref{cleqxi+vcircMpPEK}) is not increasing by replacing $\vpp$ with $[\vpp]_+$. Since we assumed that $\vpp$ satisfies the equality  at (\ref{cleqxi+vcircMpPEK}) with a given $\vA$, it implies that the same equality is satisfied for $[\vpp]_+$ and, in particular, $\Xept(\vpp)=\Xept([\vpp]_+)$.
\end{proof}
\section{All the ways lead to stable subpartitions}
Propositions \ref{equiA} and  \ref{equiP} demonstrate  two ways to test conditions for the stability of a given  subpartition  $\vA\in \OSP_{\{\vM\}}$.  The first is by showing that $\vA$ maximizes $\theta$ over $\OSP_{\{\vM\}}$, and the second by finding an equilibrium price vector $\vpp$ corresponding to $\vA$.
\par
It turns out that, in fact, {\it any} stable subpartition in $\OSP_{\{\vM\}}$ is a maximizer of $\theta$, and admits an equilibrium price vector:
\begin{theorem}\label{thmsuffness}Let $\vA\in \OSP_{\{\vM\}}$. The following conditions are equivalent:
	\begin{description}
		\item{i)}
		$\vA$ is a stable partition. \index{stable partition}
		\item{ii)}  There exists  $\vpp\in\R^N$ for which
		\be\label{eqalitycAxip}\theta(\vA)=\Xept(\vpp)+ \vpp\cdot\vM \ . \ee
		\item{iii)} $\vA$ is a maximizer of $\theta$ in $\OSP_{\{\vM\}}$.
		\item{iv)} $\vpp$ is a minimizer of $\vpp \mapsto\Xept(\vpp)+ \vpp\cdot\vM$ in $\R^N$, and $\vA$ is the corresponding competitive subpartition.
	\end{description}
\end{theorem}
\begin{proof}
	We already know that (ii,iii,iv) are equivalent by Proposition  \ref{equalityinprop}. This and Proposition \ref{equiP} guarantee that (ii,iii,iv) imply (i) as well.
	\par
	Suppose (i).
	Let
	\be\label{pidefp} p_i:=\sup\left( \sum_{l=1}^{k-1} \theta_{i_{l+1}}(x_{i_{l}})-\theta_{i_l}(x_{i_l})\right)
	- \theta_{i_k}(x_{i_k})+\theta_{i}(x_{i_k}) \ , \ee
	where the supremum is taken over all $k+1$ chains $(i_0,i_1, \ldots i_k)$ in $\I\cup\{0\}$, $k\in \N\cup\{0\}$, satisfying $i_0=0$ and $x_{i_l}\in A_{i_l}$.
	
	Note that, by cyclic subadditivity, $p_0\leq 0$. In fact, $p_0=0$ (why?).
	Let $i\in\I$.
	Let $\alpha<p_i$ and consider a $k-$chain realizing
	$$ \alpha < \left( \sum_{l=1}^{k-1} \theta_{i_{l+1}}(x_{i_{l}})-\theta_{i_l}(x_{i_l})\right)
	- \theta_{i_k}(x_{i_k})+\theta_{i}(x_{i_k}) \ . $$
	By cyclic monotonicity (c.f. \ref{iecycPw})
	$$ \alpha-\theta_i(x_{i_k}) + \theta_1(x_{i_k})\leq 0 \ , $$
	in particular $p_i<\infty$.
	
	Hence, for any $j\in\I\cup\{0\}$ and $y\in A_i$
	\begin{multline} \alpha-\theta_i(y)+\theta_j(y)
	< \left( \sum_{l=1}^{k-1} \theta_{i_{l+1}}(x_{i_l}))-\theta_{i_l}(x_{i_l})\right) \\
	- \theta_{i_k}(x_{i_k})+\theta_i(x_{i_k})  -\theta_i(y)+\theta_j(y) \leq p_j\end{multline}
	where the last inequality follows by the substitution of  the $k+1-$cycle $(i_1, i_2, \ldots, i_k, i_{k+1}=i)$  and $x_{i_{k+1}}=y$ in  (\ref{pidefp}).
	Since $\alpha$ is any number smaller than $p_i$ it follows
	$$  \theta_i(y)-p_i
	\geq\theta_j(y)-p_j$$
	for any $y\in A_i$. Taking $j=0$ we obtain, in particular, $\theta_i(y)-p_i\geq 0$ for any $i\in\I$ and $y\in A_i$. Hence
	$$  [\theta_i(y)-p_i]_+
	\geq[\theta_j(y)-p_j]_+$$
	for any $i,j\in\I\cup\{0\}$ and $y\in A_i$,  so $\vA\subseteq \vA(\vpp)$ so $\vpp$ is an equilibrium price vector (Definition \ref{equipi}). The result follows now from Proposition \ref{equalityinprop}.
	
\end{proof}


\section{Weak definition of partitions}\label{4.6}
Theorem \ref{thmsuffness}-(iii) shows a direct way to obtain a stable open subpartition in
$\OSP_{\{\vM\}}$:  Find such a subpartition which maximize $\theta$ in this set.
\par
But how can we do it? Suppose we find a sequence of open subpartitions $\vA_n\in\OSP_{\{\vM\}}$ such that\index{$\Sigma^\theta$}
\be\label{limnctheta}\lim_{n\rightarrow\infty} \theta(\vA_n)=\sup_{\vA\in \OSP_{\{\vM\}}}\theta(\vA)\equiv {\Sigma^\theta}(\vM) \ . \ee  
Can we identify an open subpartition $\vA$  which, in some sense, is the "limit" of some  subsequence of $\vA_n$?  And, if we could, can we show that  $\theta(\vA)={\Sigma^\theta}(\vM)$?  

In order to proceed, we need to assign some topology on $\OSP_{\{\vM\}}$. Suppose we had some metric on $X$. It induces a natural metric  on the set of subsets of $X$, namely the Hausdorff distance between $A_1, A_2\subset X$:
$$ d_H(A_1, A_2):= \{ \sup_{x\in A_1}\inf_{y\in A_2} d(x,y)\}\vee \{\sup_{x\in A_2}\inf_{y\in A_1}d(x,y)\} \ . $$
Surly, the Hausdorff distance can be applied to the set of subpartitions in $\OSP_{\{\vM\}}$ componentwise, and provides us with a metric on this set. However, the Hausdorff distance does not respect the measure $\mu$. In particular, if $A_n\rightarrow A$ in the Hausdorff metric and $\mu(A_n)=m$, then $\mu(A)\not=m$, in general. Thus, $\OSP_{\{\vM\}}$ is not a complete metric space under the Hausdorff distance.

To overcome this difficulty, let us consider the following definition of convergence: 
\begin{defi}
	A sequence of measurable sets $A_n$ is said to converge  weakly-*  to a measurable set $A$  ($A_n\rightharpoonup A$) if, for any continuous function $\phi$ on $X$
	$$ \lim_{n\rightarrow\infty}\int_{A_n}\phi d\mu=\int_A\phi d\mu \ . $$
\end{defi}
In particular, letting $\phi\equiv 1$ we obtain that $A_n\rightharpoonup A$ implies $\lim_{n\rightarrow\infty} \mu(A_n)=\mu(A)$. 

Using this  definition for each component of a partition  we easily obtain the continuity of the function $\theta:\vA\rightarrow \R$ with respect to the weak* convergence. What we may miss is, however, the {\em compactness} of this topology on measurable sets. Indeed, the space Borel sets is not even 
close under weak* convergence.
\begin{example}
Let $X=[0,1]$, $\mu$ the Lebesgue measure,  and $A_n:= \{x; \exists k \text{even}, x\in [k/n, (k+1)/n)\}$. Then $$ \lim_{n\rightarrow\infty} \int_{A_n}\phi dx = \frac{1}{2}\int_0^1\phi dx \  $$
but there is no set $A\in {\cal B}$ for which $\int_A\phi dx=\frac{1}{2}\int_0^1\phi dx $ for any continuous $\phi$. 
\end{example}

 Let us  represent  a subset $A\in\B$ by  the measure $1_A d\mu$, where  $1_A$ is the  {\em characteristic function} \index{characteristic function}
$$1_A(x):=\left\{\begin{array}{cc} 
1 & \text{if}  \ x\in A \\
0 & \ \text{otherwise}
\end{array}\right.$$
Stated differently, we may {\em define} the set $A$ by its action as a  {\em linear functional} on the space of continuous functions on $X$ 
 $ C(X)$;
$$ \phi\in C(X) \rightarrow \int_A\phi d\mu= \int_X \phi 1_A d\mu \in \R \ . $$
We may now {\em extend} the "space"  of Borel sets $\B$ to the space of all bounded  Borel measures on $(X,\B)$, considered as linear functionals on $C(X)$:
$$ \phi\in C(X) \rightarrow \int_X\phi d\nu  \in \R \ , $$
and define the {\em weak* convergence} \index{weak* convergence} of a sequence of Borel measures $\nu_n$ to $\nu$ by \index{Borel measure}
$$ \nu_n\rightharpoonup \nu \Leftrightarrow \lim_{n\rightarrow\infty} \int_X\phi d\nu_n = \int_X \phi d\nu \ \ \forall \phi\in C(X) \ . $$
What did we gain with this notion of convergence? It turns that the set of bounded Borel measures is closed under this notion of convergence. Moreover, it is also locally compact. In particular
\begin{tcolorbox}
	If $\{\nu_n\}$ is a sequence of  Borel measures  bounded by $\mu$, then there exists a  Borel measure $\nu\leq \mu$ and a subsequence 
	$\{\nu_{n_k}\}\subset \{\nu_n\}$ such that $\lim_{k\rightarrow\infty} \nu_{n_k}=\nu$ in the sense of weak* convergence. \index{weak* convergence}
	\end{tcolorbox}
This local compactness of the set of bounded Borel measures under weak* convergence is the key for the {\em Kantorovich relaxation}, which is the idea behind the notion of {\em weak partitions}  \index{weak partition}  defined in the next section. 

Thew notion of convergence of measures in general, and weak* convergence in particular, is a deep subject, but this result of the local compactness is all we need to know in order to proceed in this book. 
A detailed study of measure's  convergence  can be found in \cite{Bi} (and many other sources). For the convenience of the reader we extend on this subject in Appendix \ref{weakconv}.

\subsection{Kantorovich relaxation of  (sub)partitions}\label{krop}

\begin{defi}\label{weakpartition}
	A weak subpartition of $(X,\mu)$  of order $N$ is a given by $N$  non-negative Radon measures $\mu_i$ on $(X,\B)$, $i\in\I$ which satisfy
	\be\label{wp}\vmu:=(\mu_1, \ldots \mu_N) \ \ , \ \ |\vmu|:=\sum_{i\in \I}\mu_i\leq\mu \ . \ee
	$\wSP$ is the collections of all such weak partitions \index{weak partition}   of $(X,\mu)$. \\ If there is an equality in (\ref{wp})  $\vmu$  is called a weak partition. The set of weak partitions is denoted $\wP\subset \wSP$.
\end{defi}
Motivated by the above we generalize (\ref{KOSPVM}) as follows:
For any  $\vM=(m_1, \ldots m_N)\in \R^N_+$,
\be\label{KSPWM} \wSP_{\leq\vM}:= \left\{ \vmu:=(\mu_1, \ldots \mu_N)\in \wSP; \   \mu_i(X)\leq m_i, i\in\I  \right\}\ , \ee
and, more generally
\be\label{KSPWMK} \wSP_{K}:= \left\{ \vmu:=(\mu_1, \ldots \mu_N)\in \wSP; \   \vmu(X)\in K \right\}\  \ee
for a given closed set $K\subset \R^N$. 

In addition, we extend the function $\theta$ (\ref{defSigma})  to $\wSP$ as     
\be\label{defSigmaM} \theta(\vmu):=\sum_{i\in\I} \int_X\theta_i(x)\mu_i(dx)  \ . \ee
\begin{remark}\label{noatom}
	In this and the next chapter we do not need to assume the condition that $\mu$ is an atomless measure declared in the Standing Assumption \ref{sassump1*}.  In particular, we may even assume that $X$ is a {\em finite discrete set}.
\end{remark}
\begin{example} \label{examdis1}
	Let $X$  be a finite set $\{x_1, \ldots x_n\}$. \\ $\J:=\{1, \ldots n\}$. Let,
	\be\label{atomic}\mu:= \sum_{j=1}^n \alpha_j\delta_{x_j} \ , \ee
	where $\alpha_j > 0$, $\sum_{j=1}^n\alpha_j=1$ and $\delta_x$ is the Dirac measure at $x\in X$. A weak partition is given by\index{weak partition}
	\be\label{atomicweak} \mu_i:= \sum_{j=1}^n \pi(i,j)\delta_{x_j} \ \ i\in\I \ , \ee
	where $\sum_{i=1}^N\pi(i,j)=\alpha_j$ and $\pi(i,j)\geq 0$ for any $j\in \J$, $i\in\I$.
	\par 
	Under the same setting we may  present $\vtheta$ on $X$ in terms of $N\times n$ matrix  $\{\theta_{ij}\}$. Hence $\theta(\vmu)$ takes the form
	$$ \theta(\vmu):=\sum_{i=1}^N\sum_{j=1}^nc_{i,j}\pi(i,j) \ . $$
\end{example}

What is the point behind such a generalization? Recall (\ref{limnctheta}).  If we could prove the existence of a maximizer $\vA$ for $\theta$ in $\OSP_{\{\vM\}}$, we would have a stable (sub)partition in our hand! The problem is that we dont have the tool to prove the existence of such a maximizer in the set of open (sub)partitions $\OSP_{\{\vM\}}$.
\par
However, we can obtain, quiet cheaply, the existence of a maximizer for $\theta$ on sets of weak (sub)partitions. Here we take advantage of the local compactness of the space of Borel measures with   \index{Borel measure} respect to the weak* topology, and apply in componentwise to the weak partitions: \index{weak partition} 
$$ \vmu_n:=(\mu^1_n, \ldots \mu^N_n)\rightharpoonup \mu:=(\mu^1, \ldots \mu^N) \Leftrightarrow \mu^i_n\rightharpoonup \mu_i, \ \forall\ i\in\I \ .  $$
 Thus
\begin{description}
	\item{a)}  $\wSP_K$  is a compact subsets of $\wSP$.
	\item{b)}  $\theta$ is  continuous on  $\wSP_K$. That means that for any converging sequence $\vmu_n\rightharpoonup\vmu$ in $\wSP_K$,
	$$ \lim_{n\rightarrow\infty}\theta(\vmu_n)=  \theta(\vmu) \ $$
\end{description}

Indeed, let us consider a maximizing sequence $\vmu_n$ satisfying
\be\label{barSigmvm} \lim_{n\rightarrow\infty} \theta(\vmu_n) = \bar{\Sigma}_{\theta}(K):=\sup_{\vmu\in\wSP_K}\theta(\vmu) \ee  
By (a) we get the existence of a subsequence  $\vmu_{_k}\rightharpoonup\vmu\in\wSP_K$, and from (b) we obtain $ \bar{\Sigma}_{\theta}(K) =\theta(\vmu)$. Hence  $\bar{\Sigma}^K_{\theta}(\vM) = \theta(\vmu)$ and $\vmu$ is a maximizer!
\par
The  continuity (b) of $\theta$ on $\wSP_K$ follows from  Assumption \ref{sassump1}-(ii). In particular:
\begin{theorem}\label{th0}
	For any closed set $K\subset \R^N$  there exists a maximizer $\vmu$ of  $\theta$ (\ref{defSigmaM}) in $\wSP_{K}$. 
\end{theorem}
\

So, to answer   question  \ref{sumandbey}-(1) affirmatively we just have to prove that this maximizer is, in fact, in $\OSP_{\vM}$. Such a result will guarantee, in particular, that  $\bar{\Sigma}_{\theta}(K)$ defined in (\ref{barSigmvm}) is equal to ${\Sigma^\theta}(K)$ defined in (\ref{defSigmaK}). 

\subsection{Birkhoff Theorem}\label{birk} \index{Birkhoff Theorem}
In the context of Example \ref{examdis1} the sets $\wP_{\{\vM\}}$  can certainly be empty.
\par
Consider the particular case of an {\it empirical} atomic measure
\be\label{empiricmu}\mu:=\sum_{i=1}^N\delta_{x_i} \ , \text{and} \ \vM:=(1, \ldots 1)\ \ .  \ee
In that case we observe that we can embed any atomic weak partition (\ref{atomicweak})  \index{weak partition}  in the set of $N\times N$ {\it doubly stochastic matrices}
\be\label{dstoc} \Pi:=\{ \alpha_{i,j}\geq 0; \ \ \sum_{i=1}^N \alpha_{i,j}=\sum_{j=1}^N\alpha_{i,j}=1 \} \ . \ee 
A bijection $\tau: \I\leftrightarrow \{x_1, \ldots x_N\}$ corresponds to a matrix $\pi_\tau$ in the set of permutation matrices ${\cal P}$:\index{permutation}
$$ \tau \Rightarrow \pi_\tau(i,j):= \left(
\begin{array}{cc}
1 & \text{if} \ \ j=\tau(i) \\
0 & \text{if} \ \ j\not=\tau(i) \\
\end{array}
\right)\in {\cal P}\subset \Pi \  \ . $$

Now we compare (\ref{defSigma}) with $\theta_i(x_j):= \theta(i,j)$ and $m_i=1$ with  (\ref{mongectau}). 
\par
Let
\be\label{ineqcpitau} \theta(\pi):= \sum_{i=1}^N\sum_{j=1}^N \theta(i,j)\pi(i,j) \ . \ee

Then we obtain immediately that
$$\hat{\theta}(\tau):=\sum_{i=1}^N \theta_{i, \tau(i)}=\theta(\pi_\tau) \ . $$
In particular,
\be\label{ineqperlp} \max_{\pi\in \Pi} \theta(\pi)\geq \max_{\tau\in{\cal P}}\hat{\theta}(\tau) \ . \ee
It is, however, somewhat surprising that there is, in fact, an equality in (\ref{ineqperlp}). This follows from the following Theorem of Birkhoff:
\begin{theorem}
	The set $\Pi$ is the convex hull\footnote{See Appendix \ref{AppA1}} \index{convex hull} of ${\cal P}$.  Namely, for any $\pi\in \Pi$ there exists a set  of permutations $\{\tau_j\}\subset {\cal P}$ and positive numbers $\{a_j\}$ satisfying $\sum_j a_j=1$ such that 
	$$\sum_j a_j\pi_{\tau_j}=\pi \ . $$
\end{theorem}
Birkhoff's Theorem implies
\begin{prop}
	There is  an equality in (\ref{ineqperlp}).
\end{prop}

From Theorem \ref{th0} applied to the case of Example \ref{examdis1} we obtain that a maximizer $\bar{\pi}$ of (\ref{ineqcpitau}) is, in fact, a stable weak partition \index{weak partition}$$\vmu:= (\sum_{j=1}^N\bar{\pi}(1,j)\delta_{x_j}, \ldots  \sum_{j=1}^N\bar{\pi}(N,j)\delta_{x_j})$$ 
of (\ref{empiricmu}). From the equality in (\ref{ineqperlp}) due to Birkhoff Theorem \index{Birkhoff Theorem} we  get that $\bar{\pi}$ is also a permutation
\index{permutation} matrix which maximizes the right side of (\ref{ineqperlp}) as well.  Then this permutation is the Monge \index{Monge} solution $\tau$  in the sense of Definition \ref{defMonge}.
There are several proofs (mostly algebraic) of Birkhoff Theorem 
index{Birkhoff Theorem}in the literature. The following  elegant proof of Zhu is based on a variational  argument, in the spirit of this book. Here is a sketch of the argument of Zho:

	Let $\pi\in\Pi$ be a doubly stochastic matrix. Define the function $f$ on the set ${\cal G}$  of real valued $N\times N$ matrices:
	$$ f(G):= \ln\left(\sum_{\tau\in{\cal P}}e^{ G:(\pi-\tau)}\right) ,  \  \  G\in {\cal G} \ . $$ 
	The function $f$ is differentiable and its derivative at $G\in{\cal G}$ is
	$$ f^{'}(G)=\sum_{\sigma\in{\cal P}} \lambda_G(\sigma)(\pi-\sigma)$$
where
$$ \lambda_G(\sigma):= \frac{ e^{G:(\pi-\sigma)}}{\sum_{\tau\in{\cal P}} e^{G:(\pi-\tau)}} $$
satisfies $\sum_{\sigma\in{\cal P} }\lambda_G(\sigma) =1$. 
Now, assume we know that there exists a minimizer $G_0\in{\cal G}$ of $f$. Then $f^{'}(G_0)=0$, namely 
$$ \sum_{\sigma\in{\cal P}} \lambda_{G_0}(\sigma)(\pi-\sigma)=0\Longrightarrow \pi=\sum_{\sigma\in{\cal P}} \lambda_{G_0}(\sigma)\sigma$$
which implies Birkhoff's Theorem. However, the assumption that there exists a minimizer of $f$ of ${\cal G}$ is too strong. To justify this argument, Zho first claimed that $f$ is bounded from below on ${\cal G}$, and then apply  a very useful and elementary lemma:

\paragraph{Approximate Fermat principle:}
{\it If $f$ is differentiable on the entire space and bounded from below, then 
for any   $\eps>0$ there exists an {\it approximate critical point}  $G$  for which $|f^{'}(G)|\leq \eps$. }

	This, and the local compactness of ${\cal G}$ is enough to complete the argument.  The fact that $f$ is bounded from below follows from another elementary argument of Zho which implies that for any $G\in{\cal G}$ and $\pi\in\Pi$ there exists $\tau\in{\cal P}$ for which $G:(\pi-\tau)\geq 0$.

\section{Summery and beyond }\label{sumandbey}
What did we learn  so far? \\
If $\vM\in\R^N_+$ is S or US,  then \index{under saturated (US)}\index{saturated}
\begin{itemize}
	\item By  Proposition \ref{vmuinwsPP}
	\be\label{gdjft}\sup_{\vA\in \OSP_{\{\vM\}}}\theta(\vA):=\Sigma_{\vec{\theta}}(\vM)\leq \inf_{\vpp\in\R^N}\Xept(\vpp)+\vpp\cdot\vM \ . \ee
	\item From Proposition \ref{vmuinwsPP} and Theorem \ref{thmsuffness}:
	The equality in (\ref{gdjft}) is a {\em necessary} condition for the existence of a stable subpartition in $\OSP_{\{\vM\}}$.
	\item From Theorem \ref{th0}: There exists a maximizer on the set of weak sub-partitions \index{weak partition} 
	$$ \max_{\vmu\in \wSP_m}\theta(\vmu) \geq \sup_{\vA\in \OSP_{\{\vM\}}}\theta(\vA)$$

\end{itemize}
The questions still  left open, at this stage, are
\begin{enumerate}
	\item Is there a  maximizer on the left side of (\ref{gdjft})?
	\item If so, is this maximizer  unique in $\OSP_{\{\vM\}}$?
	\item What about a minimizer $\vpp\in\R^N$ of the right side of (\ref{gdjft})?
	\item Same questions regarding the maximizer in  $\OSP_K$ and minimizer in $\R^N$ of
	\be\label{gdjftSP}\sup_{\vA\in \OSP_K}\theta(\vA):=\Sigma^K_{\vec{\theta}}(\vM)\leq \inf_{\vpp\in\R^N}\Xept(\vpp)+H_K(\vpp) \  \ee 
	for a given $K\subset\R^N_+$.
\end{enumerate}
  
\part{Multipartitions}
\label{partII}
\chapter{Weak multi-partitions}\label{Wps}
{\small {\it There are many ways of going forward, but only one way of standing still,} 
(F.D.R)}

\section{Multi-partitions}\label{Genweakpar}
Let us now generalize the definition of partitions and capacity (Section \ref{secbigbrother}) in a natural way.
\par
Suppose the there is a set of   goods $\J=\{1, \ldots J\}$. Each customer  $x\in X$ consumes a given fraction $\zeta_j(x)\geq 0$ of $j\in\J$. The consumption vector $\vzeta:=(\zeta_1(x), \ldots \zeta_J(x))$ is defined such that $\sum_{j=1}^J \zeta_j(x)=1$, thus $\vzeta(x)\in\Delta^J(1)$ for any $x\in X$.  We further assume 
\be\label{vzetadef} \vzeta\in C(X; \Delta^J(1))  , \text{thus} \ \  \bar\mu(dx):= \vzeta(x)\mu(dx)\in{\cal M}_+^N(X) \ .  \ee
Each agent $i\in\I$ can supply each of the goods $j\in \J$ under a prescribed capacity. Here $m^{(j)}_i\geq 0$ is the capacity of agent $i$ for goods $j$ and
$${\vM_i}:= (m_i^{(1)}, \ldots m_i^{(J)})\in\R_+^J \ , \ i\in\I  \ .  $$
\begin{defi}\label{defD}
	The {\it capacity matrix}  $\{m^{(j)}_i\}:= \cM:= \{ \vM_1, \ldots \vM_N\}$ is a $N\times J$ matrix of positive entries. The set of all such matrices is denoted $\D(N,J)$.
\end{defi}

For an admissible weak  (sub)partition $\vmu:=(\mu_1, \ldots \mu_N)\in \wSP$, each agent should  be able to supply the part $\mu_i$ of the population. This implies
$$ \mu_i(\zeta_j):=\int_X \zeta_j(x)\mu_i(dx)\leq m^{(j)}_i \ \  ; \ \ (i,j)\in\I\times \J \ . $$

Let now $\K\subset\D(N,J)$. We  generalize (\ref{KSPWM}) for  the set of weak sub-partitions \index{weak subpartition} 
\be\label{KSPWMxi} \wSPzeta_\K:= \left\{ \vmu:=(\mu_1, \ldots \mu_N)\in \wSP; \  \vec{\mu}(\vzeta)\in \K   \right\}\ .\ee

If $\K=\{\cM\}$  is a singleton  then we denote the corresponding subpartitions  by $\wSPzeta_{\{\cM\}}$.

The set of {\em partitions} satisfying $\sum_i\mu_i=\mu$ is denoted by $\wPzeta_{\{\cM\}}\subset \wSPzeta_{\{\cM\}}$ . \index{$\wPzeta_{\{\cM\}}$}
\index{$\wSPzeta_{\{\cM\}}$}\index{$\wSPzeta_\K$}

Let
\be\label{muzeta} \cM_\zeta(\vmu):= \{ \mu_i(\zeta_j)\}
\in \D(N,J) \ . \ee 
The conditions on $\cM\in\D(N,J)$ for which the corresponding subpartition sets
are not empty ({\it feasibility conditions}) are not as simple as (\ref{OS}-\ref{US}). \index{over saturated (OS)} \index{under saturated (US)} In particular, the notions of saturation (S), under saturation (US) and over saturation (OS)  presented in Section \ref{secbigbrother} (\ref{US}, \ref{S}, \ref{OS})  should be generalized:
\begin{defi}\label{defvmSw}
	The {\em feasibility sets} with respect to $\bar\mu:= \vzeta\mu$ are (c.f. (\ref{muzeta}))  
	$$\vmS_N(\bar\mu) := \{\cM=\cM_{\zeta}(\vmu); \ \ \vmu\in \wP \} \ \ \ \ , \ \  \uvmS_N(\bar\mu):=\{\cM=\cM_{\zeta}(\vmu); \ \ \vmu\in \wSP \} \ .  $$
	Equivalently 
	$$\vmS_N(\bar\mu) := \{\cM; \  \wPzeta_{\{\cM\}}\not=\emptyset \} \ \ \ \ , \ \  \uvmS_N(\bar\mu):=\{\cM;\  \wSPzeta_{\{\vM\}}\not=\emptyset \  \} \ .  $$
\end{defi}
\begin{remark}
	Note that for any $\cM\in \uvmS_N(\bar\mu)$
	\be\label{sigmaleqmu0} \sum_{j\in \J} m^{(j)}_i\equiv \sum_{j\in \J}\mu_i(\zeta_j) = \mu_i(X)   \ee
	while
	\be\label{sigmaleqmu} \sum_{i\in\I} m^{(j)}_i\equiv \sum_{i\in\I}\mu_i(\zeta_j) \leq \mu(\zeta_j) \ .  \ee  
	By definition
	$$ \sum_{i\in\I, j\in \J} m^{(j)}_i= \sum_{i\in\I}\mu_i\left(\sum_{j\in \J}\zeta_j\right)= \sum_{i=1}^N\mu_i(X) \ . $$
	If $\cM\in \vmS_N(\bar\mu)$  (so $\sum_{i\in\I}\mu_i=\mu$) we get an equality in (\ref{sigmaleqmu}):
	\be\label{sigmaleqmu1} \sum_{i=1}^N m^{(j)}_i= \mu(\zeta_j)  \ .  \ee 
\end{remark}
\begin{example}\label{ex34}
	If $J=1$   then $\vmS_N$ is the simplex of all $N-$ vectors
	$$\vmS_N =\Delta^N(1):=\left\{ \vM:=(m_1, \ldots, m_N)\in \R^N_+; \ \ \sum_{i=1}^N m_i= 1 \right\}$$
	and $\uvmS_N(\bar\mu)$ is the sub-simplex
	$$\uvmS_N =\underline{\Delta}^N(1):=\left\{ \vM:=(m_1, \ldots, m_N)\in \R^N_+; \ \ \sum_{i\in\I} m_i\leq 1 \right\} \ . $$
\end{example}
The natural generalization of over saturation (\ref{OS}), saturation (\ref{S}) and under saturation  (\ref{US}) is as follows: \index{under saturated (US)}
\be\label{gOS}\text{Over Saturated (OS) if} \ \cM\in \D(N,J)-\uvmS_N(\bar\mu) \ . \ee
\be\label{gS} \text{Saturated (S) if} \ \    \cM\in \vmS_N(\bar\mu)  \ .\ee
\be\label{gUS} \text{Under-Saturated (US) if}\  \cM\in\uvmS_N(\bar\mu)- \vmS_N(\bar\mu)   \ . \ee

\begin{prop}\label{noempty}
	The sets $\vmS_N(\bar\mu) \subset\uvmS_N(\bar\mu)$ are both  compact, bounded and convex in $\D(N,J)$. The set $\uvmS_N(\bar\mu)$ has a non-empty interior in $\D(N,J)$.
\end{prop}
Consider Example \ref{ex34} again. Here $\vmS_N(\bar\mu)$ is an $N-1$ dimensional  simplex in $\R^N$, so it has an empty interior. It is not a-priori clear that the interior of $\vmS_N(\bar\mu)$ is not empty in the general case. However, it is the case by Corollary \ref{wSsubsetwpus} below.

\section{Feasibility conditions}\label{Fc}  

\subsection{Dual Representation of Weak (Sub)partitions}\label{dual}
Here we attempt to characterize the feasibility sets by a dual formulation. For this we return to the "market" interpretation of Chapter \ref{seconfparfreemarket}. \par
\begin{defi}\label{defDt} 
	Let $\cP$ be an $N\times J$ matrix of real entries. Any such matrix is represented by its rows:
	$$\cP= (\vpp_1, \ldots \vpp_N)$$
	where $\vpp_i\in \R^J$. The set of all these matrices is denoted by $\Dt(N,J)$.
	\par
	$\D(N,J)$ as given in Definition \ref{defD} and $\Dt(N,J)$ are considered as dual spaces, under the natural duality action
	$$ \cP:\cM:=\sum_{i \in\I}\vpp_i\cdot{\vM_i}\equiv \sum_{i=1}^N\sum_{j=1}^Jp_i^{(j)}m^{(j)}_i$$
	where $\vpp\cdot\vM$ is the canonical inner product in $\R^{J}$.
\end{defi}
Let $\vzeta:=(\zeta_1, \ldots \zeta_{J}):X\rightarrow \R_+^{J}$  verifies assumption (\ref{vzetadef}). Define, for $x\in X$ and $\cP=(\vpp_1, \ldots \vpp_N)$:
\be\label{xi0} \xi^0_\zeta(x,\cP):= \max_{i\in\I} \vpp_i\cdot\vzeta(x): X\times\Dt(N,J)\rightarrow \R\ee  
\be \label{xi0p} \xi^{0,+}_\zeta(x, \cP):= \xi^0_\zeta(x, \cP)\vee 0: X\times \Dt(N,J)\rightarrow\R_+\ee
\be\label{0Xi0} \Xee(\cP):= \mu(\xi^0_\zeta(\cdot, \cP) ):\Dt(N,J)\rightarrow\R \ \ . \ee\index{$\Xee$}
\be\label{0Xi0+} \Xeep(\cP):= \mu( \xi^{0,+}_\zeta(\cdot, \cP) ):\Dt(N,J)\rightarrow\R_+ \ . \ee\index{$\Xeep$}
Here we see a neat, equivalent definition of the feasibility sets (\ref{defvmSw}):
\begin{theorem}\label{main1} 
	$\cM\in \vmS_N(\bar\mu)$ (resp.  $\cM\in \uvmS_N(\bar\mu)$) if and only if
	\be\label{ineq0} \ a) \ \  \Xee(\cP)-\cP:\cM\geq 0 \ \ \ ; \ \text{resp.}  \  b) \  \ \Xeep(\cP)-\cP:\cM\geq 0 \ee   
	for any $\cP\in\Dt(N,J)$. Moreover, $\Xee$ and $\Xeep$ are the {\em support functions}\footnote{See Appendix \ref{AppA6}} of $\vmS_N(\bar\mu)$ and $\uvmS_N(\bar\mu)$, respectively:
	\be\label{ineq11} \sup_{\cM\in\vmS_N(\bar\mu)}\cP:\cM= \Xee(\cP) \ \ \ ; \ \ \ \sup_{\cM\in\uvmS_N(\bar\mu)}\cP:\cM= \Xeep(\cP)\ee
	holds for any $\cP\in\Dt(N,J)$.
\end{theorem}
Recall Definitions \ref{USC} and  \ref{defPH} in Appendix \ref{Convexfunctins}.
\begin{lemma}\label{poshom}
	$\Xee$ and $\Xeep$ are convex, continuous  and positively homogeneous of order 1  functions on $\Dt(N,J)$.
\end{lemma}

\begin{proof} By Proposition  \ref{propA4}   and (\ref{xi0}, \ref{xi0p}) we obtain that  $\cP \mapsto\xi^0_\zeta$ and $\cP\mapsto\xi^{0,+}_\zeta$  are convex (as functions of $\Dt(N,J)$)  for any $x\in X$. Indeed,   they are maximizers of linear (affine) functions on $\Dt(N,J)$.  Hence  $\Xee$, $\Xeep$  are convex on $\Dt(N,J)$ as well from definition (\ref{0Xi0}, \ref{0Xi0+}). Since $X$ is compact then  $\Xee, \Xeep$ are finite valued for any $\cP\in \Dt(N,J)$ so the essential domains \index{essential domain} of both coincides with $\Dt(N,J)$. Hence both functions are continuous by Proposition \ref{propinnercon}.  Both functions are positive homogeneous of order one by  definition.
\end{proof}
\begin{cor}\label{main2cor}
	$\cM$ is an inner point of $\uvmS_N(\bar\mu)$ if and only if $\cP=0$ is the {\em only}  case where  \em{(\ref{ineq0}-b)} holds with an  equality. 
\end{cor}
\noindent
\begin{proof} (of Corollary):  Since  $\Xeep(\cP) -\cP:\cM$ is positively  homogeneous by Lemma \ref{poshom} it follows that $\cP=0$ is a minimizer of (\ref{ineq0}-b) for any $\cM\in \uvmS_N(\bar\mu)$. If it is a strict minimizer then $\Xeep(\cP) -\cP:\cM>0$ for any $\cP\not=0$. Since  $\Xeep$ is continuous at any point in its essential domain \index{essential domain} (in particular at  $\cP=0$),   there exists $\alpha>0$ such that
	$$\Xee(\cP)-\cP:\cM>\alpha \ \ \text{for any} \ \ \cP\in\Dt(N,J) \ \ , \ \  |\cP|:=\sum_{i\in\I}|\vpp_i|=1 \ . $$
	Hence there exists an open  neighborhood $O$ of $\cM\in  \D(N,J)$ and $\alpha^{'}>0$ for which $\Xep(\cP) - \cP:\cM^{'}>\alpha^{'}$ for any $\cM^{'}\in O$ and $|\cP|=1$. Hence   \\
	$\Xep(\cP) -\cP:\cM^{'}>0$  for {\it any} $\cP\not=0$.  Hence
	$\cM^{'}\in \uvmS_N(\bar\mu)$ for any $\cM^{'}\in O$, hence $\cM$ is an inner point of $\uvmS_N(\bar\mu)$ by Theorem~\ref{main1}.
	\par
	Conversely, assume  there exists $\cP_0\not=0$ for which $\Xep(\cP_0) - \cP_0:\cM=0$. Then $\Xep(\cP_0) - \cP_0:\cM^{'}<0$ for any $\cM^{'}$ for which $\cP_0:(\cM-\cM^{'})<0$. By Theorem \ref{main1} it follows  that $\cM^{'}\not\in\uvmS_N(\bar\mu)$ so $\cM$ is not an inner point of $\vmS_N(\bar\mu)$.
\end{proof}
Since $\vpp\cdot{\vM_i}=\mu_i(\vpp\cdot\vzeta)$  it follows from (\ref{sigmaleqmu1},\ref{xi0}, \ref{0Xi0}) that $\underline{\cP}:\cM=\Xee(\underline{\cP})$ where
$\underline{\cP}=(\vpp, \ldots,  \vpp)$. If, in addition, $\vpp\in\R^J_+$ then $\Xep(\underline{\cP})=\Xee(\underline{\cP})$ so
$\Xep(\underline{\cP})-\underline{\cP}:\cM=0$ for any $\cM\in\vmS_N(\bar\mu)$ and any such $\underline{\cP}$. From Corollary \ref{main2cor} we obtain that $\vmS_N(\bar\mu)$ is not contain any point  in the interior of $\uvmS_N(\bar\mu)$:
\begin{cor}\label{wSsubsetwpus}
	$$\vmS_N(\bar\mu)\subset \partial\uvmS_N(\bar\mu) \ . $$
	In particular, $\vmS_N(\bar\mu)$ has no interior points in $\D(N,J)$.
\end{cor}
\subsection{Proof of Theorem \ref{main1}}

\begin{lemma}\label{lemma3.3}
	If $\cM\in \vmS_N(\bar\mu)$ then
	$$ \Xee(\cP)-\cP:\cM\geq 0$$
	for any $\cP\in\Dt(N,J)$. Likewise, if $\cM\in \uvmS_N(\bar\mu)$ then
	$$ \Xeep(\cP)-\cP:\cM\geq 0$$
	for any $\cP\in\Dt(N,J)$.
\end{lemma}
Recalling Proposition \ref{propA11} we can formulate Lemma \ref{lemma3.3} as follows: The sets $\vmS_N(\bar\mu)$ (resp. $\uvmS_N(\bar\mu)$) are  contained in the essential domains  \index{essential domain}of the Legendre transforms of $\Xee$ (resp. $\Xeep$).
\begin{proof}
	Assume $\cM\in \uvmS_N(\bar\mu)$.  By definition, there exists $\vmu\in\wSPzeta_{\{\vM\}}$ such that
	$\mu_i(\zeta_j)=M^{(j)}_i$. Also, since $\sum_{i\in\I} \mu_i\leq \mu$
	$$ \Xeep(\cP) = \mu(\xi^{0,+}_\zeta(\cdot, \cP)) \geq \sum_{i\in\I}\mu_i(\xi^{0,+}_\zeta(\cdot,\cP) )$$
	while from (\ref{xi0p}) $\xi^{0,+}_\zeta(x, \cP)\geq \vpp_i\cdot\vzeta(x)$  so
	$$ \Xeep(\cP) \geq  \sum_{i\in\I}\vpp_i\cdot\mu_i(\vzeta) = \cP:\cM \ . $$
	The case for $\Xee$ is proved similarly.
\end{proof}

In order to prove the second direction of Theorem~\ref{main1} we need the following definition of {\it regularized maximizer}:
\begin{defi}
	Let $\vec{a}:=(a_1, \ldots a_N)\in \R^N$. Then, for $\eps>0$,
	$$ max_\eps(\vec{a}):= \eps\ln\left(\sum_{i\in\I} e^{a_i/\eps}\right)$$
\end{defi}
\begin{lemma}\label{maxreg}
	For any $\eps>0$,  $max_\eps(\cdot)$ is a smooth convex function on $\R^N$. In addition  $\max_{\eps_1}(\vec{a})\geq\max_{\eps_2}(\vec{a})\geq  \max_{i\in\I}( a_i)$ for  any $\vec{a}\in\R^N$, $\eps_1>\eps_2>0$ and
	\be\label{limmaxeps}\lim_{\eps\searrow 0}max_\eps(\vec{a})= \max_{i\in \I}  (a_i) \ . \ee
\end{lemma} 
\begin{proof}
	Consider
	\be\label{maxeps}  \max_{\vec{\beta}}\left\{ -\eps\sum_{i\in\I} \beta_i\ln\beta_i + \vec{\beta}\cdot\vec{a}\right\}\ee
	where the maximum is taken on the simplex 
	$$\vec{\beta}:=(\beta_1\ldots \beta_N), \ \  \beta_i\geq 0, \ \ \sum_1^N\beta_i=1 \  .$$
	Note that (\ref{maxeps}) is {\it strictly concave} function, and its unique   maximizer is
	$$ \beta^0_i=\frac{e^{a_i/\eps}}{\sum_je^{a_j/\eps}} < 1$$
	for $i\in\I$.  Substitute this maximizer in (\ref{maxeps})) to obtain $\max_\eps(\vec{a})$.  The convexity of $\max_\eps(\cdot)$ follows from  \ref{Convexfunctins}-\ref{cf3}.  Since $\beta_i\in[0,1]$,
	$-\sum_{i\in\I}\beta_i\ln\beta_i \geq 0$ so  the term in brackets in (\ref{maxeps}) is monotone non-decreasing in $\eps>0$.
	In addition, $-\sum_{i\in\I}\beta_i\ln\beta_i $ is maximized at $\beta_i=1/N$ (show it!), so $0\leq -\sum_{i\in\I}\beta_i\ln\beta_i \leq \ln N$. It follows that
	$$\max_\eps(\vec{a}) \in [\max_{i\in\I} (a_i),\max_{i\in\I}  (a_i)+\eps\ln N] \ , $$
	and  (\ref{limmaxeps}) follows.
\end{proof}
\begin{defi}\label{def32}
	\be\label{eps11} \xi^\eps_\zeta(x, \cP):= max_{\eps}\left( \vpp_1\cdot\vzeta(x), \ldots \vpp_N\cdot\vzeta(x)\right): X\times\Dt(N,J)\rightarrow \R\ee
	\be\label{eps22} \Xeeps(\cP):= \mu( \xi^\eps_\zeta(\cdot, \cP)): \Dt(N,J)\rightarrow\R \ \ . \ee\index{$\Xeeps$} 
	Also, for each $\cP\in\Dt(N,J)$ and $i\in\I$ set
	\be\label{defmuieps} \mu_i^{(\cP)}(dx):= \frac{e^{\vpp_i\cdot\vzeta(x)/\eps}}{\sum_{k\in\I} e^{\vpp_k\cdot\vzeta(x)/\eps}} \mu(dx)\ee
	
	Likewise
	\be\label{eps11+} \xi^{\eps, +}_\zeta(x, \cP):= max_{\eps}\left( \vpp_1\cdot\vzeta(x), \ldots \vpp_N\cdot\vzeta(x), 0\right): X\times\Dt(N,J)\rightarrow \R_+\ee
	\be\label{eps22+}\Xeepsp(\cP):= \mu\left(\xi^{\eps,+}_\zeta(\cdot, \cP)\right): \Dt(N,J)\rightarrow\R_+ \ \ . \ee\index{$\Xeepsp$}
	and
	\be\label{defmuieps+} \mu_i^{(\cP,+)}(dx):= \frac{e^{\vpp_i\cdot\vzeta(x)/\eps}}{1+\sum_{k\in\I} e^{\vpp_k\cdot\vzeta(x)/\eps}} \mu(dx)\ee
\end{defi}
Since $\max_\eps(\cdot)$ is smooth  due to lemma~\ref{maxreg}, Lemma \ref{regXi} below  follows  from the above definition via an explicit differentiation.
\begin{lemma}\label{regXi}
	For each $\eps>0$, $\Xeeps$ (resp. $\Xeepsp$) is a  convex, smooth function on $\Dt(N,J)$. In addition
	$$ \frac{\partial \Xeeps(\cP)}{\partial p_i^{(j)}}=
	\mu_i^{(\cP)}(\zeta_j)\
	\ \ resp. \ \ \frac{\partial \Xeepsp(\cP)}{\partial p_i^{(j)}}= \mu_i^{(\cP,+)}(\zeta_j)  \ . $$
\end{lemma}
The proof of Theorem~\ref{main1} is obtained from Lemma \ref{correg} below, whose  proof is an easy  exercise, using (\ref{maxeps})
\begin{lemma}\label{correg}
	For any $\eps, \delta>0$ and $\cM\in\D(N,J)$
	\be\label{Xicor}\cP\rightarrow \Xeeps(\cP) + \frac{\delta}{2}|\cP|^2-\cP:\cM\ee
	is a  strictly convex function on $\Dt(N,J)$. In addition
	\be\label{XiepsgeqXi0} \Xeeps(\cP) \geq \Xee(\cP)\ee
	so, if  {\em (\ref{ineq0})} is satisfied,  then
	$$ \Xeeps(\cP)-\cP:\cM\geq 0\ . $$
	for any $\cP\in\Dt(N,J)$.  The same statement holds for $\Xeepsp$ as well.
\end{lemma}

\begin{proof} {\it (of Theorem~{\em\ref{main1}})} \\
	Lemma \ref{lemma3.3} gives us the " only if" direction.
	\par
	From Lemma~\ref{correg} we obtain the existence of a minimizer $\cP^{\eps,\delta}\in\Dt(N,J)$ of (\ref{Xicor}) for any $\eps,\delta>0$, provided (\ref{ineq0}) holds.  Moreover, from Lemma~\ref{regXi} we also get that this  minimizer satisfies
	\be\label{mi=delXi}{\vM_i}= \frac{\partial}{\partial \vpp_i^{\eps,\delta}}\Xeeps+\delta \vpp^{\eps,\delta}_i = \mu_i^{(\cP^{\eps,\delta})}(\vzeta) +\delta \vpp^{\eps,\delta}_i\ee
	By convexity of $\Xeeps$:
	\be\label{fromconvex}\Xeeps(\vec{0})\geq  \Xeeps(\cP)- \cP:\nabla\Xeeps(\cP) \ \ \ \text{for any} \ \ \cP\in\Dt(N,J) \ .\ee
	Apply $\vpp_i^{\eps,\delta}$  to   (\ref{mi=delXi}), use (\ref{fromconvex}) and sum over $i=1, \ldots N$, recalling $\cP^{\eps,\delta}=\left( \vpp_1^{\eps,\delta}, \ldots \vpp_N^{\eps,\delta}\right)$, $\cM:=(\vM_1, \ldots \vM_N)$:
	\be\label{ineq1}\cP^{\eps,\delta}:\nabla\Xeeps(\cP^{\eps,\delta})+ \delta \left| \cP^{\eps,\delta}\right|^2- \cP^{\eps,\delta}:\cM
	=\geq \Xeeps(\cP^{\eps,\delta}) - \Xeeps(\vec{0}) + \delta \left| \cP^{\eps,\delta}\right|^2-\cP^{\eps,\delta}:\cM\ee
	It follows from (\ref{ineq0}, \ref{XiepsgeqXi0},\ref{ineq1}) that
	$$  - \Xeeps(\vec{0}) + \delta \left| \cP^{\eps,\delta}\right|^2\leq 0$$
	hence
	$$ \delta\left| \cP^{\eps,\delta}\right| \leq \sqrt{\delta}\sqrt{\Xeeps(\vec{0})} \ . $$
	Hence (\ref{mi=delXi}) implies
	$$ \lim_{\delta\rightarrow 0} \mu_i^{(\cP^{\eps,\delta})}( \vzeta)  = {\vM_i} $$
	By compactness of $C^*(X)$ and since $\sum_{i\in\I}\mu_i^{(\cP^{\eps,\delta})}=\mu$ via (\ref{defmuieps} ) we can choose a subsequence $\delta\rightarrow 0$ along which the limits
	$$ \lim_{\delta\rightarrow 0}  \mu_i^{(\cP^{\eps,\delta})} := \mu_i^{(\cP^\eps)}$$
	holds for any $i\in\I$. It follows that
	$$ \sum_{i\in\I}\mu_i^{(\cP^{\eps})}=\mu \ \ \ ; \ \ \ \mu_i^{(\cP^\eps)}(\vzeta)={\vM_i} \ $$
	for any $i=1, \ldots N$,
	hence $\vmu\in\wP_{\cM}$ so $\cM\in\vmS_N(\bar\mu)$. \par
	The proof for $\cM\in \uvmS_N(\bar\mu)$ is analogous.
	\par\noindent
	Finally, the proof of (\ref{ineq11}) follows from (\ref{ineq0}) and Proposition \ref{propA11}, taking advantage on the homogeneity of $\Xee$ (resp. the positive homogeneity of  $\Xeep$).
\end{proof} 
\section{Dominance}\label{secDominance}
We now consider generalized (sub)partitions from another point of view.

A {\em Stochastic  $N-$ Matrix}  $S=\{s_i^k\}$ is an $N\times N$ matrix of non-negative entries such that
\be\label{markov} \sum_{k=1}^Ns_i^k=1 \ \ \forall i=1, \ldots N \ . \ee
We observe that if $\vmu$ is a (sub)partition then
$$ S\vmu:= \left(\sum_{i=1}^Ns_i^1\mu_i, \ldots \sum_{i=1}^Ns_i^N\mu_i\right)$$
is a (sub)partition as well. It follows from Definition \ref{defvmSw} and (\ref{defvmSw}) that if $\cM\in\vmS_N(\bar\mu)$ (resp. $\cM\in \uvmS_N(\bar\mu)$) then
$$ S\cM\in \vmS_N(\bar\mu) \ \ \ (\text{resp. }\ \ \ S\cM\in \uvmS_N(\bar\mu)) \ . $$
Here $S\cM:= \left(S\vM^{(1)}, \ldots S\vM^{(J)}\right)$ where $\cM= (\vM^{(1)}, \ldots \vM^{(J)})$, ${\vM^{(j)}}\in \R_+^N$.
\begin{defi} Let $\cM, \underline{\cM}\in\D(N,J)$. If there exists a stochastic matrix $S$ such that $\underline{\cM}=S\cM$
	then $\cM$ is said to  dominate $\underline{\cM}$ ($\cM \succ\underline{\cM})$.
\end{defi}
Assume    $S$ is such a stochastic matrix satisfying $\underline{\cM}=S\cM$.
Let
\be\label{MiMj}{\vM_i}:= (m_i^{(1)}, \ldots m_i^{(J)})\in\R_+^J, \ \ m_i:=\sum_{j=1}^Jm^{(j)}_i \ \ . \ee
(and similarly for $\vec{\underline{m}}_i, {\underline{m}}_i$).  The following
\be\label{jen0}\vec{\underline{m}}_i=\sum_{k=1}^Ns_k^i \vM_k \  \ee
holds (as an equality in $\R^J$). Summing the components in $\R^J$ of both sides of (\ref{jen0}) and dividing by $\underline{m}_i$  we obtain
\be\label{jen1}\sum_{k=1}^Ns_k^i \frac{m_k}{\underline{m}_i}=1 \ . \ee
Dividing (\ref{jen0}) by $\underline{m}_i$ we obtain
\be\label{jen2}\frac{\vec{\underline{m}}_i}{\underline{m}_i}=\sum_{k=1}^N \left(s_k^i \frac{m_k}{\underline{m}_i}\right)\frac{\vM_k}{m_k} \ . \ee
The Jensen's inequality and (\ref{jen1}, \ref{jen2}) imply
$$ F\left(\frac{\vec{\underline{m}}_i}{\underline{m}_i}\right) \leq \sum_{k=1}^N \left(s_k^i \frac{m_k}{\underline{m}_i}\right)F\left(\frac{\vM_k}{m_k}\right)$$
for any convex function $F:\R_+^J\rightarrow\R$.
Multiplying the above by $\underline{m}_i$ and summing over $ i=1, \ldots N$ we get, using (\ref{markov})
\be\label{jen4}\sum_{i=1}^Nm_iF\left( \frac{{\vM_i}}{m_i}\right)\geq \sum_{i=1}^N\underline{m}_iF\left( \frac{\vec{\underline{m}}_i}{\underline{m}_i}\right) \  .\ee
We obtained that if $\cM \succ\underline{\cM}$ then (\ref{jen4}) holds for any convex function on $\R_+^J$. It can be shown, in fact, that the reversed  direction holds as well:
\begin{prop}\label{propdom1}
	$\cM \succ\underline{\cM}$  iff (\ref{jen4}) holds for any convex $F:\R_+^J\rightarrow \R$.
\end{prop}
The definition of dominance introduced \index{dominance} above is an extension of a definition given by H.Joe (\cite{joe1}, \cite{joe2}). In these papers Joe introduced the notion of $w-$dominance  on $\R_+^N$ as follows: For a given a  vector $\vw\in \R_{++}^N$, the vector $\vx\in\R_+^N$ is said to $\vw-$dominant $\vy\in\R_+^N$ ($\vx \succ_w\vy$) iff there exists a stochastic matrix $S$ preserving $\vw$ and transporting $\vx$ to $\vy$, i.e
$$ S\vx=\vy, \ \ \  \text{and} \ \ \ S\vw=\vw \ . $$
Evidently, it is a special case of our definition where $J=2$.  The condition of $\vx\succ_w\vy$ is shown to be equivalent to
\be\label{domF12} \sum_{j=1}^Nw_j\psi\left(\frac{y_j}{w_j}\right)\leq \sum_{j=1}^N w_j \psi\left(\frac{x_j}{w_j}\right) \ \ee
for any convex function $\psi:\R_+\rightarrow\R$. The reader should observe that (\ref{domF12}) follows from (\ref{jen4}) in the case $J=2$ upon defining $\vM_1=\vx$, $\vec{\underline{m}}_1=\vy$, $\vM_2=\lambda\vw-\vx$, $\vec{\underline{m}}_2=\lambda\vw-\vy$ (where $\lambda$ is large enough such that both $\vM_2, \vec{\underline{m}}_2\in\R_+^J$) \footnote{Since  $\vw\in\R_{++}^J$ by assumption}
and setting $\psi(x)=F(s/\lambda, 1-s/\lambda)$.
\par
We now present  a generalization of  Proposition \ref{propdom1}:

Let $\vzeta:=(\zeta_1, \ldots \zeta_N):X\rightarrow \R_+^{J}$ be a measurable function,  as defined in (\ref{vzetadef}).
Consider $\cM:=\{m^{(j)}_i\}$ satisfying (\ref{sigmaleqmu0}, \ref{sigmaleqmu}). Recall  (\ref{MiMj}).
\begin{theorem}\label{thFzeta}
	$\underline{\cM}\in\vmS_N(\bar\mu)$ (resp. $\underline{\cM}\in\uvmS_N(\bar\mu)$) if and only if
	\be\label{Fzetageq}\mu(F(\vzeta)) \geq \sum_{i\in\I} \underline{m}_iF\left( \frac{\vec{\underline{m}}_i}{\underline{m}_i}\right) \ .  \ee
	is satisfied for any convex  $F:\R_+^{J}\rightarrow \R$ (resp. $F:\R_+^{J}\rightarrow \R_+$).
\end{theorem}
Recall  Remark \ref{noatom}.
Note that if we choose $X=\I:=\{1, \ldots N\}$, $\mu(\{i\}):= m_i$, the "deterministic" partition $\mu_i(\{k\}):= \mu(\{i\})\delta_{i,k}$ and $m^{(j)}_i:= \zeta_j(\{i\})\mu(\{i\})$  then   Theorem \ref{thFzeta} implies Proposition \ref{propdom1}.
\begin{proof}
	By definition of $\vmS_N(\bar\mu)$ there exists a weak partition \index{weak partition}  $\vmu=(\mu_1, \ldots \mu_N)$  such that $\underline{m}_i^{(j)}=\mu_i(\zeta_j)$. In particular $\underline{m}_i=\mu_i(X)$. Since $F$ is convex we apply Jensen's's inequality
	\be\label{Fzetageqi}\mu_i(F(\vzeta))\geq \mu_i(X)F \left(\frac{\int\vzeta d\mu_i}{\mu_i(X)}\right):=\underline{m}_iF\left( \frac{\underline{\vM}_i}{\underline{m}_i}\right) \ . \ee
	Summing over $i\in\I$ and using  $\mu=\sum_{i\in\I}\mu_i$ we obtain the result.
	
	If $\cM\in\uvmS_N(\bar\mu)$ then there exists a weak subpartition \index{weak subpartition}  $\vmu=(\mu_1, \ldots \mu_N)$  such that $m^{(j)}_i=\mu_i(\zeta_j)$ and $\sum_{i=1}^N\mu_i\leq \mu$. In that case the inequality (\ref{Fzetageq}) still follows from (\ref{Fzetageqi}), taking advantage on $F\geq 0$.
	
	Suppose now $\underline{\cM}:= (\underline{\vM}^{(1)}, \ldots \underline{\vM}^{(N)})\not\in\vmS_N(\bar\mu)$. By Lemma \ref{lemma3.3} there exists $\cP=(\vpp_1, \ldots \vpp_N)\in\Dt(N,J)$ such that
	\be\label{q1}\Xee(\cP)<\cP:\underline{\cM}:= \sum_{i\in\I} \vpp_i\cdot\underline{\vM}_i \ . \ee
	Define the function $F=F(\vzeta):\R_+^{J}\rightarrow \R$:
	\be\label{q2}F(\vzeta):= \max_{i\in\I}\vpp_i\cdot\zeta_i \  \ ( \text{resp.} \ \ F_+(\vzeta):= \max_{i\in\I}[\vpp_i\cdot\zeta_i]_+\equiv  F(\zeta)\vee 0) \ .  \ee
	So, $F, F_+$ are a convex function on $\R_+^{J}$. By definition (\ref{0Xi0})
	\be\label{q3} \Xee(\cP)=\mu(F(\vzeta)) \ \ \  ( \text{resp.} \ \ \Xep(\cP)=\mu(F_+(\vzeta))) \ . \ee
	Next, using (\ref{MiMj}) we can write $\underline{\cM}$ as
	$$ \underline{\cM}=\left(  \underline{m}_1 \frac{\underline{\vM}_1}{\underline{m}_1}, \ldots \underline{m}_N\frac{\underline{\vM}_N}{\underline{m}_N} \right)$$
	Then
	\be\label{q4} \cP:\underline\cM=\sum_{i\in\I}\underline{m}_i\vpp_i\cdot \frac{\underline{\vM}_i}{\underline{m}_i} \ . \ee
	By definition
	\be\label{q5} F\left(\frac{\underline{\vM}_i}{\underline{m}_i}\right)\geq \vpp_i\cdot \frac{\underline{\vM}_i}{\underline{m}_i} \ \ \ \text{resp.} \ \ \ F_+\left(\frac{\underline{\vM}_i}{\underline{m}_i}\right)\geq \left[\vpp_i\cdot \frac{\underline{\vM}_i}{\underline{m}_i}\right]_+
	\ee
	for any $i\in\I$. From (\ref{q1}, \ref{q3}-\ref{q5}) we obtain a contradiction to (\ref{Fzetageq}).
\end{proof}
Let us extend the definition of dominance from the set of $N\times J$ matrices $\D(N,J)$ to the set of $\R_+^J$ valued function on the general measure space $X$:
\begin{defi}\label{gendominance} 
	Let $\bar\mu=(\mu^{(1)}, \ldots \mu^{(J)})$, $\bar\nu=(\nu^{(1)}, \ldots \nu^{(J)})$ be a pair of $\R^J-$valued measures on measure spaces $X,Y$ respectively. $(X,\bar\mu)\succ(Y,\bar\nu)$ iff there exists a measure $\pi\in{\cal M}_+(X\times Y)$ such that
	$$ \int_{x\in X}\frac{d\mu^{(j)}}{d\mu}(x)\pi(dxdy)=\nu^{(j)}(dy) \ \ ; \ \ j=1, \ldots J$$
	where $\mu=\sum_1^J\mu^{(j)}$.
\end{defi}

The following Theorem is an extension of Theorem \ref{thFzeta}. Some version of it appears in  Blackwell \cite{lamport94}:
\begin{theorem}\label{Fdominance}
	$(X,\bar\mu)\succ(Y,\bar\nu)$ iff
	\be\label{domcapF} \int_XF\left(\frac{d\bar\mu}{d\mu}\right)d\mu \geq  \int_XF\left(\frac{d\bar\nu}{d\nu}\right)d\nu\ , \ee
	for any convex $F:\R^J\rightarrow \R$. Here
	$\nu:=\sum_{j=1}^J \nu^{(j)}$.
\end{theorem}
Letting $F(\vec x):= \pm\vec{1}\cdot \vec{x}$ we obtain from  Theorem \ref{Fdominance}
\begin{cor}\label{coreqwe}
	A {\em necessary} condition for the dominance $(X,\bar\mu)\succ (Y,\bar\nu)$ is the {\em balance condition}\index{dominance}
	$$\boxed{\bar\mu(X)=\bar\nu(Y)} \ . $$
\end{cor}

By Theorem \ref{thFzeta} (and its special case in Proposition \ref{propdom1}) we obtain the following characterization:
\begin{cor}\label{corzetaM} .
	Let  $(X,\bar\mu)\succ(Y,\bar\nu)$. Then
	${\vmS_N(\bar\mu)}\supseteq{\vmS_N(\veta)}$ for any $N\in\mathbb{N}$.
\end{cor}
In fact, the other direction holds as well:
\begin{theorem}\label{Domfin}
	$(X,\bar\mu)\succ(Y,\bar\nu)$ if, and only if, ${\vmS_N(\bar\mu)}\supseteq{\vmS_N(\veta)}$ for any $N\in\mathbb{N}$.
\end{theorem}
\begin{defi}\label{congru}\index{weak partition}
	Two weak partitions $\vmu=(\mu_1, \ldots \mu_N)$, $\vnu=(\nu_1, \ldots \nu_N)$ of $(X,\mu)$ and $(Y, \nu)$, resp., are $\bar\mu-\bar\nu$ {\em congruent} iff \index{congruent partitions}
	$$ \int_X \frac{d\bar\mu}{d\mu} d\mu_i=\int_Y \frac{d\bar\nu}{d\nu} d\nu_i \ , \ \ i=1, \ldots N \ .  $$
	We denote this relation by $\vmu\otimes\bar\mu\sim\vnu\otimes\bar\nu$.
\end{defi}
We may now reformulate Theorem \ref{Domfin} as follows:
\begin{theorem}\label{Domfin1}
	$(X,\bar\mu)\succ(Y,\bar\nu)$ if, and only if, for any weak partition \index{weak partition}  $\vnu$ of $(Y,\bar\nu)$ there exists a partition $\vmu$ such that $\vmu, \nu$ are $\bmu-\bnu$ congruent.
\end{theorem}
\begin{proof}
	The "only if" direction is clear. For the "if" direction, let us consider a sequence of $N-$partitions $\nu^N$  of $\nu$ such that
	\be\label{123y}\int_YF \left(\frac{d\bnu}{d\nu}\right)d\nu=\lim_{N\rightarrow\infty}\sum_1^N \nu^N_i(Y) F\left(\frac{1}{\nu_i^N(Y)}\int_Y \frac{d\bnu}{d\nu}d\nu^N_i\right)\  \ee
	for any continuous function $F$. 
	Such a sequence can be obtain, for example, by taking fine {\em strong} partitions $\{A^N_i\}$ of $Y$ such that $\nu^N=\mu\lfloor A_i^N$ (c.f Chapter \ref{optimstrong}). 
	For any such partition 
	let $\mu^N$ the  congruent $\bmu-\bnu$ partition $\vmu$ of $\mu$. Let now $F$ be a convex function. By Jensen's inequality\index{congruent partitions}
	$$ \int_X F\left(\frac{d\bmu}{d\mu}\right)d\mu^N_i  \geq \mu^N_i(X)F\left(\frac{1}{\mu^N_i(X)}\int_X\frac{d\bmu}{d\mu}d\mu^N_i\right)$$
	while 
	$$ \mu^N_i(X)F\left(\frac{1}{\mu^N_i(X)}\int_X\frac{d\bmu}{d\mu}d\mu^N_i\right)= \nu^N_i(Y) F\left(\frac{1}{\nu_i^N(Y)} \int_Y \frac{d\bnu}{d\nu}d\nu^N_i\right)$$ by congruency. 
	 Using $\mu=\sum_1^N\mu_i^N$ and summing over $i$ we get the inequality 
	(\ref{domcapF})  via (\ref{123y}).

\end{proof}
\subsection{Minimal elements}
Let $\lambda\in{\cal M}_+(Y)$ and $\lambda(Y)=\mu(X)$. By Theorem \ref{Fdominance} and the Jensen's inequality we obtain:
\begin{prop}\label{propbarX=barl}
	$$ \bar\mu(X)\lambda \prec \bar\mu \ , $$
	where $\bar\mu$ as in (\ref{vzetadef}). 
\end{prop}
We can apply this proposition to the discrete spaces $X=Y=\I:=\{1, \ldots N\}$.
Let $m_i:=\mu(\{i\})$, $m^{(j)}:= \mu^{(j)}(X)$, $\vM=(m_1, \ldots m_N)$ and $\bar m=(m^{(1)}, \ldots m^{(J)})$.  Consider the set
$$\Pi(\vM, \bar{m}):=
\left\{  \cM\in \D(N,J);  \sum_{i\in\I} m^{(j)}_i= m^{(j)} \ \forall j\in \J ; \ \ \sum_{j\in\J}m^{(j)}_i=m_i\ \ \forall i\in\I \right\} \ .  $$
It follows from Proposition \ref{propbarX=barl} that
\begin{cor} 
	$\lambda^{(j)}(\{i\}):=\{ m_im^{(j)}/m\}$ is the minimal point in $\Pi(\vM, \bar{m})$ with respect to the order relation $\succ$, where $m=\sum_{i\in\I}m_i$. That is, for any $\bar\nu$ satisfying $\nu^{(j)}(\{i\})=m_i^{(j)}$, $\bar\lambda \prec\bar\nu$.
\end{cor}

\chapter{Strong  multi-partitions}\label{optimstrong}

\section{Strong partitions as extreme  points}\label{firstsecinoptimstrong}
A {\em Strong $N-$subpartition} $\vA$ of $X$ is a subpartition of $X$ into $N$ measurable subsets which are {\em essentially disjoint}:
$$ \vA:=(A_1, \ldots A_N), \ A_i\in {\cal A}(X) ; \ \cup_{i\in\I} A_i\subset X, \ \  \ \mu(A_j\cap A_i)=0 \ \text{for} \ i\not= j \ . $$
The set of all strong subpartitions of $X$ is denoted by
\be\label{strongSP} \vSP:= \left\{ \vA; \  \vA \ \ \text{is a} \ \ \text{strong} \ N \ \text{sub-partition of} \ X  \right\} . \ee
A {\em strong $N-$ partition} is a strong $N-$subpartition  of $X$ which satisfies $\mu(\cup_1^N A_i)= \mu(X)$.   We denote the set of all strong $N-$partitions by $\vP$.
\par\noindent
We shall omit the index $N$ where no confusion is expected.

For any $\K\subset\D(N,J)$, the set $\K-$valued strong  subpartitions is
\be\label{SstrongM}  \sSPzeta_{\K}:=\left\{ \vec{A}\in \vSP, \ \ \
\left(\int_{A_1}\vzeta d\mu, \ldots, \int_{A_N}\vzeta d\mu\right)\in \K \ \right\}.. \ee
and the set of strong $\K-$valued  partitions is $\sPzeta_{\K}$.

These definitions should be compared with (\ref{KSPWMxi}).
Note that $\sPzeta_\K$ can be embedded in $\wPzeta_{\K}$ in a natural way. Just define $\vmu=(\mu_1, \ldots \mu_N)\in \wPzeta_{\K}$ by
$\mu_i:= \mu\lfloor A_i$, i.e. the restriction of $\mu$   to $A_i$. Likewise, $\sSPzeta_\K$ is embedded in $\wSPzeta_{\K}$.
\par
Motivated by the above we extend the definition of $\cM_\zeta$ (\ref{muzeta}) to strong (sub)partitions:
\be\label{muzetaS}\cM_\zeta(\vA):=\{ \int_{A_i}\zeta_jd\mu\} \in \D(N,J) \ . \ee
Now, we are in a position generalize (\ref{defvmSw}) to  the {\em Strong Feasibility sets}
\begin{multline}\label{vmSdef}\vmS^s_N(\bar\mu):=\left\{ \cM\in\D(N,J); \ \ \sPzeta_{\cM}\not=\emptyset\right\} \ ; \\  \uvmS_N^s(\bar\mu):=\left\{ \cM\in\D(N,J); \ \ \sSPzeta_{\cM}\not=\emptyset\right\}\end{multline}
By the remark above we immediately observe that
\be\label{vmSineq}\vmS^s_N(\bar\mu)\subseteq \vmS_N(\bar\mu) \ \ \text{and} \ \ \uvmS^s_N(\bar\mu)\subseteq \uvmS_N(\bar\mu)\ee
(recall (\ref{defvmSw})).
These inclusions are, in fact,   equalities:
\begin{theorem}\label{weak=strong}
	$$\vmS^s_N(\bar\mu)
	= \vmS_N(\bar\mu) \  \  \  \text{and} \  \uvmS^s_N(\bar\mu)
	=  \uvmS_N(\bar\mu)$$
\end{theorem}
Thus  we omit, from now on, the index $s$ form  $\vmS^s_N(\bar\mu)$ and  $\uvmS^s_N(\bar\mu)$.

   
\begin{proof} We  have to prove the opposite inclusion of (\ref{vmSineq}).   If $\cM\in \vmS_N(\bar\mu)$ then   $\wPzeta_{\{\cM\}}$ is not empty.
	By Radon-Nikodym \index{Radon-Nikodym derivative}  Theorem,    any $\vec{\mu}= (\mu_1, \ldots \mu_N)\in
	\wPzeta_{\{\cM\}}$   is characterized by   $\vh= (h_1, \ldots h_N)$ where $h_i$ are the Radon-Nikodym derivatives of $\mu_i$ with respect to $\mu$, namely $\mu_i=h_i\mu$. Since $\mu_i\leq \mu$ then 
	$0\leq h_i\leq 1$ $\mu$ a.e. on $X$.   Now $\wPzeta_{\{\cM\}}$ is convex and compact  in the weak*  topology \index{weak* convergence}$C^*(X)$ (c.f Appendix \ref{wconofme}) .
	By Krein-Milman Theorem \index{Krein-Milman theorem} (see Appendix \ref{AppA1}) there exists an extreme  point $\vmu$ in   $\wPzeta_{\{\cM\}}$. We  show that \index{strong  (deterministic) partition}
	\index{extreme point} \index{extreme point}  an extreme point is a strong partition, namely  $h_i\in\{0,1\}$ $\mu$-a.e on $X$, for all $i\in\I$. Since
	$\sum_Nh_i=1$ $\mu$-a.e on $X$ for any  $\vmu\in\wPzeta_{\{\cM\}}$, it is enough to show that for $i\not=j$, $h_i$ and $h_j$ cannot  both be positive on a set of positive $\mu$ measure.
	
	Let $\eps>0$,    $B\subset X$  measurable, $\mu(B)>0$ such that both $h_j>\eps$ and $h_i>\eps$ for some $i\not= j$.  Since $h_j+h_i\in[0,1]$ it follows also that $h_j$, $h_i$ are smaller than $1-\eps$ on $B$ as well.
	
	The  vector-Lyapunov convexity theorem  states that the range of a\index{Lyapunov convexity theorem }
	nonatomic vector measures with values in a finite dimensional space is
	compact and convex \cite{Lyap}. In particular
	the set
	$$R(B):= \left\{  \left(\int_A\vzeta d\mu_1, \ldots , \int_A\vzeta d\mu_N\right); \ A\in {\cal A}(X), A\subset B\right\}\subset \D(N,J)$$ is compact and convex. Obviously, $R(B)$ contains the zero point $\vec{0}\in\D(N,J)$ since $\emptyset\subset B$ and $\vmu(\emptyset)=\vec{0}$. Hence
	we can find a subset $C\subset B$ such that
	$$ \int_C\zeta_jd\mu^{(j)}=\frac{1}{2}\int_B \zeta_jd\mu^{(j)}, \ \ \ \int_C\zeta_id\mu_i=\frac{1}{2}\int_B \zeta_id\mu_i \ . $$ Set
	$w:= 1_{B}-2 \times 1_{C}$ where $1_A$ stands for the characteristic function \index{characteristic function}of a measurable set $A\subset X$. It follows that  $w$ is supported on $B$ , $|w|\leq 1$ 
	and \be\label{inA}\mu^{(j)}( w\zeta_j)=\mu_i(w\zeta_i)=0 \ . \ee By assumption, $h_j(x)\pm \eps w(x) \in [0,1]$ and $h_i(x)\pm \eps w(x) \in [0,1]$ for any $x\in B$.
	Let   $\vec{\nu}:= (\nu_1, \ldots \nu_N)$ where $\nu_j=\mu^{(j)}$, $\nu_i=-\mu_i$ and $\nu_k=0$ for $k\not= j,i$. Let
	$\vec{\mu}_1:= \vmu+ \eps w\vec{\nu}$,
	$\vec{\mu}_2:= \vmu-\eps w\vec{\nu}$.
	Then, by (\ref{inA}) both $\vec{\mu}_1, \vec{\mu}_2$ are in  $\wPzeta_{\{\cM\}}$ and $\vec{\mu}= \frac{1}{2}\vec{\mu}_1 + \frac{1}{2}\vec{\mu}_2$. This is in contradiction to the assumption that $\vmu$ is an extreme point. \index{extreme point} 
\end{proof}

\section{Structure of the feasibility sets}\label{optimstrong1}
\begin{assumption}\label{mainass3}   
	\be\label{mainass30}\mu\left[x\in X; \vpp\cdot\vzeta (x)=0\right]=0\ee
	for   any $\vpp\in \R^{J}-\{0\}$.
\end{assumption} 
Assumption \ref{mainass3} is the key to our next discussion on  cartels \index{cartel}  and fixed exchange ratios.

\paragraph{Cartels}: By a coalition we mean a subset of agents $\I^{'}\subset \I$ which decide to join together and form a cartel. \index{cartel} By a cartel we mean that the price vector $-\vpp_{I^{'}}\in\R^{J}$ for the list of goods is the same for all members of the coalition $\I^{'}$. That is:
$$\vpp_i\equiv\vpp_{\I^{'}} \ \ \forall \ i \in\I^{'} \ . $$
The capacity of a coalition $\I^{'}$ is just the sum of the capacities of its members
$$ \vM_{\I^{'}}=\sum_{i\in\I^{'}} \vM_i$$
The  price vector $-\vpp_{\I^{'}}$ for the coalition $\I^{'}$ is  determined by the capacity $\vM_{\I^{'}}$ of this coalition  (and these of the other coalitions, if there are any) via the  equilibrium mechanism. 
\par
\paragraph{Exchange ratio}
A fixed price ratio  emerges whenever the agent  recognizes a fixed exchange rate {\em between the goods $J$}. Suppose the agent can exchange one unite of the good $j$ for $z$ units of the good $k$. This implies that the price $-p_j$ she charges for $j$ is just $-zp_k$, where $-p_k$ is the price she charge for $k$.  More generally, if $\vzet:=(z_1, \ldots z_{J})$ is a fixed vector  such that $z=z_k/z_j$ is the exchange rate of $j$ to $k$, then the price vector charged by  this agent  is a multiple $\vpp=q\vzet$, where the {\em reference price} $q\in\R$ is determined, again, by the equilibrium mechanism.
\subsection{Coalitions and Cartels}\label{cocar}
Assume the set of agents $\I$ is grouped into a {\em coalitions ensemble},  \index{coalitions ensemble}that is, a set $\cI=\{\I_i\}$ of disjoint coalitions: 
Given such an ensemble $\cI$, no agent  can be
a member of two different coalitions, that is
$$\I_n\cap\I_{n^{'}}=\emptyset \ \text{for} \  n\not= n^{'} \ , $$
and {\em any} agent is a member of {\em some} coalition
$$\cup_{n} \I_n=\I\ . $$

\begin{defi}\label{adef}
	Given $\cM:= (\vM_1, \ldots \vM_N)\in\D(N,J)$ and  a {\em coalitions ensemble}  $\cI$ composed of $k:=|\cI|$ coalitions \index{coalitions ensemble}
	$\cI=(\I_1, \ldots \I_k)$ 
	$$\cI(\cM):= (\vM_{\I_1}, \ldots \vM_{\I_k})\in \D(N,k) \ , \ \ \vM_{\I_n}:=\sum_{i\in \I_n}\vM_i \ . $$
	
	For such a coalition's ensemble,   the {\em cartel price vector}  corresponding to  $\cP\in\Dt(N,k)$ is\index{cartel} 
	$$ \cI^*(\cP):= (\vpp_1, \ldots \vpp_{k})\in \Dt(N,J)$$ where $\vpp_i\in \Dt(N,|\I_i|)$ is the constant vector whose all components are equal to the $i-$ component of $\cP$.

	We also consider a partial order on the set of coalition's ensembles: $\cI\succeq \tilde{\cI}$ if for each component  $\I_l\in \cI$ there exists a component  $\tilde{\I}_l\in\tilde{\cI}$ such that $\I_l\subseteq \tilde{\I}_j$. In particular $|\cI|\geq |\tilde{\cI}|$. \index{coalitions ensemble}
	
	Note that the grand coalition $\cI=\{\I\}$  is the minimal one in this order, while the coalition of individuals $\cI=\{ \{i\}\in\I\}$ is the maximal one.\index{grand coalition}\index{grand coalition}
\end{defi}
By Definition \ref{adef} we obtain the following duality relation between these mappings:
\begin{lemma}\label{lemdual1}
	For any $\cM\in \D(N,J)$, any coalitions set $\cI$ and any $\cP\in \Dt(N,|\cI|)$
	$$\cI^*(\cP):\cM=\cP:\cI(\cM) \ . $$
\end{lemma}
Consider now a coalition's ensemble  $\cI=\{\I_1, \ldots \I_{|\cI|}\}$ and a strong (sub)partition $\vA_{\cI}:=\left(A_{\I_1}, \ldots A_{\I_{|\cI|}}\right)$
\begin{defi}\label{defembed}
	$\vmu$  is {\em embedded} in $\vA_{\cI}$ if
	$$Supp\left(\sum_{i\in\I_k}\mu_i \right) \subseteq  A_{\I_k} \  \ \text{for} \ k=1\ldots |\cI| \ . $$
\end{defi}

\begin{theorem}\label{uniquestrong} 
	Under Assumptions {\em \ref{mainass3}}:
	Let  $\cM\in\partial\uvmS_N(\bar\mu)$.
	Then there exists a unique maximal\footnote{That is, there is no coalition's ensemble $\tilde{\cI}\succeq\cI$ and a corresponding strong partition\index{strong  (deterministic) partition} $\vA_{\tilde{\cI}}$ corresponding to $\cI(\cM)$.}  coalition's ensemble $\cI$ and a unique strong    subpartition $\A_{\cI}$ such that any $\vmu\in \wSPzeta_{\{\cM\}}$ is embedded in $\A_{\cI}$.\index{coalitions ensemble}
	
	Moreover,  $\cI(\cM)$ is an extreme point in $\cI(\uvmS_N(\bar\mu))$. \index{extreme point} 
\end{theorem}

\par

The full proof of Theorem \ref{uniquestrong} is given in Section \ref{secproofs}.
\subsection{Fixed exchange ratio}\label{FER}
Suppose now each agent $i\in\I$ fixes the ratios of the prices she charge for the list of goods $J$. For this, she determines a vector $\vzet^{(i)}:= (z_1^{(i)}, \ldots z_J^{(i)})\in\R^{J}$. The prices $\vpp_i:=(p_i^{(1)}, \ldots p_i^{(J)})$ she charge her customers is a multiple of $\vzet^{(i)}$:
$$ \vpp_i= q_i\vzet^{(i)} \ \ \ , \ \ \ q_i\in \R \ . $$
\begin{defi}\label{Zet}
	Given  $ \cM:=(\vM_1\ldots \vM_N)\in\D(N,J)$,   let
	$$  \vZet(\cM):= \left( \vzet^{(1)}\cdot \vM_1, \ldots \vzet^{(N)}\cdot \vM_N\right)\in \R^N \ . $$
	The dual operation $\vZet^*:\R^N\mapsto\Dt(N,J)$ acting on $\vq:=(q_1,\ldots q_N)\in\R^N$ is defined by
	$$ \vZet^*(\vq):= \left( q_1\vzet^{(1)}, \ldots q_N\vzet^{(N)}\right) \ . $$
\end{defi}

The duality $\vZet, \vZet^*$ satisfies
\begin{lemma}\label{lemdual2}
	For any $\cM\in \D(N,J)$ and any $\vq\in\R^N$
	$$\vZet^*(\vq):\cM=\vq\cdot\vZet(\cM) \ . $$
\end{lemma}

By Proposition \ref{noempty},  $\vZet(\vmS_N(\bar\mu))$ (resp. $\vZet(\uvmS_N(\bar\mu))$) are closed convex sets in $\R^N$. We also observe that
\be\label{57}\partial \vZet(\uvmS_N(\bar\mu))\subseteq \vZet(\partial\uvmS_N(\bar\mu)) \ \ .  \ee 
This inclusion is strict, in general.
\begin{assumption}\label{mainass31}   
	$\vzet^{(i)}\in \R^{J}$, $i=i\ldots N$ are pairwise independent  (that is $\alpha\vzet^{(i)}+\beta\vzet^{(i^{'})}=0$ for $i\not=i^{'}$ iff $\alpha=\beta=0$). In addition, $\vzet^{(i)}\cdot\vzeta(x)>0$ for any $x\in X$, $i\in\I$.
\end{assumption}
\begin{theorem}\label{uniquevZet}
	Under Assumptions {\em \ref{mainass3}} and \ref{mainass31}:
	\begin{description}
		\item {i)}The boundary of $\vZet(\uvmS_N(\bar\mu))$ contained in $\R^N_{++}$ is composed of extreme points \index{extreme point}  of $\vZet(\uvmS_N(\bar\mu))$.
		\item{ii)} If $\vM\in\partial\vZet(\uvmS_N(\bar\mu))$  there exists a unique   subpartition of $X$ associated with this point, and this subpartition is a strong one. In particular there is a unique $\cM\in\uvmS_N(\bar\mu)$ such that $\vM=\vZet(\cM)$.
	\end{description}
\end{theorem}

\begin{tcolorbox}
	\begin{remark}\label{parvZ}
		Note that  unlike Corollary  \ref{wSsubsetwpus}, $$\vZet(\vmS_N(\bar\mu))\not\subset \partial \vZet(\uvmS_N(\bar\mu)) $$ in general.
	\end{remark}
\end{tcolorbox}


\subsection{Proofs}\label{secproofs}
Recall the definitions (\ref{xi0}-\ref{0Xi0+}) of $\xi^0_\zeta, \xi_\zeta^{0,+}$ and $\Xee$, $\Xeep$. For any $\cP=(\vpp_1, \ldots \vpp_N)\in\Dt(N,J)$ consider
\be\label{Aizero} A_i(\cP):= \left\{ x\in X; \vpp_i\cdot\vzeta (x)> \max_{k\in\I; k\not= i}\vpp_k\cdot\vzeta (x)\right\}\ee
\be\label{Aizerop0} A^+_0(\cP):= \left\{ x\in X; \max_{i\in\I}\vpp_i\cdot\vzeta (x)\leq 0 \right\} \ .  \ee
\be\label{Aizerop} A^+_i(\cP):= A_i(\cP)- A^+_0(\cP)  \ . \ee


We first need the following result:
\begin{lemma}\label{51} .
	Under Assumption \ref{mainass3}, if $\cP:=(\vpp_1, \ldots \vpp_N)$ such that $\vpp_i\not=\vpp_{i^{'}}$ for $i\not = i^{'}$ then $\Xee$ (resp. $\Xeep$)
	is differentiable  at $\cP$ and
	\be\label{diffm} \nabla_{\vpp_i}\Xee= \int_{A_i(\cP)} \vzeta d\mu \  \ , \text{resp.} \ \ \nabla_{\vpp_i}\Xeep= \int_{A^+_i(\cP)} \vzeta d\mu  \ee
\end{lemma}

\begin{proof}
	By definition of $\{A_i(\cP)\}$, these sets are mutually essentially  disjoint. By Assumption \ref{mainass3} and the assumption on $\cP$  we obtain that  $\mu(\cup_1^N A_i(\cP))=\mu(X)$.
	Moreover,
	\be\label{xidelP} \nabla_{\vpp_i} \xi_\zeta=\left\{ \begin{array}{cc}
		\vzeta(x)  \  \text{if} & x\in A_i(\cP) \\
		0  \ \text{if} & \exists j\not= i, \ x\in A_j(\vpp)
	\end{array}\right. \  \ee
	In particular, the $\vpp_i$ derivatives of $\xi_\zeta$ exists $\mu$ a.e in $X$, $\nabla_{\vpp}\xi_\zeta\in \mathbb{L}_1(X; \D(N,J))$ for any $\cP\in\Dt(N,J)$ and the partial derivatives are uniformly integrable.
	Since $\Xee:=\mu(\xi_\zeta)$ by definition,  its derivatives exists  everywhere
	and
	$$\nabla_{\vpp_i}\Xee=\mu\left(\nabla_{\vpp_i}\xi_\zeta(\cdot,\cP)\right)=\int_{A_i(\cP)}\vzeta d\mu \ \ . $$
	\par
	Finally,   note that $\Xee$ is a convex function, and the existence of its partial derivatives implies its differentiability (\ref{cf7}).
	
	In the case of $\Xeep$ we observe that (\ref{xidelP}) still holds for $\xi_\zeta^+$ and $A_i^+(\cP)$, $i\in\I$ while
	$\nabla_{\vpp_i}\xi_\zeta^+=0$ for any $x\in A_0^+(\cP)$. Since $\mu(\cup_{i\in\I\cup\{0\}}A_i^+)=\mu(X)$ we obtain the same result for the sup-partition induced by $\{A_i^+(\cP)\}$, $i\in\I$.
\end{proof}
\begin{cor}\label{corvzr++}
	Under Assumption \ref{mainass31}, The function $\vq\rightarrow \Xeep(\vZet^*(\vq))$ is differentiable at any $\vq\in\R_{++}^N$.
\end{cor}
\paragraph{Proof of Theorem \ref{uniquestrong}}
Given a price matrix $\cP:=(\vpp_1, \ldots \vpp_N)$, we associate with
$\vpp$ the coalitions $\I_{\vpp}:= \{k\in\I; \ \vpp_i=\vpp\}$. The collection of pairwise disjoint coalitions defined in this way constitutes the
ensemble  of coalitions $\cI_{\cP}$.\index{coalitions ensemble}
\be\label{cI_p} \cI_{\cP}:= (\I_1, \ldots \I_{|\cI(\cP)|}) \ \ , \ \ \ee
where each $\I_i$ coincides with one of the (non-empty) coalitions $\I_{\vpp}$, $\vpp\in\{\vpp_1, \ldots \vpp_N\}$.

We now recall Theorem \ref{main1} and Corollary \ref{main2cor}: If, and only if,  $\cM$  is a boundary point of  $\uvmS_N(\bar\mu)$ then there exist a non zero  $\cP_0\in\Dt(N,J)$  such that,
for any $\cP\in \Dt(N,J)$
\be\label{price0}
\Xeep(\cP_0)-\cP_0:\cM=0 \leq \Xeep(\cP)-\cP:\cM \  \ .  \ee

For any such (possibly non-unique) $\cP_0$  we associate the coalitions set $\cI:=\cI_{\cP_0}$ as defined in (\ref{cI_p}).
If there is another ${\cP}^{'}\not=\cP_0$ maximizing (\ref{price0}) then by convexity of $\Xee$, $(1-\eps)\cP_0+\eps\cP^{'}$ is a maximizer as well for any $\eps\in[0,1]$. By Definition \ref{adef}  we get that for  $\eps>0$ sufficiently small   $\cI_{\eps\cP^{'}+(1-\eps)\cP_0} \succ \cI$, and the pair of coalition's ensemble agrees iff $\cI=\cI_{\cP^{'}}$.   Thus, the maximal  coalition's ensemble is unique.\index{coalitions ensemble}

Let  $(\mu_1, \ldots \mu_{|\cI|})$ a subpartition associated with the maximal coalition $\cI$. In particular, \index{maximal coalition}
$$ \cI(\cM)_i^{(j)}=\mu_i (\zeta_j) \ , \ i\in \cI, \ j\in \J \ . $$
By definition of $\Xeep$ (\ref{0Xi0+}), (\ref{price0}) and Lemma \ref{lemdual1} we get
\begin{multline}\label{ftreyz}\Xeep(\cP_0)-\cP_0:\cM= \Xeep(\cI^*(\underline{\cP}_0))-\underline{\cP}_0\cdot\cI(\cM)\equiv
\\
\sum_{i\in \cI}\mu_i\left[\xi^{0,+}_\zeta(x, \cP_0) -\underline{\cP}_{0,i}\cdot \vzeta \right]+ \mu_0\left[\xi^{0,+}_\zeta(\cdot, \cP_0)\right] =0 \ , \end{multline}
where $\mu_0:=\mu-\sum_{i\in\cI}\mu_i$.
From the definition of $\xi^{0,+}_\zeta$ we obtain that $\xi^{0,+}_\zeta(x, \cP_0)\geq  \underline{\vpp}_{i,0}\cdot\vzeta(x)$ or any $(x,i)\in X\times\cI$ as well as  $\xi^{0,+}_\zeta\geq 0$ on $X$. Thus, (\ref{ftreyz}) implies that
$\xi^{0,+}_\zeta(x, \cP_0)=  \underline{\vpp}_{i,0}\cdot\vzeta(x)$   a.e ($\mu_i$), as well as $\xi^{0,+}_\zeta(x, \cP_0)=0$  a.e ($\mu_0$).

On the other hand,    we get via (\ref{Aizerop}, \ref{Aizerop0}) adapted to $\cI$ that \\ $\xi^{0,+}_\zeta(x, \cP_0)>  \underline{\vpp}_{i,0}\cdot\vzeta(x)$ if $x\in A^+_k(\underline{\cP}_0)$ where
for any $k\in \cI\cup\{0\}- \{i\}$. Hence
\be\label{exceptbar} Supp(\mu_i)\subset X-\cup_{k\in \cI\cup\{0\}; k\not= i} A^+_k(\underline{\cP}_0) \ . \ee

Since, by definition, the components of $\underline{\cP}_0$ are pairwise different we get by Assumptions \ref{mainass3}  and by (\ref{Aizerop}, \ref{Aizerop0}) that $\cup_{i\in \cI\cup\{0\}} A_i^+(\underline{\cP}_0)=X$.
This and (\ref{exceptbar}) imply that $\mu_i$ is the restriction of
$\mu$ to $A^+_i(\underline{\cP}_0)$, hence it is a strong partition. The uniqueness of this partition  follows as well.\index{strong  (deterministic) partition}

Finally, it follows from (\ref{price0}) that $\cM\in\partial_{\cP_0}\Xeep$.  Since $\cP_0\not=0$ it follows from Lemma \ref{51} that $\Xeep$  is differentiable at $\cP_0$. Hence $\cM$ is an extreme point via Proposition \ref{propA12}.\index{extreme point} 
\qed

The following Corollary to the proof of Theorem \ref{uniquestrong}  refers to the case of {\em maximal coalition} \index{maximal coalition} (c.f. Definition \ref{adef}). 
\begin{cor}\label{coruniquemaxcoal}
	If $\cP_0:= (\vpp_1, \ldots \vpp_N)$ satisfies $\vpp_i\not=\vpp_j$ for any $i\not= j$ then there exists a unique partition in $\sPzeta_{\{\cM_0\}}$ where 
	$\cM_0=\nabla_{\cP_0} \Xee$. 
	Moreover, this partition is a strong one, given by (\ref{Aizero}) where $\cP_0$ substituted for $\cP$. 
	\end{cor}

\paragraph{Proof of Theorem \ref{uniquevZet}}

i) \ \  Assume that  $\vM\in \partial\vZet(\uvmS_N(\bar\mu))\cap \R_{++}^N$.

Let $\cM^{'}\in \vZet^{-1}(\vM)\cap\partial\uvmS_N(\bar\mu)$. By Theorem \ref{main1} we get
$$ \Xeep(\cP)-\cP:\cM^{'}\geq 0$$
for any $\cP\in \Dt(N,J)$. In particular, we substitute $\cP=\vZet^*(\vq)$ and we get, for {\em any} $\vq\in \R^N$,
$$ \Xeep(\vZet^*(\vq))-\vZet^*(\vq):\cM^{'}\geq 0 \ . $$
From Lemma \ref{lemdual2} (and since $\vM=\vZet(\cM^{'})$ by definition):
\be\label{qwrt} \Xeep(\vZet^*(\vq))-\vq\cdot\vZet(\cM^{'})=  \Xeep(\vZet^*(\vq))-\vq\cdot\vM \geq 0 \  \ee
holds for any $\vq\in\R^N$. Since, in addition, $\vM\in \partial\vZet(\uvmS_N(\bar\mu))$ we get, as in the proof of Corollary \ref{main2cor}, that there exists a non-zero $\vq_0\in\R^N$ for which
\be\label{qwrt1}  \Xeep(\vZet^*(\vq_0))-\vq_0\cdot\vM = 0 \  \ .  \ee
We prove now that
$\vq_0\in\R_{++}^N$.

Surly, it is impossible that all components of $\vq_0$ are non-positive.
Assume with no limitation of generality, that,  $q_{0,1}>0$. By Assumption \ref{mainass3} we can find $\delta>0$ such that $\vzet^{(1)}\cdot\vzeta(x)>\delta$ for any  $x\in X$. Then $\xi_\zeta(\vZet^*(\vq),x)\geq q_{0,1}\vzet^{(1)}\cdot\vzeta(x)>\delta q_{0,1}$  on $X$.\par
Suppose  $q_{0,j}\leq 0$ for some $j\not=1$, let $\eps>0$ for which $\eps\vzet^{(j)}\cdot\vzeta< \delta q_{0,1}$ on $X$.
Then  $\xi^0_\zeta(\vZet^*(\vq_0),x)=\xi^0_\zeta(\vZet^*(\vq_0+\eps\vec{e_j}), x)$ on $X$.
Here $\vec{e}_j$ is the unit  coordinate vector pointing in the positive $j$ direction. Indeed, both $q_{0,j}\vzet^{(j)}\cdot\vzeta(x)$  and $(q_{0,j}+\eps)\vzet^{(j)}\cdot\vzeta(x)$ are smaller that $\vq_{0,1}\vzet^{(1)}\cdot\vzeta(x)$ for any $x\in X$, so the $j$ component does not contribute to the value of $\xi^0_\zeta$ at any point $x\in X$.
Hence $\Xee(\vq_0)=\Xee(\vZet^*(\vq_0+\eps\vec{e}_j))$.
so
$$\Xee(\vZet^*(\vq_0+\eps \vec{e}_j)) -(\vq_0+\eps\vec{e}_j)\cdot\vM= \Xee(\vZet^*(\vq_0)) -\vq_0\cdot\vM - \eps m_j  =-\eps m_j<0 \  $$
by (\ref{qwrt1}) (recall $m_j>0$ by assumption). This contradicts (\ref{qwrt}), hence $q_{0,j}>0$ as well and $\vq\in\R^N_+$.

We now   prove that  $\vM\in \partial \vZet(\uvmS_N(\bar\mu))\cap\R_{++}^N$ is an extreme point in $\vZet(\uvmS_N(\bar\mu))$. Consider the function
$\vq\rightarrow \Xeep(\vZet^*(\vq))$. By Assumptions  \ref{mainass3}, \ref{mainass31} and Corollary \ref{corvzr++}  we observe that this function is convex {\em and} differentiable at any $\vq\in\R_{++}^N$.    Its essential domain is $\vZet(\uvmS_N(\bar\mu))$. Thus,
(\ref{qwrt1}) and Proposition \ref{propA12} imply  that $\vM$ is an extreme point of $\vZet(\uvmS_N(\bar\mu))$.\index{extreme point} 

ii) \ \ Let now $(\mu_1, \ldots \mu_N)$ a partition associated with $\vM$. In particular \\
$\int\vzet^{(i)}\cdot \vzeta d\mu_i= m_i$. By definition of $\Xeep$ (\ref{0Xi0+}) we get
that
$$\Xeep(\vZet^*(\vq_0))-\vq_0\cdot\vM\equiv \mu\left[\xi^{0,+}_\zeta(x, \vZet^*(\vq_0))\right] -\sum_{i\in\I}q_{0,i}\mu_i\left[\vzet^{(i)}\cdot\vzeta\right]=0 \ . $$
On the other hand, since $\mu\geq \sum_{i\in\I}\mu_i$, we get
$$\mu_0\left( \xi^{0,+}_\zeta(x, \vZet^*(\vq_0))\right)+   \sum_{i=1}^N\mu_i\left(\xi^{0,+}_\zeta(x, \vZet^*(\vq_0))- q_{i,0}\vzet^{(i)}\cdot\vzeta\right)  =0 \  $$
where $\mu_0=\mu-\sum_{i\in\I}\mu_i$. Since $\xi^{0,+}\geq 0$ by definition (\ref{xi0p})  we obtain,
in particular, that $\xi^{+,0}(x, \vZet^*(\vq_0))=0$  $\mu_0$ a.e. Thus, $\mu_0$ is supported in  $A_0^+(\vZet^*(\vq_0))$ via (\ref{Aizerop0}).

From the definition (\ref{xi0p}) of $\xi^{0,+}_\zeta$ we also obtain that $\xi^{0,+}_\zeta(x, \vZet^*(\vq_0))\geq  q_{i,0}\vzet^{(i)}\cdot\vzeta(x)$ or any $x\in X$. Thus,
$\xi^{0,+}_\zeta(x, \vZet^*(\vq_0))=  q_{i,0}\vzet^{(i)}\cdot\vzeta(x)$ for $\mu_i$ a.e. $x$.

On the other hand, from (\ref{Aizerop}) (substitute $\vZet^*(\vq)$ for $\cP$) we get \\ $\xi^{0,+}_\zeta(x, \vZet^*(\vq_0))<  q_{k,0}\vzet^{(k)}\cdot\vzeta(x)$ $\mu_i$ a.s. if $x\in A^+_k(\vZet^*(\vq))$ for any $k\not= i$. Hence
\be\label{except} Supp(\mu_i)\subset X-\cup_{k\not= i} A^+_k(\vZet^*(\vq_0)) \ . \ee

By (\ref{Aizero}-\ref{Aizerop})    we obtain that the union of $A^+_i(\vZet^*(\vq_0))$, $i\in \{0\}\cup\I$,   is of full $\mu$ measure. This and (\ref{except}) imply that $\mu_i$ is the restriction of
$\mu-\mu_0$ to $A_i(\vZet^*(\vq_0))$, hence it is a strong subpartition. The uniqueness follows since the same reasoning holds for any subpartition corresponding to $\vM$.
\qed

\begin{tcolorbox} Note that, unlike $\vmS_N(\bar\mu)$, the set $\vZet(\vmS_N(\bar\mu))\subset\R_+^N$ {\em may} contains interior points (compare with Corollary \ref{wSsubsetwpus}).
\end{tcolorbox}
\begin{prop}\label{u=}
	Under {\em assumption \ref{mainass3}}, \ref{mainass31}
	$$ \R^N_{++}\cap\partial\vZet(\uvmS_N(\bar\mu))\subset\partial \vZet(\vmS_N(\bar\mu)) \ . $$
	In particular,  any sub partition corresponding to  $\vM\in \partial\vZet(\uvmS_N(\bar\mu))\cap\R^N_{++}$ is a {\em strong partition}.\index{strong  (deterministic) partition}
\end{prop}
\begin{proof}
	Let $\vM\in \partial\vZet(\uvmS_N(\bar\mu))\cap\R^N_{++}$.
	Following the proof of Theorem \ref{uniquevZet} we get  the existence of $\vq_0\in \R_{++}^N$ satisfying (\ref{qwrt1}).
	If $\vmu\in  \vZet(\wSPzeta_{\{\vM\}})$  then
	$$ \Xeep(\vZet^*(\vq_0))\equiv \mu\left[\xi^{0,+}_\zeta(x,  \vZet^*(\vq_0))\right]\geq  \sum_{i\in\I} \mu_i\left[\xi^{0,+}_\zeta(\cdot, \vZet^*(\vq_0))\right]$$
	$$\geq \sum_{i\in\I} q_{0,i}\cdot\mu_i(\vzet^{(i)}\cdot \vzeta) = \vq_0\cdot\vZet(\vM)\equiv  \Xeep(\vZet^*(\vq_0)) \ . $$
	In particular
	$$ \mu\left(\xi^{0,+}_\zeta(x,  \vZet^*(\vq_0))\right)=  \sum_{i\in\I} \mu_i\left(\xi^{0,+}_\zeta(x, \vZet^*(\vq_0))\right) \ .  $$
	Since $\vq_0\in \R_{++}^N$, Assumption \ref{mainass3} and the definition of $\xi_\zeta^{0,+}$ imply that $\xi^{0,+}_\zeta(\cdot,\vZet^*(\vq_0))$ is positive and continuous on $X$. This, and  $\sum_{i\in\I} \mu_i\leq \mu$ imply that, in fact, $\sum_{i\in\I} \mu_i=\mu$  so $\vmu\in \wPzeta_{\{\vM\}}$ is a   {\it strong partition}. In particular $\vM\in\vZet(\vmS_N(\bar\mu))$. Since
	$\vZet(\vmS_N(\bar\mu))\cap\partial\vZet(\uvmS_N(\bar\mu))\subset\partial\vZet(\vmS_N(\bar\mu))$    then $\vM\in\partial \vZet(\vmS_N(\bar\mu))$  as well.
\end{proof}

\subsection{An application: Two States for two Nations}\label{2s2n}
Suppose $X$ is a territory held by two ethnic groups living unhappily together, say ${\cal J}$ and ${\cal P}$. Let $\mu$ the distribution of the total population in $X$. Let $\zeta_{\cal J}:X\rightarrow[0,1]$ be the relative density of the population ${\cal J}$. Then $\zeta_{\cal P}:= 1-\zeta_{\cal J}$ the relative density of the population ${\cal P}$.

It was suggested by some wise men and women that the territory $X$ should be divided between the two groups, to establish a ${\cal J}-$state $A_{\cal J}$ and a ${\cal P}-$state $A_{\cal P}$:
$$A_{\cal J}\subset X, \ \ A_{\cal P}\subset X \ ; \ \ A_{\cal J}\cup A_{\cal P}=X, \ \ \ \mu(A_{\cal J}\cap A_{\cal P})=0$$
Under the assumption that nobody is forced to migrate from one point to another in $X$, what are the possibilities of such divisions?

The question can be reformulated as follows. Let us assume that an $A_{\cal J}$ state is formed whose ${\cal J}$ population is $m_{\cal J}$ and whose ${\cal P}$ population is $m_{\cal P}$:\footnote{Of course, the  ${\cal J}-{\cal P}$ populations of the ${\cal P}$ state are, respectively,  $M_{\cal J}-m_{\cal J}$ and
	$M_{\cal P}-m_{\cal P}$.}
$$ \int_{A_{\cal J}}\zeta_{\cal J}d\mu=m_{\cal J} \ , \ \ \int_{A_{\cal J}}\zeta_{\cal P}d\mu=m_{\cal P} \ . $$
The evident  constraints are
\be\label{rect}0\leq m_{\cal J}\leq \mu(\zeta_{\cal J}) \ \ ; \ \ 0\leq  m_{\cal P}\leq \mu(\zeta_{\cal P})\ . \ee
Assuming for convenience that the total population $\mu$ is normalized ($\mu(X)=1$, so $\mu(\zeta_{\cal J})+ \mu(\zeta_{\cal P})=1$), we may use Theorem \ref{thFzeta} to characterize the feasibility set $S$ in the rectangle domain (\ref{rect}) by
\begin{multline}(m_{\cal J}, m_{\cal P})\in S \Leftrightarrow \mu\left(F(\zeta_{\cal J}, \zeta_{\cal P})\right) \geq \\  (m_{\cal J}+m_{\cal P}) F\left( \frac{m_{\cal J}, m_{\cal P}}{m_{\cal J}+m_{\cal P}}\right)+ (1-m_{\cal P}-m_{\cal J}) F\left( \frac{\mu(\zeta_{\cal J})-m_{\cal J}, \mu(\zeta_{\cal P})-m_{\cal P}}{1-m_{\cal J}-m_{\cal P}}\right) \ . \end{multline}

From Proposition  \ref{propbarX=barl} we also obtain that the diagonal of the rectangle (\ref{rect}) is always contained in $S$:
$$\cup_{\alpha\in[0,1]}\alpha (\mu(\zeta_{\cal J}), \mu(\zeta_{\cal P}))\subset S \ . $$
c.f. Fig \ref{figprojectdiagonal}. 

What else can be said about the feasibility set $S$, except being convex and containing the diagonal of the rectangle \ref{rect}? If $\mu(\zeta_1(x)/\zeta_2(x)=r)=0$ for any $r\in[0, \infty]$,  then the assumption of Theorem \ref{uniquevZet} is satisfied with $J=2$, $\vzet^{(1)}=(1,0)$, $\vzet^{(2)}=(0,1)$. In particular we obtain
\begin{prop}
	All points of the boundary  $\partial S\cap\R_{++}^2$ are extreme points. \index{extreme point} For each $(m_{\cal J}, m_{\cal P})\in\partial S\cap\R_{++}^2$ there exists $r\in[0,\infty]$ such that the corresponding partition
	$$ A_{\cal J}:= \left\{ x\in X; \ \ \zeta_{\cal J}(x)/\zeta_{\cal P}(x)\geq r\right\}, \ \  A_{\cal P}:= \left\{ x\in X; \ \ \zeta_{\cal J}(x)/\zeta_{\cal P}(x)\leq r\right\}$$
	is unique.
	\par
	In particular, $S$ is contained in the parallelogram
	$$ \inf_{x\in X} \frac{\zeta_{\cal J}(x)}{\zeta_{\cal P}(x)}\leq
	\frac{m_{\cal J}}{m_{\cal P}}\wedge\frac{\mu(\zeta_{\cal J})-m_{\cal J}}{\mu(\zeta_{\cal P})-m_{\cal P}} \leq
	\frac{m_{\cal J}}{m_{\cal P}}\vee \frac{\mu(\zeta_{\cal J})-m_{\cal J}}{\mu(\zeta_{\cal P})-m_{\cal P}}
	\leq \sup_{x\in X} \frac{\zeta_{\cal J}(x)}{\zeta_{\cal P}(x)} \ , $$
	\begin{figure}
		\centering
		\includegraphics[height=6.cm, width=10.cm]{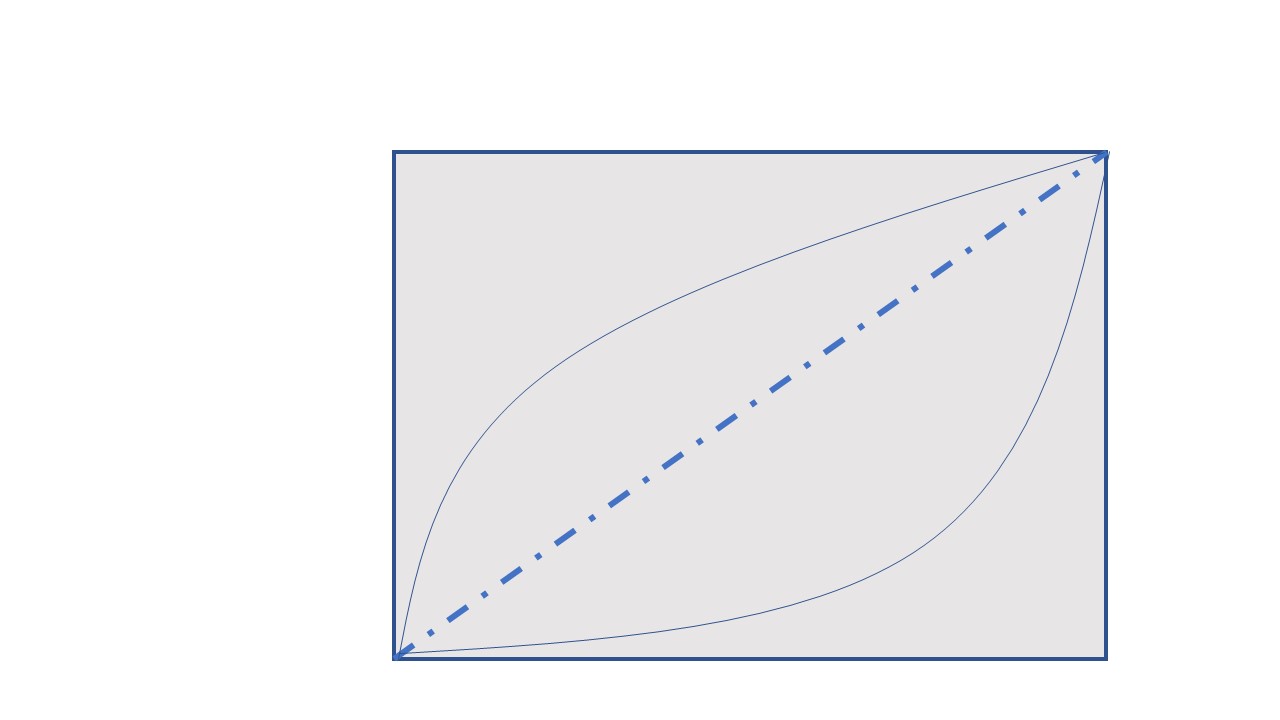}\\
		\caption{Projection on the diagonal. }\label{figprojectdiagonal}
	\end{figure}
\end{prop} 
\section{Further comments}
The special case of Theorem \ref{uniquestrong} where $N=1$ can be formulated as follows: Let $\vec{\sigma}(dx):=\vzeta(x)\mu(dx)$ be an $\R^J-$valued measure on $X$. The set $\uvmS_N(\bar\mu)$ corresponds, in that case, to the image of $\vec{\sigma}$ over all measurable subsets of $X$:
$$ \uvmS_N(\bar\mu):=\left\{ \vec{\sigma}(A) \ ; \ \ A\subset X \ \ \text{is} \ \mu \ \text{measurable} \ \right\}\subset \R^J \ . $$
The geometry of such sets was discussed by several authors. In particular, several equivalent  sufficient and necessary conditions for the strict convexity of
$\uvmS_N(\bar\mu)$ were introduces at \cite{Sch},  \cite{bian1}, \cite{bian2}. One of these conditions is the following:
\begin{theorem}\label{thbian}
	[\cite{Sch}, \cite{bian1}] The set $\uvmS_N(\bar\mu)$ is {\em strictly convex} iff the following condition holds: For any measurable set $A\subset X$ for which $\vec{\sigma}(A)\not=0$ there exists measurable sets $A_1, \ldots A_J\subset A$ such that the vectors $\vec{\sigma}(A_1), \ldots \vec{\sigma}(A_J)$ are linearly independent.
\end{theorem} 
Theorem \ref{thbian} can be obtained as  a special case  of Theorem \ref{uniquestrong}. Indeed, if $N=1$ then there is only one possible "coalition", composed of the single agent, hence the strict convexity of this set (namely the property that any boundary point is an extreme point) \index{extreme point} is conditioned on Assumption \ref{mainass3}. Let us show that, for a continuous $\vzeta$ (\ref{vzetadef}), Assumption \ref{mainass3} is, indeed, equivalent to the assumption of Theorem  \ref{thbian}.

If Assumption \ref{mainass3} fails then there exists a nonzero $\vpp\in\R^J$ and a measurable set $A$ such that $\mu(A)>0$ and $\vpp\cdot\vzeta(x)=0$ on $A$.  Hence, for any measurable $B\subset A$, $\vpp\cdot \vec{\sigma}(B)\equiv \int_B\vpp\cdot\vzeta d\mu=0$ as well. Hence $\vpp$ is not spanned by any collection of $J$ subsets in $A$.

Conversely, suppose $\mu(A)>0$ and let $k$ be the maximal dimension of $Sp\left(\vec{\sigma}(A_1)\ldots \vec{\sigma}(A_J)\right)$ where $A_1, \ldots A_J$ run over all $\mu-$ measurable subsets of $A$. We can find $k$ subsets $A_1\ldots A_k$ of $A$ such that the dimension of $Sp\left(\vec{\sigma}(A_1)\ldots \vec{\sigma}(A_k)\right)$ equals $k$. If $k<J$ then there exists $\vpp\in\R^J$ such that $\vpp\cdot\vec{\sigma}(A_i)=0$ for $i\in \{1, \ldots k\}$. If there exists a measurable $B\subset A$ such that $\vpp\cdot\vec{\sigma}(B)\not=0$ then the dimension of the space spanned by $\vec{\sigma}(A_i)$, $i=1, \ldots k$ and $\vec{\sigma}(B)$ is $k+1$. This contradicts the assumed maximality of $k$. Thus, $\vpp\cdot\vec{\sigma}(B)=0$ for any measurable subset of $A$, which implies that $\vpp\cdot\vzeta=0$ on $A$.

As a special case of Theorem \ref{uniquevZet} we may consider $\vzet^{(i)}$ to be the principle coordinates of $\R^N$ (in particular, $J=N$). The set $\vZet(\uvmS_N(\bar\mu))$ in $\R^N$ is, then, given by
$$ \vZet(\uvmS_N(\bar\mu))=\left\{ \left(\int_{A_1}\zeta_1d\mu, \ldots \int_{A_N}\zeta_Nd\mu\right)\right\}, $$
where $\vA:=(A_1, \ldots A_N)$ runs over the set $\vSP$ of all strong subpartitions of $X$ (c.f. Section \ref{firstsecinoptimstrong}). Such sets are the object of study in \cite{wol}.
The case of $N=2$ is the case we considered in section \ref{2s2n}. A detailed study of this case can be found in  \cite{LW}.

\chapter{Optimal multipartitions}\label{chomp}
\section{Optimality within the weak partitions}\label{secton7}
\subsection{Extension to hyperplane}\label{sduoptim1}\index{weak partition}
In order to consider the optimization of $\theta$ on $\wPzeta_{\{\cM_0\}}$ where $\cM_0\in \vmS_N(\bar\mu)$ (resp. on $\wSPzeta_{\{\cM_0\}}$  where $\cM_0\in \uvmS_N(\bar\mu)$), we introduce the following extension
of Theorem \ref{main1}:

Let $\SDt$ be a subspace of $\Dt(N,J)$. Let $\SDt^\perp\subset \D(N,J)$ be the subspace of annihilators of $\SDt$, that is
$$ \SDt^\perp:= \{ \cM\in \D(N,J), \cP:\cM=0 \ \ \forall \cP\in \SDt \} \ . $$
Given such $\SDt$ and $\cP_0\in\Dt(N,J)$, the following Theorem  extends Theorem \ref{main1} to the  hyperplane $\SDt+s\cP_0$.
\begin{theorem}\label{main2}
	For any $s\in\R$
	\be\label{main2_1} \inf_{\cQ\in \SDt}\Xee(\cQ+s\cP_0)-\cQ:\cM_0=\sup_{\cM\in(\SDt^\perp+\cM_0)\cap \vmS_N(\bar\mu)}s\cP_0:\cM \ee as well as
	\be\label{main2_2} \inf_{\cQ\in \SDt}\Xeep(\cQ+s\cP_0)-\cQ:\cM_0=\sup_{\cM\in(\SDt^\perp+\cM_0)\cap \uvmS_N(\bar\mu)}s\cP_0:\cM \ . \ee
\end{theorem}
The  case $\SDt=\Dt(N,J)$   reduces to Theorem \ref{main1}. Indeed, if $\cM_0\not\in \vmS_N(\bar\mu)$ (resp. $\cM_0\not\in \uvmS_N(\bar\mu)$)
then the right side of (\ref{main2_1}, \ref{main2_2}) is a supremum  over a null set (since $\SDt^\perp=\{0\}$) and, by definition of the supremum, it equals $-\infty$.

If, on the other hand, $(\SDt^\perp+\cM_0)\cap \vmS_N(\bar\mu)\not=\emptyset$ (resp. $(\SDt^\perp+\cM_0)\cap \uvmS_N(\bar\mu)\not=\emptyset$) then  the supremum on the right sides of (\ref{main2_1}) (resp. (\ref{main2_2})) is always attended, since both $\vmS_N(\bar\mu), \uvmS_N(\bar\mu)$ are compact sets. Thus, there exists $\cM_*\in (\SDt^\perp+\cM_0)\cap \vmS_N(\bar\mu)$ (resp. $\cM_*\in (\SDt^\perp+\cM_0)\cap \uvmS_N(\bar\mu)$) such that
\begin{multline}\label{alwaysdo}s\cP_0:\cM_*=\sup_{\cM\in(\SDt^\perp+\cM_0)\cap \vmS_N(\bar\mu)}s\cP_0:\cM \\ \text{resp.} \ \ s\cP_0:\cM_*=\sup_{\cM\in(\SDt^\perp+\cM_0)\cap \uvmS_N(\bar\mu)}s\cP_0:\cM\end{multline}

\begin{remark}
	We can make a natural connection between  reduction to coalition's ensemble  introduced in section \ref{cocar} and the  duality with respect to affine subsets. Indeed,  given a coalition $\cI$ we may define $\SDt_{\cI}:=\cI^*(\Dt(N,J))$.\index{coalitions ensemble}
	
	If we imply Theorem \ref{main2} in the special case $s=0$ (and arbitrary $\cP_0$) we can get Theorem \ref{uniquestrong} from the following statement: For any $\cM_0\in\partial\uvmS_N(\bar\mu)$ there exists a unique maximal coalition ensemble $\cI$ such that\index{maximal coalition} 
	the inequality
	$$ \Xeep(\cP)-\cP:\cM_0\geq  0$$
	holds for any  $\cP\in\SDt_{\cI}$, and there exists a {\em unique} $\underline{\cP}\not=0$ in $\SDt_{\cI}$ along which the above inequality turns into an equality on the ray spanned by $\underline{\cP}$ ($\cP=\{\alpha\underline{\cP}\}$, $\alpha\geq 0$). This $\underline{\cP}$ induces the unique
	strong subpartition $\vA_{\cI}$.

	There is also a  natural connection  between  Theorem \ref{uniquevZet}   and  Theorem \ref{main2} which is explained below:
	\par\noindent
	Let $\SDt=\vZet^*(\R^N)\subset\Dt(N,J)$.
	We may imply Theorem \ref{main2}  for $s=0$ and  get Theorem \ref{uniquevZet} from the following statement: For any $\vM_0\in\partial\vZet(\uvmS_N(\bar\mu))\cap\R^N_+$ there exists a unique $\cM_0\in \vZet^{-1}(\vM_0)$ such that
	the inequality
	$$ \Xeep(\cP)-\cP:\cM_0\geq  0$$
	holds for any $\cQ\in \SDt$,
	and there exists a {\em unique} $\underline{\cP}\not=0$ in $\SDt$ along which the above inequality turns into an equality
	on the ray spanned by $\underline{\cP}$ ($\cP=\{\alpha\underline{\cP}\}$, $\alpha\geq 0$). This $\underline{\cP}$ induces the unique
	strong subpartition  $\vA\in \uvmS_N(\bar\mu)_{\{\cM_0\}}$.

\end{remark}
The minimizer $\cQ\in\SDt$  on the left side of (\ref{main2_1}, \ref{main2_2}), however,  is not necessarily attained.
Recall also Definition \ref{weakpartition} of the  weak (sub)partition \index{weak subpartition}  sets $\wP, \wSP$ and (\ref{KSPWMxi}). Since $\cM_*\in \vmS_N(\bar\mu)$ (resp. in $\uvmS_N(\bar\mu)$) then (\ref{alwaysdo}) implies that, for such pairs $(\cM_0, \cP_0)\in \D(N,J)\times\Dt(N,J)$,  there exist $\cM_*\in \SDt^\perp+\cM_0$ and (sub)partitions $\vmu$ which maximizes $s\cP_0:\cM_\zeta(\vmu))$ on $\wP_{\SDt^\perp+\cM_0}$ (resp. on $\wSP_{\SDt^\perp+\cM_0}$), i.e
$$ s\cP_0:\cM_\zeta(\vmu))=\inf_{\cQ\in \SDt}\Xee(\cQ+s\cP_0)-\cQ:\cM_0\ , $$
resp.
$$s\cP_0:\cM_\zeta(\vmu))=\inf_{\cQ\in \SDt}\Xeep(\cQ+s\cP_0)-\cQ:\cM_0\ . $$
Letting $s=1$ we obtain
\begin{prop}\label{prop12}
	For each $(\cP_0, \cM_0)\in \Dt(N,J)\times\D(N,J)$ there exists $\cM_*\in \vmS_N(\bar\mu)\cap (\cM_0+\SDt^\perp)$ (resp. $\cM_*\in \uvmS_N(\bar\mu)\cap(\cM_0+\SDt^\perp)$) and  $\vmu\in\wP_{\cM_*}$ (resp. $\vmu\in\wSP_{\cM_*}$) such that
	$\vmu$ maximizes $\cP_0:\cM_\zeta(\vmu)$ on $\wP_{\SDt^\perp+\cM_0}$ (resp. on $\wSP_{\SDt^\perp+\cM_0}$), and, moreover,
	$$ \cP_0:\cM_\zeta(\vmu)=\inf_{\cQ\in\SDt}\Xee(\cQ+\cP_0)-\cQ:\cM_0 \ , $$
	resp.
	$$ \cP_0:\cM_\zeta(\vmu)=\inf_{\cQ\in\SDt}\Xeep(\cQ+\cP_0)-\cQ:\cM_0  \ . $$
\end{prop}

Let us consider (\ref{main2_1}) in  the  case $s=0$ and $\SDt\subset\Dt(N,J)$ a  subspace $\not=\{0\}$.
We view $\Xee$ as defined on the subspace $\SDt$ whose dual is the quotient space $\D(N,J)/\SDt^\perp$. The action of $\SDt$ on $\D(N,J)/\SDt^\perp$ is define, naturally, as
$\cQ:\cM$ where $\cQ\in\SDt$ and $\cM$ is {\em any representative} form $\D(N,J)/\SDt^\perp$.
Hence (\ref{main2_1}) reads
\be\label{main2_1s=0} \inf_{\cQ\in \SDt}\Xee(\cQ)-\cQ:\cM_0=0 \ \
\text{iff} \ \  \cM_0\in \vmS_N(\bar\mu)/\SDt^\perp \ . \ee
In the general case we may view $(\cQ,s)\rightarrow \Xee(\cQ+s\cP_0)$  as a positively homogeneous function on the space $\SDt\otimes\R$. The dual of this space is  \\ $\D(N,J)/\SDt^\perp\oplus\R$, and the duality action is
$$(\cQ,t):(\cM,s):= \cQ:\cM +ts \  $$
where $\cQ\in\SDt$, $\cM$ is any representative from $\D(N,J)/\SDt^\perp$ and $ts$ is just  the  product of $t$ and $s$ in $\R$.

Then (\ref{main2_1}) applied to all $s\in\R$ reads as
$$ \inf_{(\cQ,s)\in (\SDt\oplus\R)}\Xee(\cQ+s\cP_0)-(\cQ,s):(\cM_0,t)=0 $$
iff $(\cM_0,t)\in \widetilde{\vmS_N(\bar\mu)}\subset \D(N,J)/\SDt^\perp\otimes \R$, where
\be\label{extendedS} \widetilde{\vmS_N(\bar\mu)}:=
\left\{ (\cM_0,t), \ \cM_0\in \vmS_N(\bar\mu)/\SDt^\perp, \ \ \inf_{\cM\in(\SDt^\perp+\cM_0)}cP_0:\cM\leq {\it t}\leq \sup_{\cM\in(\SDt^\perp+\cM_0)}\cP_0:\cM \right\}\ . \ee
Similarly
$$ \inf_{(\cQ,s)\in (\SDt\oplus\R)}\Xeep(\cQ+s\cP_0)-(\cQ,s):(\cM_0,t)=0 $$
iff $(\cM_0,t)\in \widetilde{\uvmS_N(\bar\mu)}\subset \D(N,J)/\SDt^\perp\otimes \R$, where
\be\label{extendeduS} \widetilde{\uvmS_N(\bar\mu)}:= \left\{ (\cM_0,t), \ \cM_0\in \uvmS_N(\bar\mu)/\SDt^\perp, \ \ \inf_{\cM\in(\SDt^\perp+\cM_0)}\cP_0:\cM\leq {\it t}\leq \sup_{\cM\in(\SDt^\perp+\cM_0)}\cP_0:\cM \right\}\ . \ee
Recalling Proposition \ref{propA11} we observe that $\widetilde{\vmS_N(\bar\mu)}$ (resp. $\widetilde{\uvmS_N(\bar\mu)}$) is the essential domain of the Legendre transform  \index{Legendre transform} of $(\cQ,s)\rightarrow\Xee(\cQ+s\cP_0)$ (resp. $(\cQ,s)\rightarrow\Xeep(\cQ+s\cP_0)$) as functions on $\SDt\otimes\R$. It is, in fact, an {\em extension} of $\vmS_N(\bar\mu)$ (resp. $\uvmS_N(\bar\mu)$) from $\D(N,J)$ to $\D(N,J)/\SDt^\perp\otimes\R$.


\subsubsection{Proof of Theorem \ref{main2}}
The inequalities
\begin{multline}\label{ineqSDt}\inf_{\cQ\in \SDt}\Xee(\cQ+\cP_0)-\cQ:\cM_0\geq \sup_{\cM\in(\SDt^\perp+\cM_0)\cap \vmS_N(\bar\mu)}\cP_0:\cM, \ \  \\ \text{resp.} \ \inf_{\cQ\in \SDt}\Xeep(\cQ+\cP_0)-\cQ:\cM_0\geq \sup_{\cM\in(\SDt^\perp+\cM_0)\cap \uvmS_N(\bar\mu)}\cP_0:\cM\end{multline}
hold by Theorem \ref{main1}. In order to prove the reverse inequality we need
the Hahn-Banach Theorem
\begin{theorem}\index{Hahn-Banach}
	(Hahn-Banach) Let $V$ be a real vector space, $p:V\rightarrow\R$
	a sublinear function and $f:U\rightarrow\R$ a linear functional
	on a linear subspace $U\subseteq V$ s.t. $f\left(x\right)\leq p\left(x\right)$
	for every $x\in U$. Then there exists a linear functional $F\in V^{*}$
	s.t. $F\left(u\right)=f\left(u\right)\forall u\in U$ and $F\left(x\right)\leq p\left(x\right)\forall x\in V$.
\end{theorem}
\index{Hahn-Banach}
Hahn-Banach Theorem is valid for any linear space. Here we use it for the finite dimensional space $V\equiv \Dt(N,J)$. Let
$$p(\cP):= \inf_{\cQ\in\SDt}\Xi(\cP+\cQ)-(\cP+\cQ):\cM_0 \   $$
where $\Xi$ stands for either $\Xee$ or $\Xeep$. Note that
\be\label{pgadolntpx} p\geq 0 \ \  \text{on} \ \  \Dt(N,J)\ee  by Theorem \ref{main1} since  $\cM_0\in\Ss$ where $\Ss= \vmS_N(\bar\mu)$ (resp. $\Ss=\uvmS_N(\bar\mu)$). Recall that a function $p$ is sublinear iff
\begin{enumerate}
	\item $p(s\cP)=s p(\cP)$ for any $\cP\in \Dt(N,J)$ and $s>0$.
	\item $p(\cP_1+\cP_2)\leq p(\cP_1) + p(\cP_2)$.
\end{enumerate}
Note that $\Xi$ is sublinear by definition (\ref{xi0}-\ref{0Xi0+}). Since $\SDt$ is a subspace it follows that
$$ p(s\cP)= \inf_{\cQ\in \SDt}\Xi(s(\cP+\cQ))-s(\cP+\cQ):\cM_0 = s p(\cP) \  $$
where $s\geq 0$. For any $\eps>0$ there exists $\cQ_1, \cQ_2\in\SDt$ such that
$$ p(\cP_1)\leq \Xi(\cP_1+\cQ_1)-(\cP_1+\cQ_1):\cM_0 +\eps , \ \ p(\cP_2)\leq \Xi(\cP_2+\cQ_2)-(\cP_2+\cQ_2):\cM_0 +\eps \ , $$
thus, by sub-linearity of $\Xi$ and definition of $p$
\begin{multline}p(\cP_1+\cP_2) \leq \Xi(\cP_1+\cP_2+\cQ_1+\cQ_2)-(\cP_1+\cP_2+\cQ_1+\cQ_2):\cM_0 \\ \leq \Xi(\cP_1+\cQ_1)-(\cP_1+\cQ_1):\cM_0
+ \Xi(\cP_2+\cQ_2)-(\cP_2+\cQ_2):\cM_0 \leq p(\cP_1)+p(\cP_2) +2\eps
\end{multline}
so $p$ is sub-linear on $\Dt(N,J)$. Let $U$ be the one-dimensional space of $\Dt(N,J)$ spanned by $\cP_0$. Define
$f(s\cP_0):= s p(\cP_0)$ for {\em any} $s\in\R$. Thus, $f$ is a linear functional on $U$ and satisfies
$f(\cP)\leq p(\cP)$ for any $\cP\in U$. Indeed, it holds with quality if $\cP=s\cP_0$ where $s\geq 0$ by definition, while $f(\cP)\leq 0\leq p(\cP)$ if $s\leq 0$ by (\ref{pgadolntpx}). By Hahn-Banach \index{Hahn-Banach}Theorem there exists a linear functional $F\equiv \cM_*\in \D(N,J)$ such that $\cP:\cM_*\leq p(\cP)$ for any $\cP\in \Dt(N,J)$ while $\cP_0:\cM_*= p(\cP_0)$. Thus
$$ \cP:\cM_* \leq \Xi(\cP+\cQ)-(\cP+\cQ):\cM_0$$
holds for any $\cP\in \Dt(N,J)$ and any $\cQ\in \SDt$. Thus
$$ (\cP+\cQ):(\cM_*+\cM_0)\leq \Xi(\cP+\cQ)+ \cQ:\cM_*$$
holds for {\em any} $\cP\in\Dt(N,J)$ and $\cQ\in \SDt$. Setting $\cQ=0$ we obtain that $\cM_*+ \cM_0\in  \Ss$ by Theorem \ref{main1}, and setting $\cP=-\cQ$ we obtain $\cQ:\cM_*\geq 0$ on $\SDt$. Since $\SDt$ is a subspace it follows that $\cQ:\cM_*=0$ for any $\cQ\in \SDt$, so $\cM_*\in \SDt^\perp$. We obtained that
$$ \sup_{\cM\in(\SDt^\perp+\cM_0)\cap \Ss} \cP_0:\cM\geq \cP_0:\cM_*=\inf_{\cQ\in\SDt} \Xi(\cP_0+\cQ)-\cQ:\cM_0$$
This implies the opposite inequality to (\ref{ineqSDt}).
$\Box$

\subsection{Optimal multi-(sub)partitions: Extended setting}\label{subsecaffine}
Given $\vzeta\in C(X, \R_+^{J})$ as in (\ref{vzetadef}),  $\vtheta\in C(X, \R^N)$ as in Assumption \ref{sassump1}, we consider the function
$$ \widehat{\zeta}:=(\zeta_1, \ldots \zeta_{J+N}):=(\vzeta, \vtheta)\in C(X, \R^{N+J}) \ . $$
This definition suggests that we extend the set of "goods" from $\J$ to $\J\cup \I$. Thus, we
consider  the {\em extended spaces} $\widehat{\Dt}:= \Dt(N,J)\times \Dt(N,N)$, where $\Dt(N,J)$ as in Definition \ref{defDt} and $\Dt(N,N)\sim \R^{N^2}$ parameterized by $\Dt(N,N)=(\vpp_{*,1}, \ldots \vpp_{*,N})$, $\vpp_{*,i}\in\R^N$. This space  is parameterized as
$$\widehat{\bf P}:=(\cP, \cP_*)=(\vpp_1, \ldots \vpp_N; \vpp_{*,1}, \ldots \vpp_{*,N})\sim \R^{N(N+J)}$$
(Recall $\vpp_i\in\R^{J}$  and $\vpp_{*,i}\in \R^N$ for $1\leq i\leq N$).

Similarly, the dual space
$\widehat{\D}:= \D(N,J)\times \D(N,N)$,
thus
$$\widehat{\bf M}:=(\cM, \cM^{'})=(\vM_1, \ldots \vM_N; \vM_{*,1}, \ldots \vM_{*,N})\sim \R^{N(N+J)}$$
as well.

The duality action of $(\cP, \cP^{'})$ on $(\cM, \cM_*)$ is the direct sum
$$ \widehat{\bf P}:\widehat{\bf M}:= \cP:\cM+ \cP^{'}:\cM^{'}:= \sum_{i\in\I}\left(\vpp_i\cdot{\vM_i}\right)_{\R^{J}}+
\sum_{i\in\I}\left(\vpp_{*,i}\cdot\vM_{*,i}\right)_{\R^N} $$
where the inner products refer to the corresponding spaces indicated for clarity.

Let
\be\label{tutu}\widehat{\Xee}(\cP, \cP_*):= \mu\left(\max_{i\in\I}\left( \vpp_i\cdot \vzeta + \vpp_{*,i}\cdot\vtheta\right)\right)\ee
resp.
\be\label{tutu1}\widehat{\Xeep}(\cP, \cP_*):= \mu\left(\max_{i\in \I}\left( \vpp_i\cdot \vzeta + \vpp_{*,i}\cdot\vtheta, \right)\vee 0\right) \ . \ee
Comparing with (\ref{xi0}-\ref{0Xi0}) we observe that (\ref{tutu}, \ref{tutu1}) are just the application of these definition to the current setting:
\begin{defi}\label{defsubQ} . \\
	\begin{description}
		\item{i)} \ $\SDt:= \left\{ (\cP,\vec{0}_{\Dt(N,N)}) ; \ \cP\in \Dt(N,J)\right\}$, then \\
		$\SDt^\perp:= \left\{ (\vec{0}_{\cM},\cM_* ); \ \cM_*\in \D(N,N)\right\}$.
		\item{ii)} $\widehat{\bf P}_0:=(\vec{0}_{\Dt(N,J)}, {\bf I}^{'}_0)$ where  ${\bf I}^{'}_0$ is the identity  $N\times N$ matrix.
		\item{iii)} \  $\widehat{\bf M}_0:= (\cM_0, \vec{0}_{\D(N,N)})$ where $\cM_0\in \D(N,J)$ is given.
	\end{description}
\end{defi}
With this notation we get (cf.  Definition \ref{agentvalue}) below)
\be\label{cthetanewdef}\theta(\vmu)\equiv \sum_{i\in\I}\mu_i( \theta_i )\equiv \widehat{{\bf P}_0}:\vmu\left(\widehat{\zeta}\right)
\equiv {\bf I}^{'}_0:\vmu(\vtheta)
\ . \ee
Let
\be\label{xitheta} \xiphi(x, \cP) := \max_{i\in \I}\left\{ \theta_i(x) + \vpp_i\cdot\vzeta(x), 0\right\}\ee
\be\label{xi+0} \xiphip(x, \cP) := \xiphi(x, \cP)\vee 0 \ee
Then, (\ref{tutu}, \ref{tutu1}) can be written as:
\be\label{Xiphi+}  \widehat{\Xeep}\left(\cP,{\bf I}^{'}_0\right)=\Xipsip(\cP):= \mu(\xiphip(\cdot, \cP))\ee
\be\label{Xiphitheta}   \widehat{\Xee}\left(\cP,{\bf I}^{'}_0\right)=\Xipsi(\cP):= \mu(\xiphi(\cdot, \cP) ) \ . \ee

Proposition \ref{prop12} can now be written as:
\begin{theorem}\label{main4}
	Given $\cM\in \vmS_N(\bar\mu)$ (resp. $\cM\in \uvmS_N(\bar\mu)$), then the maximum of $\theta(\vmu)$ in $\wP_{\{\cM_0\}}$ (resp.
	the maximum of $\theta(\vmu)$ in $\wSP_{\{\cM_0\}}$) is given by 
	\be\label{xistar} {\hatXi}^+(\cM_0) = \inf_{\cP\in \Dt(N,J)} \left[ \Xipsip(\cP) - \cP:\cM_0\right]  \ ,\ee  \index{${\hatXi}^+$}
	resp.
	\be\label{xistartheta} \hatXi(\cM_0) = \inf_{\cP\in \Dt(N,J)} \left[ \Xipsi(\cP) - \cP:\cM_0\right]  \ . \ee\index{$\hatXi$}
\end{theorem}

In Theorem \ref{main4} we left open the question of existence of a minimizer $\cP$ of
(\ref{xistar}, \ref{xistartheta}). See Theorem \ref{main5} below. 
\begin{defi}\label{escelate}
 $\cM\in\vmS_N(\bar\mu)$ is an   {\em an escalating capacity} if there is no $\cP$  minimizing    (\ref{xistartheta}).
\end{defi}
The reason for this notation will be explained in section \ref{escl}.   See also the box above Theorem \ref{main6}  and section \ref{EoE}.
\vskip .3in

\begin{defi}\label{agentvalue}   
	Given a weak (sub)partition\index{weak subpartition}  $\vmu$. 
	let
	$\cM_*(\vmu)\in \D^{'}(N,N)$ given by
	$ \{\mu^{(j)}(\theta_i)\}_{1\leq i,j\leq N}$. 
	The {\em extended feasibility set} is an extension of Definition \ref{defvmSw}
	$$\widehat{\vmS_N(\bar\mu)}:= \cup_{\vmu\in \wP}\left\{ \cM(\vmu), \cM_*(\vmu)\right\} \ \  \ \ \ ; \ \ \ \widehat{\uvmS_N(\bar\mu)}:= \cup_{\vmu\in \wSP}\left\{ \cM(\vmu), \cM_*(\vmu)\right\}  \ , $$
	and
	$$ {\vmS_N(\bar\mu)}^{'}:= \cup_{\vmu\in \wP}\left\{ \cM_*(\vmu)\right\} \ \ \ , \ \text{resp.} \ \ {\uvmS_N(\bar\mu)}^{'}:= \cup_{\vmu\in \wSP}\left\{ \cM_*(\vmu)\right\} \ .  $$
	The diagonal elements of $\cM_*(\vmu)$
	are called the {\em surplus values} of the agents under the (sub)partition $\vmu$:
	$$ Diag \left(\cM_*(\vmu)\right)\equiv \left( \mu_1(\theta_1), \ldots \mu_N(\theta_N)\right) $$
	where $\mu_i(\theta_i)$ is the surplus value of agent $i$.
	\par
	Consistently with Definition \ref{defsubQ} and (\ref{extendedS}, \ref{extendeduS}) we define
	$$\widetilde{\vmS_N(\bar\mu)}:=\left\{(\cM,t)\in \D(N,J)\otimes\R; \ \ (\cM, \cM_*)\in\widehat{\vmS_N(\bar\mu)}; t=Tr(\cM_*)) \right\} \ , $$
	resp.
	$$\widetilde{\uvmS_N(\bar\mu)}:=\left\{(\cM,t)\in \D(N,J)\otimes\R; \ \ (\cM, \cM_*)\in\widehat{\uvmS_N(\bar\mu)}; t=Tr(\cM_*) \right\} \ , $$
\end{defi}
Note: In terms of this definition, as well as with Definition \ref{defsubQ}-(ii)
$$\theta(\vmu)\equiv {\bf I}_0^{'}:\cM_*(\vmu)\equiv Tr\left(\cM_*(\vmu)\right)$$
is another equivalent formulation of (\ref{cthetanewdef}).  In particular, Theorem \ref{main4} implies the following, alternative definition for the optimal value of $\theta(\vmu)$ on $\wP_{\{\cM\}}$ (resp. on  $\wSP_{\{\cM\}}$ ).
\be\label{tracedef}  \hatXi(\cM):= \sup_{(\cM, \cM_*)\in \widehat{\vmS_N(\bar\mu)}} Tr(\cM_*) \equiv \sup\left\{t; (\cM,t)\in \widetilde{\vmS_N(\bar\mu)}\right\}\ee
resp.
\be\label{tracedefu}{\hatXi}^+(\cM):= \sup_{(\cM, \cM_*)\in \widehat{\uvmS_N(\bar\mu)}} Tr(\cM_*) \equiv \sup\left\{t; (\cM,t)\in \widetilde{\uvmS_N(\bar\mu)}\right\}\ . \ee
From (\ref{tracedefu}) we obtain that $(\cM, {\hatXi}^+(\cM))\in\partial\widetilde{\uvmS_N(\bar\mu)}$ {\em for any} $\cM\in\uvmS_N(\bar\mu)$.
It is also evident that $(\cM, {\hatXi}(\cM))\in\partial\widetilde{\vmS_N(\bar\mu)}$, since $\vmS_N(\bar\mu)$ (hence $\widetilde{\vmS_N(\bar\mu)}$) contains no interior points.
We now imply Corollary \ref{main2cor} to obtain
\begin{cor}\label{main2corex}
	$(\cM,t)$ is an inner point of $\widetilde{\uvmS_N(\bar\mu)}$  if and only if $(\cP, s)=0$ is the {\em only}  minimizer of
	$$ \inf_{(\cP,s)\in (\Dt(N,J)\oplus\R)}\widehat{\Xeep}\left(\cP,s{\bf I}^{'}_0\right)-(\cP,s):(\cM,t)=0 $$
\end{cor}
Hence,
for any $\cM\in\uvmS_N(\bar\mu)$ (resp. $\cM\in\vmS_N(\bar\mu)$) there exist $(\cP,s)\not=0$ such that
\be\label{a12}\widehat{\Xeep}\left(\cP,s{\bf I}^{'}_0\right)=\cP:\cM + s{\hatXi}^+(\cM)\ee
resp.
\be\label{b12} \widehat{\Xee}\left(\cP,s{\bf I}^{'}_0\right)=\cP:\cM + s{\hatXi}(\cM) \ . \ee

To understand the meaning of (\ref{a12}, \ref{b12}) we compare it to Theorem \ref{main4}. By (\ref{xitheta}- \ref{Xiphitheta}) we may write
$$\widehat{\Xeep}\left(\cP,s{\bf I}^{'}_0\right)= \Xi^{s\theta,+}_\xi(\cP) \ \ \text{resp.} \ \ \widehat{\Xee}\left(\cP,s{\bf I}^{'}_0\right)= \Xi^{s\theta}_\xi(\cP) $$
so (\ref{a12}, \ref{b12}) are equivalent to the following:
\begin{theorem}\label{main5}
	For any $\cM\in\vmS_N(\bar\mu)$ (resp. $\cM\in\uvmS_N(\bar\mu)$) there exists $(\cP_0,s_0)\not=0$ such that
	\begin{multline}\label{a13}\inf_{(\cP,s)\in \Dt(N,J)\times \R} \left[ \Xi_\zeta^{s\theta,+}(\cP) -s{\hatXi}^+(\cM)- \cP:\cM\right]= \\ \Xi_\zeta^{s_0\theta,+}(\cP_0) -s_0{\hatXi}^+(\cM)- \cP_0:\cM=0,  \end{multline}
	resp.
	\begin{multline}\label{b13}\inf_{(\cP,s)\in \Dt(N,J)\times \R} \left[ \Xi_\zeta^{s\theta}(\cP) -s{\hatXi}(\cM)- \cP:\cM\right]=\\ \Xi_\zeta^{s_0\theta}(\cP_0) -s_0{\hatXi}(\cM)- \cP_0:\cM=0 \ .  \end{multline}
\end{theorem}
Now: $(\cP_0,s_0)\not=0$ implies that {\em either} $s_0\not=0$ or $\cP_0\not=0$ (or both). If $s_0\not=0$ (in that case the reader can show that, in fact,  $s_0>0$)  we can divide (\ref{a13}, \ref{b13}) by $s_0$, using (\ref{xitheta}- \ref{Xiphitheta}) to observe $s^{-1}\Xi_\zeta^{s\theta}(\cP)=
\Xi_\zeta^\theta(\cP/s)$,
and conclude that {\em there exists a minimizer} $s_0^{-1}\cP_0$ to  (\ref{xistar}, \ref{xistartheta}) in Theorem \ref{main5}.
\par
In particular:
\begin{tcolorbox}
	\begin{quote}
		The case of escalation  \index{escalation}(Definition \ref{escelate}) corresponds to  $s_0=0$ (hence $\cP_0\not=0$) in Theorem \ref{main5}.
	\end{quote}
\end{tcolorbox}

The Theorem  below implies another  characterization of the optimal (sub)partition:
\begin{theorem}\label{main6}
	{\em Any} optimal (sub)partition $\vmu$ corresponding to $\cM\in\uvmS_N(\bar\mu)$ (resp. $\cM\in\vmS_N(\bar\mu)$) satisfies the following:  
	\begin{multline}\label{ApnoP} \ supp(\mu_i)\subset A_i^{\theta,+}:= \\ \left\{ x\in X; \vpp_{0,i}\cdot\vzeta(x)+s_0\theta_i(x)= \max_{k\in\I}\left[\vpp_{0,k}\cdot\vzeta(x)+s_0\theta_k(x)\right]_+\right\} \ . \end{multline}
	resp.
	\begin{multline}\label{AnoP}supp(\mu_i)\subset A_i^{\theta}:=\\ \left\{ x\in X; \vpp_{0,i}\cdot\vzeta(x)+s_0\theta_i(x)= \max_{k\in\I}\vpp_{0,k}\cdot\vzeta(x)+s_0\theta_k(x)\right\} \ , \end{multline}
	where $\cP_0=(\vpp_{0,1}, \ldots \vpp_{0,N})$ and $s_0\in\R$ are as given by Theorem \ref{main5}.
\end{theorem}

\begin{proof}
	Let $\vmu:= (\mu_1, \ldots \mu_N)$ be an optimal (sub)partition and let $\mu_0=\mu-\sum_{i\in\I}\mu_i$. By (\ref{a13}),
	$$ 0=\Xi_\zeta^{s_0\theta,+}(\cP_0) -s_0{\hatXi}^+(\cM)- \cP_0:\cM$$
	while, by (\ref{Xiphitheta}) and since $\vmu\in \wSP_{\{\cM\}}$ and is an {\em optimal} (sub)partition
	\begin{multline}\label{multrrg}\Xi_\zeta^{s_0\theta,+}(\cP_0)=\sum_{i\in\I\cup\{0\}}
	\mu_i\left(\max_{k\in\I}\left([s_0\theta_k+\vpp_{0,k}\cdot\vzeta]_+\right)\right)\geq \\
	$$\sum_{i=1}^N\left[ s_0\mu_i(\theta_i)+ \mu_i(\vpp_{0,i}\cdot\vzeta)\right]= s_0{\hatXi}^+(\cM)+ \cP_0:\cM \end{multline}
	so the inequality above is an equality. In particular, for $\mu_i$ a.e
	$$ \max_{k\in\I}[s_0\theta_k(x)+\vpp_{0,k}\cdot\vzeta(x)]_+=s_0\theta_i(x)+ \vpp_{0,i}\cdot\vzeta(x) . $$
	By (\ref{ApnoP}) and the continuity of $\vtheta$, $\vzeta$ we obtain that $supp(\mu_i)\subset A_i^{\theta,+}$.
	
	The case  $\cM\in\vmS_N(\bar\mu)$ is proved similarly.
\end{proof}

\subsection{Price adaptation and Escalation}\label{escl}
So far we considered the equilibrium vector $\vpp$ as a tool for achieving  optimal (sub)partitions (Sections \ref{secbigbrother}, \ref{seconfparfreemarket}, see also Proposition \ref{equalityinprop}).  One may expect that, in the case of multi-partition, the price $-p_i^{(j)}$ should be  interpreted as  the equilibrium price {\em charged} by agent $i$ for the good $j$ in order to obtain the
required capacity $m_{0,i}^{(j)}$.
\par
However, we didn't consider {\em how the agent determines these prices}. It is conceivable that this process is made by trial and error. Thus, when the agent $i$ "guess" the price vector $-p_i^{(j)}$ for the good $j\in \J$, she should consider the number of consumers of $j$  who accept this prices and compare it with the desired capacities $m^{(j)}_{0,i}$.  If she is underbooked, namely  $m^{(j)}_{0,i}$ is above the number of her consumers for $j$, she will decrease the price in order to attract more consumers. If, on the other hand, she is overbooked, then she will increase the price to get rid of some.
\par
But how does the agent $i$ determines the number of consumers of $j$ who accept the price $-p_i^{(j)}$? Recall that each consumer $x$ need the fraction $\zeta_j(x)$ of the good $j$. Hence the price paid by consumer $x$ to agent $i$ for the {\em basket} $J$ is $-\vzeta(x)\cdot \vpp_i$. Thus, she only need to determine the {\em entire} set of her consumers $\mu_i$. Once $\mu_i$ is known, she knows the current capacity  $m^{(j)}_i(\cP)=\mu_i(\zeta_j)$ for the current price matrix $-\cP$.
\par
Recalling (\ref{A+0}) we  obtain that the set of all {\em candidates} $A_i^+(\cP)$ who {\em may} hire $i$ at the price level $-\vpp_i$ is the set of consumers who makes a non-negative profit for trading with $i$, and this profit is {\em at least as large} as the profit they may get form trading with any other agent. Thus
\be\label{A+pzeta} A^+_i(\cP):= \{x\in X; \theta_i(x)+\vpp_i\cdot \vzeta\geq  [\theta_k(x)+\vpp_k\cdot\vzeta(x) ]_+\ \ \forall k\in\I\} \  .  \  \ee
In fact, there may be a set of "floating" consumers who belong to two (or more) such sets (note that $\mu(A^+_i(\cP)\cap A^+_k(\cP))$ is not necessarily zero for $i\not=k$). The only information on which  $i$ can be sure of, upon her choice of the price vector   $-\vpp_i$,  is that {\em all} her consumers are in the set $A^+_i(\cP)$.
\par
Note that $\Xipsi$ and ${\Xipsi}^+$ are convex functions. By Proposition \ref{propA7}  (recalling (\ref{Xiphitheta}) and Definition \ref{cf5}) we get that, under the choice $\cP$, the corresponding capacities set $\cM(\cP)$ is given by the sub-gradient 
\be\label{cMsubgrad} \cM(\cP)\in\partial_{-\cP} {\Xipsi}^+\not=\emptyset\ee
\par
Let $m^{(j)}_i(\cP)$ be in the $(i,j)$ component of $\cM(\cP)$. According to the above reasoning, the agent $i$ will decrease $p_i^{(j)}$ if
$$m^{(j)}_{0,i}>\max_{\cM\in \cM(\cP)}m^{(j)}_i(\cP)$$ and will increase $p_i^{(j)}$ if
$$m^{(j)}_{0,i}<\min_{\cM\in \cM(\cP)}m^{(j)}_i(\cP) \ . $$ If
$$\min\left\{ m^{(j)}_i(\cP) \ ; \ \cM\in \cM(\cP)\right\}\leq m^{(j)}_{0,i}\leq \max\left\{ m^{(j)}_i(\cP) \ ; \ \cM\in \cM(\cP)\right\}$$
then $i$ will, probably, not change $p_i^{(j)}$.

So, if $-\cP(t)$ is the value of the price matrix at time $t$ and $\cP_0$ is its initial value at $t=t_0$ we presume that the {\em forward derivative} $d^+\cP(t)/dt$ exists and would like to state that $d^+\cP(t)/dt\in \cM(\cP(t))-\cM_0 $. However, $d^+\cP(t)/dt$, if exists, is in the space $\Dt(N,J)$ while $\cM(\cP)$ and $\cM_0$ are in $\D(N,J)$. So, we have to identify $\Dt(N,J)$ and $\D(N,J)$ in some way. For this we define a linear mapping ${\bf J}:\D(N,J)\rightarrow\Dt(N,J)$ such that
$$ {\bf J}\cM:\cM>0 \ \ \ \forall \cM\not=0 \ \ \text{in} \ \ \D(N,J) \ . $$
This definition makes $\D(N,J)$  and $\Dt(N,J)$  inner product space, and
\be\label{defnorm}|\cM|:=\sqrt{{\bf J}\cM:\cM} \ \ ; \ \ |\cP|:=\sqrt{\cP:{\bf J}^{-1}\cP}\ee
are natural norms.

So,   we presume that
\be\label{difincl} \frac{d^+}{dt}\cP(t)\in {\bf J}\left(\cM(\cP(t))-\cM_0\right) \ \ , \ \ t\geq t_0 \ \ ; \ \ \cP(t_0)=\underline{\cP}\in \Dt(N,J)\ . \ee
The condition (\ref{difincl}) is an example of a {\em differential inclusion}. It is a generalization of a system of {\em Ordinary Differential Equations} (ODE). In fact, by (\ref{cMsubgrad}) we observe that it is an ODE if the subgradient \index{subgradient}of ${\Xipsi}^+$ is a singleton, which is equivalent, via Proposition \ref{cf7}, to the assumption that ${\Xipsi}^+$ is differentiable anywhere.

The existence and uniqueness of $\cP(\cdot)$ satisfying (\ref{difincl}) is a common knowledge,  due the convexity of ${\Xipsi}^+$ (\cite{amb1}, \cite{amb2}). For the sake of completeness we introduce below some of the steps toward the proof of this result.
\par 
Let $\eps>0$ and $t_j:=t_0+j\eps$. If $\cP(t_j)$ is known, define
\be\label{defPj+1}\cP(t_{j+1}):=  \min!_{\cP\in\Dt(N,J)}\left\{ (2\eps)^{-2}\left|\cP(t_j)-\cP\right|^2+ {\Xipsi}^+(-\cP)+\cP:\cM_0\right\} \ . \ee
Since $\cP\rightarrow{\Xipsi}^+(-\cP)$ is convex, the term in brackets above is {\em strictly convex} and $\cP(t_{j+1})$ is unique (c.f. Definition \ref{defA1}). Moreover, it follows from (\ref{defPj+1}) and (\ref{cMsubgrad}) that $\cP(t_{j+1})$ satisfies the {\em implicit} inclusion
\be\label{findifcl}\cP(t_{j+1})\in \cP(t_j)+\eps{\bf J}\left(\cM(\cP(t_{j+1})) -\cM_0\right) \ . \ee
Next, we interpolate on time to define $t\rightarrow \cP_\eps(t)$ for any $t\geq t_0$ as
$$ \cP_\eps(t)= \eps^{-1}\left[(t-t_j)\cP(t_{j+1})+(t_{j+1}-t)\cP(t_{j})\right] \ \ \text{for} \ \ t_j\leq t < t_{j+1}\ ,\ j=0,1,\ldots \ . $$
Using (\ref{findifcl}) it can be proved that  $\cP(t)=\lim_{\eps\rightarrow 0}\cP_\eps(t)$ for any $t\geq t_0$ is the unique solution of the inclusion (\ref{difincl}).

It is also evident from (\ref{defPj+1}) that
$$ {\Xipsi}^+(-\cP_{t_{j+1}})+\cP_{t_{j+1}}:\cM_0\leq {\Xipsi}^+(-\cP_{t_{j}})+\cP_{t_{j}}:\cM_0  $$
for any $j=0,1,2\ldots$. Hence
$$ t \rightarrow {\Xipsi}^+(-\cP(t))+\cP(t):\cM_0$$
is non-increasing. Moreover by (\ref{cMsubgrad}, \ref{defnorm}, \ref{difincl})
\begin{multline}\frac{d^+}{dt}\left[{\Xipsi}^+(-\cP(t))+\cP(t):\cM_0\right]=  -{\bf J}\left(d^+\cP(t)/dt\right): \left(\partial_{-\cP(t)}{\Xipsi}^+-\cM_0\right)= \\ -\left|\left(\cM(\cP(t))-\cM_0\right)\right|^2\ . \end{multline}
Recall from Theorem \ref{main4} that
$${\hatXi}^+(\cM_0) = \inf_{\cP\in \Dt(N,J)} \left[ \Xipsip(\cP) - \cP:\cM_0\right] \leq \Xipsip(\cP(t)) - \cP(t):\cM_0$$
for any $t\geq t_0$.
We now obtain the reason for the terminology of "escalating capacity" in Definition \ref{escelate}:
\begin{theorem}\label{thescelation}
	The solution of (\ref{difincl}) satisfies
	$$ \lim_{t\uparrow\infty} \Xipsip(\cP(t)) - \cP(t):\cM_0= {\hatXi}^+(\cM_0) \ . $$
	If $\cM_0$ is non-escalating then $\lim_{t\uparrow \infty}\cP(t)=\cP_0$ where $\cP_0$ is a minimizer of (\ref{xistar}). Otherwise, the limit of $\cP(t)$ does not exist and $\lim_{t\uparrow\infty}|\cP(t)|=\infty$.  However,
	$$ \lim_{t\uparrow\infty}\frac{\cP(t)}{|\cP(t)|}:= \cP_0$$ 
	exists, where $\cP_0$ is a minimizer of (\ref{a13}).\footnote{Note that in case of escalating $\cM_0$, $s_0=0$ while  $\cP_0 \not=0$ in  (\ref{a13}).}
\end{theorem}





\section{Optimal Strong multipartitions}\label{scunique}
Recall the definition of $\theta$ on the set of strong $N-$(sup)partitions:
$$ \theta(\vA):=\sum_{i=1}^N\int_{A_i}\theta_id\mu \ . $$
Let $\K\subset\D(N,J)$ be a compact convex set. The main question we address in this section is:

\begin{tcolorbox}
	Under which conditions there is a {\em unique}, strong (sub)partition $\vA$ which maximizes $\theta$ in the set of weak (sub)partitions $\wPzeta_{\{\K\}}$ (resp. $\wSPzeta_{\{\K\}}$)?
\end{tcolorbox}

Following the discussion of Chapter \ref{S(M)P} and Theorems \ref{main5}, \ref{main6}, we focus on the "natural suspects" 
\be\label{Aiphi} A^\theta_i(\cP):= \left\{ x\in X; \vpp_i\cdot\vzeta(x)+\theta_i(x)> \max_{j\not= i}\vpp_j\cdot\vzeta(x)+\theta_j(x)\right\}\ee
\be\label{Aiphip} A^{\theta,+}_i(\cP):= A^\theta_i(\cP)- A^\theta_0(\cP)\ee
where
$$A^\theta_0(\cP):=\left\{ x\in X; \vpp_i\cdot\vzeta(x)+\theta_i(x)\leq 0 \ \forall i\in\I \right\} \ . $$
Recall that the utility of a consumer $x$  of agent $i$ charging price $\vpp_i$ is $\theta_i(x)-\vpp_i\cdot\vzeta(x)$. Thus,
the set of (sub)partitions $A_i^\theta(-\cP)$ ($A_i^{\theta,+}(-\cP)$) represents  subsets of consumers   who  prefer the agent $i$ over all other agents, given the  price matrix $\cP$.

As suggested by Theorem \ref{main5}, there is a close relation between optimal $\vmu$ and strong (sub)partitions  of the form (\ref{Aiphi}, \ref{Aiphip}). Thus   we rephrase our question as:

\begin{tcolorbox}
	Under which conditions there is a {\em unique} $\cP\in\Dt(N,J)$ such that \\
	(\ref{Aiphi}) (resp.  (\ref{Aiphip})) are $\theta$ optimal strong (sub)partitions in $\wPzeta_{\{\K\}}$ (resp. $\wSPzeta_{\{\K\}}$)?
\end{tcolorbox}
At the first stage we concentrate in the case where $\K=\{\cM\}$ is a singleton.
Recall (\ref{xistar},\ref{xistartheta}):
\be\label{xistar1} {\hatXi}^+(\cM) = \inf_{\cP\in \Dt(N,J)} \left[ \Xipsip(\cP) - \cP:\cM\right]  \ ,\ee
\be\label{xistartheta1} \hatXi(\cM_0) = \inf_{\cP\in \Dt(N,J)} \left[ \Xipsi(\cP) - \cP:\cM\right]  \ . \ee
%
%
%
where, from (\ref{Xiphi+}, \ref{Xiphitheta})
\be\label{Xiphi+1}  \Xipsip(\cP)\equiv \mu\left(\xiphip(x, \vpp)\right)\ \ \ ; \ \
\xiphip(x,\cP)\equiv \max_i(\theta_i(x)+\vpp_i\cdot\vzeta(x))_+\ee
\be\label{Xiphitheta1}   \Xipsi(\cP)\equiv \mu\left(\xiphi(x, \cP)\right) \ \ ; \ \ \xiphi(x,\cP)\equiv\max_i (\theta_i(x)+\vpp_i\cdot\vzeta(x))\ee   
We now consider the following adaptation of Assumption \ref{mainass3}: 
\begin{assumption}\label{mainass2} 
	\begin{description}.
		\item{i)} For any $i\not=j\in\I$ and any $\vpp\in\R^{J}$,   \\ $\mu\left(x\in X \ ; \ \  \vpp\cdot\vzeta(x)+ \theta_i(x)-\theta_j(x)=0\right)=0$ .
		\item{ii)}  For any $i\in\I$ and any $\vpp \in\R^{J}$,   \\ $\mu\left(x\in X \ ; \ \  \theta_i(x)=\vpp\cdot\vzeta(x)\right)=0$ .
	\end{description}
\end{assumption}

By Assumption \ref{mainass2}-(i) it follows that $\{A^\theta_i(\cP)\}$ is, indeed, a strong partition for {\em any}  $\cP\in\Dt(N,J)$.\index{strong  (deterministic) partition}
Likewise, Assumption  \ref{mainass2}-(i,ii) implies that $\{A^{\theta,+}_i(\cP)\}$ is  a strong subpartition.
In particular
\be\label{disjointAi}A^\theta_i(\cP)\cap A^\theta_j(\cP)=A^{\theta,+}_i(\cP)\cap A^{\theta,+}_j(\cP)=\emptyset\ee
for $i\not= j$.
\subsection{Example of escalation}\label{EoE}\index{escalation}
Let $\cP_0, \cM_0$ be given as in Corollary \ref{coruniquemaxcoal}. Then $(A_1(\cP_0), \ldots A_N(\cP_0))$ given by  (\ref{Aizero}) where $\cP_0$ substituted for $\cP$, is the only partition in $\sPzeta_{\{\cM_0\}}$.  Assume that $\theta_i-\theta_j$ is independent of the components of $\vzeta$ on  $\bar{A}_i(\cP_0)\cap \bar{A}_j(\cP_0)$. \footnote{$\bar{A}$ stands for the closure of the set $A$. }This  implies that for any $(\lambda_1, \ldots \lambda_N)\in \R^N$ there exists $x\in \bar{A}_i(\cP_0)\cap \bar{A}_j(\cP_0)$  such that $\theta_i(x)-\theta_j(x)\not=\vlambda\cdot\vzeta(x)$. 
\begin{prop}
	For $\cM_0$ and $\theta$ as above, $\cM_0$ is an escalating capacity for the given $\theta$. 
\end{prop}
\begin{proof}
	If a minimizer $\cP=(\vpp_1\ldots \vpp_N)$ of (\ref{xistartheta}) exists then  (\ref{Aiphi})
	is a partition in $\sSPzeta_{\{\cM_0\}}$ which, by Corollary \ref{coruniquemaxcoal} must be the same as $A_i(\cP_0)$. In particular
	$$ \bar{A}_i(\cP_0)\cap \bar{A}_j(\cP_0)= \bar{A}_i^\theta(\cP)\cap \bar{A}_j^\theta(\cP)  \ . $$

	Any point  in the  set  $\bar{A}_i^\theta(\cP)\cap \bar{A}_j^\theta(\cP)$ must satisfies 
	$$ \theta_i(x)+ \vpp_i\cdot \vzeta(x)=  \theta_j(x)+ \vpp_j\cdot \vzeta(x)$$
	so $\theta_i(x)-\theta_j(x)=(\vpp_j-\vpp_i)\cdot \vzeta(x)$.  Since any such point is in $\bar{A}_i(\cP_0)\cap \bar{A}_j(\cP_0)$ as well, we obtain a contradiction to the assumption on $\theta$. 
\end{proof}
\subsection{Uniqueness for a prescribed capacity}
It turns that Assumption \ref{mainass2}, {\em standing alone}, is enough for  a uniqueness of optimal subpartition for  $\cM\in\uvmS_N(\bar\mu)$, provided  $\cM$ is an {\em interior point} of $\uvmS_N(\bar\mu)$. The key to this result is the following observation, generalizing
Lemma \ref{51}:
\begin{lemma}\label{difcor}
	Under Assumption \ref{mainass2}-(i), $\Xipsi$ is differentiable at any $\cP\in\Dt(N,J)$ and satisfies (\ref{diffmphi})-(a).
	If, in addition, Assumption \ref{mainass2}-(ii) is satisfied than $\Xipsip$ is differentiable as well and (\ref{diffmphi})-(b) holds. Here
	\be\label{diffmphi} a) \  \frac{\partial\Xipsi}{\partial \vpp_i}(\cP)= \int_{A_i^\theta(\vpp)} \vzeta d\mu \ \ , \ \ b) \ \frac{\partial\Xipsip}{\partial \vpp_i}(\cP)= \int_{A_i^{\theta,+}(\vpp)}\vzeta d\mu \  \ . \ee
\end{lemma}
\begin{remark}
	Notice that the conditions of Lemma \ref{51} as well as Assumption \ref{mainass3}, are not required in Lemma \ref{difcor}.
\end{remark}

\begin{theorem}\label{main3} \ . 
	Let Assumption \ref{mainass2}(i+ii). If
	$\cM$ is an interior point of $\uvmS_N(\bar\mu)$
	then
	there exists a unique subpartition
	which maximize
	$\theta$ 
	in $\wSPzeta_{\{\cM\}}$,  and this subpartition is a strong one, given by $\{A^{\theta,+}_i(\cP^0)\}$ (\ref{Aiphip}) for some uniquely determined $\cP^0\in\Dt(N,J)$.
\end{theorem}
Note that $int(\vmS_N(\bar\mu))=\emptyset$ so Theorem \ref{main3} is void for $\cM\in\vmS_N(\bar\mu)$.
\begin{proof} 
	Let   $-\hatuXi$ be the Legendre transforms of  $\Xipsip(-\cdot)$, i.e\index{Legendre transform} 
	\be\label{ssdfty} \hatuXi(\cM)=\inf_{\cP\in\Dt(N,J)}\Xipsip(\cP)-\cP:\cM \ . \ee 
	We prove that the essential domain of $-\hatuXi$  is the same as the essential domain of  $-\Sigma_\zeta^{0,+}$,  namely $\uvmS_N(\bar\mu)$. Indeed, by definition (\ref{Xiphi+1}, \ref{0Xi0+}) 
	$$\Xeep(\cP)+\|\vtheta\|_\infty \mu(X) \geq \Xipsip(\cP)\geq \Xeep(\cP)-\|\vtheta\|_\infty \mu(X)$$
	hence
	$$ \inf_{\cP\in \it\Dt(N,J)}\Xipsip(\cP)-\cP:\cM>-\infty \Leftrightarrow\inf_{\cP\in \it\Dt(N,J)}\Xeep(\cP)-\cP:\cM>-\infty $$
	which implies the claim via Theorem \ref{main1}.
	
	If $\cM$ is an interior point in the essential domain of $-\hatuXi$ then,  by (\ref{cf5}) the subgradient \index{subgradient} $\partial_{\cM}(-\hatuXi)$ is not empty. Any $\cP\in-\partial_{\cM}(-\hatuXi)$ is a minimizer of
	(\ref{ssdfty}).  Let $\cP^0$ be such a minimizer. Let $\vmu$ is a maximizer of $\theta$ in $\wSPzeta_{\{\cM\}}$.
	Then, by (\ref{Xiphi+1})
	\begin{multline}\label{734} \hatuXi(\cM)=\Xipsip(\cP^0)-\cP^0\cdot\cM=\mu(\xiphip(x,\cP^0))-\cP^0\cdot\cM \geq \\ \sum_{i\in\I}\mu_i(\xiphip(\cdot,\cP^0))-\cP^0\cdot\cM \geq \sum_{i\in\I}\mu_i(\theta_i+\vpp^0_i\cdot\vzeta)-\cP^0\cdot\cM \ . \end{multline}
	By Lemma \ref{difcor}  we obtain that $\Xipsip$ is differentiable at $\cP^0$ and, by  (\ref{diffmphi}):\footnote{Here is the {\em only} place in the proof we use the differentiability of $\Xipsip$.}
	$$ \mu_i(\vzeta)=\int_{ A^{\theta,+}_i(\cP^0)}\vzeta d\mu=\vM_i$$
	  Hence
	$$
	(\ref{734})=  \sum_{i\in\I}\mu_i(\theta_i)\equiv \theta(\vmu)=\hatuXi(\cM) \
	$$
	where the last equality follows from Theorem \ref{main4}.
	Hence the middle inequality in (\ref{734}) is an equality. Since $\xiphip(x,\cP)\geq \theta_i+\vpp_i\cdot\vzeta$ everywhere by (\ref{Xiphi+1}) we obtain that
	$\xiphip(x,\vpp^0)= \theta_i(x)+\vpp^0_i\cdot\vzeta(x)$ for any  $x\in supp(\mu_i)$. That is, $supp(\mu_i)\supseteq A^{\theta,+}_i(\cP^0)$ by  (\ref{Aiphip}).  This, (\ref{disjointAi}) and  $\mu_i\leq \mu$ imply that
	$\mu_i$ is the restriction of $\mu$ to $A^{\theta,+}_i(\cP^0)$. In particular, the maximizer  $\vmu\in\wSPzeta_{\{\cM\}}$ is unique, and is a strong subpartition given by $\vA^{\theta,+}(\cP^0)$.
\end{proof}
\vskip .3in

As a byproduct of the uniqueness of the minimizers $\cP^0$ of (\ref{ssdfty}) and via   Proposition \ref{cf7} we obtain
\begin{cor}\label{coruniquediff}
	$\hatuXi$ is differentiable at any inner point of its essential domain $\uvmS_N(\bar\mu)$.
\end{cor}
Combining both Theorems \ref{uniquestrong} and  \ref{main3} we obtain
\begin{theorem}\label{main33} 
	Let Assumption \ref{mainass2} (i+ii) {\em and} \ref{mainass3} .
	If  $\cM\in\uvmS_N(\bar\mu)$
	then there exists a unique, maximal  coalition's ensemble $\cI$ and a strong subpartition $\underline{\vA}_{\cI}$ such that any $\vmu\in \wSPzeta_{\{\cM\}}$\index{coalitions ensemble}
	which maximize
	$\theta$
	in $\wSPzeta_{\{\cM\}}$ is embedded in $\underline{\vA}_{\cI}$.
	
\end{theorem}

\begin{proof}
	If $\cM$ is an interior point of $\uvmS_N(\bar\mu)$ then Theorem \ref{main3} implies the uniqueness of the subpartition corresponding to the coalition's ensemble of individuals, which is the maximal possible coalition.
	\par
	If $\cM\in\partial\uvmS_N(\bar\mu)$ then Theorems \ref{uniquestrong}   implies the uniqueness of maximal coalition's \index{maximal coalition}  ensemble $\cI$ \index{coalitions ensemble}and a unique strong subpartition $\vA_{\cI}$ corresponding to $\cI(\cM)$. Evidently, any subpartition in $\wSPzeta_{\{\cM\}}$ must be embedded in $\vA_{\cI}$ .
\end{proof}
If each agent $\{i\}$ agrees on a fixed exchange rate $\vzet^{(i)}$ subject to Assumption \ref{mainass31}, then we can get {\em unconditional} uniqueness. In fact
\begin{theorem}\label{main3Z}
	Under Assumption \ref{mainass2},  \ref{mainass3} and  \ref{mainass31}:
	If $\vM\in \vZet(\uvmS_N(\bar\mu))$  then there exists a unique subpartition which maximizes $\theta$ in $\cup_{\cM\in \vZet^{-1}(\vM)}\wSPzeta_{\{\cM\}}$, and this subpartition is a strong one.
\end{theorem}
\begin{proof}
	We may assume that all components of $\vM$ are in $\R_{++}^N$ for, otherwise, we restrict to a subset of $\I$ on which the components of $\vM$ are all positive, and note that all the assumptions of the Theorem are valid also for the restricted system.
	
	If $\vM\in\partial\vZet(\uvmS_N(\bar\mu))$ then, by the above assumption  and Theorem \ref{uniquevZet}, there exists a unique subpartition and there is nothing to prove.
	So, we assume $\vM$ is an inner point of $\vZet(\uvmS_N(\bar\mu))$.
	By  Theorem \ref{main4}
	\be\label{list1}\sup_{\cM}\left\{\hatuXi(\cM) \  ; \vZet(\cM)=\vM\right\}=\sup_{\vmu}\left\{ \theta(\vmu) ; \ \ \mu_i(\vzet_i\cdot\vzeta)=m_i\right\} . \ee
	On the other hand
	\begin{multline}\label{list2} \sup_{\cM}\left\{\hatuXi(\cM) \  ; \vZet(\cM)=\vM\right\}=\\ \sup_{\cM}\inf_{\cP}\left\{\Xipsip(\cP)-\cP:\cM \ ; \
	\cP \in\Dt(N,J), \  \cM\in \vZet^{-1}(\vM)\right\} \  \end{multline}
	In addition, if $\vZet(\cM)=\vM$ then by Lemma \ref{lemdual2}
	\be\label{rs}\inf_{\cP}\Xipsip(\cP)-\cP:\cM\leq \inf_{\vq\in\R^{J}}\Xipsip(\vZet^*(\vq))-\vZet^*(\vq)\cdot\cM=\inf_{\vq\in\R^{J}}\Xipsip(\vZet^*(\vq))-\vq\cdot\vM\ee
	Since $\vM$ is an inner point of $\vZet(\uvmS_N(\bar\mu))$ (which is the essential domain of \\ $\Xipsip\circ\vZet^*$) and $\Xipsip$ is differentiable at any point by Lemma \ref{difcor} and Proposition \ref{propA7} imply that
	the infimum of the right side of (\ref{rs}) is attained at some $\vq_0\in\R^{J}$ and
	$$ m_j=\frac{\partial}{\partial q_j} \Xipsip\circ\vZet^*(\vq_0)=\int_{A^{\theta,+}_j(\vZet^*(\vq_0))}\vzet_j\cdot \vzeta d\mu \ . $$
	However,  $\xiphip(x, \vZet^*(\vq_0))=\vzet_j\cdot\vzeta(x)$ for any $x\in A^{\theta,+}_j(\vZet^*(\vq_0))$, hence 
	\begin{multline}\label{list3}\inf_{\vq\in\R^{J}}\Xipsip(\vZet^*(\vq))-\vq\cdot\vM= \sum_{i\in\I}\int_{A^{\theta,+}_j(\vZet^*(\vq_0))}\vzet_j\cdot\vzeta(x)d\mu  \\ \leq \sup_{\vmu}\left\{ \theta(\vmu) ; \ \ \mu_i(\vzet_i\cdot\vzeta)=m_i\right\} . \end{multline}
	(\ref{list1}-\ref{list3}) imply that $ \{A^{\theta,+}_j(\vZet^*(\vq_0))\}$ is an optimal strong subpartition of $\theta$ in
	$\cup_{\cM\in \vZet^{-1}(\vM)}\wSPzeta_{\{\cM\}}$. The uniqueness of this partition is proved as in Theorem \ref{main3}.
\end{proof}
\subsection{Uniqueness within the feasibility domain}
Let us recall the generalized definition of Under Saturation (US), Saturation (S) and Over Saturation (OS) (\ref{gUS}, \ref{gS}, \ref{gOS}). 
 \index{over saturated (OS)} \index{under saturated (US)}Theorems \ref{main33}, \ref{main3Z} deal with the existence and uniqueness of a strong (sub)partition maximizing $\theta$  for each {\em prescribed} $\cM\in\uvmS_N(\bar\mu)$. Here we discuss the uniqueness of optimal (sub)partition within closed  convex sets $\K\subset\D(N,J)$.

Recall that $\hatuXi(\cM)$ (\ref{ssdfty})  is the maximal  value of $\theta$ for subpartitions for a prescribed $\cM\in\uvmS_N(\bar\mu)$. If we look for a  subpartition maximizing $\theta$ on
$\sPzeta_{\K}$ (\ref{SstrongM}) then it must belong to $\sPzeta_{\cM_0}$ where $\cM_0$ is a maximizer
of $\hatuXi$ on $\uvmS_N(\bar\mu)\cap \K$.   Granted the uniqueness of a maximal subpartition of $\theta$ in $\sPzeta_{\{\cM\}}$ for any $\cM\in\uvmS_N(\bar\mu)$  we obtain

\begin{tcolorbox}
	The uniqueness of $\theta$ maximizer in  $\sPzeta_{\K}$ is equivalent to a uniqueness of the maximizer of $\hatuXi$ on $\uvmS_N(\bar\mu)\cap \K$.
\end{tcolorbox}

\begin{assumption}\label{asstheta+}
	All the components of  $\vtheta$ are non-negative on $X$ and $|\vtheta|:=\sum_{i\in\I}\theta_i(x)> 0$ for any $x\in X$.
\end{assumption}

\begin{prop} Under Assumption \ref{asstheta+}, if $\K\subset \D(N,J)$ is  closed  then any maximizer of $\hatuXi$ on $\uvmS_N(\bar\mu)\cap \K$
	is necessarily obtained at the boundary of $\K\cap\uvmS_N(\bar\mu)$. In particular, if $\K\supset\uvmS_N(\bar\mu)$ then any such maximizer is in $\partial\uvmS_N(\bar\mu)$. Moreover, in that case any maximizing subpartition is a partition.
\end{prop}

Indeed, if $\cM_0\not\in\vmS_N(\bar\mu)$ is such a maximizer then there is a strong subpartition $\vA$ realizing the maximum of $\theta$ in
$\sPzeta_{\{\cM_0\}}$. In that case there exists a  measurable set $A_0= X-\cup_{i=1}^n A_i$ such that $\mu(A_0)>0$. Since at least one of the components of $\vtheta$ is positive  it follows that $\int_{A}\theta_id\mu>0$, $\int_A\zeta_id\mu=\eps$ for some $i\in \I$, $A\subset A_0$ and $\eps>0$ small enough. If $\cM_0$ is an internal point of $\uvmS_N(\bar\mu)\cap \K$ then $\cM_0+\eps  \vec{e}_i\in \uvmS_N(\bar\mu)\cap \K$ and $\hatuXi(\cM_0+\eps\vec{e}_i)> \hatuXi(\cM_0)$, which is a contradiction.

We now extend Theorem \ref{main3Z} for a convex  $\cK\subset\R^N$.

\begin{theorem}\label{main3ZK} 
	Let $\cK\subset \R^N$ be a closed convex set.
	
	Under Assumptions \ref{mainass2},  \ref{mainass3} and  \ref{mainass31}:  If $\cK\cap \vZet(\uvmS_N(\bar\mu))\not=\emptyset$ there exists a unique subpartition in $\wSPzeta_{\vZet^{-1}(K)}$ which maximizes $\theta$, and this subpartition is a strong one.
	
	If $K\supset\vZet(\uvmS_N(\bar\mu))$ and  Assumption \ref{asstheta+} is granted as well,
	then the above subpartition is a partition. In the last case
	Assumption \ref{mainass2} can be replaced by  Assumption \ref{mainass2}-(i).
\end{theorem}
\noindent
{\bf Proof of Theorem  \ref{main3ZK} } \\
By Theorem \ref{main3Z} we only have to prove the {\em uniqueness} of the maximizer of
$$\vM\mapsto\overline{\hatXi}(\vM) \ \ , \ \ \vM\in \cK\cap \vZet(\vmS_N(\bar\mu))$$
where
$$\overline{\hatXi}(\vM):= \sup_{\cM}\left\{\hatXi(\cM) \  ; \vZet(\cM)=\vM\right\}$$

Let $\vM_0$ be this maximizer. Then, by (\ref{list1}-\ref{rs})
$$ \overline{\hatXi}(\vM)=\inf_{\vq\in\R^N}\Xipsi(\vZet^*(\vq))-\vq\cdot\vM  \  $$
hence   $-\overline{\hatXi}$ is the Legendre transform  \index{Legendre transform} of $\Xipsi\circ\theta(-\vZet^*)$. By assumption this function  is differentiable at  any point in $\R^N$ (c.f \ref{diffmphi}), so Proposition \ref{cf7} in
Appendix \ref{Convexfunctins}   imply that $\overline{\hatXi}$ is strictly concave at any interior point of its essential domain, namely at any $\vM\in \text{Int}\left(\vZet(\vmS_N(\bar\mu))\right)$.
\par
If $\vM_0\in \partial \vZet(\vmS_N(\bar\mu))$ then Theorem \ref{uniquevZet} implies that it is an extreme point. \index{extreme point} This, and the strict concavity of $\overline{\hatXi}$ at inner points imply the uniqueness of $\vM_0$.
$\Box$

\subsection{The MinMax Theorem: A unified formulation}\label{summerych4}
So we finally got our result regarding {\it both} existence and uniqueness of a strong generalized (sub)partition verifying the maximal allocation of consumers under given capacities of the agents.
\par
The mere existence of optimal strong partition is achieved with little effort.\index{strong  (deterministic) partition}
Indeed,  Theorem \ref{th0} implies the existence of weak (sub)partition \index{weak subpartition}  by "soft" analysis. On the other hand, the proof of  Theorem \ref{weak=strong} implies that for any feasible $\cM$ the set of strong  (sub)partitions is the extreme points of the set $\wPzeta_{\{\cM\}}$
 
($\wSPzeta_{\{\cM\}}$) of weak ones. Since the set of extreme points \index{extreme point} must contain the set of optimal partitions, we get existence of strong partitions in a rather  cheap way...

The main "hard" analysis we had to go  so far was in order to prove the {\em uniqueness} of the optimal partitions, as well as their characterization by the dual problem \index{optimal partition}
on $\Dt(N,J)$. One additional bonus we got is that these strong optimal (sub)partitions are {\em open (sub)partitions} in the sense of Definition \ref{defiopenpartition}.
\par
The duality formalism we extensively used is reflected in the MinMax Theorem.
The MinMax Theorem is of fundamental importance in optimization theory.  This theorem, basically follows from the Hahn-Banach \index{Hahn-Banach}Theorem, has many versions. For our case we only need the following, restricted version:

\par\noindent {\bf MinMax Theorem}:
{\it Let $\Dt$ be a vector space over $\R$, $\K$ a convex, compact  domain.  Assume  $\Theta:\Dt\times \K\rightarrow \R$  is  convex in $\cP\in \Dt$ for any  $\cM\in \K$ and  concave in $\cM$ for any $\cP\in \Dt(N,J)$. Then
	$$ \inf_{\cP\in \Dt} \max_{\cM\in \K}\Theta(\cP,\cM)= \max_{\cM\in \K} \inf_{\cP\in \Dt}\Theta(\cP,\cM) \ :=\alpha $$
	Moreover, there exists $\cM_0\in\K$ such that
	\be\label{minmaxalpha} \inf_{\cP\in \Dt} \Theta(\cP,\cM_0)=\alpha \ . \ee}
\vskip .3in\noindent
In our case we take $\Dt=\Dt(N,J)$,  $\K$ a convex  compact subset of $\D(N,J)$ and
$$ \Theta(\cP,\cM):= \Xipsip(\cP)-\cP:\cM \  \  $$
verifies the conditions of the MinMax Theorem. Indeed,
we know, by now, that $\inf_{\cP\in\Dt(N,J)}\Theta(\cP, \cM)=-\infty$ {\em unless} $\cM\in \uvmS_N(\bar\mu)$ (that is, $\cM$ is in the essential domain of $\hatXi(\cM)\equiv -\inf_{\cP\in\Dt(N,J)}\Theta(\cP, \cM)$).
Since $\uvmS_N(\bar\mu)$ is a compact  subset of $\D(N,J)$ we may use
the MinMax Theorem,  replacing  $\Theta$ by
$$ \Xipsip(\cP)+ H_\K(-\cP) \ , $$
where
\be\label{defpcircMM}  H_\K(\cP):=\max_{\cM\in \K}\cP:\cM \ \ee
is the {\em support function} of $\K$ (compare with (\ref{defpcircM})) (see Appendix A).

Using the MinMax Theorem,
Theorem \ref{main3} and Proposition \ref{difcor} imply (\ref{ssdfty}) which, in turn, yields
\be\label{cKtheta}\max_{\vmu\in\wSPzeta_\K}\theta(\vmu)= \max_{\vA\in\sSPzeta_\K}\theta(\vA)
\equiv \inf_{\cP\in\Dt(N,J)} \Xipsip(\cP) +H_\K(-\cP) \  \ee

In conclusion,  we  obtain a {\em unified} description finding the optimal sub-partition for both the under-saturated ($\K\subset \vmS_N(\bar\mu)$) and the over-saturated ($\K-\uvmS_N(\bar\mu)\not=\emptyset$) cases.
Likewise, if $K\subset \R^N$, then
\be\label{sfretyZ} \vpp\mapsto \Xipsip(\vZet^*(\vpp)) +H_K(-\vpp)\ee
where
\be\label{defpcircMZ}  H_K(\vpp):=\max_{\vM\in K}\vpp\cdot\vM \  \ee
is convex on $\R^N$.

The MinMax Theorem via (\ref{minmaxalpha})  also guarantees the existence of $\cM_0\in \K\cap\uvmS_N(\bar\mu)$ for which  (\ref{cKtheta}) can be replaced by
\be\label{infP0} \inf_{\cP\in\Dt(N,J)} \Xipsip(\cP) -\cP:\cM_0\ee
where the optimal partition is obtained via Theorems \ref{main33}, \ref{main3Z}, dealing with the case of a singleton $\K=\{\cM_0\}$ (resp. $K=\{\vZet(\cM_0)\}$).    However, the  \index{optimal partition}
{\it uniqueness} of this $\cM_0$ is beyond the mere statement of the MinMax Theorem. This uniqueness, and the uniqueness of the corresponding (sub)partition,  is the subject of Theorem \ref{main3ZK}.

Even if we take for granted the uniqueness of $\cM_0$, neither the  existence  nor uniqueness of a minimizer $\cP_0$ of (\ref{infP0})  follows from the MinMax Theorem. In fact, by Theorem \ref{main3} we know both existence and uniqueness of this minimizer only if $\cM_0$ happen to be an interior point of $\uvmS_N(\bar\mu)$. If $\cM_0$ is a boundary point of $\uvmS_N(\bar\mu)$ then we know the uniqueness and existence of an {\em optimal partition} by Theorems \ref{uniquestrong}, \ref{uniquevZet}, while  an equilibrium price vector $\cP$ may not exists (see section \ref{escl}).
\chapter{Applications to  learning theory}\label{learningT}
{\small{\it Where is the wisdom we have lost in knowledge?
	Where is the knowledge we have lost in information?}
	(T.S. Eliot)}

\section{Maximal likelihood of a classifier}\label{1learn}    \index{classifier}\index{likelihood}
Let $X$ be the probability space. We can think about it as a space  of random samples (e.g. digital data representing figures of different animals). Let $\I$ be a finite set of cardinality $N$. We can think of $\I$ as the set of {\em labels}, e.g. a lion, elephant, dog, etc.

Suppose that $Z$ is a random variable on $X\times \I$. We can think about $Z$ as a {\em classifier}:  \index{classifier}For each given data point  $x\in X$ it produces the random variable $x\mapsto\E(Z|x)$ on the set of labels $\I$ (see below).

Let $\E(Z|X)$ be the $X-$marginal of $Z$.   We can think of it as a  random variable predicting the input data in $X$.  Likewise, $\E(Z|\I)$ is the $\I-$marginal of $Z$. It can be considered as  a random variable predicting the  output labels in $\I$. We assume that the input distribution of $\E(Z|X)$ is given  by the probability law $\mu$ on $X$. 

The distribution of $Z$ over $X\times \I$ is given by a {\em weak partition}  \index{weak partition} $\vmu=(\mu_1, \ldots \mu_N)$ of $(X,\mu)$, where $\mu=|\vmu|:=\sum_1^N\mu_i$. It means that
the probability that a data $x\in X$ will trigger the label $i\in \I$ is $d\mu_i/d\mu(x)$.

Let $\vM=(m_1, \ldots m_N)\in \Delta^N(1)$ be the distribution of 
$\E(Z|\I)$, namely $m_i:=\mu_i(X)$ is the probability that $Z=i$.
The Shannon information of $\E(Z|\I)$  is
$$H(Z|\I)=-\sum_{i\in\I} m_i\ln m_i \ . $$
It represents the amount of information stored in a random process composed of independent throws of a dice of $N=|\I|$ sides, where the probability of getting output  $i$ is $m_i$. The Shannon information is always non-negative. Its minimal value $H=0$ is attained  iff there exists $i\in \I$ for which $m_i=1$ (hence $m_k=0$ for $k\not=i$, so the dice falls always on the side $i$, and we gain no information during the process), and is maximal $H=\ln |\I|$ for a "fair dice" where $m_i=1/N$.

The information corresponding to $Z$ where $\E(Z|X)$ is known is given by
$$ H(Z|X)=-\sum_{i\in\I} \int_X\ln\left(\frac{d\mu_i}{d\mu}\right)d\mu_i \ . $$

The marginal information of $Z$ given $X$ is defined by
\be\label{IZdef} I_Z(X,\I):= H(Z|\I)-H(Z|X)=\sum_{i\in\I} \int_X\ln\left(\frac{d\mu_i}{d\mu}\right)d\mu_i-\sum_{i\in\I}m_i\ln m_i  \ . \ee
This information is always non-negative via Jensen's inequality and the convexity of $-H$ as a function of the distribution. This agrees with the interpretation that the correlation between the signal $X$ and the output label $\I$ contributes to the marginal information. In particular, if the marginals $\E(Z|\I)$ and $\E(Z|X)$ are independent (so $\mu_i=m_i\mu$)  then $H(Z|X)=H(Z|\I)$ so $I_Z(X,\I)=0$.

Let  $\theta(i,x):=\theta_i(x)$ measures the level of  likelihood that an  input data $x$ corresponds to a  label $i$. The average likelihood due to a classifier $Z$ is, thus \index{classifier}
\be\label{liklihood}\vmu(\vtheta):=\sum_{i\in\I}\mu_i(\theta_i) \ . \ee
The object of a learning machine is to develop a classifier $Z$ which will produce a maximal likelihood under a controlled  amount of marginal information.

In the worst case scenario, a relevant function is the {\em minimal possible marginal information} \index{marginal information} for a given likelihood $\vmu(\vtheta)=\alpha$. For this we define this minimal information  as
$$ R(\alpha):= \inf_Z\{I_Z(X,\I), \ \vmu(\vtheta)=\alpha, \E(Z|\I)=\vM, \E(Z|X)=\mu\} \ . $$ 
From (\ref{IZdef}, \ref{liklihood}) we may rewrite 
$$ R(\alpha)= \inf_{\vmu}\left\{\sum_{i\in\I} \int_X\ln\left(\frac{d\mu_i}{d\mu}\right)d\mu_i; \ \theta(\vmu)=\alpha, |\vmu(X)|=\vM,   \ |\vmu|=\mu\right\}-\sum_{i\in\I} m_i\ln m_i \ .  $$
\par
From this definition and the linearity of $\vmu\rightarrow \vmu(\theta)$ it follows that $R$ is a concave function. 
The concave dual of $R$ is
$$R_*(\beta):=\inf_{\alpha} R(\alpha)-\alpha\beta$$
which is a concave function as well.  By the min-max theorem we recover
$$ R(\alpha)=\sup_{\beta}R_*(\beta)+\alpha \beta \ . $$

\begin{prop}\label{prop81}
	Let
	\be\label{duQsup}Q(\beta, \phi):= \beta\sum_{i\in\I} m_i\ln\left(\int_Xe^{\frac{\theta_i+\phi}{\beta}}d\mu\right) -\int_X\phi d\mu \ . \ee
	Then
	\be\label{Qdefinfo}R_*(\beta):= \inf_{\phi\in C(X)}Q(\beta,\phi) \  . \ee

\end{prop}

Note that $Q(\beta,\phi)=Q(\beta,\phi+\lambda)$ for any constant $\lambda\in\R$. Thus, we use (\ref{duQsup}, \ref{Qdefinfo}) to write
\be\label{Q=0}R_*(\beta)= \inf_{\phi\in C(X); \int_X\phi d\mu=0}\beta\sum_{i\in\I} m_i\ln\left(\int_Xe^{\frac{\theta_i+\phi}{\beta}}d\mu\right) \ . \ee
\vskip.2in
The parameter $\beta$ can be considered as the "temperature", which indicates the amount of uncertainty of the optimal classifier $Z$. \index{classifier}
\vskip .2in\noindent
In the "freezing limit" $\beta\rightarrow 0$ we get
$$ \lim_{\beta\rightarrow 0}\beta\ln\left(\int_X e^{\frac{\theta_i+\phi}{\beta}}d\mu\right)= \max_{x\in X}\theta_i(x)+\phi(x) \ , $$
so
\be\label{Qbetainfty} R_*(0)=\inf_{\int\phi d\mu=0}\left\{ \sum_{i\in\I} m_i\max_{x\in X}[\theta_i(x)+\phi(x)]\right\} \ . \ee
It can be shown that $R_*(0)$ is obtained by the optimal partition  \index{optimal partition}of $X$ corresponding to the utility $\{\theta_i\}$ and capacities $\vM$ which we encountered in Chapter \ref{S(M)P}:
\be\label{QwithvM}\boxed{R_*(0)=\Sigma^\theta(\vM):=\sup_{\vA}\left\{\sum_{i\in\I}\int_{A_i}\theta_id\mu ; \ \ \mu(A_i)=m_i\right\}}\ee 
where $\vA=\{A_i, \ldots A_N\}$ is a strong partition of $X$. Indeed, \index{strong  (deterministic) partition} a minimizing sequence of $\phi$ in (\ref{Qbetainfty}) converges pointwise to a limit which is a constant $-p_i$ on each of the optimal components  $A_i$ in $\vA$, and these constants are the equilibrium prices which minimize
$\Xi(\vpp)-\vpp\cdot\vM$
on $\R^N$, where
$$\Xi(\vpp):= \int_X\max_{i\in \I}[\theta_i-p_i]d\mu \ . $$
Compare with
(\ref{defSigma}, \ref{defcmu}). In particular, the optimal classifier at the freezing state corresponds to \index{classifier}
$\mu_i=\mu\lfloor A_i$ where
$$A_i\subset \{x\in X; \theta_i(x)-p_i= \max_{j\in \I}\theta_j(x)-p_j \} \ , $$
and verifies the conditions of strong partitions $A_i\cap A_j=\emptyset$ for $i\not= j$ and $\cup_iA_i=X$.

Thus, the optimal $Z$ will predict the output $i$ for a data $x\in X$ with probability $1$ iff $x\in A_i$, and with probability $0$ if $x\not\in A_i$.
\vskip .2in
In the limit $\beta=\infty$ we look for a classifier $Z$ satisfying $I_Z(X,\I)=0$, that is, the amount of  information  $H(Z|X)$  is the maximal one. \index{classifier} Since

$$ \lim_{\beta\rightarrow\infty} \beta\ln\left(\int_Xe^{\frac{\theta_i+\phi}{\beta}}d\mu\right)=\int_X(\theta_i+\phi)d\mu=\int_X\theta_id\mu$$
it follows that  the optimal likelihood in the limit $\beta=\infty$   is\index{likelihood}
$$\boxed{R_*(\infty)= \sum_{i\in\I} m_i\mu(\theta_i) \  }$$
corresponding to the {\em independent variables} $\E(Z|X), \E(Z|\I)$, where $\vmu=\vM\mu$.

\begin{proof}of Proposition \ref{prop81}: \ 
	Let us maximize
	\be\label{withphi}\sum_{i\in\I}\int_X\theta_id\mu_i-  \beta\sum_{i\in\I} \int_X\ln\left(\frac{d\mu_i}{d\mu}\right)d\mu_i+\sum_{i\in\I}\int_X\phi d\mu_i-\int_X\phi d\mu \ee
	under the constraints $\mu_i(X)=m_i$. Here $\theta\in C(X)$ is the Lagrange multiplier \index{Lagrange multiplier} for the constraint $|\vmu|=\mu$. Taking the variation of
	(\ref{withphi}) with respect to $\mu_i$ we get
	$$\theta_i-\beta \ln\left(\frac{d\mu_i}{d\mu}\right)+\phi=\gamma_i$$
	where $\gamma_i$ is the Lagrange multiplier due to the constraint $|\vmu(X)|=|\vM|$. Thus
	$$\frac{d\mu_i}{d\mu}=\frac{m_ie^{\frac{\theta_i+\phi}{\beta}}}{\int_Xe^{\frac{\theta_i+\phi}{\beta}}d\mu} \ . $$
	Substitute this in (\ref{withphi}) to obtain that (\ref{withphi}) is maximized at $Q(\beta,\phi)-\beta\sum_{i\in\I}m_i\ln m_i$. Minimizing over $\phi\in C(X)$ we obtain (\ref{Qdefinfo}).
\end{proof}
\section{Information bottleneck} \label{bottleN}\index{information bottleneck}
The Information-Bottleneck (IB) method was first introduced by Tishby,
Pereira and Bialek \cite{Tis} in 1999. Here we attempt to obtain a geometric characterization of this concept.

Suppose a classifier $U$ {\em is given} on $X\times \J$, where the label set $\J$ is  finite of cardinality $|\J|<\infty$. The object of a learning machine  is to reduce the details of the data space $X$ to a finite space $\I$ whose cardinality is $|\I|\leq |\J|$. Such a learning machine  can be described by a r.v   $V$  on $X\times \I$ which is  faithful, i.e    the $X-$ marginal of $V$ on $X$ coincides  with that of $U$:
\be\label{faithfull}\E(U|X)=\E(V|X) \ . \ee
We denote this common distribution on $X$ by $\mu$. Such a random variable will provide a classifier $W$ on $\I\times \J$ by composition: \index{classifier}
$$ Prob(i\in \I,j\in \J|W):=\E(U=(x,i), V=(x,j)|\I, \J) \ . $$
We note on passing that such a composition never increases the marginal information, so\index{marginal information} 
\be\label{ineqUW} I_U(X,\J)\geq I_W(\I,\J) \ \ee
(see below).
\par
As in section \ref{1learn} we represent the given distribution of $U$ in terms of a weak $J-$partition \index{weak partition} 
$\bar\mu:=(\mu^{(1)}, \ldots \mu^{(j)})$ where $\mu^{(j)}$ is a positive measure on $X$ and  $|\bmu|:= \sum_{j\in\J}\mu^{(j)}=\mu$ is the marginal distribution of $U$ on $X$. Let us denote
$$ \zeta_j:= \frac{d\mu^{(j)}}{d\mu} \ \ ; \ \ \ \vzeta:=(\zeta_1, \ldots \zeta_J) $$
so $\mu^{(j)}=\zeta_j\mu$ and $|\vzeta|:=\sum_{j\in\J}\zeta_j=1$ on $X$ (compare with (\ref{vzetadef})).

The (unknown) distribution of the classifier $V$ can be introduced in terms of weak $N$  partition of $\mu$: $\vmu=(\mu_1, \ldots \mu_N)$ wheclassifierre $|\vmu|:= \sum_{i\in\I}\mu_i= \mu$ via (\ref{faithfull}).

The decomposition of $U$ and $V$ provides the classifier $W$ on $\I\times\J$. The distribution of this classifier is given \index{classifier}
by $N\times J$ matrix $\cM=\{m^{(j)}_i\}$ where
\be\label{Mijdefs}m^{(j)}_i:=\mu_i(\zeta_j)\ee
is the probability that $W=(i,j)$. 

Let
\begin{multline}\label{sumofMij} \vM^{(j)}:= (m_1^{(j)}, \ldots m_I^{(j)}), \ \  |\vM^{(j)}| := m^{(j)}=\mu^{(j)}(X) \\
\bm_i:= (m_i^{(1)}, \ldots m_i^{(J)} ),  \ |\bm_i|:=m_i=\mu_i(X) \ .  \end{multline}
The information of $W$ and  its $\I$ and  $\J$ marginal are  given by
$$ H(W|\I,\J)=-\sum_{i\in\I}\sum_{j\in\J}m^{(j)}_i\ln\left(m^{(j)}_i\right) \ \ , $$
$$ H(W|\I)+H(W|\J)=-\sum_{j\in\J} m^{(j)}\ln m^{(j)}-\sum_{i\in\I} m_i\ln m_i \ ,$$
The marginal information of $W$   is given as\index{marginal information} 
\begin{multline}\label{IWtotal} I_W(\I,\J)=H(W|\I)+H(W|\J)-H(W|\I,\J)
\\
=-\sum_{i\in\I} m_i\ln m_i-\sum_{j\in\J} m^{(j)}\ln m^{(j)}+\sum_{i\in\I}\sum_{j\in\J}m^{(j)}_i\ln m^{(j)}_i \ . \end{multline}
Note that
$$ H(U|X)=-\sum_{j\in\J} \int_X \zeta_j\ln\left(\zeta_j\right)d\mu \ , \ H(U|\J)=H(W|\J)=-\sum_{j\in\J} m^{(j)}\ln m^{(j)}$$
so the marginal information due to $U$ is\index{marginal information} 
\be\label{IUtotal} I_U(X,\J)=-\sum_{j\in\J} m^{(j)}\ln m^{(j)}+\sum_{j\in\J} \int_X \zeta_j\ln\left(\zeta_j\right)d\mu \ \ .  \ee
Note that $s\mapsto s\ln s$ is a convex function. Since $|\vmu|=\mu$, $|\vmu(X)|:= |\vM|$ and $m^{(j)}_i:= \int\zeta_jd\mu_i$  we get by the Jensen's inequality
$$ \int_X \zeta_j\ln\left(\zeta_j\right)d\mu=\sum_{i\in\I}m_i\int_X \zeta_j\ln\left(\zeta_j\right)\frac{d\mu_i}{m_i}\geq\sum_{i\in\I} m^{(j)}_i\ln \frac{m^{(j)}_i}{m_i}$$
so (recalling $\sum_{j\in\J}m_i^{(j)}=m_i$ )
$$ \sum_{j\in\J}\int_X \zeta_j\ln\left(\zeta_j\right)d\mu\geq \sum_{i\in\I}\sum_{j\in\J} m^{(j)}_i\ln\left( \frac{m^{(j)}_i }{m_i}\right) $$
so by (\ref{IWtotal}, \ref{IUtotal}) we verify (\ref{ineqUW}).
Recall that the Jensen's inequality turns into an equality iff $\zeta_j=m^{(j)}_i/m_i$ a.e $\mu_i$. Thus

\begin{tcolorbox}
	The difference between the marginal information in $W$ and the marginal information in $U$ is the {\em distortion}\index{marginal information} 
	$$I_U(\J,X)-I_W(\I,\J)= \sum_{j\in\J}\int_X \zeta_j\ln\left(\zeta_j\right)d\mu- \sum_{i\in\I}\sum_{j\in\J} m^{(j)}_i\ln\left( \frac{m^{(j)}_i }{m_i}\right) \geq 0 \ . $$
	The information gap can be made  zero  only if $|\I|\geq |\J|$  {\em and the classifier  \index{classifier}$U$ is a deterministic one}, i.e. $\zeta_j\in\{0,1\}$. In that case $\I\supset\J$ and any choice of an immersion $\tau:\J\rightarrow\I$ implies that   $V=\tau\circ U$  is an optimal choice to minimize the information gap.
\end{tcolorbox}

	\begin{figure}
	\centering
	\includegraphics[width=50mm,scale=0.1]{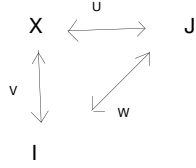}
	\caption{A diagram of the random variables vs. the spaces }\label{figinter}
\end{figure} 
\subsection{Minimizing the distortion} \label{mingap} \index{distortion}
For a given r.v $U$   subjected to the distribution $\bmu$, all possible
distributions of $W$ for a given cardinality $|\I|$ are represented by  points in $\vmS_N(\bar\mu)$. In particular, we can look for the optimal $W$ which minimizes the information gap with respect to $U$ in terms of its distribution  $\cM_0\in \vmS_N(\bar\mu)$. Since $m^{(j)}$ are independent of $U$ it follows, by
(\ref{IWtotal}), that is $\cM_0$ is a maximizer of 
$$   {\bf h}(\cM):=\sum_{i\in\I}\sum_{j\in\J}m^{(j)}_i\ln \left(\frac{m^{(j)}_i}{m_i}\right) \ \ \forall \cM\in\vmS_N(\bar\mu)  $$
where we recall (\ref{sumofMij}).  Thus
\be\label{barhdef} {\bf h}(\cM)=\sum_{i\in\I}h({\bm_i}) \ \ \ , \ \  h(\bm):= \sum_{j\in\J}m^{(j)}\ln m^{(j)}-|\bm|\ln(|\bm|) \ .  \ee
\begin{lemma}\label{hathconvex}
	${\bf h}$ is positively homogeneous (c.f. Definition \ref{defPH}) and strongly convex on the simplex $\Delta_N(\bmu)$.
\end{lemma}
\begin{proof}
	Direct observation implies $h(\lambda\bm)=\lambda h(\bm)$ for any $\lambda\geq 0$, $\bm\in\R^J$.

	Differentiating $h$ twice in $\R^J_{++}$ we obtain
	$$ \frac{\partial h}{\partial m^{(i)}\partial m^{(j)}}= \delta_i^j \left(m^{(i)}\right)^{-1} -\left(\sum_{k\in\J}m^{(k)}\right)^{-1} \ . $$
	Given a vector $\bar{\alpha}=(\alpha_1, \ldots \alpha_J)\in\R^J$ we obtain
	\be\label{hquadform}\sum_{i,j\in\J} \frac{\partial h}{\partial m^{(i)}\partial m^{(j)}} \alpha_i\alpha_j=\sum_{j\in\J} \frac{\alpha_j^2}{m^{(j)}}-\frac{(\sum_{j\in\J}\alpha_j)^2}{\sum_{j\in\J}m^{(j)}} \ . \ee
	Using Cauchy-Schwartz  inequality
	$$ \sum_{j\in\J}\alpha_j=\sum_{j\in\J}\frac{\alpha_j}{\sqrt{m^{(j)}}} \sqrt{m^{(j)}}\leq \left(\sum_{j\in\J} \frac{\alpha_j^2}{m^{(j)}}\right)^{1/2}\left(\sum_{j\in\J} m^{(j)}\right)^{1/2} \  $$
	which implies that (\ref{hquadform}) is non-negative. Moreover, an equality in Cauchy-Schwartz implies $\bar{\alpha}=\lambda\bm$ for some  $\lambda>0$.  
	It follows that 
	$h$ is strongly convex  on  the simplex $\Delta^J(1)$.
	Since $\Delta_N(\bmu)$ is constrained by
	$\sum_{i\in\I} \bm_i=\bmu(X)$ it follows that ${\bf h}$ is strongly convex in $\Delta_N(\bmu)$. 
\end{proof}
From the convexity of $h$ and (\ref{barhdef}) we obtain the convexity of $I_W(\I,\J)$ as a function of $\cM$ on $\vmS_N(\bar\mu)$. Since $\vmS_N(\bar\mu)$ is a compact and convex set we obtain immediately the existence of a maximizer $\cM_0:=\{ m_{i,0}^{(j)}\}$ in the relative boundary of $\vmS_N(\bar\mu)$. Moreover, the set of maximizers is a convex subset of $\vmS_N(\bar\mu)$.
\begin{lemma}\label{MaxM0} 
	The maximizer $\cM_0$ of ${\bf h}$ in $\Delta_N(\bmu)$ is unique if and only if  the vectors
	$(\bm_{i,0}, h(\bm_{i,0} ))\in \R^{J+1}$ , $i=1\ldots N$
	are independent in  $\R^{J+1}$.
\end{lemma}

\begin{tcolorbox}
	In particular, $N\leq J+1$ is a necessary condition for uniqueness of the maximizer of $I_W(\I,\J)$.
\end{tcolorbox}

\begin{proof}
	By the strong convexity of $h$ Lemma \ref{hathconvex} we obtain that ${\bf h}(\cM)={\bf h}(\cM_0)$ iff there exist $\lambda_1, \ldots \lambda_N>0$ such that $\bm_i=\lambda_i\bm_{i,0}$  for $i=1, \ldots N$. If this is the case, then ${\bf h}(\cM)=\sum_{i\in\I}\lambda_i h(\bm_{i,0})$. Thus
	\be\label{hM} \sum_{i\in\I}\lambda_i h(\bm_{i,0})={\bf h}(\cM_0)
	\ . \ee
	In addition, we recall from (\ref{sumofMij}) that any $\cM\in\vmS_N(\bar\mu)$ is subjected to the constraint
	$\sum_{i\in\I}{\bm_i}=\bmu(X)$. Hence $\sum_{i\in\I}\lambda_i \bm_{i,0}, =\bmu(X)$ which,  together with (\ref{hM}) imply 
	\be\label{hM2} \sum_{i\in\I}\lambda_i \left( \bm_{i,0}, h(\bm_{i,0})\right) =\left(\bmu(X), {\bf h}(\cM_0)\right)\ . \ee
	The system  (\ref{hM}, \ref{hM2}) admits the solution $\lambda_1=\ldots =\lambda_N=1$, and this is the unique solution of this system iff the vectors $(\bm_{i,0}, h(\bm_{i,0}))$, $i\in\I$ are independent in $\R^{J+1}$.
\end{proof}
\begin{remark}
	The uniqueness of the maximizer $\cM_0$ does not necessarily implies the uniqueness of the optimal classifier $V$ \index{classifier} realizing the minimal  information gap. In fact, a classifier $V$ is determined by the partition $\vmu=(\mu_1, \ldots \mu_N)$ of $X$, and the uniqueness of $\cM_0$ only implies that the corresponding partition must satisfy $\mu_i(X)=|\bm_{i,0}|:=\sum_{j\in\J} m^{(j)}_{i,0}$.
	
\end{remark} 

Recall that $\cM_0$ is a boundary point of $\vmS_N(\bar\mu)$.
\begin{theorem} 
	Let $\cM_0\in\partial\vmS_N(\bar\mu)$ be a maximizer of ${\bf h}(\cM)$ in $\vmS_N(\bar\mu)$ and satisfies the condition of Lemma \ref{MaxM0}.
	Assume $\vzeta$ satisfies Assumption \ref{mainass3}. Let $\vzet^{(i)}\in \R^J$, $i=1, \ldots N$ be given  satisfying Assumption \ref{mainass31}. If $\vZet(\cM_0)\in \partial \vZet(\vmS_N(\bar\mu))$  (c.f. Definition \ref{mainass31}) then
	the minimal information gap for a given cardinality $N$ is realized  is  unique, deterministic classifier  \index{classifier}$V$. In particular, its distribution is given by a partition $\vmu=\mu\lfloor\vA$, that is $\mu_i=\mu\lfloor A_i$ where $\vA=(A_1, \ldots A_N)$ is a strong partition.\index{strong  (deterministic) partition}
\end{theorem}
The proof of this Theorem follows from Theorem \ref{uniquevZet} (ii). Note that
$\partial\vZet(\uvmS_N(\bar\mu))\subset \partial\vZet(\uvmS_N(\bar\mu))$ (since $\vmS_N(\bar\mu)\subset\uvmS_N(\bar\mu)$) and any subpartition corresponding to $\cM\in\vmS_N(\bar\mu)$ is necessarily a partition.
@@@@@ 

\subsection{The information bottleneck in the dual space}\index{information bottleneck}
We are given  a random variable $U$ on the state space $X\times\J$ as an input. As before, we can view $U$ as a classifier  \index{classifier} over the set of features $X$ into the set of labels $\J$.

A network get this classifier as an input, and (stochastically) represent the data $x\in X$   by    internal states $i\in\I$ of the network. We assume that $\I$ is a finite set $|\I|=N$.

As a result of the training we get a classifier  \index{classifier}on the set $\I\times\J$, where $\I$ is the reduction of the feature space $X$.

The objectives of the "information bottleneck" as described by Tishbi and coauthors are\index{information bottleneck}
\begin{itemize}
	\item  Predictability: to preserve as much of the marginal information of the induced classifier $W$ as possible, that is, to minimize the information gap between $U$ and $W$.
	\item Compressibility: to minimize as much as possible the marginal information stored in the classifier  \index{classifier} $V$.\index{marginal information} 
\end{itemize}

In addition we include the possibility of a likelihood function $\theta:\I\times X\rightarrow \R$ as in section \ref{1learn}. So, we add another objective
\begin{itemize}
	\item To increase as much as possible the expected likelihood of   $V$ as a classifier. \index{classifier}\index{likelihood}
\end{itemize}
Now, we consider the {\em Information Bottleneck} (IB) variational problem. The IB was originally introduced by Tishby and co-authors \cite{Tis}
who suggested to  minimize\footnote{Compare with \cite{Tis1, Tis3},   where $\beta$ corresponds to  $\beta^{-1}$}
\be\boxed{ P(V,\beta,\gamma):= I_V(\I,X)- \beta I_W(\I,\J)-\gamma \E(\theta(V))}\tag{IB}   \ee
where $\beta, \gamma\geq 0$ (in the current literature $\gamma=0$).
\par
The rational behind (IB) is as follows: The desired classifier $V$  \index{classifier} should induce maximal marginal information on the induced $W$, as well as maximal likelihood. On the other hand, the price paid for maximizing this information is the complexity of $V$  measured in terms of the marginal information
\index{marginal information}  stored ($I_V$).\index{marginal information} 
\par
The limit of $\beta$ large corresponds to maximal information in $W$ (i.e. the minimal information gap). Likewise, the limit of large $\gamma$ emphasizes the importance of the likelihood of $V$.\index{likelihood}
\vskip .3in\noindent

Let us calculate  the marginal information $I_V(\I,X)$:
$$ H(V|X)=-\sum_{i\in\I} \int_X \ln\left(\frac{d\mu_i}{d\mu}\right)d\mu_i$$
while
$$ H(V|\I)=H(W|\I)=-\sum_{i\in\I} m_i\ln m_i$$
so
\be\label{IVIX} I_V(\I,X)=-\sum_{i\in\I} m_i\ln m_i + \sum_{i\in\I} \int_X \ln\left(\frac{d\mu_i}{d\mu}\right)d\mu_i\ee

Finally, we recall that the expected likelihood of $V$ is\index{likelihood}
$$ \E(\theta(V)):=\sum_{i\in\I} \int_X \theta_id\mu_i := \vmu(\vtheta)\ . $$

Note that $m^{(j)}$ are independent of $V$. In terms of the distribution $\vmu$ of $V$  we obtain $P(V)-\beta \sum_{j\in\J} m^{(j)}\ln m^{(j)}\equiv P(\vmu)$ where
\begin{multline}\label{Q0}P(\vmu)= \sum_{i\in\I} \int_X \left(\ln\left(\frac{d\mu_i}{d\mu}\right)-\gamma\theta_i\right)d\mu_i\\-\sum_{i\in\I} m_i\ln m_i -\beta \sum_{i\in\I}\sum_{j\in\J}m_i^{(j)}\ln\left(\frac{m^{(j)}_i}{m_i}\right) \ .
\end{multline}
Let
$$ P_1(\vmu):=\sum_{i\in\I} \int_X \left(\ln\left(\frac{d\mu_i}{d\mu}\right)-\gamma\theta_i\right)d\mu_i \ . $$


Here  $\Dt(N,J)$ as given in Definition \ref{defDt}, $\cP:=(\vpp_1, \ldots \vpp_N)$.
\begin{lemma}\label{unPcm}
	$$ \inf_{\vmu\in\wPzeta_{\{\cM\}}}P_1(\vmu)= \inf_{\cP\in \Dt(N,J)
	}\left\{\int_X \ln\left(\sum_{k=1}^N e^{-\vpp_k\cdot\vzeta + \gamma\theta_k}\right)d\mu +\cP:\cM\right\}+1 $$
\end{lemma}

\begin{proof}
	
	Recall that $\wPzeta_{\{\cM\}}\not=\emptyset$ iff there exists a weak partition \index{weak partition}  $\vmu=(\mu_1, \ldots \mu_N)$ of $\mu$ such that
	$\int_X\zeta_jd\mu_i=m^{(j)}_i$ and $|\vmu|=\mu$.
	In particular
	$$ \sup_{\cP\in \Dt(N,J), \phi\in C(X)}\sum_{i\in\I} \int_X \left(\phi-\vpp_i\cdot\vzeta\right)d\mu_i+\cP:\cM-\int_X\phi d\mu= \left\{ \begin{array}{cc}
	0 & \text{if} \ \vmu\in\wPzeta_{\{\cM\}} \\
	\infty & \text{if} \ \vmu\not\in\wPzeta_{\{\cM\}}
	\end{array}\right.
	$$
	where $\vpp_i\in\R^J$,  $\cP=(\vpp_1, \ldots \vpp_N)\in \Dt(N,J)$ and $\phi\in C(X)$.
	Then
	\begin{multline} \sup_{\cP\in \Dt(N,J), \phi\in C(X)}\sum_{i\in\I} \int_X \left(\ln\left(\frac{d\mu_i}{d\mu}\right)-\gamma\theta_i+\phi-\vpp_i\cdot\vzeta\right)d\mu_i+\cP:\cM-\int_X\phi d\mu
	\\ = \left\{\begin{array}{cc}
	P_1(\vmu) & \text{if} \ \vmu\in\wPzeta_{\{\cM\}} \\
	\infty & \text{if} \ \vmu\not\in\wPzeta_{\{\cM\}}
	\end{array}\right.
	\end{multline}
	
	It follows that
	$$ \inf_{\vmu\in\wPzeta_{\{\cM\}}}P_1(\vmu)=\inf_{\vmu}\sup_{\cP,\phi}\sum_{i\in\I} \int_X \left(\ln\left(\frac{d\mu_i}{d\mu}\right)-\gamma\theta_i+\phi-\vpp_i\cdot\vzeta\right)d\mu_i+\cP:\cM-\int_X\phi d\mu$$
	where the supremum is over $\cP\in \Dt(N,J)$, $\phi\in C(X)$ and the infimum is unconstrained By the Min-Max theorem\index{Min-Max theorem}
	\be\label{1minmax}\inf_{\vmu\in\wPzeta_{\{\cM\}}}P_1(\vmu)= \sup_{\cP,\phi}\inf_{\vmu}\sum_{i\in\I} \int_X \left(\ln\left(\frac{d\mu_i}{d\mu}\right)-\gamma\theta_i+\phi-\vpp_i\cdot\vzeta\right)d\mu_i+\cP:\cM-\int_X \phi d\mu\ee
	and, moreover,
	\be\label{2minmax}\inf_{\vmu\in\wPzeta_{\{\cM\}}}P_1(\vmu)<\infty\ee
	since $\wPzeta_{\{\cM\}}\not=\emptyset$.
	We now consider the unconstrained infimum
	\be\label{Q1} Q_1(\phi, \cP):= \inf_{\vmu}\sum_{i\in\I} \int_X \left(\ln\left(\frac{d\mu_i}{d\mu}\right)-\gamma\theta_i+\phi+\vpp_i\cdot\vzeta\right)d\mu_i \ee
	We find that
	the  minimizer of (\ref{Q1}) exists, and takes the form
	\be\label{fdmi} \frac{d\mu_i}{d\mu}= e^{\gamma\theta_i-\phi+\vpp_i\cdot\vzeta-1} \ . \ee  
	The condition $|\vmu|=\mu$ implies that
	\be\label{fdmi1}\phi+1=\ln\left(\sum_{k\in\I}e^{ \gamma\theta_k+\vpp_k\cdot\vzeta}\right)\ee
	and 	from (\ref{1minmax})
	\begin{multline}\label{longchain}\inf_{\vmu\in\wPzeta_{\{\cM\}}}P_1(\vmu)=\sup_{\cP}\left\{- \int_X \ln\left(\sum_{k=1}^N e^{\vpp_k\cdot\vzeta + \gamma\theta_k}\right)d\mu +\cP:\cM\right\}-1
	\\ =-\inf_{\cP}\left\{\int_X \ln\left(\sum_{k=1}^N e^{\vpp_k\cdot\vzeta + \gamma\theta_k}\right)d\mu -\cP:\cM\right\}-1 \ .
	\end{multline}
\end{proof}

\begin{lemma}\label{unPcmnew}
	$$\inf_{\cP}\left\{\int_X \ln\left(\sum_{k=1}^N e^{\vpp_k\cdot\vzeta + \gamma\theta_k}\right)d\mu -\cP:\cM\right\}>-\infty$$ iff  $\cM\in\vmS_N(\bar\mu)$, where $\vmS_N(\bar\mu)$ as defined in (\ref{defvmSw}). 
\end{lemma}
\begin{proof} 
	Recall from Theorem \ref{main1} that $\cM\not\in\vmS_N(\bar\mu)$ iff

	\be\label{intXmaxivpp}\inf_{\cP\in\Dt(N,J)} \Xee(\cP)-\cP:\cM =-\infty \ . \ee
	where $\Xee(\cP)=
	\int_X\max_i\vpp_i\cdot\vzeta d\mu$.
	
	Since $\ln\left(\sum_{k=1}^N e^{\vpp_k\cdot\vzeta + \gamma\theta_k}\right)= \max_i\vpp_i\cdot\vzeta +O(1)$ 
	then (\ref{intXmaxivpp}) implies the bound. 
	
\end{proof}

\begin{theorem} The minimal value of (IB) is the minimum of
	\begin{multline}\label{longex}-\inf_{\cP}\left\{\int_X \ln\left(\sum_{k=1}^N e^{\vpp_k\cdot\vzeta + \gamma\theta_k}\right)d\mu -\cP:\cM\right\} \\ -\sum_{i\in\I} m_i\ln m_i -\beta \sum_{i\in\I}\sum_{j\in\J}m_i^{(j)}\ln\left(\frac{m^{(j)}_i}{m_i}\right) -1\end{multline}
	over $\cM\in\vmS_N(\bar\mu)$. If the infimum in $\cP_0\in \Dt(N,J)$ is attained for a minimizer $\cM_0\in\D(N,J)$ then  the distribution of the minimizer $V$ of (IB) is given by the weak partition \index{weak partition} 
	\be\label{compare1}\mu_i(dx)=\frac{e^{\gamma\theta_i(x)+\vpp^0_i\cdot\vzeta(x)}}{\sum_ke^{\gamma\theta_k(x)+\vpp^0_k\cdot\vzeta(x)}}\mu(dx) \ . \ee
\end{theorem}
\begin{proof}
	Follows from (\ref{Q0}), Lemma \ref{unPcm} and Lemma  \ref{unPcmnew}. The minimizer $\vmu$ follow from (\ref{fdmi}, \ref{fdmi1}).
\end{proof}
In the notation of \cite{Tis} where $\gamma=0$, the optimal distribution $\mu_i$ takes the form
\be\label{compare2}\frac{d\mu_i}{d\mu}= {\cal Z}^{-1}M_ie^{-\beta D_{KL}}\ee
where $D_{KL}(U|W)$ is the Kullback-Leibler divergence  \index{Kullback-Leibler divergence} \cite{KL} for the distribution of $(U,W)$, and ${\cal Z}$ is the partition function which verifies the constraint $\sum_{i\in\I}\mu_i=\mu$. In our notation
$$ D_{KL}(U|W)= \sum_{j\in\J}\zeta_j(x)\ln\left(\frac{\zeta_j(x) m_i}{m^{(j)}_i}\right) \ . $$
To relate (\ref{compare1}) with (\ref{compare2}) we assume that the optimal $\cM$ in (\ref{longex}) is a {\em relative  internal point} of $\vmS_N(\bar\mu)$. Then we equate the derivative of (\ref{longex}) with respect to $m^{(j)}_i$ to zero , {\em at the optimal} $\cP^0$, to obtain
$$ p^{(j),o}_i=\ln \left(m_i\right) + \beta \ln\left( \frac{m^{(j)}_i}{m_i}\right)+ \lambda^{(j)} \  $$
where $\lambda^{(j)}$ is the Lagrange multiplier corresponding to the constraint\\ $\sum_{i\in\I} m_i^{(j)}=m^{(j)}$. \index{Lagrange multiplier} 
Since $\sum_{j\in\J}\zeta^{(j)}(x)=1$ we get 
\\
$\vpp^0_i\cdot\vzeta= \ln m_i + \beta\sum_{j\in\J}\zeta_j(x)\ln\left( \frac{m^{(j)}_i}{m_i}\right) +\bar{\lambda}\cdot\vzeta(x)$. Thus, (\ref{compare1}) takes the form (where $\gamma=0$)
$$\mu_i(dx)=\frac{m_i}{\cal Z}e^{\beta\sum_{j\in\J}\zeta_j(x)\ln\left( \frac{m^{(j)}_i}{m_i}\right)+\bar{\lambda}\cdot\vzeta(x)}\mu(dx) \  $$
where ${\cal Z}$ is the corresponding partition function.
Now, we can add and subtract any function of $x$ to the powers of the exponents since any such function is canceled out with the updated definition of ${\cal Z}$. If we add  the function $x\mapsto\sum_{j\in\J}\zeta_j\ln\zeta_j$ and subtract $\bar{\lambda}\cdot\vzeta$  to get (\ref{compare2}).

\begin{tcolorbox}
	The representation (\ref{compare2}) is valid \emph{only if} the minimizer of (\ref{longex}) is a relative interior point of $\vmS_N(\bar\mu)$. From section \ref{mingap} we realize that this is may  {\em not the case} if $\beta$ is sufficiently large.
\end{tcolorbox}

\part{From optimal partition to O.T and back}\label{From Multipartitions to Multitransport}
\chapter{Optimal transport for scalar measures}  \label{scalarcase} \index{optimal partition}
{\small{\it  A plan is the transport medium which conveys a person from the station of dreams to the destination of success. Goals are the transport fees
}	(Israelmore Ayivor)}
\section{General setting}
So far we considered the transport problem from the {\em source}, given by  a measure space $(X,\mu)$ to a {\em target} given by discrete measure space $(\I, \vm)$. Here we consider the extension where the target is a general measure space $(Y,\nu)$.
We pose the following assumption:
\begin{assumption}\label{assreg}
	$X,Y$ are compact spaces, $\theta\in C(X,Y)$ is non-negative  and $\mu\in{\cal M}_+(X)$, $\nu\in{\cal M}_+(Y)$ are regular Borel measures.  \index{Borel measure}
\end{assumption}
We define 
\be\label{supmonge} \otheta(\mu,\nu):= \max_{\pi\in\oPi(\mu,\nu)}\int_X\int_Y \theta(x,y)\pi(dxdy)  \ee
where
\begin{multline}\label{underPidef}\oPi(\mu,\nu):= \left\{\pi\in {\cal M}_+(X\times Y) \ ; \mu(dx)\geq \pi(dx, Y) , \ \nu(dy)\geq \pi(X, dy) \  \right\} \ . \end{multline}
In the {\em balanced case} $\mu(X)=\nu(Y)$ we may replace $\oPi$ by 
\be\label{Pidef}\Pi(\mu,\nu):= \left\{\pi\in {\cal M}_+(X\times Y) \ ; \mu(dx)= \pi(dx, Y) , \   \ \nu(dy)=\pi(X, dy)  \right\} \ . \ee

The optimal $\pi$ is called  an \underline{\em Optimal Transport Plan} (OTP).


\begin{example}
	If $\mu=\alpha\delta_x$, $\nu=\beta\delta_y$ .
	where $\alpha, \beta>0$ then $\oPi(\mu,\nu):= \{(\alpha\wedge\beta)\delta_x\delta_y\}$ is a single measure .
	$$\otheta(\alpha\delta_x,\beta\delta_y)=(\alpha\wedge\beta)\theta(x,y)\  . $$
\end{example}


\begin{example}\label{genparttrans}
	If $\mu\in {\cal M}_+(X)$, and $\nu=\sum_{i\in\I} m_i\delta_{y_i}$. If $\mu(X)\geq\sum_{i\in\I} m_i$ 
	$$\oPi(\mu,\nu)= \{\sum_{i\in\I} \delta_{y_i}(dy)\otimes\mu_i(dx)\ \ , \ \ \text{where} \ \int_X\mu_i=m_i \ \text{and} \ \sum_{i\in\I}\mu_i\leq \mu \  . \}$$
	In that case $\otheta(\mu,\nu)$ corresponds to the under-saturated case (\ref{US}). \index{under saturated (US)} If  \\ $\mu(X)\leq\sum_{i\in\I} m_i$  then
	$$\oPi(\mu,\nu)= \{\sum_{i\in\I} \delta_{y_i}(dy)\otimes\mu_i(dx)\ \ , \ \ \text{where} \ \int_X\mu_i\leq m_i \ \text{and} \ \sum_{i\in\I}\mu_i= \mu \  . \}$$
	then $\otheta(\mu,\nu)$ corresponds to the over saturated case (\ref{OS}). \index{over saturated (OS)}
\end{example}

As we see from Example \ref{genparttrans},
these definitions also  extend our definition of weak partitions \index{weak partition}  in Chapter \ref{Wps} where  $Y:=\I :=\{1, \ldots N\}$ and $\nu:=\sum_{i\in\I} m_i\delta_{(i)}$.

\section{Duality} \label{dualitysection}  
\index{Kantorovich duality}
Recall that in Chapter \ref{S(M)P} (\ref{defSigma},\ref{defcmu}) 
we considered {\em strong (sub)partition}, where the maximizers of (\ref{supmonge}, \ref{infmonge}) are
obtained as the {\em deterministic partition} $\mu_i=\mu\lfloor A_i$. The analogues of strong (sub)partitions in the general transport case is an {\em Optimal Transport Map} (OTM) \\ $T:X\rightarrow Y$ such that, formally, the optimal plan $\pi$ takes the form $\pi_T(dxdy)= \mu(dx)\delta_{y-T(x)}dy$. Thus, 
\begin{description} 
	\item{i)} If $\mu(X)<\nu(Y)$  then $\pi_T\in\oPi(\mu,\nu)$  iff $T_\#\mu\leq \nu$, that is,  for any Borel set $B\subset Y$
	$\mu(T^{-1}(B))\leq \nu(B)$. Equivalently
	$$ \int_X \phi(T(x))d\mu(x)\leq \int_Y \phi(y)d\nu(y) \ \ \ \forall \phi\in C(X) \ . $$
	\item{ii)} If $\mu(X)>\nu(Y)$ then $\pi_T\in\oPi(\mu,\nu)$   iff $T_\#\mu\geq \nu$, that is, for any Borel set $B\subset Y$
	$\mu(T^{-1}(B))\geq \nu(B)$,
	Equivalently
	$$ \int_X \phi(T(x))d\mu(x)\geq \int_Y \phi(y)d\nu(y) \ \ \ \forall \phi\in C(X) \ . $$
	\item{iii)} If $\mu(X)=\nu(Y)$ then $\oPi(\mu,\nu)=\Pi(\mu,\nu)$ and  $\pi_T\in\Pi(\mu,\nu)$   iff $T_\#\mu= \nu$, that is, for any Borel set $B\subset Y$
	$\mu(T^{-1}(B))= \nu(B)$,
	Equivalently
	$$ \int_X \phi(T(x))d\mu(x)= \int_Y \phi(y)d\nu(y) \ \ \ \forall \phi\in C(X) \ . $$
\end{description}

The way from the "stochastic" OTP to the deterministic OTM which we did for the semi-discrete case is concealed in the dual formulation. If the target space $Y$ is a finite space, then we obtained, under assumption    \ref{mainass2}  (in case $J=1$ ), that the optimal weak (sub)partition \index{weak subpartition}  is  given by the strong (sub)partition determined by the {\em prices} $p\in\R^{|Y|}$.   

To show the connection with  Monge-Kantorovich Theory  (\cite{Vil1}, \cite{Vil2}), define
\be\label{overlineJtheta}\overline{\cal J}_\theta:=\{ (\xi,p)\in C(X)\times C(Y); \ \ \xi(x)+p(y)\geq \theta(x,y) \ \ \forall (x,y)\in X\times Y \} \ . \ee
Consider first the saturation case $\mu(X)=\nu(Y)$. Then, for any $\pi\in \oPi(\mu,\nu)$ and any $(\xi,p)\in \overline{\cal J}_\theta$,
\be\label{conothetasat} \int_{X\times Y} \theta d\pi \leq \int_{X\times Y} [\xi(x)+p(y)] \pi(dxdy)= \int_X \xi d\mu+\int_Y p d\nu \ , \ee
hence, in particular,
\be\label{conotheta} \otheta(\mu,\nu)\leq \inf_{(\xi,p)\in \overline{\cal J}_\theta} \int_X \xi d\mu+\int_Y p d\nu  \ . \ee
Assume  $\mu(X)>\nu(Y)$. Then   (\ref{conotheta}) cannot be valid since  the infimum on the right is $-\infty$. Indeed, we obtain for any constant $\lambda$ that $(\xi, p)\in \overline{\cal J}_\theta$ iff  $(\xi-\lambda, p+\lambda)\in \overline{\cal J}_\theta$, and
$$  \int_X (\xi-\lambda) d\mu+\int_Y (p+\lambda) d\nu=  \int_X \xi d\mu+\int_Y p d\nu + \lambda(\nu(Y)-\mu(X)) \rightarrow -\infty$$
as $\lambda\rightarrow \infty$. However,  
(\ref{conotheta}) is still valid for $\pi\in  \oPi(\mu,\nu)$  if we restrict the pair $(\xi, p)$ to $(\xi,p)\in\overline{\cal J}_\theta$ such that $\xi\geq 0$. 
Indeed, (\ref{conothetasat}) implies that
$$ \int_{X\times Y} \theta d\pi \leq \int_X \xi d\hat{\mu}+\int_Y p d\nu$$
for any $\pi\in \oPi(\mu,\nu)$ where $\hat{\mu}(dx)=\pi(dx, Y)\leq \mu$ satisfying $\hat{\mu}(X)=\nu(Y)$. If $\xi\geq 0$ then
$\int_X\xi d\mu\geq \int_X\xi d\hat{\mu}$, thus
\be\label{conothetaun} \otheta(\mu,\nu)\leq \inf_{(\xi,p)\in \overline{\cal J}_\theta , \xi\geq 0} \int_X \xi d\mu+\int_Y p d\nu  \
\ee
holds in the  case $\mu(X)>\nu(Y)$, for any $\pi\in \oPi(\mu, \nu)$.
\par 
Now, suppose   (\ref{conothetaun}) is satisfied with equality. 
Let $\hat{\nu}>\nu$ such that $\hat{\nu}(Y)<\mu(X)$.   Since $\theta\geq 0$ by assumption then $\otheta(\mu, \hat{\nu}) \geq\otheta(\mu, \nu)$ by definition. If
$(\xi_\eps, p_\eps)\in \overline{\cal J}_\theta$, $\xi_\eps\geq 0$  satisfies
$\int_X \xi_\eps d\mu+\int_Y p_\eps d\nu\leq \otheta(\mu,\nu)+\eps$ for some $\eps>0$ then  from
$\int_X \xi_\eps d\mu+\int_Y p_\eps d\hat{\nu}\geq \otheta(\mu,\hat{\nu})$ we obtain 
$$ \int_Y p_\eps(d\hat{\nu}-d\nu)\geq -\eps \ . $$
Since we may take $\hat{\nu}(Y)$ as close as we wish to $\mu(X)$ (e.g. $\tilde{\nu}=\nu+\alpha\delta_{y_0}$ for any $\alpha<\mu(X)-\nu(Y)$ and any $y_0\in Y$) we get 
$$p_\eps\geq -\eps/(\mu(X)-\nu(Y)) \ . $$

Since $\eps>0$ is arbitrary we obtain that the infimum must be attained at $p\geq 0$. In particular (compare with Proposition \ref{prhe} and the remark thereafter)
\begin{prop}\label{procond} Suppose $\theta\geq 0$ , $\mu(X)>\nu(Y)$ and (\ref{conothetaun}) is an equality. Then 
	\be\label{bartheta=+=}\otheta(\mu,\nu)=\inf_{(\xi,p)\in \overline{\cal J}_\theta, \xi\geq 0,p\geq 0} \int_X \xi d\mu+\int_Y p d\nu \ .  \ee
	By the same reasoning (flipping $\mu$ with $\nu$) we obtained: If $\mu(X)<\nu(Y)$ and $\otheta(\mu,\nu)=\inf_{(\xi,p)\in \overline{\cal J}_\theta, p\geq 0} \int_X \xi d\mu+\int_Y p d\nu $
	then (\ref{bartheta=+=}) holds as well. 
\end{prop}

It is remarkable that in the case of saturation $\mu(X)=\nu(Y)$, an equality in (\ref{conothetaun})  {\em does not}, in general, imply (\ref{bartheta=+=}).  Evidently, we may restrict $\overline{\cal J}_\theta$ to {\em either}
$p\geq 0$ {\em or} $\xi\geq 0$ by replacing $(\xi,p)$ with $(\xi+\lambda, p-\lambda)$ for an appropriate  constant $\lambda$,  but not both! 

To remove the conditioning in Propositions \ref{procond} we use the corresponding equalities in the saturation case.
This is the celebrated  duality theorem discovered by Kantorovich \index{Kantorovich duality}
\cite{Ka} and Koopmans \cite{Ko} - for which they shared the Nobel Memorial Prize in
economics. 
\begin{theorem}\label{satur=} If $\mu(X)=\nu(Y)$, 
	$$\otheta(\mu,\nu)=\inf_{(\xi, p)\in \overline{\cal J}_\theta} \int_X\xi d\mu+\int_Yp d\nu \ . $$
\end{theorem}
\begin{remark}\label{remtheta>0}
	In the balanced case  we can surely remove the assumption that $\theta$ is non-negative. Indeed, we may always change $\theta$ by an additive  constant.  However, in the imbalanced case $\mu(X)\not=\nu(Y)$, we cannot remove the assumption $\theta\geq 0$. If, e.g., $\theta$ is a non-positive function then $\theta(\mu, \nu)=0$ by choosing $\pi=0$ in $\oPi(\mu,\nu)$ (\ref{underPidef}).  \end{remark}
We extend this theorem to the unbalanced cases as follows:
\begin{theorem}\label{thmnocond} Suppose  $\mu(X)\not= \nu(Y)$.   Then
	$$\otheta(\mu,\nu)=\inf_{(\xi,p)\in \overline{\cal J}_\theta, \xi\geq 0,p\geq 0} \int_X \xi d\mu+\int_Y p d\nu $$
	holds.
\end{theorem}
\begin{proof}
	We prove the first claim for $\theta$ in the case
	$\mu(X)>\nu(Y)$. The other claims follow by symmetry.
	
	By definition and the assumption $\theta\geq 0$ we obtain 
	$$ \otheta(\mu,\nu)=\sup_{\tilde{\mu}\leq \mu, \  \tilde{\mu}(X)=\nu(Y)}\otheta(\tilde{\mu}, \nu) \ .  $$
	By Theorem \ref{satur=}
	$$ \otheta(\mu,\nu)=\sup_{\tilde{\mu}\leq \mu, \  \tilde{\mu}(X)=\nu(Y)}\inf_{(\xi,p)\in\overline{\cal J}_\theta} \left\{ \int_X \xi d\tilde{\mu} +\int_Y p d\nu\right\}$$ 
	Since $X$ is compact, the set $\tilde{\mu}\leq \mu, \  \tilde{\mu}(X)=\nu(Y)$  is compact in ${\cal M}_+(X)$ with respect to the $C^*(X)$ topology. Hence, the MinMax Theorem implies
	\be\label{otheta=sup}\otheta(\mu,\nu)=\inf_{(\xi,p)\in \overline{\cal J}_\theta} \left\{\left( \sup_{ \tilde{\mu}\leq \mu, \  \tilde{\mu}(X)=\nu(Y)}\int_X \xi d\tilde{\mu}\right) +\int_Y p d\nu\right\} \ . \ee
	For a given $\xi\in C(X)$ let $\bar{A}_\xi(\lambda):=\{x\in X; \xi(x)\geq\lambda\}$ and  $A_\xi(\lambda):=\{x\in X; \xi(x)>\lambda\}$. The function
	$\lambda \mapsto\mu(\bar{A}_\xi(\lambda))$ is monotone non-increasing, lower semi continuous, \index{Lower-semi-continuous (LSC)} while $\lambda \mapsto\mu(A_\xi(\lambda))$
	is monotone non-increasing, upper semi continuous. In addition, $\mu(A_\xi(\lambda))\leq \mu(\bar{A}_\xi(\lambda))$ for any $\lambda$. Thus, there exists $\lambda_0$ such that $\mu(\bar{A}_\xi(\lambda_0))\geq \nu(Y)\geq \mu(A_\xi(\lambda_0))$. Since $\mu$ is regular and contains no atoms, there exists a Borel set $B\subset X$ such that $ A_\xi(\lambda_0) \subseteq B\subseteq \bar{A}_\xi(\lambda_0)$ and $\mu(B)=\nu(Y)$. Let $\bar{\mu}:=\mu \lfloor B$  the restriction of $\mu$ to $B$. We leave it to the reader to verify that $\bar{\mu}\leq \mu, \  \bar{\mu}(X)=\nu(Y)$ and
	$$  \sup_{ \tilde{\mu}\leq \mu, \  \tilde{\mu}(X)=\nu(Y)}\int_X \xi d\tilde{\mu}= \int_B \xi d\mu= \int_X\xi d\bar{\mu} \ . $$
	Since $(\xi,p)\in \overline{\cal J}_\theta$  then  $([\xi-\lambda_0]_+, p+\lambda_0)\in \overline{\cal J}_\theta$ as well. Since
	$$ \int_X [\xi-\lambda_0]_+d\mu= \int_X\xi d\bar{\mu} -\lambda_0 \nu(Y) , \ \ \ \int_Y (p+\lambda_0)d\nu=\int_Y pd\nu + \lambda_0 \nu(Y)$$
	we get
	$$ \left(\sup_{\tilde{\mu}\leq \mu, \  \tilde{\mu}(X)=\nu(Y)}\int_X \xi d\tilde{\mu}\right) +\int_Y p d\nu =\int_X [\xi-\lambda_0]_+ d\mu +\int_Y (p+\lambda_0) d\nu \ . $$
	Since $[\xi-\lambda_0]_+\geq 0$ on $X$ and $([\xi-\lambda_0]_+, p+\lambda_0)\in \overline{\cal J}_\theta$ it  follows that
	$$ \inf_{(\xi,p)\in \overline{\cal J}_\theta} \left\{\left( \sup_{\tilde{\mu}\leq \mu, \  \tilde{\mu}(X)=\nu(Y)}\int_X \xi d\tilde{\mu}\right) +\int_Y p d\nu\right\}  \geq \inf_{(\xi,p)\in \overline{\cal J}_\theta, \xi\geq 0 } \int_X \xi d\mu  +\int_Y p d\nu \ , $$
	so
	$$ \otheta(\mu,\nu)\geq \inf_{(\xi,p)\in \overline{\cal J}_\theta, \xi\geq 0} \int_X \xi d\mu +\int_Y p d\nu \ . $$
	On the other hand, by (\ref{conotheta}) we get
	$$\otheta(\mu,\nu)\leq \inf_{(\xi,p)\in \overline{\cal J}_\theta} \int_X \xi d\mu+\int_Y p d\nu \leq \inf_{(\xi,p)\in \overline{\cal J}_\theta, \xi\geq 0} \int_X \xi d\mu+\int_Y p d\nu$$
	so the equality is verified for $\otheta(\mu,\nu)$ in the case  $\mu(X)>\nu(Y)$. The claim the follows from Proposition \ref{procond}. 
\end{proof}

\section{Deterministic transport}
The subject of existence (and uniqueness) of a deterministic transport plan plays a major  part of the optimal transport literature. Here we only sketch the fundamental ideas,  extended as well to unbalanced transport.  \index{unbalanced transport}

The existence of optimal deterministic transport is related to the existence of optimizers to the dual problem as given by Theorems \ref{satur=} and \ref{thmnocond}. 

Following  the current literature in optimal transport (see, e.g. \cite{Vil1, SntA}...) we 
define the transform  $p\in C(Y)\rightarrow p^\theta\in C(X)$: 
\be\label{ptheta} p^\theta(x)=\sup_{y\in Y} \theta(x,y) -p(y) \ . \ee
Likewise the transform  $\xi\in C(X)\rightarrow \xi_\theta\in C(Y)$: 
\be\label{xitheta} \xi_\theta(y)=\sup_{x\in X} \theta(x,y) -\xi(x) \ . \ee
Note that if $X=Y$ and $\theta$ is a symmetric function ($\theta(x,y)=\theta(y,x)$ \ $\forall(x,y)\in X\times X$) then both definitions are reduced to the same one. In that case, the functions  of the form $p^\theta$ are called $\theta-$convex.  We shall adopt this notation in the general case:
\begin{defi}
	A function $\xi\in C(X)$ is {\em $\theta_X$ convex} if $\xi=p^\theta$ for some $p\in C(Y)$. Likewise, $p\in C(Y)$ is {\em $\theta_Y$ convex} if $p=\xi_\theta$ for some $\xi\in C(X)$. We denote $\Theta_X$ (res  $\Theta_Y$) the set of $\theta_X$ (resp. $\theta_Y$) convex functions.  
\end{defi}
By the assumed compactness of $X,Y$ and continuity  (hence uniform continuity) of $\theta$, the $\theta-$convex functions are always continuous.  In particular:

\begin{prop}\label{propthetacon} \ . 
	\begin{description}
		\item[(i)] For any $p\in C(Y)$, $p^\theta\in C(X)$ and $(p^\theta, p)\in\overline{\cal J}_\theta$. Likewise, for any $\xi\in C(X)$, $\xi_\theta\in C(Y)$ and $(\xi, \xi_\theta)\in\overline{\cal J}_\theta$. 
		\item[(ii)] For any $p\in C(Y)$ and $y\in y$, $p^\theta_\theta(y):= (p^\theta)_\theta (y) \leq  p(y)$. Likewise, for any $\xi\in C(X)$ and $x\in X$, $\xi_\theta^\theta(x):= (\xi_\theta)^\theta (x) \leq  \xi(x)$. 
		\item[(iii)]   $\xi$ is $\theta_X$convex  iff   $\xi_\theta^\theta=\xi$. Same implies for $\theta_Y$ convex $p$. . 
		\item[(iv)] For any $\theta_X$ convex function $\xi$  and any 
		$x_1, x_2\in X$, $$\xi(x_1)-\xi(x_2)\leq \max_{y\in Y} \theta(x_1, y)-\theta(x_2, y) \ . $$
		Likewise, for any $\theta_Y$  convex function $p$  and any 
		$y_1, y_2\in Y$, $$p(y_1)-p(y_2)\leq \max_{x\in X} \theta(x, y_1)-\theta(x, y_2) \ . $$
	\end{description}
\end{prop}
\begin{proof}
	The proof  follows directly from the definitions. We shall only present the proof of the {\em only if}  part in (iii) and leave the rest for the reader. 
	
	If $\xi$ is $\theta_X$ convex then there exists $p\in C(Y)$ such that $\xi=p^\theta$. We show that $\xi^\theta_\theta:= p^{\theta\theta}_\theta=p^\theta\equiv \xi$. 	
	From  definition $$p^{\theta\theta}_\theta(x)=\sup_{y}\inf_{x^{'}}\sup_{y^{'}}\theta(x, y)-\theta(x^{'}, y) +\theta(x^{'}, y^{'}) -p(y^{'})\ . $$
	If we substitute $y=y^{'}$ we get the inequality $p^{\theta\theta}_\theta(x)\geq p^\theta(x)$. If we substitute $x=x^{'}$ we get the opposite inequality. 
\end{proof}
Proposition \ref{propthetacon}-(i) and Theorems \ref{satur=}, \ref{thmnocond} enable us to reduce the minimization of the dual problem from the set of pairs $\overline{\cal J}_\theta$ to the set of $\theta-$convex  functions on either $X$ or $Y$. 
\begin{theorem}\label{thuimbaba}
	If $\mu(X)=\nu(Y)$ then
	$$\otheta(\mu,\nu)=\inf_{\xi\in \Theta_X} \int_X \xi d\mu+\int_Y \xi_\theta d\nu= \inf_{p\in \Theta_Y} \int_X p^\theta d\mu+\int_Y p d\nu  \ , $$
	while if 	$\mu(X)<\nu(Y)$ and $\theta\geq 0$  then 
	$$\otheta(\mu,\nu)=\inf_{\xi\in \Theta_X; \xi\geq 0} \int_X \xi d\mu+\int_Y[\xi_\theta ]_+d\nu$$
	and if $\mu(X)<\nu(Y)$ then 
	$$	\otheta(\mu,\nu)=	 \inf_{p\in \Theta_Y; p\geq 0} \int_X [p^\theta]_+ d\mu+\int_Y p d\nu  \ . $$
\end{theorem}


\subsection{Solvability of the dual problem}


Let us start from  the balanced case.  Let $p\in C(Y)$,  $p^\theta\in C(X)$. Let $y_2\in Y$ be a maximizer in (\ref{ptheta}). Then
$$ p^\theta(x_1)-p^\theta(x_2) \leq \theta(x_1, y_2)-p(y_2)- p^\theta(x_2)=\theta(x_1, y_2)-p(y_2)-\left[\theta(x_2, y_2)-p(y_2)\right]$$
$$= \theta(x_1, y_2)-\theta(x_2, y_2)\leq \max_{y\in Y}|\theta(x_1, y)-\theta(x_2, y)| \  .$$
Let us assume that $X$ is a {\em metric} compact spaces, and $d_X$ the metric on $X$. It follows that there exists a continuous, non-negative valued function $\sigma$ on $\R_+$ such that $\sigma(0)=0$ and
$$ \max_{y\in Y}|\theta(x_1, y)-\theta(x_2, y)| \leq \sigma(d_X(x_1, x_2)) \  $$
In particular it follows that for {\em any} $p\in C(Y)$, $p^\theta$ is subjected to a modulus of continuity $\sigma$ determined by $\theta$:
$$ |p^\theta(x_1)-p^\theta(x_2)|\leq \sigma(d_X(x_1,x_2)) \  \ \forall x_1, x_2\in X . $$
If we further assume that $Y$ is a compact metric space and $d_Y$ the associated metric, we obtain the same result for $\xi_\theta$ , where $\xi\in C(X)$ (\ref{xitheta}): 
$$ |\xi_\theta(y_1)-\xi_\theta(y_2)|\leq 
\sigma(d_X(x_1,x_2)) \  \ \forall y_1, y_2\in Y . $$

We may reduced now the sets $\overline{J}_\theta$, $\underline{J}_\theta$  in Theorems  \ref{satur=}, \ref{thmnocond} to uniformly bounded and equi-continuous pair of functions. Moreover, we may assume that the pairs are bounded in supremum norm as well (why?).  By the Arzel\`{a}-Ascoli   Theorem  \index{Arzel\`{a}-Ascoli   Theorem } we get the uniform convergence of minimizing/maximizing sequence to an optimizer. Thus we replace the $inf$ by $min$ and $sup$ by $max$ in Theorems \ref{satur=}, \ref{thmnocond}. 
In particular we obtained: 
\begin{lemma}\label{cormesacemnew}
	In the balanced case there exists $(\xi_0, p_0)\in \overline{\cal J}_\theta$ such that $\xi_0=p_0^\theta$, $p_0=(\xi_0)_\theta$ and  
	\be\label{p0optimotheta} \otheta(\mu,\nu)=  \int_X\xi_0\mu(dx)+\int_Y p_0\nu(dy)  
	\ .   \ee
	If $\mu(X)>\nu(Y)$ then there exists such a pair for which   $p_0\in C(Y; \R_+)$   while if  $\mu(X)<\nu(Y)$ then $\xi_0\in C(X; \R_+)$. 
\end{lemma}

\begin{lemma}\label{9.2ori}
	Any optimal plan $\pi_0$ for  $\otheta(\mu,\nu)$  is supported in the set $\{(x,y): \xi_0(x)+p_0(y)=\theta(x,y)\}$.
\end{lemma}
\begin{proof} 
	By Theorems \ref{satur=} and \ref{thmnocond} and Lemma \ref{cormesacemnew} it follows that if $\pi_0$ is optimal then
	$$\otheta(\mu,\nu)= \int_{X\times Y}\theta(x,y)\pi_0(dxdy) = \int_X\xi_0d\mu+\int_Yp_0d\nu \ . $$
	Balanced case: \  we get $\pi_0\in \Pi(\mu,\nu)$, so
	$$\int_X\int_Y\xi_0(x)+p_0(y) - \theta(x,y)\pi_0(dxdy)=0\ . $$
	Since $\xi_0(x)+p_0(y)\geq  \theta(x,y)$ we get the claim for the balanced case. 
	\par
	In the unbalanced case   $\mu(X)<\nu(Y)$, let $\tilde{\mu}\leq \mu$ be the $X$ marginal of $\pi_0$. Then
	$$ \int_X\xi_0d\mu+\int_Yp_0d\nu=\int_X\xi_0d\tilde{\mu}+\int_Yp_0d\nu + \int_X\xi_0(d\mu-d\tilde{\mu})$$
	$$= \int_{X\times Y}[\xi_0(x)+p_0(y)]d\pi_0+ \int_X\xi_0(d\mu-d\tilde{\mu})\geq\otheta(\mu,\nu)$$ 
	where the last inequality follows from $\xi_0(x)+p_0(y)\geq  \theta(x,y)$  and $\xi_0\geq 0$ via Theorem \ref{thmnocond}. It implies again that the support of $\pi_0$ is contained in the set $\{(x,y): \xi_0(x)+p_0(y)=\theta(x,y)\}$ and, in addition, that $\xi_0=0$ on the support of $\mu-\tilde{\mu}$. The analogues argument applies for the case $\mu(X)>\nu(Y)$.
\end{proof}
We now sketch the way to obtain existence and uniqueness of a {\em deterministic} transport map $\pi_0$. For this we replace the assumption that $X,Y$ are compact sets by $X=Y=\R^d$, but $supp(\mu), \ supp(\nu)$ are compact subsets in $\R^d$. In addition we assume that $\theta\in C^1(\R^d\times\R^d)$ and the function
$y\rightarrow \nabla_x\theta(x,y)$
is injective for any $x$, i.e 
\be\label{twist} \nabla_x\theta(x,y_1)=\nabla_x\theta(x,y_2)\Rightarrow y_1=y_2 \ . \ee
\begin{theorem}\label{Tmap}
	Assume the supports of both $\mu$ and $\nu$ are bounded in $\R^n$, and that the twist condition (\ref{twist}) is satisfied for $\theta$. Let $( \xi_0, p_0)$ be the dual pair verifying (\ref{p0optimotheta}). Then $\xi_0$ is differentiable $\mu$ a.e. and there exists a measurable mapping $T$ on $supp(\mu)$ verifying
	$$ \nabla_x\theta(x, y)=\nabla_x\xi_0(x) \ \ \ \mu \ a.e \  $$
	where $y=T(x)$. 
	Moreover, any optimal plan $\pi\in\Pi(\mu,\nu)$ of $u(\mu,\nu)$ is supported in the graph of $T$:
	$$ \sup(\pi)\subset \left\{(x,y); y=T(x) \right\}  \ . $$
\end{theorem}
In particular, such $T$ satisfies 
\be T_\#\mu=\nu, \ \ \text{that is}\ \ \ \mu(T^{-1}(B))=\nu(B) \ \ \forall \ \ B\subset Y \ \ \text{measurable}\ee
and this mapping is the solution of the Monge problem  \index{Monge}
\be\label{monge} \max_{S_\#\mu=\nu} \int \theta(x, S(x))d\mu \ . \ee
{\bf Sketch of proof}: Let $(\xi_0, p_0)\in \overline{\cal J}_\theta$ be the optimal solution of the dual problem. Assuming $(x,y)\in supp(\pi_0)$  then by Lemma \ref{9.2ori} we get that  the equality $\xi_0(x)+p_0(y)=\theta(x,y)$, while $\xi_0(z)+p_0(y)\geq \theta(z,y)$ for any $z$ by definition. If $\xi_0$ is differentiable then this implies $\nabla_x\xi_0(x)=\nabla_x\theta(x, y)$. By the twist condition (\ref{twist}), this determines $y$ and we denote $y:=T(x)$. 
\vskip .3in
\section{Metrics on the set of probability measures}\label{metriconmeasures}
Let us invert maximum to minimum in (\ref{supmonge}) we obtain
$$c(\mu,\nu):=  \min_{\pi\in\oPi(\mu,\nu)}\int_X\int_Y c(x,y)\pi(dxdy)  $$
where $c\in C(X,\times Y)$ is now considered as a {\em cost} of transportation. This can be easily observed as equivalent to the (\ref{supmonge}), upon choosing $c=-\theta$. In the dual formulation we have to invert the inequality in 
$\overline{\cal J}_\theta$ and consider  
\be\label{underlineJtheta}\underline{\cal J}_c:=\{ (\xi,p)\in C(X)\times C(Y); \ \ \xi(x)+p(y)\leq c(x,y) \ \ \forall (x,y)\in X\times Y \} \   \ . \ee
If we restrict ourselves to the balanced case $\mu(X)=\nu(Y)$ then Theorem \ref{satur=} takes the form 
$$c(\mu,\nu)=\sup_{(\xi, p)\in \underline{\cal J}_c} \int_X\xi d\mu+\int_Yp d\nu \ . $$
Note, however, that if we assume that $c$ is non-negative (as we did for $\theta$ in Assumption \ref{assreg}), then 
we have to invert the inequalities in the definition of $\oPi(\mu,\nu)$ (\ref{underPidef}) in order to avoid a trivial minimizer $\pi=0$ in the imbalanced case $\mu(X)\not=\nu(Y)$ (see Remark \ref{remtheta>0}). 

In the special case of $X=Y=\R^d$ we may consider $c_q(x,y)=|x-y|^q$. Of particular interest is the case  $q\geq 1$, which leads to the definition of {\em  metrics} on the set of probability measures on $\R^d$  of finite $q$ moment:
\be\label{M1supq} {\cal M}_1^{(q)}:= \{ \mu\in {\cal M}_1(\R^d), \ \ \int |x|^q d\mu<\infty \} \ . \ee
Indeed, it turns out that 
\be\label{wasq} W_q(\mu,\nu):= c_q(\mu,\nu)^{1/q}\ee 
is a {\em metric} on ${\cal M}_1^{(q)}$, called (perhaps unjustly, see \cite{Vershik}) the {\em Wasserstein metric} \cite{Br, Vil1}.
\subsection{Special cases}\label{9.4.1}
\begin{example}\label{9.3.1ex}
	Suppose
	$\theta(x,y)=x\cdot y$ 
	is the  inner product in $\R^d$.  
	Since \\ $|x\cdot y|\leq (|x|^2+|y|^2)/2$, we get that $\theta(\mu,\nu)$ is bounded on ${\cal M}^{(2)}_1$. The connection with $W_2$ ia apparent via (\ref{wasq})
	for $q=2$
	$$ W^2_2(\mu,\nu):=\inf_{\pi\in \Pi(\mu,\nu)}\int_{X\times X} |x-y|^2\pi(dx)
	= $$
	$$
	\int|x|^2\mu(dx) +\int |x|^2\nu(dx) -\sup_{\pi\in \Pi(\mu,\nu)}\int x\cdot yd\pi$$
	and $\sup_{\pi\in \Pi(\mu,\nu)}\int x\cdot yd\pi= \theta(\mu,\nu)$.  The definition $\theta(\mu,\nu)$ where $\theta(x,y)=x\cdot y$ stands for {\em the statistical  correlation} between random variables distributed according to $\mu, \nu$. Thus, the Wasserstein $W_2$ metric is related to the matching of such two random variable with  {\em maximal correlation}.  
	
	In this special case $\theta(x,y)=x\cdot y$ corresponding to the Wasserstein metric $W_2$  we get that the optimal mapping $T$ is just the gradient of the function $\xi_o$:
	\be\label{brenier} T(x)=\nabla_x\xi_0(x) \ . \ee
	In a pioneering paper, Brenier \cite{Br} considered the quadratic  cost function $c(x,y)=|x-y|^2$,  and proved   that the optimal potential $\xi_0$ is a  {\em convex function}. In particular
	\begin{theorem}\label{brenierTh}{\em \cite{Br}} For any pair of probability measures $\mu,\nu\in {\cal M}_1^{(2)}(\R^d)$ (\ref{M1supq}) 
		where $\mu$ is absolutely continuous with respect to Lebesgue measure,  there exists a unique convex function $\xi$ such that 
		$\nabla\xi_\#\mu=\nu$, and 
		$$ \int|x-\nabla\xi(x)|^2d\mu < \int |S(x)-x|^2 d\mu$$
		for any $S\not=\nabla\xi$ satisfying $S_\#\mu=\nu$. 
	\end{theorem}
	This result is one of the most quoted papers in the corresponding literature.

	\begin{cor}\label{9.4.1cor}
		Let $\mu\in {\cal M}^{(2)}_1(\R^d)$ is absolutely continuous with respect to Lebesgue measure and $\phi:\R^d\rightarrow \R$ is convex. Then $T:=\nabla\phi$ is a measurable mapping and $\nu:=\nabla \phi_\#\mu  \in {\cal M}_1^{(2)}$. Moreover,  $T$ is the only  solution of the Monge problem with respect to the cost $c(x,y)=|x-y|^2$ for $\mu,\nu$. 
	\end{cor}
\end{example}

\begin{example}\label{9.4.2}
	Suppose $X=Y$ is a metric space and $d$ the corresponding metric. 
	
	The {\em metric Monge distance} between $\mu$ and $\nu$ is defined as\index {Monge distance}
	 \index{Monge}
	$$ d(\mu,\nu):= \min_{\pi\in \Pi(\mu,\nu)}\int_{X\times X}d(x,y) \pi(dxdy) \ . $$
	Let us define
	\be\label{d-lam} \theta(x,y)=\lambda-d(x.y)\ee
	where $\lambda\geq \max_{x,y\in X}d(x,y)$ (here we take advantage on our assumption that $X$ is a compact space). Thus
	$$ d(\mu,\nu)=\lambda-\otheta(\mu,\nu) \geq 0 \ .  $$
	Using (\ref{xitheta})  $$\xi_\theta(y)=\max_{x\in X} \lambda-d(x,y)-\xi(x)= \lambda-\min_{x\in X}d(x,y)+\xi(x):=\lambda-\xi_d(y) \ . $$
	From its definition, $\xi_d(y)=\min_{x\in X}d(x,y)+\xi(x)\in Lip(1)$ where $Lip(1)$  is the set of $1-$Lipschitz functions
	\be\label{lipineqxid} \xi_d(x_1)-\xi_d(x_2)\leq d(x_1, x_2) \ , x_1, x_2\in X \ .  \ee
	Indeed, if $z_1=\arg\min d(x_1, \cdot) +\xi(\cdot)$  then for any $x_2,z \in X$
	$$ \xi_d(x_2)-\xi_d(x_1)\leq d(x_2, z)-\xi(z) + d(x_1, z_1)-\xi(z_1) \  , $$
	and, by choosing $z=z_1$ we get (\ref{lipineqxid}). 
	Moreover, we easily observe that $Lip(1)$ is a   {\em self-dual} space, i.e $\xi_d=\xi$ if and only if $\xi \in Lip(1)$. 
	
	From Theorem \ref{thuimbaba} it follows that
	\begin{tcolorbox}
		In the balanced case $\mu(X)=\nu(X)$
		\be\label{wasser1} d(\mu,\nu)=\sup_{\xi\in Lip(1)} \int_X \xi d(\nu-\mu) \ , \ee
	\end{tcolorbox}
	which is the celebrated Kantorovich Rubinstein dual formulation of the metric Monge problem  \cite{Vil1}. In particular we obtain that $d(\mu,\nu)$ depends {\em only} on $\mu-\nu$, and, in this sense, is a {\em norm} on the set of probability measures which lift the metric $d$ from the case space $X$ to the set of probability measures on $X$. Indeed, we may identify
	$d(x,y)$ with $d(\delta_x, \delta_y)$. 
	
	In the unbalanced case $\mu(X)>\nu(X)$ we use 
	(\ref{d-lam}) and Theorem \ref{thuimbaba} to obtain
	$$ d(\mu,\nu):= \lambda\nu(X)-\inf_{\xi\in Lip(1), \xi\geq 0}\int_{X}\xi(x)\mu(dx) + [\lambda-\xi(x)]_+\nu(dx)\ , $$
	which holds for any $\lambda>\max_{x,y\in X} d(x,y)$. In particular we can take $\lambda>\max_X\xi$ so   $[\lambda-\xi]_+=\lambda-\xi$, and obtain 
	$$ d(\mu,\nu):= -\inf_{\xi\in Lip(1), \xi\geq 0}\int_{X}\xi(x)(\mu(dx)-\nu(dy))=\sup_{\xi\in Lip(1), \xi\leq 0}\int_{X}\xi(x)(\mu(dx)-\nu(dy)) \ . $$
	\begin{tcolorbox}
		If $\mu(X)>\nu(X)$, 
		$$d(\mu,\nu)= \sup_{\xi\in Lip(1), \xi\leq 0}\int_{X}\xi(x)(\mu(dx)-\nu(dy))  \ . $$
	\end{tcolorbox}
	Likewise $\mu(X)<\nu(X)$ 
	\begin{tcolorbox}
		If $\mu(X)<\nu(X)$, 
		$$d(\mu,\nu)= \sup_{\xi\in Lip(1), \xi\geq 0}\int_{X}\xi(x)(\mu(dx)-\nu(dy))  \ . $$
	\end{tcolorbox}
	\begin{remark}
		$d$ is not extend to a norm (and neither a metric) on the set of positive measures. Only its restriction ot the probability measures ${\cal M}_1$ is a norm. 
	\end{remark}
\begin{remark}
The norm $d$  on ${\cal M}_1$ is a metrization of the weak* topology introduced in section \ref{4.6}.  See Appendix \ref{wconofme}.  
\end{remark}
\end{example}
 
\subsection{McCann Interpolation}\label{secmecinter} 
\index{McCann Interpolation}
Let $T$ be a measurable mapping in Euclidean space $X$.  Let $\mu,\nu\in{\cal M}_1(X)$, and $\nu=T_\#\mu$.  Define {\em the interpolation} of $T$ with the identity $I$ as $T_s:= (1-s)I + sT$, where $s\in[0,1]$. This induces an interpolation  between $\mu$ and $\nu$ via $T$ as follows
$$ \mu_{(s)}:= {T_s}_\#\mu \ , \ s\in[0,1] $$
Evidently $\mu_{(0)}=\mu$ and $\mu_{(1)}=\nu$, while $\mu_{(s)}\in {\cal M}_1(X)$ for any $s\in[0,1]$.
 \index{Monge}
Suppose now $T$ is the optimal Monge map for $\mu,\nu\in {\cal M}_1f^{(2)}(\R^d)$ with respect to the quadratic cost $c(x,y)=|x-y|^2$. By Theorem \ref{brenierTh} 
$T=\nabla\xi$ for some convex function $\xi$. Then 
$T_s=\nabla\xi_s$ where $\xi_s(x)= \left[(1-s) |x|^2/2+ s\xi(x)\right]$ is a convex function for any $s\in[0,1]$. In particular, by Corollary \ref{9.4.1cor}, $T_s=\nabla\xi_s$ is the optimal mapping of $\mu$ to $\mu_{(s)}$, that is 
$$ W_2(\mu, \mu_{(s)})=\sqrt{\int|\nabla\xi_s(x)-x|^2 d\mu} \ . $$
Since $\nabla\xi_s(x)-x= s(\nabla\xi(x)-x)$ we get
\be\label{likewise1501} W_2(\mu, \mu_{(s)})=s\sqrt{\int|\nabla\xi(x)-x|^2 d\mu}=  sW_2(\mu,\nu) \ . \ee
Likewise
\be\label{likewise150} W_2(\nu, \mu_{(s)})=(1-s)W_2(\mu,\nu) \ , \ee
and $\mu_{(s)}$ is the only measure  which minimize $(1-s)W^2_2(\mu,\lambda)+sW^2_2(\nu,\lambda)$ over $\lambda \in {\cal M}_1^{(2)}$. 

This remarkable identity implies that the orbit $\mu_{(s)}$ defined in this way is, in fact, a {\em geodesic path} in the set ${\cal M}_1^{(2)}$. See \cite{inter1, inter2}. 
\chapter{Interpolated costs}\label{mtaiu}
\section{Introduction}
Assume there exists a compact set $Z$ and a pair of functions \\ $\theta^{(1)}\in C(X\times Z; \R^+)$, $\theta^{(2)}\in C(Y\times Z; \R^+)$, such that
\be\label{pizul}
\theta(x,y):= \max_{z\in Z}\theta^{(1)}(x,z)+\theta^{(2)}(y,z)
\ \ . \ee                             \begin{tcolorbox}   	It is more natural, in the current context, to invert the point of view from utility (which should be maximized) to a cost (which should be minimized). Indeed, this is what we did in Section \ref{metriconmeasures} and there is nothing new about it whatsoever. All we need is to define the cost $\utheta(x,y)
	=-\theta(x,y)$ and replace maximum by minimum and v.v. In particular (\ref{czul}) is replaced by 
	\be\label{czul}		\utheta(x,y):= \min_{z\in Z}\utheta^{(1)}(x,z)+\utheta^{(2)}(y,z)
	\ \ . \ee
\end{tcolorbox}
\begin{example}\label{exlpnew}
	
	If $X=Y=Z$ is a compact convex set in $\R^d$, $r\geq 1$.  Then
	$\utheta^{(1)}(x,i)=2^{r-1}|x-z|^r, \ \ \ \utheta^{(2)}(y,i)=2^{r-1}|y-z|^r$
	verifies (\ref{czul}) for
	$\utheta(x,y)=|x-y|^r$. If $r>1$ then the maximum is obtained at the mid-point $z=(x+y)/2$, and if $r=1$ it is obtained at any point in the interval $\tau x + (1-\tau)y$, $\tau\in[0,1]$.
	
	More generally, if $\alpha>0$ then
	$$ \utheta_\alpha^{(1)}(x,i)=\frac{(1+\alpha^{1/(r-1)})^r}{\alpha+\alpha^{r/(r-1)}}|x-z|^r \ , \ \  \theta_\alpha^{(2)}(y,i)=\frac{\alpha(1+\alpha^{1/(r-1)})^r}{\alpha+\alpha^{r/(r-1)}}|y-z|^r \ , $$
	which reduces the the previous case if $\alpha=1$.
\end{example}

\begin{example}\label{exlpgen}
	Let  $X$ be a compact Riemannian manifolds and $l=l(x,v)$ is a  Lagrangian function on the tangent space $(x,v)\in \mathbb{T}X$, that is
	\begin{itemize}
		\item $l\in C(\mathbb{T}X)$
		\item $l$ is strictly convex on the fiber   $v$ for $(x,v)$.
		\item  $l$ is superlinear in each fiber, i.e., $\lim_{\|v\|\rightarrow\infty} \frac{l(x,v)}{\|v\|}=\infty$ for any $x\in X$
	\end{itemize}
	For any $T>0$ define $\theta_T:X\times X\rightarrow \R$ as the {\em minimal action}
	$$ \utheta(x,y):= \utheta_T(x,y):= \min_{w\in C^1(0,T; X), w(0)=x, w(1)=y} \int_0^1 l(w(t)\dot{w}(t))dt \ . $$
	Then, for any $0<T_1<T$
	$$ \utheta(x,y)=\min_{z\in X} \utheta_{T_1}(x,z) +\utheta_{T-T_1}(y,z)$$
	so, by definition with $\utheta=\utheta_T$ we get $\utheta^{(1)}(x,z)=\utheta_{T_1}(x,z)$ and $\utheta^{(2)}(y,z)=\utheta_{T-T_1}(y,z)$.
\end{example}
Note that Example \ref{exlpnew} is, indeed, a special case of Example \ref{exlpgen}, where $l(x,v):= \|v\|^r$ and $T=1$. More generally, we can extend Example \ref{exlpnew} to a {\em geodesic space} $X$ where $d:X\times X\rightarrow \R$ is the corresponding metric:
\be\label{cmetric} d^r(x,y)= \min_{z\in X}\frac{(1+\alpha^{1/(r-1)})^r}{\alpha+\alpha^{r/(r-1)}}d^r(x,z) + \frac{\alpha(1+\alpha^{1/(r-1)})^r}{\alpha+\alpha^{r/(r-1)}}d^r(y,z) \ . \ee
\subsection{Semi-finite approximation: The middle way}
Let $Z= Z_m:=\{z_1, \ldots z_m\}\subset Z$ is a finite set. Denote
\be\label{sd} c^{Z_m}(x,y):= \min_{1\leq i\leq m}c^{(1)}(x,z_i)+c^{(2)}(z_i,y)\geq \utheta(x,y)\ee
{\it the ($Z_m$) semi-finite approximation}  \index{semi-finite approximation} of $c$ given by (\ref{czul}). \par

The Kantorovich lifting of $c^Z_m$  to the set of measures is given by 
\be\label{optpar}c^{Z_m}(\mu,\nu)  := \inf_{\pi\in \Pi(\mu,\nu)} \int_{X\times Y} c^{Z_m}(x,y)\pi(dxdy)  \  .  \ee

\par
An advantage of the semi-discrete method described above is that it has a dual formulation which convert the optimization (\ref{optpar}) to a convex optimization on $\R^m$. Indeed, we prove that for a given $Z_m\subset Z$  there exists a concave function $\Xi_{\mu, Z_m}^\nu:\R^m\rightarrow \R$ such that
\be\label{mixdef} \max_{\vpp\in\R^m}\Xi_{\mu, Z_m}^\nu(\vpp)=c^{Z_m}(\mu,\nu)  \ . \ee
and, under some conditions on either $\mu$ or $\nu$,  the maximizer is unique up to a uniform translation $\vpp \rightarrow \vpp+\beta(1, \ldots 1)$ on $\R^m$. Moreover,  the maximizers of $\Xi_{\mu, Z_m}^\nu$ yield a {\it unique} congruent optimal  partition.\index{congruent partitions}
\par

The accuracy of the approximation of $\utheta(x,y)$ by $c^{Z_m}(x,y)$ depends, of course, on the choice of the set $Z_m$.  In the special (but interesting) case $X=Y=Z=\R^d$ and $\utheta(x,y)=|x-y|^q$, $q>1$ it can be shown that, given a compact set $K\subset \R^d$,  for a fairly good choice of $Z_m\subset K$ one may get   $\utheta^{Z_m}(x,y)-\utheta(x,y)=O(m^{-2/d})$ for any $x,y\in K$.
\par
From (\ref{sd}) and the above reasoning we obtain in particular
\be\label{deltac} c^{Z_m}(\mu,\nu)-\utheta(\mu,\nu)\geq 0 \  \ee
for any pair of probability measures, and that, for a reasonable choice of $Z_m$,
(\ref{deltac}) is of order $m^{-2/d}$ if the supports of $\mu,\nu$ are contained in a compact set.

For a given $m\in\mathbb{N}$ and  pair of probability measures $\mu,\nu$ and ,  the optimal choice of $Z_m$ is the one which minimizes (\ref{deltac}). Let
\be\label{cm} \phi^m(\mu,\nu):= \inf_{Z_m\subset Z}c^{Z_m}(\mu,\nu)-c(\mu,\nu) \geq 0 \  \ee
where the infimum is over all sets of $m$ points in $Z$. Note that the optimal choice now depends on the measures $\mu,\nu$ themselves (and not only on their supports).  A natural question is then to evaluate  the asymptotic limits
$$ \bar{\phi}(\mu,\nu):= \limsup_{m\rightarrow \infty} m^{2/d}\phi^m(\mu,\nu) \ \ \ ; \ \ \  \underline{\phi}(\mu,\nu):= \liminf_{m\rightarrow \infty} m^{2/d}\phi^m(\mu,\nu) \ . $$
Some preliminary results regarding these limits are discussed in this chapter.

\section{Optimal congruent partitions}\label{secOptpar} 

\begin{defi}\label{defP}\noindent 
	 
	Given a pair of probability measures $\mu\in{\cal M}_1(X), \nu\in{\cal M}_1(Y)$ and $m\in \mathbb{N}$, a \index{congruent partitions} {\em weak congruent $m$ partition of $(X,Y)$ subject to $(\mu,\nu)$} is a pair of weak partitions \index{weak partition} $\vmu:=(\mu_1, \ldots \mu_m)$, $\vnu:=(\nu_1, \ldots \nu_m)$ where $\mu_i\in {\cal M}_+(X), \nu_i\in {\cal M}_+(Y)$ such that
	$$ \mu_i(X)=\nu_i(Y) \ \ \   1\leq i\leq m \ . $$
	The set of all weak congruent $m-$ partitions is denoted by ${\wSP}_{\mu,\nu}(m)$. \index{congruent partitions}
	Since, by assumption,  nether $\mu$ nor $\mu$ contains atoms it follows that $\wSP_{\mu,\nu}(m)\not=\emptyset$ for any $m\in \mathbb{N}$.

\end{defi}

\begin{lemma}\label{lemapart}	
	$$ c^{Z_m}(\mu,\nu)=\min_{(\vec{\mu}, \vec{\nu})\in\wSP_{\mu,\nu}(m)}\sum_{1\leq i\leq m} \left[\int_X c^{(1)}(x, z_i)\mu_i(dx) + \int_Y c^{(2)}(y, z_i)\nu_i(dy\right] \  $$
	where $c^{Z_m}(\mu,\nu)$ as defined by (\ref{optpar}) and $(\vec{\mu}, \vec{\nu})\in \wSP_{\mu,\nu}(m)$.
\end{lemma}
\begin{proof}
	First note that the existence of minimizer follows by 
	compactness of the measures in the weak* topology (see section \ref{krop}). \index{weak* convergence}
	\par
	Define, for $1\leq i\leq m$,  $$\Gamma_i:= \{(x,y)\in X\times Y; \ c^{(1)}(x,z_i)+c^{(2)}(y, z_i)= c^{Z_m}(x,y) \}\subset X\times Y$$
Note that, in general, the choice of $\{\Gamma_i\}$ is not unique. However, we may choose 	  $\{\Gamma_i\}$ as measurable, pairwise disjoint sets  in $X\times Y$.
	\par
	Given $\pi\in\Pi_X^Y(\mu,\nu)$, let $\pi_i$ be the restriction of $\pi$ to $\Gamma_i$. In particular $\sum_{1\leq i\leq m}\pi_i=\pi$. Let $\mu_i$ be the $X$ marginal of $\pi_i$ and $\nu_i$ the $y$ marginal of $\pi_i$.  Then $(\vec{\mu},\vec{\nu})$ defined in this way is in $\wSP_{\mu,\nu}(m)$. Since by definition $c^{Z_m}(x,y)=c^{(1)}(x, z_i)+ c^{(2)}(y, z_i)$ a.s. $\pi_i$,  \begin{multline}\int_X\int_Yc^{Z_m}(x,y) \pi(dxdy) = \sum_{1\leq i\leq m} \int_X\int_Y c^{Z_m}(x,y)\pi_i(dxdy)\\
	=\sum_{1\leq i\leq m} \int_X\int_Y (c^{(1)}(x, z_i)\pi_i(dxdy)+ \int_X (c^{(2)}(y, z_i)\pi_i(dxdy)\pi_i(dxdy) \\
	=\sum_{1\leq i\leq m}\left[ \int_X c^{(1)}(x, z_i)\mu_i(dx)+ \int_Y c^{(2)}(y, z_i)\nu_i(dy)\right]\end{multline}
	Choosing $\pi$ above to be the optimal transport plan we get the inequality
	$$  c^{Z_m}(\mu,\nu)\geq\inf_{(\vec{\mu}, \vec{\nu})\in\wSP_{\mu,\nu}(m)}\sum_{1\leq i\leq m} \left[\int_X c^{(1)}(x, z_i)\mu_i(dx) + \int_Y c^{(2)}(y, z_i)\nu_i(dy\right] \ . $$
	To obtain the opposite inequality, let $(\vec{\mu},\vec{\nu})\in \wSP_{\mu,\nu}(m)$ and set $r_i:=\mu_i(X)\equiv \nu_i(Y)$.  Define
	$\pi(dxdy)=\sum_{1\leq i\leq m}r_i^{-1}\mu_i(dx)\nu_i(dy)$.  Then $\pi\in\Pi_X^Y(\mu,\nu)$ and, from (\ref{sd})
	\begin{multline} \int_X\int_Y c^{Z_m}(x,y)\pi(dxdy)= \sum_{1\leq i\leq m}\int_X\int_Y c^{Z_m}(x,y)r_i^{-1}\mu_i(dx)\nu_i(dy)
	\\ \leq \sum_{1\leq i\leq m}\int_X (c^{(1)}(x, z_i)+ c^{(2)}(y, z_i))r_i^{-1}\mu_i(dx)\nu_i(dy)\\
	= \sum_{1\leq i\leq m}\left[ \int_X c^{(1)}(x, z_i)\mu_i(dx)+ \int_Y c^{(2)}(y, z_i)\nu_i(dy)\right]
	\end{multline}
	and we get the second inequality.
\end{proof}
Given $\vpp=(p_{z_1}, \ldots p_{z_m})\in\R^m$, let
\be\label{xi}  \xi_{Z_m}^{(1)}(\vpp,x):= \min_{1\leq i\leq m} c^{(1)}(x, z_i)+p_i \ \ ; \ \ \  \xi^{(2)}_{Z_m}(\vpp,y):= \min_{1\leq i\leq m} c^{(2)}(y, z_i)+p_i\ee
\be\label{Xi12} \Xi_\mu^{Z_m}(\vpp):=\int_X \xi_{Z_m}^{(1)}(\vpp,x)\mu(dx) \ \ ; \ \ \Xi_\nu^{Z_m}(\vpp):=\int_Y \xi^{(2)}_{Z_m}(\vpp,y)\nu(dy) \ . \ee
\be\label{Xi} \Xi^{Z_m}_{\mu,\nu}(\vpp):= \Xi_\mu^{Z_m}(\vpp) + \Xi_\nu^{Z_m}(-\vpp) \ . \ee
\vskip .3in 
For any  $\vec{r}$ in the simplex $\in\bar{\Delta}^m(1)$ (recall section \ref{notations}), let 
\be\label{Xi*mu}(-\Xi_\mu^{Z_m})^*(-\vec{r}):= \sup_{\vpp\in\R^m}\Xi_\mu^{Z_m}(\vpp)-\vpp\cdot\vec{r} . \ee
Analogously, for  $\nu\in{\cal M}_1(Y)$
\be\label{Xi*nu}(-\Xi_\nu^{Z_m})^*(-\vec{r}):= \sup_{\vpp\in\R^m}\Xi_\nu^{Z_m}(\vpp)-\vpp\cdot\vec{r}  \ . \ee
Compare these with the function $\Xi^+$   in section \ref{BBMFM}. 
\begin{lemma} \label{lemdu}
	$$(-\Xi_\mu^{Z_m})^*(-\vec{r})= c^{(1)}\left(\mu, \sum_{1\leq i\leq m} r_i\delta_{z_i}\right), \ \ \ 
	(-\Xi_\nu^{Z_m})^*(-\vec{r})= c^{(2)}\left(\nu, \sum_{1\leq i\leq m}\delta_{z}\right)\ \ . $$
\end{lemma}
\begin{proof}
	This is a special case  (for the scalar case $J=1$)  of the partition problems discussed in Section \ref{secton7}.   See also \cite{Vil1}. It is also a special case of generalized partitions,
	see Theorem 3.1 and its proof in \cite{wol}.

\end{proof}
\begin{theorem}\label{firstdu}
	\be\label{ap} \sup_{\vpp\in\R^m}\Xi^{Z_m}_{\mu,\nu}(\vpp)= c^{Z_m}(\mu,\nu) \ . \ee
\end{theorem}
\begin{proof}  
	From Lemma \ref{lemapart}, Lemma \ref{lemdu} and Definition \ref{defP} we obtain
	\be\label{c=Xi} c^{Z_m}(\mu,\nu)= \inf_{\vec{r}\in\bar{\Delta}^m(1)}\left[ (-\Xi_\mu^{Z_m})^*(-\vec{r})+(-\Xi_\nu^{Z_m})^*(-\vec{r})\right] \ . \ee
	Note that $(-\Xi_\mu^{Z_m})^*$, $(-\Xi_\nu^{Z_m})^*$ as  defined in ( \ref{Xi*mu}, \ref{Xi*nu}), are, in fact, the Legendre  transforms of \index{Legendre transform} $-\Xi_\mu^{Z_m}$, $-\Xi_\nu^{Z_m}$, respectively. As such, they are defined formally on the whole domain $\R^m$ (considered as the dual of itself under the canonical  inner product). It follows that $(-\Xi_\mu^{Z_m})^*(\vec{r})=(-\Xi_\nu^{Z_m})^*(\vec{r})=\infty $   for $\vec{r}\in\R^m-{\Delta}^m(1)$. Note that this definition is consistent with the right hand side of ( \ref{Xi*mu}, \ref{Xi*nu}), since 
	$c^{(1)}(\mu, \sum_1^mr_i\delta_{z_i})=c^{(2)}(\nu, \sum_1^mr_i\delta_{z_i})=\infty$
	if $\sum_{i=1}^m r_i\delta_{z_i}$ is not a probability measure, i.e.  $\vec{r}\not\in{\Delta}^m(1)$.
	
	On the other hand, $\Xi_\mu^{Z_m}$ and $\Xi_\nu^{Z_m}$ are both finite and continuous on the whole of $\R^m$. The  Fenchel-Rockafellar  duality theorem (see \cite{Vil1}- Thm 1.9) then implies
	\be\label{Xi*=Xi} \sup_{\vpp\in\R^m} \Xi_\mu^{Z_m}(\vpp)+ \Xi_\nu^{Z_m}(-\vpp) = \inf_{\vec{r}\in \R^m} (-\Xi_\mu^{Z_m})^*(\vec{r})+ (-\Xi_\nu^{Z_m})^*(\vec{r})\ . \ee
	The proof follows from (\ref{Xi}, \ref{c=Xi}).
	
	\par\noindent 
	{\bf An alternative proof}: \\ We can prove (\ref{ap}) directly by constrained minimization, as follows: $(\vec{\mu}, \vec{\nu})\in \wSP_{\mu,\nu}(m)$ iff $F(\vpp, \phi,\psi):=$ 
	$$\sum_{1\leq i\leq m} p_i\left(\int_X d\mu_i-\int_Y d\nu_i\right) + \int_X \phi(x)\left(\mu(dx)-\sum_{1\leq i\leq m} \mu_i(dx)\right) + $$
	$$ \int_Y \psi(y)\left(\nu(dy)-\sum_{1\leq i\leq m} \nu_i(dy)\right)\leq 0$$
	for any choice of $\vpp\in\R^m$, $\phi\in C(X)$, $\psi\in C(Y)$. Moreover, $\sup_{\vpp,\phi,\psi}F=\infty$ unless $(\vec{\mu}, \vec{\nu})\in \wSP_{\mu,\nu}(m)$. 
	We can then obtain from Lemma \ref{lemapart}: $c^{Z_m}(\mu,\nu)=$
	\begin{multline}\label{multi}\inf_{\{\mu_i\in {\cal M}_+(X), \nu_i\in{\cal M}_+(Y)\}}\sup_{\vpp\in\R^m, \phi\in C(X), \psi\in C(Y)}\sum_{1\leq i\leq m} \left[\int_X c^{(1)}(x, z_i)\mu_i(dx) + \int_Y c^{(2)}(y, z_i)\nu_i(dy)\right]  \\ + F(\vpp, \phi, \psi) \ \\
	=  \sup_{\vpp\in\R^m, \phi\in C(X), \psi\in C(Y)}\inf_{\{\mu_i\in {\cal M}_+(X), \nu_i\in{\cal M}_+(Y)\}}
	\sum_{1\leq i\leq m} \int_X\left( c^{(1)}(x, z_i)+p_i-\phi(x)\right)\mu_i(dx)
	\\ + \sum_{1\leq i\leq m} \int_Y\left( c^{(2)}(y, z_i)-p_i-\psi(y)\right)\nu_i(dy) +\int_X\phi \mu(dx) +\int_Y\psi\nu(dy) \ . \end{multline}
	We now observe that the infimum on $\{\mu_i, \nu_i\}$
	above is $-\infty$ unless $c^{(1)}(x, z_i)+p_i-\phi(x)\geq 0$ and $c^{(2)}(y, z_i)+p_i-\psi(y)\geq 0$ for any $1\leq i\leq m$. Hence, the two sums on the right of  (\ref{multi}) are non-negative, so the infimum   with respect to $\{\mu_i, \nu_i\}$ is zero. To obtain the supremum on the last two integrals on the right of (\ref{multi}) we choose $\phi, \psi$ as large  as possible under this constraint, namely
	$$ \phi(x)=\min_{1\leq i\leq m} c^{(1)}(x, z_i)+p_i \ \ \ , \ \ \ \psi(y)=\min_{1\leq i\leq m} c^{(2)}(y, z_i)-p_i$$
	so $\phi(x)\equiv \xi_{Z_m}^{(1)}(\vpp,x)$, $\psi(y)\equiv\xi_{Z_m}^{(2)}( -\vpp,y)$  by definition via (\ref{xi}).
\end{proof}

\section{Strong partitions}\label{secstrong}
We now define strong partitions as a special case of weak congruent $m-$ partitions (Definition \ref{defP}).
\begin{defi}\label{sdefP}\noindent 
	Given a pair of probability measures $\mu\in{\cal M}_1(X), \nu\in{\cal M}_1(Y)$ and $m\in \mathbb{N}$, a {\em weak congruent $m$ partition of $(X,Y)$ subject to \index{congruent partitions}$(\mu,\nu)$} is a pair of {\em strong} partitions $\vA:=(A_1, \ldots A_m)$, $\vec{B}:=(B_1, \ldots B_m)$ where $A_i\subset X, \ B_i\subset Y$ are measurable strong partitions of $X,Y$, correspondingly,   such that\index{strong  (deterministic) partition}
	$$ \mu(A_i)=\nu_i(B_i) \ \ \   1\leq i\leq m \ . $$
	The set of all strong congruent $m-$ partitions is denoted by ${\cal SP}_{\mu,\nu}(m)$.\index{congruent partitions}
\end{defi}
\begin{assumption}\label{assAB}
	\begin{description} .
		\item{a)}
		$\mu(x; c^{(1)}(x, z_i)-c^{(1)}(x, z^{'})=p)=0$ for any $p\in\R$ and any $z_i, z_i^{'}\in Z_m$.
		\item{b)} $\nu(y; c^{(2)}(y, z_i)-c^{(2)}(z^{'},y)=p)=0$ for any $p\in\R$ and any $z_i, z_i^{'}\in Z_m$.
	\end{description}
\end{assumption}
Let us also define, for $\vpp\in\R^m$
\be\label{sparAB}
\begin{split}
	& A_i(\vpp):= \{ x\in X; \ c^{(1)}(x, z_i)+p_i= \xi^{(1)}_{Z_m}(\vpp,x)\} \\ & B_i(\vpp):= \{ y\in Y; \ c^{(2)}(y, z_i)+p_i= \xi^{(2)}_{Z_m}(\vpp,y)\} \ .
\end{split} \ee
Note that, by (\ref{xi}, \ref{Xi12})   
\be\label{Ximu0} \Xi_\mu^{Z_m}(\vpp)= \sum_{1\leq i\leq m}\int_{A_i(\vpp)} (c^{(1)}(x, z_i)+p_i)\mu(dx)\ee
likewise
\be\label{Xinu0} \Xi_\nu^{Z_m}(\vpp)= \sum_{1\leq i\leq m}\int_{B_i(\vpp)} (c^{(2)}(y, z_i)+p_i)\nu(dy) \ . \ee
\begin{lemma}\label{Xidiff}Under assumption \ref{assAB} (a) (resp. (b)) 
	\begin{description}
		\item{i)} For any $\vpp\in\R^m$, $\{A_i(\vpp)\}$ (resp. $\{B_i(\vpp)\}$) induces  essentially disjoint partitions of $X$ (resp. $Y$).
		\item{ii)} $\Xi_\mu^{Z_m}$  (resp. $\Xi_\nu^{Z_m}$) is continually  differentiable functions on $\R^m$,
		$$ \frac{\partial \Xi_\mu^{Z_m}}{\partial p_i}=\mu(A_i(\vpp)) \ \ \ \text{resp.}  \ \ \ \frac{\partial \Xi_\nu^{Z_m}}{\partial p_i}=\nu(B_i(\vpp)) \ . $$
	\end{description}
\end{lemma}
This Lemma is a special case of  Lemma 4.3 in  [W].
\begin{theorem}\label{strongth}
	Under  assumption \ref{assAB} there exists a unique
	minimizer $\vec{r}_0$ of (\ref{c=Xi}). In addition, there exists a  maximizer $\vpp_0\in\R^m$ of $\Xi_{\mu,\nu}^{Z_m}$, and  $\{A_i(\vpp_0),B_i(-\vpp_0)\}$  induces a unique, strong congruent  corresponding partitions in $X,Y$ satisfying  $\mu(A_i)=\nu(B_i):= r_{0,i}$, and\index{congruent partitions}
	\be\label{pirep} \pi_0(dxdy):= \sum_1^m (r_{0,i})^{-1}{\bf 1}_{A_i(\vpp_0)}(x){\bf 1}_{B_i(-\vpp_0)}(y) \mu(dx)\nu(dy) \ee
	is the unique optimal transport plan for  $c^{Z_m}(\mu, \nu)$.
\end{theorem}
\begin{proof}
	The proof is based on the differentiability of $\Xi^{Z_m}_{\mu,\nu}$ via Lemma \ref{Xidiff} and Proposition \ref{cf7}. See the proof of Theorem \ref{old} for details. 
	
	To prove that $\pi_0$ given by (\ref{pirep}) is an optimal plan,
	observe that $\pi_0\in\Pi(\mu,\nu)$, hence
	$$ c^{Z_m}(\mu,\nu)\leq \int_X\int_Y c^{Z_m}(x,y)\pi_0(dxdy) \ . $$
	Then we get, from (\ref{sd})
	$$c^{Z_m}(\mu,\nu)\leq  \int_X\int_Y c^{Z_m}(x,y)\pi_0(dxdy)\leq \sum_{1\leq i\leq m}\int_{A_i(\vpp_0)\times B_i(-\vpp_0)} (c^{(1)}(x, z_i)\mu(dx) + c^{(2)}(y,z_i)\nu(dy))$$
	$$ = \sum_{1\leq i\leq m}\left( \int_{A_i(\vpp_0)} c^{(1)}(x, z_i)\mu(dx) + \int_{B_i(-\vpp_0)} c^{(2)}(y, z_i)\nu(dy)\right) = \Xi^{Z_m}_{\mu,\nu}(\vpp_0)\leq c^{Z_m}(\mu,\nu) \ $$
	where the last equality from Theorem \ref{firstdu}. In particular, the first inequality is an equality so $\pi_0$ is an optimal plan indeed.
\end{proof}
\section{Pricing in hedonic market}\label{hedonic}\index{hedonic market}
In adaptation to the  model of Hedonic market \cite{[Hed]} there are 3 components: The space of  consumers  (say, $X$), space of producers (say $Y$) and space of commodities, which we take here to be a finite set $Z_m:=\{z_1, \ldots z_m\}$.    The function $c^{(1)}:= c^{(1)}(x, z_i)$ is  the {\it negative} of the utility
of commodity $1\leq i\leq m$ to consumer $x$, while  $c^{(2)}:=c^{(2)}(y, z_i)$ is the cost of producing commodity $1\leq i\leq m$ by the producer  $y$.
\par
Let $\mu$ be a probability measure on $X$ representing the distribution of  consumers, and $\nu$ a probability measure on $Y$ representing the distribution   of the producers. Following \cite{[Hed]} we add the "null commodity" $z_0$ and assign the zero utility and cost
$c^{(1)}(x, z_0)=c^{(2)}( z_0,y)\equiv 0$ on $X$ (resp. $Y$). We understand the meaning that a consumer (producer) chooses the null commodity is that he/she avoids consuming (producing) any item from $Z_m$.

The object of pricing in Hedonic market is to find   equilibrium prices for the commodities which will balance supply and demand: Given  a price $p_i$ for $z$, the consumer at $x$ will buy the commodity $z$ which minimize   its loss  $c^{(1)}(x, z_i)+p_i$, or will  buy nothing (i.e. "buy" the null commodity $z_0$) if $\min_{1\leq i\leq m}c^{(1)}(x, z_i)+p_i>0$),  while producer at $y$ will prefer to produce commodity $z$ which maximize its profit $-c^{(2)}(y, z_i) + p_i$, or will produce nothing if $\max_{1\leq i\leq m}-c^{(2)}(y, z_i)+p_i<0$. Using notation (\ref{xi}-\ref{Xi}) we define
\be\label{Xi0}  \xi_X^0(\vpp,x):= \min\{\xi^{(1)}_{Z_m}(\vpp,x), 0\} \ \ ; \ \ \  \xi^0_Y(\vpp,y):= \min\{\xi^{(2)}_{Z_m}(\vpp,y),0\}\ee 
\be\label{Xi120} \Xi^0_\mu(\vpp):=\int_X \xi_X^0(\vpp,x)\mu(dx) \ \ ; \ \ \Xi^0_\nu(\vpp):=\int_Y \xi_Y^0(\vpp,y)\nu(dy) \ . \ee 
\be\label{Xi0} \Xi^{0,\nu}_\mu(\vpp):= \Xi^0_\mu(\vpp) + \Xi^0_\nu(-\vpp) \ . \ee
Thus, $\Xi^{0,\nu}_\mu(\vpp)$ is the difference between the total loss of all consumers and  the total profit of all producers, given the prices vector $\vpp$. It follows that an equilibrium price vector balancing supply and demand is the one which (somewhat counter-intuitively) {\it maximizes} this difference. 
The corresponding  optimal strong $m-$partition represent the matching between producers of  ($B_i\subset Y$) to consumers ($A_i\subset X$) of $z\in Z$.  The introduction of null commodity allows the possibility that only part of the consumer (producers) communities actually consume (produce), that is $\cup_{1\leq i\leq m} A_i\subset X$ and $\cup_{1\leq i\leq m} B_i\subset Y$, with $A_0=X-\cup_{1\leq i\leq m} A_i$ ($B_0=Y-\cup_{1\leq i\leq m}B_i$) being the set of non-buyers (non-producers).
\par
From the dual point of view, an adaptation  $c^{Z_m}_0(x,y):=\min\{ c^{Z_m}(x,y), 0\}$ of (\ref{sd}) (in the presence of null commodity) is the {\it cost of direct  matching} between producer $y$ and  consumer $x$. The {\it optimal matching} $(A_i, B_i)$ is the one which {\it minimizes} the total cost  $c_0^{Z_m}(\mu,\nu)$  over all congruent sub-partitions  as defined in Definition \ref{sdefP},  with the possible inequality $\mu(\cup A_i)=\nu(\cup B_i)\leq 1$.\index{congruent partitions}

\section{Dependence on the sampling set}\label{secmonotone}
So far we considered the sampling set $Z_m\subset Z$ as a fixed set. Now we consider the effect of optimizing $Z_m$ within the sets of cardinality $m$ in $Z$.
\par
As we already know  ( \ref{sd}), $c^{Z_m}(x,y)\geq c(x,y)$ on $X\times Y$  for any $(x,y)\in X\times Y$ and $Z_m\subset Z$. Hence also $c^{Z_m}(\mu,\nu)\geq c(\mu,\nu)$ for any $\mu,\nu\in {\cal M}_1$ and any $Z_m\subset Z$ as well. An {\it improvement} of $Z_m$ is a new choice $Z_m^{new}\subset Z$ of the {\it same} cardinality $m$ such that $c^{Z_m^{new}}(\mu,\nu)<c^{Z_m}(\mu,\nu)$.
\par
In  section \ref{monim} we propose a way to improve a given $Z_m\subset Z$, once the optimal partition is calculated. Of course, the improvement depends on the measure $\mu,\nu$. \index{optimal partition}
\par
In section \ref{secass} we discuss the limit $m\rightarrow \infty$ and prove some asymptotic estimates.
\subsection{Monotone improvement}\label{monim}
\begin{prop}\label{promon}
	Define $\Xi_{\mu, Z_m}^\nu$ on $\R^m$ with respect to $Z_m:= \{z_1, \ldots z_m\}\in Z$ as in (\ref{Xi}). Let $(\vec{\mu}, \vec{\nu})\in \wSP_{\mu,\nu}(m)$ be the optimal partition
	corresponding to $c^{Z_m}(\mu,\nu)$. Let $\zeta(i)\in Z$ be a minimizer of
	\be\label{Znew} Z\ni \zeta \mapsto \int_X c^{(1)}(x,\zeta)\mu_{z_i}(dx) + \int_Y c^{(2)}(\zeta,y)\nu_{z_i}(dy)  \ .  \ee
	Let $Z_m^{new}:= \{ \zeta(1), \ldots \zeta(m)\}$.  Then $c^{Z_m^{new}}(\mu,\nu)\leq c^{Z_m}(\mu,\nu)$.
\end{prop}
\begin{cor}\label{cormon}
	Let Assumption \ref{assAB} (a+b), and $\vpp_0$ be the minimizer of $\Xi_\mu^{\nu, Z_m}$ in $\R^m$. Let  $\{A_i(\vpp_0), B_i(-\vpp_0)\}$  be the strong partition \index{strong  (deterministic) partition}corresponding to $Z_m$ as in (\ref{sparAB}). Then the components of  $Z_m^{new}$ are obtained as the minimizers of
	$$ Z\ni\zeta \mapsto \int_{A_i(\vpp_0)} c^{(1)}(x,\zeta)\mu(dx) + \int_{B_i(-\vpp_0)} c^{(2)}(\zeta,y)\nu(dy) \ . $$
\end{cor}
\begin{proof}(of Proposition \ref{promon}): Let $\Xi_\mu^{\nu, new}$  be defined with respect to $Z_m^{new}$. By Lemma \ref{lemapart} and Theorem \ref{firstdu}
	$\Xi^{\nu, new}_\mu(\vpp)\leq \Xi^\nu_\mu(\vpp^*):= \max_{\vpp\in\R^m}\Xi_\mu^{\nu, Z_m}(\vpp)$ for any $\vpp\in\R^m$,   so  $\max_{\R^m}\Xi_\mu^{\nu,new}(\vpp)\equiv c^{Z_m^{new}}(\mu,\nu) \leq \max_{\vpp\in\R^m}\Xi_\mu^{\nu, Z_m}(\vpp)\equiv c^{Z_m}(\mu,\nu)$.
\end{proof}
\begin{remark}\label{remmean} If  $X=Y=Z$ is an Euclidean space and  $c(x,y)=|x-y|^2$ then $z^{new}$ is the center of mass of $(A_i(\vpp_0),\mu)$ and $(B_i(-\vpp_0),\nu)$:
	$$ z^{new}:= \frac{\int_{A_i(\vpp_0)}x\mu(dx) + \int_{B_i(-\vpp_0)}y\nu(dy)}{\mu(A_i(\vpp_0)) + \nu(B_i(-\vpp_0))} \ . $$
\end{remark}
Let
$$ \uc^m(\mu,\nu):= \inf_{Z_m\subset Z\ ; \ \#(Z_m)=m} c^{Z_m}(\mu,\nu) \ . $$
Let $Z_m^k:=\{ z_1^k, \ldots z_m^k\}\subset Z$ be a sequence of sets such that $z_i^{k+1}$ is obtained from $Z_m^k$ via  (\ref{Znew}). Then by Proposition \ref{promon}
$$ c(\mu,\nu)\leq \uc^m(\mu,\nu) \leq \ldots c^{Z_m^{k+1}}(\mu,\nu)\leq c^{Z_m^{k}}(\mu,\nu) \leq \ldots c^{Z_m^{0}}(\mu,\nu)  \ . $$
{\bf Open problem:} Under which additional conditions one may guarantee
$$ \lim_{k\rightarrow\infty}c^{Z_m^{k}}(\mu,\nu)= \uc^m(\mu,\nu)  \ \ ? $$

\subsection{Asymptotic estimates}\label{secass}
Recall the definition (\ref{cm})
$$ \phi^m(\mu,\nu):= \inf_{Z_m\subset Z}c^{Z_m}(\mu,\nu)-c(\mu,\nu):= \uc^m(\mu,\nu)-c(\mu,\nu) \geq 0 \ \ .  $$
Consider the case  $X=Y=Z=\R^d$ and
$$c(x,y)=\min_{z\in\R^d}h(|x-z|)+ h(|y-z|)$$
where $h:\R^+\rightarrow \R^+$ is convex, monotone increasing, twice continuous  differentiable. 
\begin{lemma}\label{ssy}
	Suppose both $\mu$ and $\nu$ are supported on in a compact set $K\subset \R^d$. Then there exists $D(K)<\infty$  such that
	\be\label{ssyeq} \limsup_{m\rightarrow \infty} m^{2/d}\phi^m(\mu,\nu)\leq D(K) \ . \ee
\end{lemma}
\begin{proof}
	By Taylor expansion of $z\rightarrow h(|x-z|)+ h(|y-z|)$ at $z_0=(x+y)/2$ we get $h(|x-z|)+ h(|y-z|)=$ 
	$$ 2h(|x-y|/2) + \frac{1}{2|x-y|^2}h^{''}\left(\frac{|x-y|}{2}\right)\left[ (x-y)\cdot(z-z_0)\right]^2+o^2(z-z_0)\ \ . $$
	Let now $Z_m$ be a regular grid of $m$ points which contains the support $K$. The distance between any $z\in K$ to the nearest point in the grid does not exceed $C(K)m^{-1/d}$, for some constant $C(K)$. Hence $c_m(x,y)-c(x,y)\leq \sup |h^{''}| C(K)^2m^{-2/d}$ if $x,y\in K$. Let $\pi_0(dxdy)$ be the optimal plan corresponding to $\mu,\nu$ and $c$. Then, by definition,  
	$$ c(\mu,\nu)=\int_X\int_Y c(x,y)\pi_0(dxdy)  \ \ ; \ \  c_m(\mu,\nu)\leq \int_X\int_Y c_m(x,y)\pi_0(dxdy)$$
	so $$\phi^m(\mu,\nu)\leq  \int_X\int_Y (c_m(x,y)-c(x,y))\pi_0(dxdy)\leq \sup |h^{''}| C(K)^2m^{-2/d} \ , $$
	since $\pi_0$ is a probability measure.
\end{proof}
If $h(s)=2^{q-1}s^q$ (hence $c(x,y)=|x-y|^q$) then the condition of Lemma \ref{ssy} holds if $q\geq 2$. Note that if $\mu=\nu$ then $c(\mu,\mu)=0$ so $\phi^m(\mu,\mu)=\inf_{Z_m\in Z} c^{Z_m}(\mu,\mu)$.  In that particular case we can improve the result of Lemma \ref{ssy} using Zador's Theorem for {\em vector quantization}. 
\begin{theorem}\label{zador}\cite{[GL],[Z]} Let $f\in \mathbb{L}_1(\R^d)$ be a density (with respect to Lebesgue) of a probability measure (in particular $f\geq 0$ and $\int_{\R^d} f =1$.) Let 
	Then 
	$$ \lim_{n\rightarrow\infty} m^{q/d} \min_{ Z_m\subset \R^d}\int_{\R^d} \min_{z\in Z_m} |x-z|^q f(z)dz= C_{d,q} \left[\int_{\R^d} f^{d/(d+q)}\right]^{(d+q)/d} \ . $$
\end{theorem}

\begin{cor}\label{prozador}
	If $c(x,y)=|x-y|^q$, $q \geq 2$, $X=Y=Z=\R^d$ and  $\nu=\mu=f(x)dx$
	\be\label{mquant}  \lim_{m\rightarrow \infty}m^{q/d}\phi^m(\mu,\mu)= 2^qC_{d,q}\left(\int f^{d/(d+q)}dx\right)^{(d+q)/d} \ \ee
	where $C_{d,q}$ is some universal constant.
\end{cor}
\begin{proof}
	From (\ref{Xi}), $\Xi_{\mu,\mu}^{Z_m}(\vpp)= \Xi_\mu^{Z_m}(\vpp)+\Xi_\mu^{Z_m}(-\vpp)$ is an even function. Hence its maximizer must be $\vpp=0$. By Theorem \ref{firstdu}
\be\label{zadurdual0}\Xi_{\mu,\mu}^{Z_m}(0)=2\Xi_\mu^{Z_m}(0)= c^{Z_m}(\mu,\mu) \ . \ee	
	Using (\ref{xi}, \ref{Xi12}) with $c^{(1)}(x,y)=c^{(2)}(y,x)= 2^{q-1}|x-y|^q$ we get
	\be\label{zadurdual} \Xi_\mu^{Z_m}(0)= 2^{q-1}\int_{\R^d}\min_{1\leq i\leq m}|x-z|^q \mu(dx) \ . \ee
	Let now  $\mu=fdx$. Since, evidently,  $c(\mu,\mu)=0$ we get
	 (\ref{mquant}) from (\ref{zadurdual}, \ref{zadurdual0}) and from Theorem \ref{zador}.
\end{proof}
Note that Corollary \ref{prozador} does not contradict Lemma \ref{ssy}. In fact $q\geq 2$ it is compatible with the Lemma,  and (\ref{ssyeq}) holds with $D(K)=0$
if $q>2$. If $q\in [1,2)$, however, then the condition of the Lemma is not satisfied (as $h^{''}$ is not bounded near $0$), and the Proposition is a genuine extension of the Lemma, in the particular case $\mu=\nu$.
\par
In the particular case $q=2$ we can extend Corollary \ref{prozador} to the general case $\mu\not=\nu$, under certain conditions.
 
Let $X=Y=Z=\R^d$,  $c(x,y)=|x-y|^2$, $\mu, \nu\in {\cal M}_1^{(2)}$ (recall (\ref{wasq})).  Assume $\mu, \nu$ are absolutely continuous with respect to Lebesgue measure on $\R^d$. In that case, Brenier Polar factorization Theorem \ref{brenierTh} implies the existence of a unique solution to the quadratic Monge problem,  \index{Monge} i.e a Borel mapping  $T$ such that $T_\#\mu=\nu$  (\ref{brenier}). Let $\lambda$ be the McCann \index{McCann Interpolation} interpolation between $\mu$ and $\nu$ corresponding to the middle point $s=1/2$ (see section \ref{secmecinter}).  It turns out that $\lambda$ is absolutely continuous 
with respect to Lebesgue measure ${\cal L}$ as well. Let  $f:= d\lambda/ d{\cal L}\in \mathbb{L}_1(\R^d)$. 
\begin{theorem}
	Under the above assumptions,
	$$  \limsup_{m\rightarrow \infty}m^{2/d}\phi^m(\mu,\nu)\leq 4C_{d,2}\left(\int f^{d/(d+2)}dx\right)^{(d+2)/d} \ . $$
\end{theorem}

\begin{proof}
 Let  $S_1$ to be the Monge mapping transporting $\lambda$ to $\mu$, and 
 $S_2$  the Monge mapping transporting $\lambda$ to $\nu$. In particular $\mu=S_{1\#}\lambda$, $\nu=S_{2,\#}\lambda$ and (recall $c(\cdot,\cdot):=W_2^2(\cdot, \cdot)$) we get by (\ref{likewise150}, \ref{likewise1501})	
	\be\label{cccde}c(\lambda,\mu)=\int|S_1(z)-z|^2 d\lambda= c(\lambda,\nu)=\int|S_2(z)-z|^2 d\lambda=\frac{1}{4} c(\mu,\nu) \ . \ee
	\begin{figure}
		\centering
		\includegraphics[height=4.cm, width=16.cm]{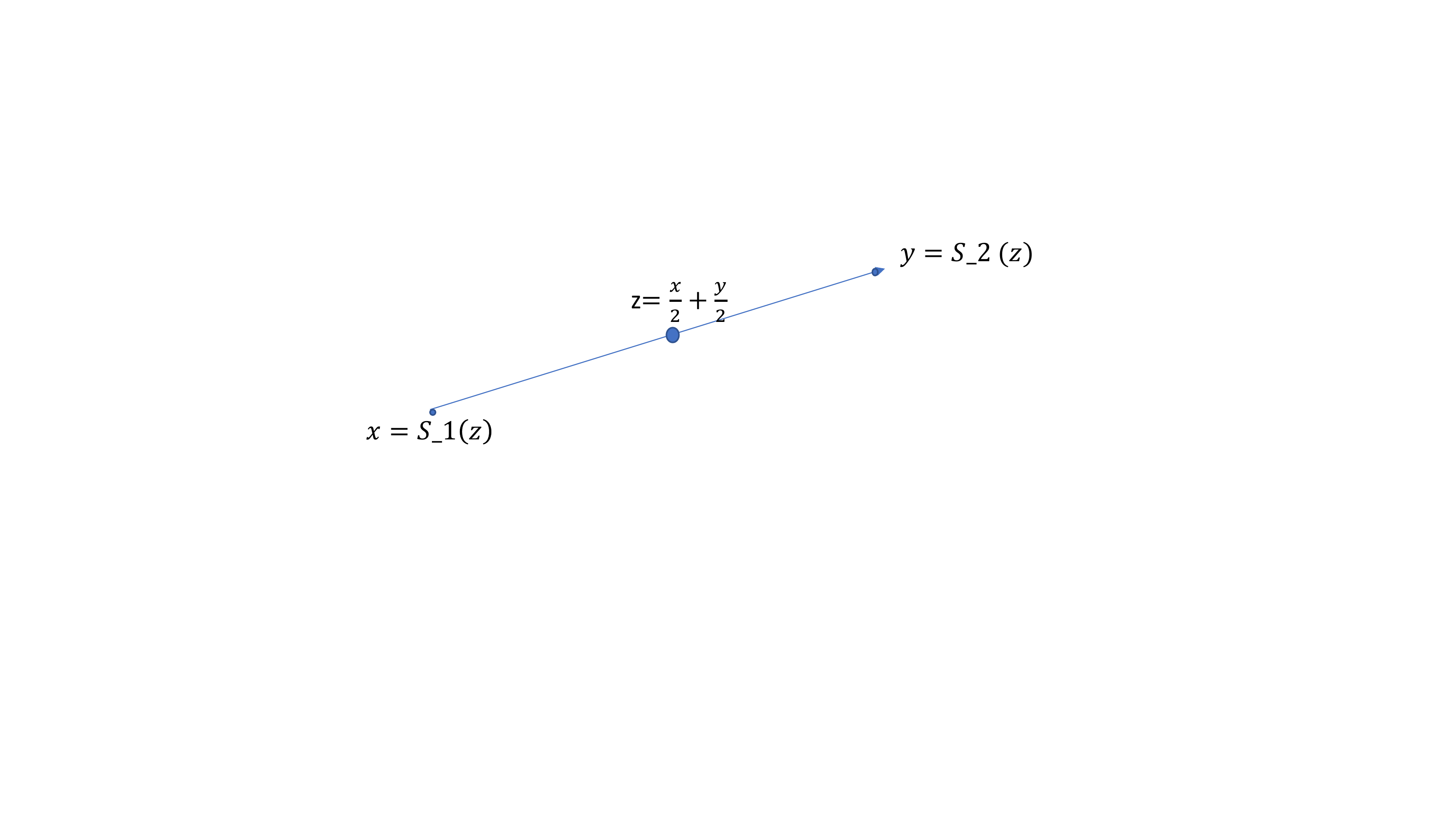}\\
		\caption{Interpolation: z is the mid point between $x$ to $y=T(x)$. }\label{figinter}
	\end{figure} 


	Given $1\leq i\leq m$, let $\lambda_1, \ldots \lambda_m$ be a weak $m-$partition of $\Lambda$. In particular $\sum_1^m\lambda_i=\lambda$. Let
	$z_i$ be the center of mass of $\lambda_i$, so
	\be\label{com} \int zd\lambda_i= \lambda_i(\R^d)z_i  \ . \ee
	From (\ref{cccde}) it follows 
	\be\label{ss1}c(\mu,\nu)=2\left[\sum_{1\leq i\leq m}\int |S_1(z)-z|^2d\lambda_i +\sum_{1\leq i\leq m}\int|S_2(z)-z|^2d\lambda_i\right]\ . \ee
Let $\mu_i:= S_{1,\#} \lambda_i$, $\nu_i:= S_{2,\#}\lambda_i$. In particular $\nu_i(\R^d)=\mu_i(\R^d)=\lambda_i(\R^d)$ so $\{\mu_i\}, \{\nu_i\}$ is a congruent weak\index{congruent partitions}
 partition (Definition \ref{defP}). 
	Form Lemma  \ref{lemapart}
	$$ c^{Z_m}(\mu,\nu)\leq 2\left(\sum_{1\leq i\leq m} \int |x-z_i|^2\mu_i(dx)+ \sum_{1\leq i\leq m} \int |y-z_i|^2\nu_i(dy)\right)$$
	\be\label{ss2}= 2\sum_{1\leq i\leq m}\int\left\{|S_1(z)-z_i|^2+ |S_2(z)-z_i|^2\right\}d\lambda_i \ .  \ee
	Hence (\ref{cm})
	$$\phi^{Z_m}(\mu,\nu)\leq  2\sum_{1\leq i\leq m}\int_{V_i}\left[|S_1(z)-z_i|^2-|S_1(z)-z|^2\right]d\lambda
	$$
	$$+2\sum_{1\leq i\leq m} \int_{V_i}\left[|S_2(z)-z_i|^2-|S_2(z)-z|^2\right]d\lambda $$
	Using the identity
	$$ |S_\kappa(z)-z_i|^2-|S_\kappa(z)-z|^2= |z_i|^2 - |z|^2 -2S_\kappa(z)\cdot (z_i-z)$$
	for $\kappa=1,2$ we get 
		$$ |S_1(z)-z_i|^2-|S_1(z)-z|^2+
	|S_2(z)-z_i|^2-|S_2(z)-z|^2=$$
	$$ 2|z_i|^2 -2|z|^2 - 2(S_1(z)+ S_2(z))\cdot (z_i-z)= 2|z_i|^2 -2|z|^2 - 4z\cdot (z_i-z)$$
	where we used $\frac{1}{2}(S_1(z)+S_2(z))=z$ (c.f Fig \ref{figinter}).   Then,  (\ref{com}) and the above imply
	$$ \int\left\{|S_1(z)-z_i|^2- |S_1(z)-z|^2 + |S_2(z)-z_i|^2-|S_2(z)-z_i|^2\right\}d\lambda_i= $$
	$$4\int|z|^2d\lambda_i -2\lambda_i(\R^d) |z_i|^2=4\int |z_i-z|^2d\lambda_i \ . $$
 
	together with (\ref{ss1}, \ref{ss2}) and (\ref{cm}) implies
	\be\label{rree} \phi^m(\mu,\nu)\leq 4\sum_{1\leq i\leq m} \int|z-z_i|^2 d\lambda_i  \ee
	for any weak partition $\lambda_1, \ldots \lambda_m$.  Taking the minimal weak partition    it turns out,  by Corollary \ref{prozador} that the right side of (\ref{rree}) is as small as $4C_{d,2}\left(\int f^{d/(d+2)}dx\right)^{(d+2)/d} $ where $f$ is the density of $\lambda$.  
\end{proof}
\section{Symmetric transport and congruent partitions}\label{symcong} \index{congruent partitions}
The optimal transport between two  $\R_+^J$ valued measures, discussed in Part \ref{partII}, can be naturally generalized to an optimal transport between two {\em general} vector valued measures. 
 Here we replace the measures $\mu,\nu$ by {\em $\R_+^J$-valued} measures
$$\bar\mu:= (\mu^{(1)}, \ldots \mu^{(J)})\in {\cal M}_+^J(X), \ \ \bar\nu:= (\nu^{(1)}, \ldots \nu^{(J)})\in {\cal M}_+^J(Y)\ , $$
and we denote $\mu=|\bmu|:=\sum_1^J \mu^{(j)}$, $\nu=|\bnu|:=\sum_1^J\nu^{(j)}$. The
set $\Pi(\mu,\nu)$ (\ref{Pidef}) is generalized into 
\be\label{Pidefmunuvecnew} \Pi(\bmu,\bnu):=\{ \pi\in {\cal M}_+(X\times Y); \  \int_X\frac{d\mu^{(j)}}{d\mu}(x)\pi(dxdy)=\nu^{(j)}(dy) \ \ , \ j=1\ldots J \} \ . \ee 
where $d\mu_i/d\mu$, $d\nu_i/ d\nu$  stands for the Radon-Nikodym derivative. \index{Radon-Nikodym derivative}
\par
In general the set $\Pi(\bmu,\bnu)$ can be an empty one.  If  $\Pi(\bmu, \bnu)\not=\emptyset$ then 
$\bmu\succ\bnu$  (c.f Definition \ref{gendominance}).  
The generalization of the Kantorovich problem (\ref{supmonge}) takes the form
\begin{tcolorbox}
	\be\label{infmongevec} \theta(\bar\mu,\bar\nu):= \max_{\pi\in\Pi(\bar\mu,\bar\nu)}\int_X\int_Y \theta(x,y)\pi(dxdy) \ . \ee 
	$\theta(\bmu,\bnu)=\infty$  if $\bmu\not \succ\bnu$.
\end{tcolorbox}
If $J>1$ then   $\theta(\bmu,\bnu)\not= \theta(\bnu,\bmu)$  in general, even if 
$\bmu$ and $\bnu$ are living on the same domain $X$ and $\theta(x,y)=\theta(y,x)$ for any $x,y\in X$. Indeed we obtain from (\ref{infmongevec} ) that $\theta(\bnu,\bmu)=\infty$ if  $\bnu\not\succ\bmu$, while $\theta(\bmu,\bnu)<\infty$ if 
$\bmu\succ\bnu$.  This is in contrast to the case $J=1$. 

Let $Z$ be  measure space and $\theta$ satisfies (\ref{pizul}). 
Then  we define  
\be\label{genopttrans} \bar\theta(\bmu,\bnu):=\sup_{\blambda\prec\bmu\wedge\bnu}\theta_1(\bmu,\blambda) +\theta_2(\bnu, \blambda) \ . \ee
In particular 
\begin{tcolorbox}
	If $X=Y$, $\theta_1=\theta_2$ then $\bar\theta(\bmu,\bnu)=\bar\theta(\bnu,\bmu)$ for any $\bmu,\bnu$. $\bar\theta(\bnu,\bmu)<\infty$   if and only if $\bmu(X)=\bnu(X)$.  
	
	If $J=1$ then $\bar\theta(\mu,\nu)<\infty$ only if $\mu(X)=\nu(Y)$. 
\end{tcolorbox}

From now on we assume that   $Z$ is a finite space. One of the motivations for this model is an extension of the hedonic market \index{hedonic market} (section \ref{hedonic} ) to several commodities: 
\par\noindent
Consider a market of $\J=(1, \ldots J)$ goods. The domain $X$ is the set of consumers of these goods, and $\mu^{(j)}$ is the distribution of consumers of $j\in\J$. Likewise, $Y$ is the set of manufacturers of the goods, and $\nu^{(j)}$ is the distribution of the manufacturers of $j\in\J$. 

In addition we presume the existence of $N$ "commodity centers" 

$$Z_N:=\{z_1, \ldots z_N\} \ . $$ Let $\theta^{(j)}_1(x,z_i)$ be the utility of the good $j$ for a consumer  $x$ at the center $z_i$, same as  $\theta^{(j)}_2(y,z_i)$  for a producer $y$ of $j$   at the center $z_i$.
  
We may extend definition \ref{sdefP} of  congruent  $N-$partition to this setting:

A partition $\vec{A}=(A_1, \ldots A_N)$ of $X$ and $\vec{B}=(B_1, \ldots B_N)$ of $Y$ are congruent \index{congruent partitions} with respect to $\bmu=(\mu^{(1)}, \ldots \mu^{(J)})$, 
$\bnu=(\nu^{(1)}, \ldots \nu^{(J)})$
 if 
\be\label{bcond}\mu^{(j)}(A_{i})=\nu^{(j)}(B_{i}) \ \ 1\leq i\leq N; , \ \ 1\leq j\leq J \ . \ee
Any such possible congruent partition represents a possible matching between the consumers and the producers: all consumers in $A_{i}$ and all producers in $B_{i}$ are associated with the single center $z_i$.  The balance condition (\ref{bcond}) guarantees that the center $z_i$ can satisfies the supply and demand for {\em all} goods $J$, simultaneously.

The total utility of such a congruent partition is\index{congruent partitions}
\begin{multline}\Theta(\{A_i\}, \{B_i\}):= \sum_{j=1}^J\sum_{z\in Z_N}\int_{A_i} \theta_1^{(j)}(x, z) \mu^{(j)}(dx) +  \int_{B_i}\theta_2^{(j)}(y, z) \nu^{(j)}(dy) \\ \equiv \sum_{z\in Z_N}\int_{A_i} \theta_1(x,z)\mu(dx) + \int_{B_i} \theta_2(y,z)\nu(dy)\end{multline}
where 
$$\mu:=\sum_{j=1}^J \mu^{(j)} \ \ , \nu:=\sum_{j=1}^J \nu_j , \ \ \theta_1:= \sum_{j}\theta_1^{(j)}d\mu^{(j)}/d\mu \ , \  \theta_2:= \sum_{j}\theta_2^{(j)}d\nu^{(j)}/d\nu \ . $$ 
The efficient partition is the one which maximize the total utility among all possible congruent partitions.\index{congruent partitions}
\par
Other motivation concerns an  application of Monge metric to colored images.  The Monge metric  \index {Monge distance}(often called the "earth movers metric") became very popular in computer imaging in recent years. The general practice  for black \& white images is to     consider these images as probability measures on an Euclidean domain (say a rectangle $B$), demonstrating the level degree of gray. The matching between the  two images is reduced to solving the Monge \index{Monge} problem for the two corresponding measures $\mu, \nu$ on $B$, and  is given by the optimal matching $T:B\rightarrow B$ in (\ref{mongeclass}) where, in general, $\theta(x,y)=-|x-y|^2$. The motivation is either to quantify the difference between two such images, or to interpolate  between the two images in order to obtain a video connecting two possible states. 
\par
If these measures are colored, then 
the general practice is to consider them as probability measures in a lifted space $B\times C$ where the color space $C$ is, in general, a three dimensional domain representing the level of the RGB (Red-Green-Blue) values.
The matching is still given by a solution of the Monge problem  \index{Monge} (\ref{mongeclass}) where, this time, the measures are defined on  $B\times C$ and the optimal matching is a mapping in this space as well.

The alternative paradigm suggested by vectorized transport is to view the images as {\em vector valued} (RGB) measures.

It is remarkable, as shown in Lemma \ref{lemapart},  that the case of a single good ($J=1$) is reduced to an optimal transport of $(X,\mu)$ to $(Y,\nu)$  with respect to the utility
$$ \theta_{Z_N}(x,y):=\max_{z\in Z_N} \theta_1(x, z)+\theta_2(y, z) \ . $$
This, unfortunately, is {\em not}  the case for the vectorized case. However, 
Theorem \ref{strongth} can be extended to the vectorized case, where we define 
$$ \Xi^{(1)}(\cP)\equiv 
\mu\left(\xi^{(1)}(\cdot,\cP )\right)  
\ \ ; \ \ \xi^{(1)}(x,\cP)\equiv\max_i (\theta_1(x, z_i)+\vpp_i\cdot d\bmu/d\mu(x)) , $$
$$ \Xi^{(2)}(\cP)\equiv 
\nu\left(\xi^{(2)}(\cdot,\cP )\right)  
\ \ ; \ \ \xi^{(2)}(x,\cP)\equiv\max_i (\theta_2(y, z_i)+\vpp_i\cdot d\bnu/d\nu(y)) , $$
$$\Xi(\cP):= \Xi^{(1)}(\cP) + \Xi^{(2)}(-\cP) \ . $$
The proof of the Theorem below is very similar to the proof of Theorem \ref{strongth} so we skip it. 
\begin{theorem}
	If {\em any} $J\geq 1$ and under Assumption \ref{assAB} 
	\be\label{ThetainfP} \max_{\{A_i\}, \{B_i\}} \Theta(\{A_i\}, \{B_i\})=\inf_{\cP}\Xi(\cP; \bmu,\bnu) \  . \ee
	where the infimum  is over all $N\times J$ matrices $\cP=\{p_i^{(j)}\}$ and maximum is over  all $\bmu-\bnu$ congruent partitions. 
	If a minimizer $\cP_0$ is obtained  
	then the optimal congruent  partitions $\{A_i^0\}, \{B_i^0\}$ satisfy\index{congruent partitions}
$$A_i^0=A^{\theta_1}_i(\cP_0), \ \ B_i^0=A^{\theta_2}_i(-\cP_0)  $$
	where $A^\theta_i(\cP)$     are defined as in (\ref{Aiphi}) , where $\vzeta = d\bmu/d\mu$ (resp. $\vzeta=d\bnu/d\nu$). 
\end{theorem}
\part{Cooperative and non-Cooperative partitions}\label{part3}
\chapter{Back to Monge: Individual values}\label{BTM}
{\small {\it You don't get paid for the hour. You get paid for the value you bring to the hour}.  Jim Rohn}
\vskip .3in
Theorems \ref{main33}, \ref{main3Z} are the most general result we obtained so far, regarding the existence and uniqueness of generalized, strong (sup)partitions. In particular it  provides a full answer to the questions raised in Section \ref{sumandbey}, together with a constructive algorithm via a  minimization of a convex function for finding the optimal (sub)partitions.  What we need are just   Assumptions \ref{mainass3} and (resp.)  \ref{mainass2}(i+ ii).
\par
Yet,  it seems that we still cannot answer any of these questions  regarding the Saturation and Over-Saturation cases for {\it non-generalized} (sub)partitions, discussed in Sections
\ref{secbigbrother}-\ref{BBMFM}.
\par
Let us elaborate this point. Theorem \ref{main33} provides us with uniqueness {\em only up to a coalition's ensemble}.  \index{coalitions ensemble}So, if the ensemble's units  are not singletons, the theorem only gives us uniqueness {\em up to the given ensemble}.  On the other hand, Theorem \ref{main3Z} (as well as Theorem \ref{main3ZK}) provides uniqueness without reference to any coalition. However, the assumption behind this Theorem require the fixed exchange ratios $\{\vzet_i\}$ defined in section \ref{FER}, and the corresponding Assumption \ref{mainass31}. The Monge partition problem, as described in Chapter \ref{S(M)P}, corresponds to the case $\vzeta$ is real valued  (i.e. $J=1$). This, indeed, is equivalent to the case of fixed exchange rates in $\R^{J}$, $J>1$ where all $\vzet_i\in\R^{J}$ {\em equal each other}. This, evidently, defies Assumption \ref{mainass31}.
\par
So, what about Theorem  \ref{main3}? It only   requires  Assumption \ref{mainass2} which, under the choice $\vzeta\equiv 1$, takes the form:
\begin{assumption}\label{mainass4}.
	\begin{description}
		\item{i)} For any $i,j\in\I$ and any $r\in \R$,   \ \ \  $\mu\left(x\in X \ ; \ \  \theta_i(x)-\theta_j(x)=r\right)=0$ .
		\item{ii)}  For any $i\in\I$ and any $r \in\R$,   \ \ \  $\mu\left(x\in X \ ; \ \  \theta_i(x)=r\right)=0$ .
	\end{description}
\end{assumption}
Hence, Theorem \ref{main3} can be applied for non-generalized (sub)partitions, granting Assumption \ref{mainass4}. However, this Theorem only guarantees the existence and uniqueness of a strong (sub)partition for an {\it interior} points of $\uvmS_N(\bar\mu)$.

Which of the points in $\uvmS_N(\bar\mu)$ are interior points? \index{over saturated (OS)} \index{under saturated (US)}
It is evident that under the choice $\vzeta=1$ the US,S,OS condition (\ref{gUS},\ref{gS},\ref{gOS}) are reduced to (\ref{US},\ref{S},\ref{OS}). Hence, an interior point must be a  US point  (\ref{US}). In particular, we still cannot  deduce the uniqueness of stable partitions \index{stable partition}for (S) and (OS) capacities..... \index{over saturated (OS)} \index{under saturated (US)}

But, alas, "Despair is the conclusion of fools".\footnote{
	Benjamin Disraeli, The Wondrous Tale of Alroy, Part 10, Chapter 17.} It turns out that we  can still prove this result, using {\em only} Assumption \ref{mainass4}-(i).

We recall the setting of the Monge problem   \index{Monge} (Chapter \ref{S(M)P}). Here $J=1$ so we set $\D(N,J)=\Dt(N,J):=\R$ and $\vzeta=1$.
In addition we  make the following  change of notation from chapters \ref{Wps}-\ref{scunique}: replace $\vpp$ by $\vpp=-\vpp$. This notation is more natural if we interpret $\vpp$ as the {\em price} vector of the agents. Under this  change
\be\label{Xinoplus}\vpp \mapsto\Xepto(\vpp):=\mu\left(\max_{i\in\I}(\theta_i(\cdot)-p_i)\right) \ , \ee \index{$\Xepto$} 
$\Xept$ (\ref{Xi+def}):
\be\label{Xiwithplus}\vpp \mapsto\Xept(\vpp):=\mu\left(\max_{i\in\I}[\theta_i(\cdot)-p_i]_+\right) \ . \ee\index{$\Xept$} 

Recall $N=|\I|$ is the number of agents in $\I$. Let 
$\vm=(m_1, \ldots m_N)\in \R^N_+$ and $|\vm|:= \sum_{i=1}^Nm_i$. In that case  the definitions of $\vmS_N$ and $\uvmS_N$ (Definition \ref{defvmSw}) are reduced to 
\be\label{simplexM} \vmS_N(\mu):=\{\vm\in\R^N_+;  |\vm|=\mu(X)\}, \ \ \ \uvmS_N(\mu):=\{\vm\in\R^N_+; |\vm|\leq\mu(X)\} \ . \ee 
\begin{theorem} \ \label{old} 
	\begin{description}
		\item{a)} Let Assumption \ref{mainass4}-(i). Let  $K\subset\R^N_+$ is a closed convex set such that $|\vec{m}|\geq \mu(X)$ for any $\vM\in K$.
		
		Then there exists an equilibrium price  vector $\vpp^0$ , unique up to an additive translation
		\be\label{adtrans}p^0_i\rightarrow p^0_i+\gamma,\  i\in\I \ , \ \gamma\in\R \ee   which is a minimizer of
		$$\vpp \mapsto\Xepto(\vpp)+H_K(\vpp)$$
		on $\R^N$ (recall (\ref{defpcircMZ}). Moreover, the associated partition
		$$ \vA^\theta(\vpp^0):=(A^\theta_1(\vpp^0), \ldots ,A^\theta_N(\vpp^0))   $$
		where
		\be\label{AsubiP}\ A^\theta_i(\vpp^0):=\{x\in X; \theta_i(x)-p^0_i>\max_{j\not= i}\theta_j(x)-p^0_j \}\ee  
		is the {\em unique} optimal partition which maximizes $\theta(\vA)$ on $\OP_{K\cap \vmS_N(\mu)}$. \index{optimal partition}
	
		\item{b)} Let Assumption \ref{mainass4}-(i,ii). Let  $K\subset\R^N_+$ is a closed convex set such that $K\cap\uvmS_N(\mu)\not=\emptyset$.
		
		Then there exists an   equilibrium price  vector $\vpp^0$ which is a minimizer of
		$$\vpp \mapsto\Xi^+_{\theta}(\vpp)+H_K(\vpp)$$
		on $\R^N$. Moreover, the associated (sub)partition
		$$ \vA^{\theta,+}(\vpp^0):=(A^{\theta,+}_1(\vpp^0), \ldots ,A^{\theta,+}_N(\vpp^0))  \ , $$
		where
		\be\label{AsubiP+}\ A^{\theta,+}_i(\vpp^0):=A^{\theta}_i(\vpp^0) -A^{\theta}_0(\vpp^0) , \ \ A^{\theta}_0(\vpp^0):= \{x\in X; \max_{1\leq j\leq N}\theta_j(x)-p^0_j\leq 0\}\ee
		is the {\em unique} optimal sub-partition which maximizes $\theta(\vA)$ on $\OSP_{K\cap \uvmS_N(\mu)}$. If $\mu(A^{\theta}_0(\vpp^0))>0$ then the vector $\vpp^0$ is    unique, and if $\mu(A^{\theta}_0(\vpp^0))=0$ then  $\vpp^0$ is unique {\em up to a negative  additive translation}
		\be\label{adtrans-}p^0_i\rightarrow p^0_i-\gamma,\  i\in\I \ , \ \gamma\in\R_+ \ .  \ee
	\end{description}
\end{theorem}
In particular, recalling Section \ref{escl} we obtain  that, {\em in spite of the unboundedness of the equilibrium price $\vpp^0$ (\ref{adtrans})},
\begin{cor}
	There is no escalation for the Monge problem under Assumption  \ref{mainass4}. \index{Monge}\index{escalation}
\end{cor}
Another conclusion which we obtain  yields a {\em unified} representation in the under saturation, saturation and  over saturation cases. Here we consider   $K=\{\vec{s}\in\R^N_+; \vec{s}\leq \vm\}$, so $H_K(\vpp)=[\vpp]_+\cdot \vm $ where   $[\vpp]_+:= ([p_1]_+, \ldots [p_N]_+)$.   \begin{cor}\label{aconc} 
	Under Assumption  \ref{mainass4}, there exists a (sub)partition $\vA_0$ such that
	$$\theta(\vA_0)=\Sigma^{\theta,+}(\vM):=\min_{\vpp\in\R^I}\Xi^{\theta,+}(\vpp)+[\vpp]_+\cdot \vM =\Xi^{\theta,+}(\vpp^0)+[\vpp^0]_+\cdot \vM \ .  $$
	Moreover, $\vA_0= A^{\theta,+}_i(\vpp^0)$ is given by (\ref{AsubiP+}). 
\end{cor} 
The claim below is an extension, for Monge (sub)partitions, of Corollary \ref{coruniquediff} which uses  the uniqueness result of the equilibrium vector $\vpp^0$ and Proposition \ref{cf7}:
\begin{cor}\label{coruniquediff2}
	Under Assumption  \ref{mainass4},
	The function $\Sigma^{\theta,+}$ is differentiable at any  interior point
	$\vM\in \uvmS_N(\mu)$, and
	$$ \frac{\partial\Sigma^{\theta,+}}{\partial m_i}=p^0_i \ \ ; \ \ i\in\I$$ 
	If $\vM\in\vmS_N(\mu)$  then  $\Sigma^{\theta,+}$ is differentiable in the "negative" direction, i.e.
	\be\label{diff-}\frac{\partial^-\Sigma^{\theta,+}}{\partial m_i}:=-\lim_{\eps\downarrow 0} \eps^{-1}\left(\Sigma^{\theta, (+)}(\vM-\eps\vec{e}_i)-\Sigma^{\theta, (+)}(\vM)\right)=  p^0_i\ee
	while    $\Sigma^{\theta}$ is differentiable on the tangent space of $\vmS_N(\mu)$, i.e.
	
	\be\label{diff0}\lim_{\eps\downarrow 0} \eps^{-1}\left(\Sigma^{\theta, (+)}(\vM+\eps\vec{\zeta})-\Sigma^{\theta, (+)}(\vM)\right)= \vzeta\cdot \vpp^0\ee
	for any $\vzeta=(\zeta_1, \ldots \zeta_N)$ satisfying $\sum_{i\in\I}\zeta_i=0$, $\zeta_i\geq 0$ if $m_i=0$.
\end{cor}
\begin{remark}\label{remmaxprice} The vector $\vpp^0$ defined in the saturated case  by (\ref{diff-}) is the {\em maximal price vector}   (\ref{adtrans-}). It is the maximal price which the agents can charge such that {\em any} consumer will attend some agent.
\end{remark}
\begin{remark}
	The two parts of the Theorem contain the three cases (recall (\ref{OS}, \ref{S}, \ref{US})  \index{over saturated (OS)} \index{under saturated (US)} 
	\begin{description}
		\item{US)} The Under Saturated $\vM\in int(\uvmS_N(\mu))$ in part (b) where $K=\{\vM\}$,
		\item{S)}  The Saturated  $\vM$ in both  (a) and (b) where $\vM\in \vmS_N(\mu)$,
		$K=\{\vM\}$,   and
		\item{OS)}  The Over Saturated   where $\vM\not \in \uvmS_N(\mu)$. If the components $\theta_i$ are all non-negative   then case (a) is valid  since the {\em only}  maximizer of
		$\Sigma^{\theta,+}$
		is in $\vmS_N(\mu)$ (show it!).
	\end{description}

\end{remark}

\begin{proof} of Theorem \ref{old}:
	\par\noindent (a)
	The inequality (\ref{cleqxi+vcircMpPE}) of Proposition  \ref{vmuinwsPP} is valid also if we replace $\Xept$ by $\Xepto$.
	Indeed, (\ref{sec11})  is extended to\footnote{Note the change of notation from  $\vpp$ to $-\vpp$ between  section \ref{BBMFM} and here.  This is because $\vpp$ is more natural as a price vector in section \ref{BBMFM}. }
	$$ \theta_i(x)\leq \xi(\vpp,x)+p_i \ \ \ \text{where} \ \ \xi(\vpp,x):=\max_{1\leq j\leq N}\theta_j(x)-p_j \ ,  $$
	so
	\be\label{cleqxi+vcircMpPEnew}\theta(\vA)\leq \Xepto(\vpp)+H_K(\vpp) \ . \ee
	holds for any $\vA\in\OSP_K$ and $\vpp\in\R^N$. \index{over saturated (OS)}
	In case of an equality (\ref{cleqxi+vcircMpPEnew}), Proposition \ref{prhe} is valid as well.
	\par
	Assume first $K=\{\vM\}$ where $\vM$ is a saturated vector ($\vM\in\vmS_N(\bar\mu)$). Then (\ref{cleqxi+vcircMpPE}) takes the form
	\be\label{cleqxi+vcircMpPEM}
	\theta(\vA)\leq \Xepto(\vpp)+\vpp\cdot\vM\ . \ee
	
	\par Note that Proposition \ref{difcor} can be applied  since Assumption ~\ref{mainass2}(i) is compatible with Assumption \ref{mainass4}. In particular it follows that $\Xepto$ is differentiable on $\R^N$. The first  equality in (\ref{diffmphi}) is translated into
	\be\label{diffmphi1}  \frac{\partial\Xepto}{\partial p_i}(\vpp)=- \int_{A_i^\theta(\vpp)}  d\mu  \  \ .  \ee

	We now prove the existence of such a minimizer $\vpp^0$.
	
	Observe that \be\label{additive}\Xepto(\vpp +\alpha\vec{1})=\Xepto(\vpp)-\alpha\mu(X) \ \ \ ; \ \ \ \vec{1}:=(1, \ldots 1)\in\R^N\ . \ee
	In particular $\nabla\Xepto(\vpp)=\nabla\Xepto(\vpp+\alpha\vec{1})$  and,  in the saturated case $\vec{1}\cdot\vM=\mu(X)$:
	\be\label{additive=} \Xepto(\vpp)+\vpp\cdot\vM=  \Xepto(\vpp+\alpha\vec{1})+(\vpp+\alpha\vec{1}\cdot\vM)\ee
	for any $\alpha \in\R$. So, we restrict the domain of $\Xepto$ to
	\be\label{normal} \R^N_0:=\{\vpp\in \R^N \ \ , \ \ \vpp\cdot\vec{1}=0 \} \ .  \ee
	\par

	Let $\vpp_n$ be a minimizing  sequence of  $\vpp \mapsto\Xepto(\vpp)-\vpp\cdot\vM$ in $\R^N_0$, that is
	$$\lim_{n\rightarrow\infty} \Xepto(\vpp_n)+\vpp_n\cdot\vM=\inf_{\vpp\in\R^N}\Xepto(\vpp_n)+\vpp_n\cdot\vM \ . $$
	
	Let $\|\vpp\|_2:= (\sum_{i\in\I}p^2_i)^{1/2}$ be the Euclidean norm of $\vpp$. If we prove that for any minimizing sequence $\vpp_n$ the norms $\|\vpp_n\|_2$ are uniformly bounded, then there exists  a converging subsequence whose limit is the minimizer  $\vpp^0$. This follows since $\Xepto$ is, in particular, a continuous   function.
	\par
	Assume there exists a subsequence along which $\|\vpp_n\|_2\rightarrow\infty$. Let $\widehat{\vpp}_n:= \vpp_n/\|\vpp_n\|_2$.
	Then
	\begin{multline}\label{try1}\Xepto(\vpp_n)+\vpp_n\cdot\vM:= \left[\Xepto(\vpp_n)-\vpp_n\cdot\nabla_{\vpp}\Xepto(\vpp_n)\right] + \vpp_n\cdot\left( \nabla_{\vpp}\Xepto(\vpp_n)+\vM\right)\\
	=\left[\Xepto(\vpp_n)-\vpp_n\cdot\nabla_{\vpp}\Xepto(\vpp_n)\right] + \|\vpp_n\|_2\widehat{\vpp}_n\cdot\left( \nabla_{\vpp}\Xepto(\vpp_n)+\vM\right) \ . \end{multline}
	Note  that
	\be\label{Ximu0}    \Xepto(\vpp)-\vpp\cdot\nabla\Xepto(\vpp)= \sum_{i\in\I}\int_{A_i(\vpp)} \theta_id\mu\ , \ee
	so, in particular
	\be\label{try2}0 \leq \int_X \min_{i\in\I} \theta_i d\mu \leq \left[\Xepto(\vpp)-\vpp\cdot\nabla_{\vpp}\Xepto(\vpp)\right]=\sum_{i\in\I} \int_{A_i(\vpp)}\theta_i(x)d\mu \\ \leq \int_X \max_{i\in\I} \theta_i d\mu<\infty \ . \ee
	By (\ref{try1}- \ref{try2}) we obtain, for $\|\vpp_n\|_2\rightarrow\infty$,
	\be\label{lim0phat}\lim_{n\rightarrow\infty} \widehat{\vpp}_n\cdot\left( \nabla_{\vpp}\Xepto(\vpp_n)+\vM\right)=0 \ . \ee
	Since $\widehat{\vpp}_n$ lives in the unit sphere in $\R^N$  (which is a compact set), there exists a subsequence for which $\widehat{\vpp}_n\rightarrow \widehat{\vpp}_0:= (\hat{p}_{0,1}, \ldots \hat{p}_{0,N})$. Let $P_+:= \max_{i\in\I} \widehat{p}_{i,0}$ and $J_+:= \{ i \ ; \widehat{p}_{0,i}=P_+\}$.
	
	Note that for $n\rightarrow\infty$ along such a subsequence, $p_{n,i}-p_{n, k}\rightarrow\infty$ iff $i\in J_+, k\not\in J_+$. It follows that $A^\theta_{k}(\vpp_n)=\emptyset$ if $k\not\in J_+$ for $n$ large enough, hence $\mu(\cup_{i\in J_+}A^\theta_i(\vpp_n)) =\mu(X)=\mu(X)$ for $n$ large enough. Let
	$\mu_i^n$ be the restriction of $\mu$ to $A^\theta_i(\vpp_n)$. Then the limit $\mu^n_i\rightharpoonup \mu_i$ exists (along a subsequence) where $n\rightarrow\infty$.  In particular, by (\ref{diffmphi1})
	$$\lim_{n\rightarrow\infty} \frac{\partial\Xepto}{\partial p_{n,i}}(\vpp_n)=-\lim_{n\rightarrow\infty} \int_X d\mu_i^n= - \int_X d\mu_i$$
	while  $\mu_i\not=0$  only if $i\in J_+$, and $\sum_{i\in J_+}\mu_i=\mu$. Since $\widehat{p}_{0,i}=P_+$ for $i\in J_+$ is the maximal value of the coordinates of $\widehat{\vpp}_0$, it follows that
	$$\lim_{n\rightarrow\infty} \widehat{\vpp}_n\cdot\left( \nabla_{\vpp}\Xepto(\vpp_n)+\vM\right)=\widehat{\vpp}_0\cdot\vM -P_+\sum_{i\in J_+}\int_Xd\mu_i=\widehat{\vpp}_0\cdot\vM -P_+\mu(X) \ . $$
	Now, by definition, $\widehat{\vpp}_0\cdot\vM <P_+\mu(X)$ unless $J_+=\{1, \ldots N\}$. In the last case we obtain a contradiction of (\ref{normal}) since it implies $\widehat{\vpp}_0=0$ which contradicts $\widehat{\vpp}_0$ is in the unit sphere in $\R^N$. If $J_+$ is a proper subset of $\{1,\ldots N\}$ we obtain a contradiction to (\ref{lim0phat}).
	Hence $\|\vpp_n\|_2$ is uniformly bounded, and any limit $\vpp^0$ of this set is a minimizer.
	\par
	The proof of uniqueness of optimal partition is identical to the proof of this part in Theorem \ref{main3} (see  (\ref{734})). This also implies the uniqueness (up to a shift) of $\vpp^0$ via (\ref{additive=}). \index{optimal partition}
	
	To complete the proof we need to show that
	\be\label{hrcjohr}\vM\in K\cap\vmS_N(\mu)\mapsto{\Sigma^\theta}(\vM):= \min_{\vpp\in\R^N}\Xepto(\vpp)+\vpp\cdot\vM\ee
	admits a
	unique  maximizer.

	Recall that the function  $\Xepto$ is convex function on $\R^N$. Moreover, its partial derivatives exists at any point in $\R^N$, which implies that its  sub-gradient is s singleton.   Its Legendre transform  \index{Legendre transform}  takes finite values  only on the simplex of saturated vectors $\vmS_N(\mu)$. Indeed, by (\ref{additive})
	$$\Xepto(\vpp +\alpha\vec{1})+(\vpp+\alpha\vec{1})\cdot\vM =\Xepto(\vpp)+\alpha(\sum_{i\in\I} m_i-\mu(X)) \ , $$
	so
	$$\Xi^{\theta,*}(\vM):= \sup_{\vpp\in\R^N}\vpp\cdot\vM-\Xepto(\vpp)=\infty$$
	if 
	$\vm\not\in\vmS_N(\mu)$.  
	In fact, we already know  that $\vmS_N(\mu)$ is the {\em essential domain} of $\Xi^{\theta,*}$.

	Now, $K\cap \vmS_N(\mu)$ is a compact, convex set. The uniqueness of the maximizer (\ref{hrcjohr}) follows if  $\Xi^{\theta,*}$ is {\em strictly convex} on its essential domain $\vmS_N(\mu)$.  This follows from the differentiability of $\Xepto$ and from Proposition \ref{cf7}.
	\par\noindent
	b) \ The proof of case (b) follows directly from the proof of case (a), where we add the agent $\{0\}$ to $\{1, \ldots ,N\}$, and set $\theta_0\equiv 0$. The uniqueness of $\vpp^0=(p^0_1, \ldots p^0_N)$ in that case follows from the uniqueness up to a shift of $(p^0_0, p^0_1, \ldots , p^0_N)$, where we "nailed" this shift by letting $p_0^0=0$.
\end{proof}

\section{The individual surplus values}\label{indivalue} 
The main conclusion we may draw from Theorem  \ref{old} is the existence of an "individual value" (i.v) for an agent. This is the value which the consumers  attribute to their agents. If the price vector of agents is $\vpp$, then the individual value for agent $i$ is \index{individual (surplus) value}
\be\label{indipro}V^\theta_i(\vpp):= \int_{A^\theta_i(\vpp)}\theta_id\mu \ee
where $A^{\theta}_i(\vpp)=\{x; \theta_i(x)-p_i= \max_{j\not=i}[\theta_j(x)-p_j]_+\}$. Under the conditions of Theorem \ref{old} we know that the partition is uniquely determined by the capacities $\vM$, so we may consider the partition $\vA$ and the individual values $\vec{V}$ as functions of the capacity vector $\vM$, rather than the price vector $\vpp$. Thus, we sometimes refer to
$$ V_i(\vM):= \int_{A^\theta_i(\vM)}\theta_id\mu $$
where $A^{\theta}_i(\vM)=A^{\theta}_i(\vpp(\vM))$.

\begin{example}\label{singajex}{\em The case of a single agent}:\\
	For $\theta\in C(X)$ is the utility function of a single agent, let (Fig \ref{theta_p})
	$$A^\theta_p:= \{x\in X; \theta(x)\geq p\} \ , $$
	\begin{figure}
		\centering
		\includegraphics[height=8.cm, width=12.cm]{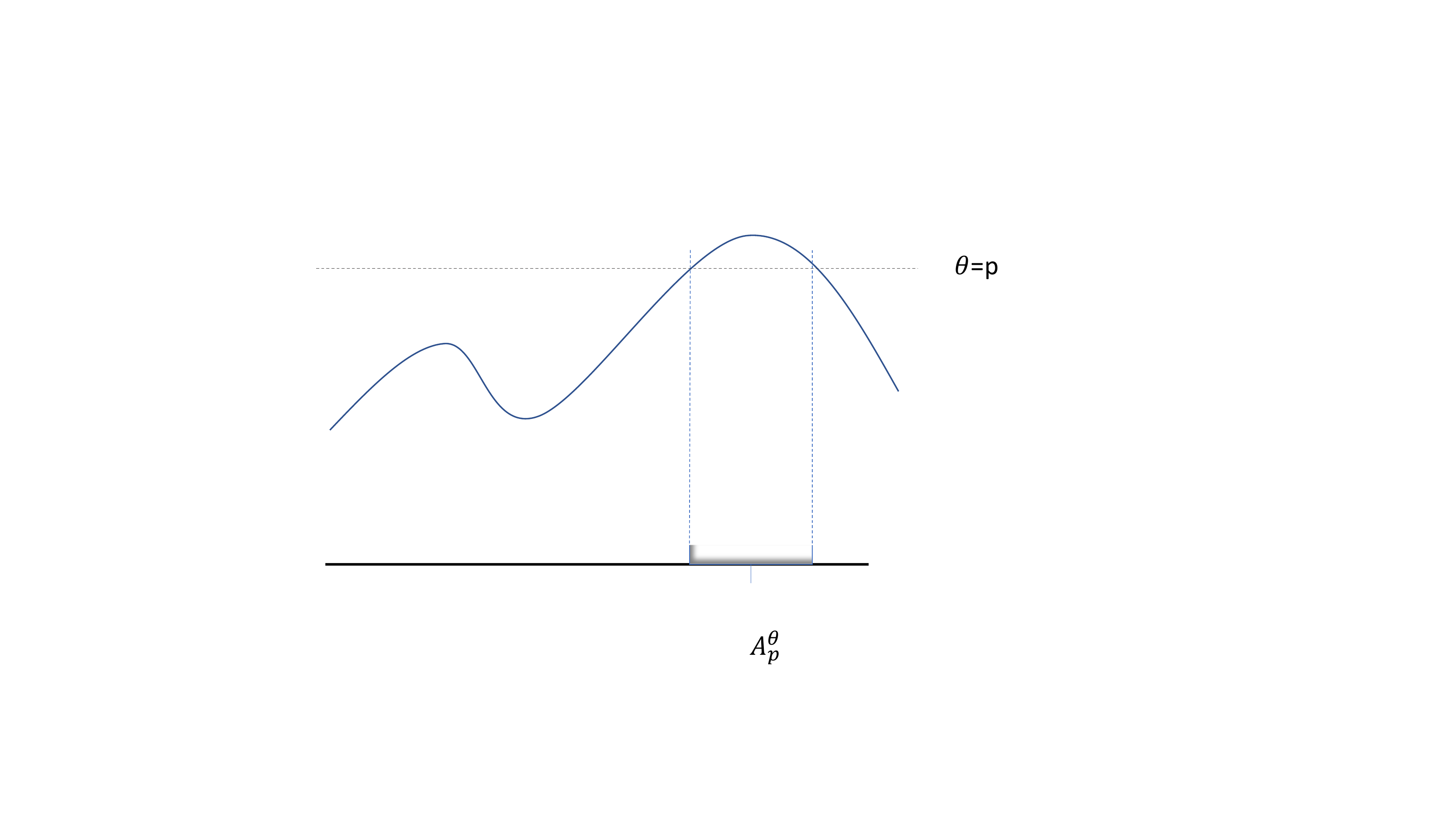}\\
		\caption{Single agent}\label{theta_p}
	\end{figure}
	$$m_\theta(p):=\mu(A^\theta_p)\ , F_\theta(t):=\int_t^\infty m_\theta(s)ds \ , \ {\cal F}_\theta(m):=\inf_{t\in\R}[mt+F_\theta(t)]\ . $$ Note that $F_\theta$ is defined since $\theta$ is bounded on $X$, so $m_\theta(t)=0$ for $t>\max_X\theta$. Moreover, $F_\theta$  and ${\cal F}_\theta$ are   concave functions, and 
	$$ -\int_p^{\infty} t dm_\theta(t) = \int_{A_p^\theta}\theta d\mu \equiv V^\theta(m_\theta(p))\ . $$
	Integration by parts  and duality implies \index{Kantorovich duality}
	$$  -\int_p^{\infty} t dm_\theta(t) =pm_\theta(p)+F_\theta(p)\equiv {\cal F}_\theta(m_\theta(p))  \ . $$
	Substitute $m_\theta(p)=m$ we obtain that the i.v for the single agent of capacity $m$ is just ${\cal F}_\theta(m)$, so, for any $m>0$,
	
	\be\label{P=F}V^\theta(m)={\cal F}_\theta(m) \ \forall m\in(0, \mu(X)] \ .  \ee
	
	Note that ${\cal F}_\theta(m)=-\infty$ if  $m >\mu(X)$.
	
	The equilibrium price $p=p_\theta(m)$ corresponding to capacity $m$ is the inverse of the function $m_\theta(p)$ and, by duality
	\be\label{pduality} p_\theta(m)=\frac{d {\cal F}_\theta(m)}{dm} \ . \ee
	
	Also,
	by definition, ${\cal F}_\theta(0)=0$ so
	\begin{tcolorbox}
		$$ \lim_{m\rightarrow 0} m^{-1}V^\theta(m)=\frac{d}{dm}{\cal F}_\theta(0)=\max_X\theta \ ,   $$
	\end{tcolorbox}
	as expected.

\end{example}
\begin{example}\label{exAlambda}{\em The marginal case of two agents under saturation}\\
	Assume  $N=2$ and  $m_1+m_2=\mu(X)$.
	Using the notation of  Example \ref{singajex} we consider (Figure \ref{theta2_p})
	\be\label{com1} A_p^{\theta_1-\theta_2}:= \left\{x\in X; \ \ \theta_1(x)\geq \theta_2(x)+p \right\} \ \ .  \ee
	\begin{figure}
		\centering
		\includegraphics[height=8.cm, width=12.cm]{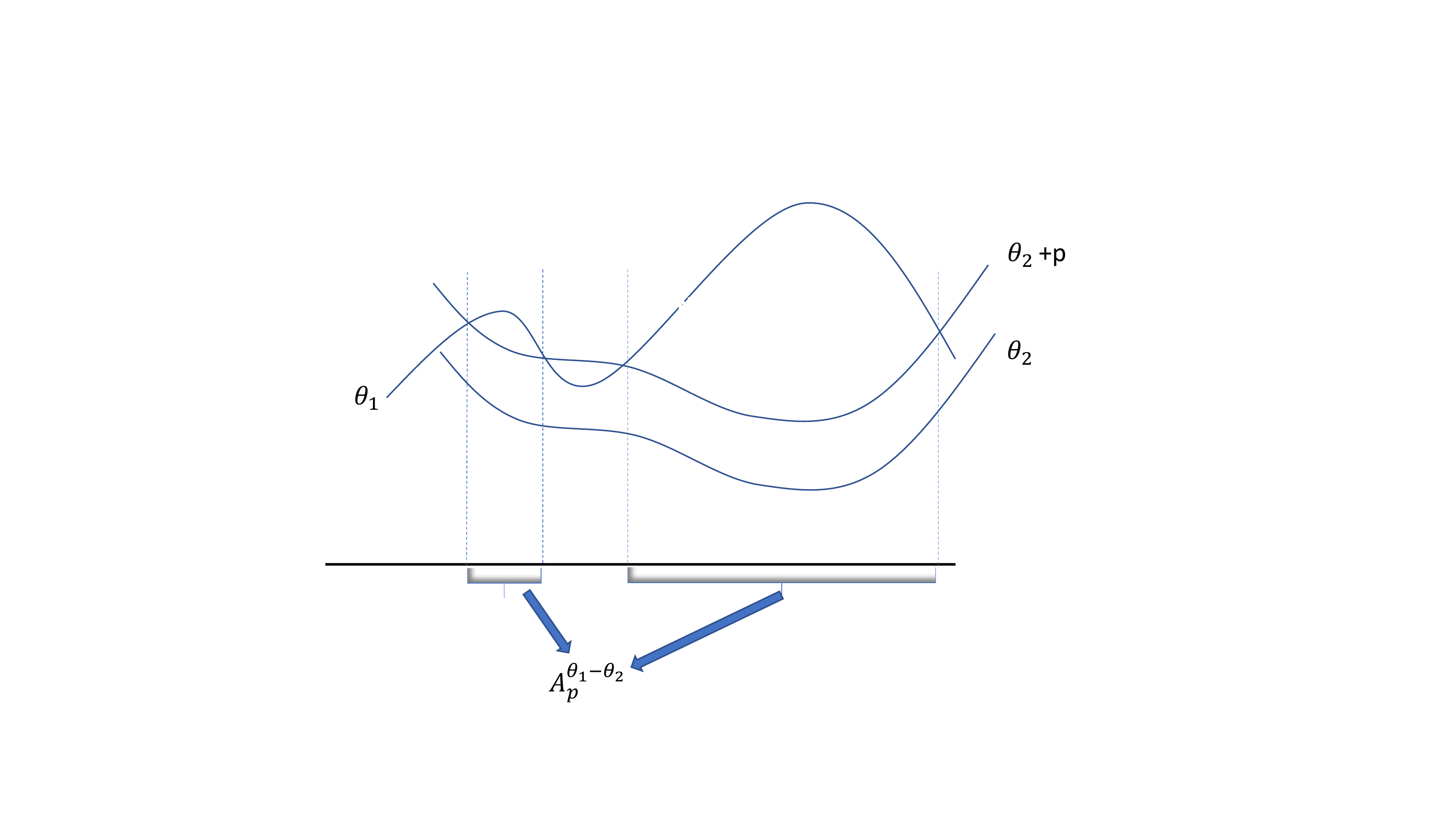}\\
		\caption{Two agents in saturation}\label{theta2_p}
	\end{figure}
	The complement of this set is, evidently, $A_{-p}^{\theta_2-\theta_1}$. Since
	$A^+_0=\emptyset$ in the saturated case, we obtain by Theorem \ref{old} (a) that the equilibrium price is determined by any $(p_1, p_2)$ such that $p=p_2-p_1$ verifies $\mu\left(A_p^{\theta_1-\theta_2}\right)=m_1$. Since $m_2=\mu(X)-m_1$, it implies that  $\mu\left(A_{-p}^{\theta_2-\theta_1}\right)=m_2$ as well.
	
	However, the i.v is {\em not} given by (\ref{P=F}) as in Example \ref{singajex}. In particular, in the limit $m_1\rightarrow 0$
	\begin{tcolorbox}
		$$ \lim_{m_1\rightarrow 0}\frac{V_1(m_1)}{m_1}\in Conv(\{\theta_1(x_1)\})$$
		where $\{x_1\}$ is the set of maximizers of $\theta_1-\theta_2$ and $Conv(\cdot)$ is the convex hull \index{convex hull}  of this set in $\R^1$.
	\end{tcolorbox}
\end{example}

\begin{example}\label{exlp}
	Suppose $\theta$  is a non-negative, continuous function on $X$ verifying $\mu(x; \theta(x)=r)=0$ for any $r\geq 0$. Let $\lambda_N>\lambda_{N-1}>\ldots \lambda_1>0$ be constants. We assume that $\theta_i:=\lambda_i\theta$ where $\mu(x; \theta(x)=r)=0$ $\forall r\in \R$ (in particular $\vtheta=(\theta_1, \ldots \theta_N)$ verifies assumption  \ref{mainass4}).
	Let $\vm$ in th unit simplex $\Delta^N(1)$.
	
	From (\ref{A+0}, \ref{A+i}) we obtain 
	$$ A_0^+(\vpp)\equiv \{ x; \theta(x)<\min_{i\in\I}\lambda_i^{-1}p_i\} $$
	$$A_i^+(\vpp)\equiv \left\{ x; \min_{j>i}\frac{p_j-p_i}{\lambda_j-\lambda_i}> \theta(x) > \max_{j<i}\frac{p_j-p_i}{\lambda_j-\lambda_i}\right\}-A_0^+(\vpp) \ . $$
	In particular, the partitions $A_i^+(\vpp)$  consist of unions of level sets of the function $\theta$.
	\par
	At optimal partition we observe that the i.v of the "top agent" $N$ is just $\lambda_N$ times the i.v of a {\em single agent} whose utility function is $\theta$ and capacity $m_N$, i.e, by (\ref{P=F}) \index{optimal partition}
	$$V_N(m_N)=\lambda_N{\cal F}_\theta(m_N) \  $$
	where (recall (\ref{pduality})) \index{Kantorovich duality}
	$$ A_N(\vM):=  \left\{x\in X; \ \ \theta(x)\geq p_\theta(m_N) \right\} \  $$
	is the level set of agent $N$.   
	For any $1\leq i < N-1$ let ${\cal M}_i:=\sum_{k=i}^Nm_k$.  We obtain
	\be\label{ivfct} V_i(\vM)= \lambda_i\left( {\cal F}_\theta({\cal M}_i)-{\cal F}_\theta({\cal M}_{i+1})\right) \ . \ee
	where
	\be\label{lambdabar}A_i(\vM):= \left\{x\in X; \ \ p_\theta({\cal M}_{i+1})\geq \theta(x)\geq p_\theta({\cal M}_i) \right\} \  ,  \ee
	is the level set of agent $i$.

	The sum of i.v is, thus, 
	\be\label{unidi}{\Sigma^\theta}(\vM)=\sum_{i\in\I} \lambda_i\int_{A_i}\theta d\mu\equiv \sum_{i\in\I} \lambda_i({\cal F}_\theta({\cal M}_{i})-{\cal F}_\theta\left({\cal M}_{i+1})\right) \ , \ee
	where ${\cal M}_{N+1}:= 0$. 
\end{example}
\section{Will wiser experts always get higher values?}
Suppose that, after some education and training,  one of the agents (say  no.$1$) improves her skill so the utility function she produces for her customers increases from $\theta_1$ to $\tilde{\theta}_1\geq \theta_1$ on $X$. Assuming that the utilities of all  other agents are unchanged, what is the impact of this change on the i.v of agent $1$?
\par

For example, consider a system of two experts in saturation and that there is no change in the other parameters of the problem (namely $m_1, m_2, \theta_2$).
\par
We expect that the i.v of  the first expert $\tilde{V}_1(\vM):=V^{\tilde{\theta}_1,\theta_2}_1(\vM)$ will increase under this change.   Is it so, indeed?
\par
Well, not necessarily! Suppose $m_1<<\mu(X)$ and let $x_1$ be a unique maximizer of $\theta_1-\theta_2$. By Example \ref{exAlambda}, $V_1 \approx m_1\theta_1(x_1)$. Let now $\tilde{x}_1$ be a unique  maximizer of $\tilde{\theta}_1-\theta_2$. So
$\tilde{V}_1\approx m_1\tilde{\theta}_1(\tilde{x}_1)$.  But it may happen  that  $\tilde{\theta}_1(\tilde{x}_1)< \theta_1(x_1)$, even though $\tilde{\theta}_1(x)>\theta_1(x)$ for any $x\in X$ !
(Fig \ref{deltatheta12}).
\begin{figure}
	\centering
	\includegraphics[height=8.cm, width=12.cm]{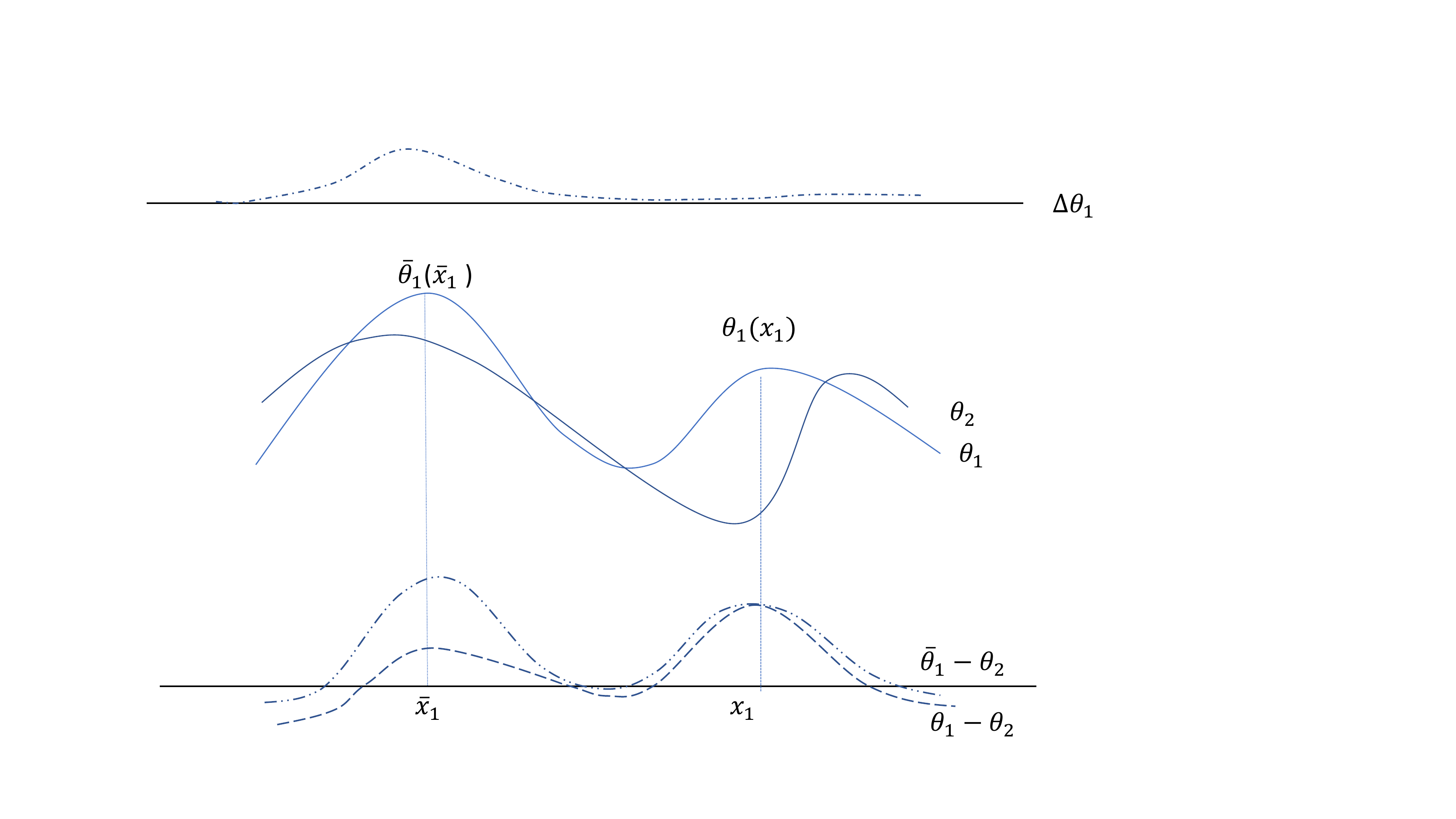}\\
	\caption{$\Delta\theta_1=\bar{\theta}_1-\theta_1 $.   Increasing $\theta_1$ implies decreasing $V_1$. }\label{deltatheta12}
\end{figure}
\par
Definitely, there are cases for which  an increase in the utility of a given expert will increase its i.v, independently of his own capacity, as well as  the utilities and capacities  of the other experts. In particular,  we can think about two cases where the above argument fails:
\begin{description}
	\item{Case 1:} \ $\tilde{\theta}_1=\theta_1+\lambda$ where $\lambda>0$ is a constant.
	\item{Case 2:} \   $\tilde{\theta}_1=\beta \theta_1$ where $\beta>2$ is a constant.
\end{description}
In the first case the "gaps" $\theta_1-\theta_2$ and $\tilde{\theta}_1-\theta_2$ preserves their order, so if $x_1$ is a maximizer of the first, it is also a maximizer of the second. In particular  the optimal partition is unchanged, and we can even predict that $\tilde{V}_1=V_1+\lambda m_1>V_1$ (c.f Theorem \ref{new0} below). \index{optimal partition}
\par
In the second case the order of gaps may change. It is certainly possible that $\tilde{\theta}_1(\tilde{x}_1)-\theta_2(\tilde{x}_1)> \tilde{\theta}_1(x_1)-\theta_2(x_1)$ (where $x_1$, $\tilde{x}_1$ as above), but, if this is the case, an elementary calculation yields $\tilde{\theta}_1(\tilde{x}_1)>
\theta_1(x_1)$, so the above argument fails. Indeed, if we assume both $\beta\theta_1(\tilde{x}_1)-\theta_2(\tilde{x}_1)> \beta\theta_1(x_1)-\theta_2(x_1)$ and $\beta\theta_1(\tilde{x}_1)<
\theta_1(x_1)$, then (since $\beta\geq 2$),
$\theta_1(x_1)-\theta_2(x_1)< -\theta_2(\tilde{x}_1)<\theta_1(\tilde{x}_1)-\theta_2(\tilde{x}_1)$  so  $x_1$ cannot be the maximizer of $\theta_1-\theta_2$ as assumed.
\par
In fact, we can get the same  result if either $\tilde{\theta}_1\geq 2\theta_1$ or if $\tilde{\theta}_1= \beta\theta_1$ and $\beta\geq 1$ (but, remarkably, not in the case $\tilde{\theta}_1\geq \beta\theta_1$ where $\beta<2$ !). This follows from the following results:
\begin{theorem}\label{new0}\cite{wol1}
	Let $\vtheta:=(\theta_1, \ldots , \theta_N)$ and $\tilde{\vtheta}:=(\tilde{\theta}_1, \theta_2,\ldots , \theta_N)$. Assume both $\vtheta, \tilde{\vtheta}$ verify
	Assumption  \ref{mainass4}. Let $\vM\in\R_+^N$, $V_1$ the i.v of agent 1 corresponding to $\vtheta$ and the capacity  $\vM$, and $\tilde{V}_1$ the same  corresponding to $\tilde{\vtheta}$ and the same capacity $\vM$.
	\begin{description}
		\item{i)} If $\tilde{\theta}_1=\beta\theta_1$ for a constant $\beta>0$ then $\tilde{V}_1\geq\beta{V}_1$ if $\beta>1$, $\tilde{V}_1\leq\beta V_1$ if $\beta<1$.
		\item{ii)} If $\vM$ is either saturated or under saturated, \index{under saturated (US)} and   $\tilde{\theta}_1=\theta_1+\lambda$ for a constant $\lambda>0$ then $\tilde{V}_1=V_1+\lambda m_1$.
	\end{description}
\end{theorem}
In Theorem \ref{new} we expand on case (i) of Theorem \ref{new0} and obtain the somewhat surprising result:

\begin{theorem} \label{new}\
	Under the same conditions as Theorem \ref{new0}
	\begin{description}
		\item{i)}
		Suppose  $\tilde{\theta}_1\geq \beta\theta_1$ where $\beta >1$ is a  constant.  Then
		\be\label{ini}\tilde{V}_1\geq (\beta-1)V_1 \ . \ee
		\item{ii)} For any $\beta>2$, $s>\beta-1$  there exists such a system ($\vtheta,\vM$) and $(\tilde{\vtheta},\vM)$, where $\vM$ is a saturation vector,  such that $\tilde{\theta}_1\geq \beta\theta_1$, $\tilde{\theta}_i=\theta_i$ for $i\not=1$, both $\vtheta$, $\tilde{\vtheta}$   verify
		Assumption  \ref{mainass4}, and
		$$\tilde{V}_1<sV_1 \ . $$
		In particular, the inequality  (\ref{ini}) is sharp  in the case $\beta>2$.
	\end{description}
\end{theorem}
\begin{cor}
	The i.v of an agent cannot decrease if its  utility $\theta_i$ is replaced by $\tilde{\theta}_i\geq 2\theta_i$, without changing any of the capacities and the  utilities of other agents.
\end{cor}
In Theorem \ref{new1} we obtain sharp conditions for the {\it decrease} of i.v, given an increase of the corresponding utility: 
\begin{theorem}\label{new1}   Under the assumption of Theorem \ref{new0}, if $\vM$ is either  under saturated or saturated,
	\begin{description}
		\item{i)}If $1<\beta <2$, $\lambda\geq 0$ and
		\be\label{c2}\beta\theta_1(x)\leq \tilde{\theta}_1(x)\leq \beta\theta_1(x)+\lambda \ ,\ee
		then
		\be\label{barP1} \tilde{V}_1 \geq V_1 -\frac{m_1\lambda(2-\beta)}{\beta-1}\ . \ee
		\item{ii)} For any $1<\beta<2, \lambda>0, s<(2-\beta)/(\beta-1)$  there exists a system  $(\vtheta, \vM)$ and $(\tilde{\vtheta}, \vM)$ such that
		$\beta\theta_1\leq \tilde{\theta}_1\leq \beta\theta_1+\lambda$, $\tilde{\theta}_i=\theta_i$ for $i\not=1$, both $\vtheta, \tilde{\vtheta}$  verify Assumption \ref{mainass4} such that
		$$\tilde{V}_1<V_1-m_1\lambda s \   \ . $$
		In particular, the inequality  (\ref{barP1}) is sharp  in the case $1<\beta <2$.
	\end{description}
\end{theorem}
\subsection{Proofs}
The key Lemma is an adaptation of Lemma \ref{difcor}:
\begin{lemma}\label{lemc2}
	Let $a>0$ and $\vtheta:=\vtheta(x,t):X\times [0,a]\rightarrow \R^N_+$  for any $t\in [0,a]$.
	Let $\vtheta$  and verifies Assumption \ref{mainass4} for  $t=0$ and $t=a$.
	Assume further that each component $t \mapsto\theta_i(x,t)$ is convex and differentiable on $[0,a]$ for any $x\in\R^N$ and $$\partial_t\theta_i := \dot{\theta_i}\in \mathbb{L}^\infty(X\times [0,a])$$ for any $i\in\I$. Then the function $(\vpp,t) \mapsto \Xi_{\vtheta(\cdot, t)}(\vpp)$  (\ref{Xinoplus}) is convex  on $\R^N\times [0,a]$, and, {\em if} its $t$ derivative  $\dot{\Xi}_{\vtheta(\cdot,t)}(p)$ exists at $( \vpp,t)$ then
	\be\label{17}
	\dot{\Xi}_{\vtheta(\cdot,t)}(p) = \sum_{i\in\I}\int_{A_i(\vpp,t)}\dot{\theta}_i(x,t) d\mu \ . \ee
	Here
	\be\label{Aipt}A_i(\vpp,t):= \{x\in X; \ \theta_i(x,t)-p_i>\theta_j(x,t)-p_j \ \ \forall j\not= i\} \ . \ee
	The same holds if we replace $\Xepto$ by $\Xept$  (\ref{Xiwithplus}) and (\ref{Aipt}) by
	\be\label{Aipt+}A^+_i(\vpp,t):= \{x\in X; \ \theta_i(x,t)-p_i>[\theta_j(x,t)-p_j]_+ \ \ \forall j\not= i\} \ . \ee
\end{lemma}

\begin{proof} The proof  follows as in  Lemma \ref{difcor}. Here we define
	$$\xi:X\times\R^N\times[0,a]\rightarrow \R \ , \ \ \ \xi(x,\vpp,t)=\max_{i\in\I}[\theta_i(x,t)-p_i]$$ and  $\Xepto(\vpp,t):=\int_X\xi(x,\vpp,t)\mu(dx)$. Again, $\xi$ is convex on $\R^N\times[0,a]$ for any $x\in X$ so $\Xepto$ is convex on $\R^N\times[0,a]$ as well, while
	$$ \dot{\xi}=\left\{ \begin{array}{cc}
	\dot{\theta}(x, t) \  \  \ if  & x\in A_i(\vpp,t) \\
	0  \ \  \ if& \exists j\not= i, \ x\in A_j(\vpp,t)
	\end{array}\right.  \  \ \  $$
	implies (\ref{17}).
\end{proof}

\begin{proof} of Theorem \ref{new0}
	\begin{description}
		\item{i)} Let $\vt:=(t_1, \ldots t_N)\in \R^N$. Let
		$\Theta(x,t):=(t_1\theta_1(x), \ldots t_N\theta_N(x))$.
		Consider
		\be\label{Xitm}\Xi(\vpp, \vt):=\Xi^{ \Theta} \ . \ee
		\par
		By Lemma \ref{lemc2}, $(\vpp, \vt) \mapsto\Xi^\Theta$ is mutually convex on $\R^N\times \R^N$, and
		\be\label{pvt} \partial_{t_i}{\Xi^\Theta}(\vpp,\vt)=\int_{A_i(\vpp, \vt)}\theta_i d\mu\equiv t_i^{-1}V_i(\vt)\ee
		where
		\be\label{pvt1}A_i(\vpp,\vt):= \{x\in X; \ t_i\theta_i(x)-p_i<t_j\theta_j(x)-p_j \ \ \forall j\not= 1\} \ ,\ee
		whenever $ \partial_{t_i}{\Xi}^\Theta$ exists. It follows that
		both
		\be\label{XiSigma}\Sigma(\vM, \vt):= \min_{\vpp\in\R^N} \Xi^\Theta(\vpp, \vt)+\vM\cdot \vpp\ee in the US,S case, or \index{over saturated (OS)} \index{under saturated (US)}
		$$\Sigma(\vM,\vt):= \max_{\vM\leq \vM}\min_{\vpp\in\R^N} \Xi^\Theta(\vpp, \vt)+\vM\cdot \vpp$$
		in the OS case are convex with respect to $\vt$ as well. Then \index{over saturated (OS)}
		$$ \partial_{t_i}{\Sigma}=\int_{A_i(\vpp^0, \vt)}\theta_i d\mu\equiv t_i^{-1}V_i(\vt)$$
		holds as well, where $\vpp^0:=\vpp^0(\vM,\vt)$ is the unique equilibrium price vector (perhaps up to an additive constant) guaranteed by Theorem \ref{old} for the utility vector $\Theta$.
		Hence, for $\vt_{(\beta)}:=(\beta, 1;\ldots 1)$ we obtain
		$$  V_1(\vt_{(\beta)})/\beta\equiv \partial_{\beta}{\Sigma}(\vM, \vt_{(\beta)})\geq \partial_{t_1}{\Sigma}(\vM,\vec{1})\equiv V_1(\vt_{(1)}) \ , $$
		where $V_1(\vt_{(1)})\equiv V_1$ and $V_1(\vt_{(\beta)})\equiv \tilde{V}_1$ by (\ref{pvt}).
		\item{ii)}
		If $\theta_1\rightarrow \theta_1+\lambda$ then the optimal partition in the S, US cases is unchanged.   \index{over saturated (OS)} \index{under saturated (US)}Then \index{optimal partition}
		$$\tilde{V}_1:=\int_{A_1}(\theta_1+\lambda)d\mu = \int_{A_1}\theta_1d\mu+ \lambda\int_{A_1}d\mu=V_1+\lambda m_1 \ . $$
	\end{description}
\end{proof}

\begin{proof}  of Theorem \ref{new}
	\begin{description}
		\item{i)} Let $\sigma:=\tilde{\theta}_1-\beta\theta_1\geq 0$, $\alpha:=\beta-1\geq 0$. Let a function  $\phi:[0,1] \mapsto\R$ satisfying
		\be\label{psi2}\phi(0)=\dot{\phi}(0)=0 \ \text{and}   \ \ddot{\phi}\geq 0 \ \ \text{for any} \  t\in [0,1] \ \ , \ \phi(1)=1  \ .  \ee
		Define
		\be\label{psi1}\theta(x,t):= (1+\alpha t)\theta_1(x) + \sigma(x)\phi(t) \ . \ee
		So
		\be\label{1eqtilde}\theta(x,1)=\tilde{\theta}_1(x) \  \ee
		and  $\theta$ is convex in $t\in[0,1]$ for any $x$. Also
		$\dot{\theta}(x,t)=\alpha\theta_1(x)+\sigma(x)\dot{\phi}(t)$. Let now $\delta>0$. Then
		\be\label{newpsi}\theta(x,1)\geq \dot{\theta}(x,1)-\delta\|\sigma\|_\infty\ee
		provided
		\be\label{z1} \sigma(x)\dot{\phi}(1)\leq \sigma(x)+\theta_1(x)+\delta\|\sigma\|_\infty \ . \ee
		Since $\theta_1$ and $\sigma$ are non-negative, the later is guaranteed if $\dot{\phi}(1)\leq 1+\delta$. So, we choose
		$\phi(t):= t^{1+\eps}$ for some $\eps\in(0, \delta]$. This meets (\ref{psi2},\ref{z1}).

		Let now
		$$\Sigma(\vM,t):=\inf_{\vpp\in\R^N}\Xi^\Theta(\vpp,t) +\vpp\cdot\vM$$
		where $\Theta(x,t):= (\theta(x,t), \theta_2(x), \ldots \theta_N(x))$.
		By Lemma \ref{lemc2}, $(\vpp, t)\mapsto \Xi^\Theta(\vpp,t)$ is convex. So  $\Sigma$ is convex in $t$ for a fixed $\vM$.
		In the OS case \index{over saturated (OS)}
		$$\Sigma(\vM,t):=\sup_{\vm\leq \vM}\inf_{\vpp\in\R^N}\Xi^\Theta(\vpp,t) +\vpp\cdot\vm$$
		is convex (as maximum of convex functions) as well. By the same Lemma
		\be\label{verified} \dot{\Sigma}(\vM, 0)=\int_{A_1(0)}\dot{\theta}(,0)d\mu=\alpha\int_{A_1(0)}\theta_1d\mu\equiv \alpha V_1\ee
		where $A_1(0)$ is the first component in the optimal partition associated with $\vtheta$, while, at $t=1$ we obtain from convexity and (\ref{newpsi})
		\begin{multline}\label{alpha} \dot{\Sigma}(\vM, 1)=\int_{A_1(1)} \dot{\theta}(x,1)d\mu \leq \int_{A_1(1)}( \theta(x,1)+\delta\|\sigma\|_\infty)d\mu \\  \leq \int_{A_1(1)} \theta(x,1)d\mu +\delta\mu(X)\|\sigma\|_\infty \ \end{multline}
		where $A_1(1)$ is the first component in the optimal partition associated with  $\Theta(1,t)$.   \index{optimal partition}Since $\tau \mapsto \phi(\tau)$ is convex, $\tau\mapsto \Sigma(\vM,\tau)$ is convex as well by Lemma \ref{lemc2} and we get
		\be\label{beta} \dot{\Sigma}(\vM,1)\geq  \dot{\Sigma}(\vM,0) \ . \ee
		From (\ref{alpha}, \ref{beta})
		$$ \int_{A_1(1)} \theta(x,1)d\mu \geq \alpha V_1-\delta\mu(X) \|\sigma\|_\infty \ . $$
		
		Now, recall  $\beta:=1+\alpha$ and  $\theta(x,1):=\tilde{\theta}_1$ by (\ref{1eqtilde}), so $\int_{A_1(1)} \theta(x,1)d\mu\equiv \tilde{V}_1$.  Since $\delta>0$ is arbitrary small, we obtain the  result.
		\item{ii)}
		Assume $N=2$, $m_1+m_2=\mu(X)$. We show the existence of non-negative, continuous  $\theta_1, \theta_2$, $x_1, x_2\in X$ and $\lambda>0$    such that, for given $\delta>0$
		\begin{description}
			\item{a)} $\Delta(x):= \theta_1(x)-\theta_2(x)< \Delta(x_1)$ for any $x\in X-\{x_1\}$.
			\item{b)}  $\Delta_\beta(x):= \beta\theta_1(x)-\theta_2(x)< \Delta_\beta(x_1)$ for any $x\in X-\{x_1\}$.
			\item{c)}   $\Delta_\beta(x_2)+\lambda =\Delta_\beta(x_1)+\delta$.
		\end{description}
		We show that  (a-c) is consistent with
		\be\label{s<}s\theta_1(x_1)>\beta\theta_1(x_2)+\lambda\ee
		for given $s>\beta-1$. \par
		Suppose (\ref{s<}) is  verified.  Let
		\be\label{theta0}\theta_0:= \left\{\begin{array}{cc}
			1-\frac{|x-x_2|}{\eps} & \text{if} \ |x-x_2|\leq \eps \\
			0 & \text{if} \ |x-x_2|> \eps
		\end{array}\right.\ee
		(assuming, for simplicity, that $X$ is a real interval).
		Set $\tilde{\theta}_1:=\beta\theta_1+\lambda\theta_0$.  If $\eps$ is small enough then
		$\tilde{\theta}_1-\theta_2$ is maximized at $x_2$ by (b,c), while $\theta_1-\theta_2$ is maximized at $x_1$ by (a).
		Letting  $M_1<<1$ we find, by Example \ref{exAlambda}, that
		$V_1\approx M_1\theta_1(x_1)$ and $\tilde{V_1}\approx M_1(\beta\theta_1(x_2)+\lambda\theta_0(x_2))= M_1(\beta\theta_1(x_2)+\lambda)$.
		By (\ref{s<}) we obtain the result.
		\par
		So, we have only to prove that (\ref{s<}) is consistent with (a-c). We rewrite it as
		$$\frac{s}{\beta-1} \left[ \Delta_\beta(x_1)-\Delta(x_1)\right] > \frac{\beta}{\beta-1} \left[ \Delta_\beta(x_2)-\Delta(x_2)\right]+\lambda \ . $$
		From  (c) we obtain
		$$\frac{s}{\beta-1} \left[ \Delta_\beta(x_1)-\Delta(x_1)\right] > \frac{\beta}{\beta-1} \left[ \Delta_\beta(x_2)-\Delta(x_2)\right]+\Delta_\beta(x_1)-\Delta_\beta(x_2)+\delta \ , $$
		that is
		\be\label{cucu} (s-\beta+1)\Delta_\beta(x_1)-\Delta_\beta(x_2)> (s-\beta)\Delta(x_2)+s(\Delta(x_1)-\Delta(x_2))+\delta(\beta-1) \ . \ee
		We now set $\Delta_\beta(x_1)$ and $\lambda$ large enough, keeping $\delta, \Delta_\beta(x_2),\Delta(x_1), \Delta(x_2)$ fixed. Evidently, we can do it such that (c) is preserved.
		Since $s-\beta+1>0$ by assumption, we can get  (\ref{cucu}).
	\end{description}
\end{proof}

\begin{proof} of Theorem \ref{new1}. \\
	\begin{description}
		\item{i)} Let $\beta=1+t$ where $t\in(0,1)$.
		We change (\ref{psi1}) into   \be\label{psi11}\theta(x,t):= (1+ t)(\theta_1(x)+\gamma) + \sigma(x)\phi(t)\ee
		and
		\be\label{magilush} \tilde{\theta}_1(x):= (1+t)\theta_1(x)+\sigma(x)\phi(t) \  \ee
		where $\gamma>0$ is a constant and $\sigma\geq 0$  on $X$.   Then
		$\dot{\theta}(x,t)=\theta_1(x)+\gamma+\sigma(x)\dot{\phi}(t)$, and we obtain
		\be\label{yomtov}\theta(x,t)\geq \dot{\theta}(x,t), \ \ t>0; \ \ \dot{\theta}(x,0)=\theta_1(x)+\gamma  \ee
		provided
		\be\label{z11} \sigma(x)\dot{\phi}(t)\leq \sigma(x)\phi(t)+t(\theta_1(x)+\gamma) \ ; \ \dot{\phi}(0)=0 \ .  \ee
		Since $\theta_1,\sigma$ are non-negative, the later is guaranteed if
		\be\label{z2} \dot{\phi}(t)\leq \phi(t)+\frac{t\gamma}{\|\sigma\|_\infty} \ ; \ \dot{\phi}(0)=0 \ .  \ee
		Since $t<1$ (by assumption $\beta:=1+t<2$), the choice $\phi(\tau):=\tau^{1+\eps}$ for $0\leq \tau\leq t$ and $\eps>0$ small enough (depending on $t$) verifies (\ref{z2}) provided
		\be\label{uuu}\|\sigma\|_\infty < \gamma t/(1-t) \ . \ee
		\par
		
		Hence we can let $\sigma$ to be any function verifying   (\ref{uuu}). Then  (\ref{psi11}, \ref{magilush}) imply
		\be\label{zirgug}(1+ t)\theta_1(x) \leq \tilde{\theta}_1(x)\leq (1+ t)\theta_1(x)+\frac{\gamma t^{2+\eps}}{1-t} \ . \ee
		Now, we note from the second part of (\ref{yomtov}) that
		\be\label{verified1} \dot{\Sigma}(\vM, 0)=\int_{A_1(0)}\dot{\theta}(,0)d\mu=\int_{A_1(0)}(\theta_1+\gamma)d\mu\equiv V_1+\gamma m_1 \  \ee
		since $A_1(0)$ is independent of $\gamma$ in the S, US cases.  \index{under saturated (US)} In addition,  (\ref{psi11}, \ref{magilush},\ref{z2}) imply
		$$ \dot{\Sigma}(\vM, t)=\int_{A_1(t)}\dot{\theta}(,t)d\mu\leq \int_{A_1(t)}\theta(\cdot,t)d\mu= \int_{A_1(t)}(\tilde{\theta}_1+(1+t)\gamma)d\mu$$
		$$ \equiv \tilde{V}_1+(1+t)\gamma m_1 \ , $$
		where $A_1(t)$ is the first component in the optimal partition associated with  $\Theta$.  Since $\tau \mapsto \phi(\tau)$ is convex, $\tau\mapsto \Sigma(\vM,\tau)$ is convex as well by Lemma \ref{lemc2} and we get, as in (\ref{beta}) \index{optimal partition}
		\be\label{gamma} \dot{\Sigma}(\vM,t)\geq  \dot{\Sigma}(\vM,0) \ . \ee
		where, again, we used that $A_1(t)$ is independent of $\gamma$ and $t>0$.
		Recalling $\beta:=1+ t$, let  $\lambda:=\gamma(\beta-1)^2/(2-\beta)$ and $\eps$ small enough we get (\ref{c2}, \ref{barP1}), using (\ref{zirgug},\ref{verified1}, \ref{gamma}).
		\par
		\item{ii)}
		Assume $N=2$, $m_1+m_2=1$, that $\theta_1-\theta_2$ attains its maximum at $x_1$, and $x_2\not=x_1$. Let  $\tilde{\theta}_1:=\beta\theta_1+\lambda \theta_0$ where $\theta_0$ as defined in (\ref{theta0}). We assume, as in part (ii) of the proof of Theorem \ref{new}, that $x_1$ is a maximizer of $\beta\theta_1-\theta_2$ as well.

		Next, assume
		\be\label{this1}\beta\theta_1(x_1)-\theta_2(x_1)< \lambda+\beta\theta_1(x_2)-\theta_2(x_2) \  \ee
		which implies, in particular, that $x_2$ is the maximizer of $\tilde{\theta}_1-\theta_2$ (see part (ii) of the proof of Theorem \ref{new}).
		If, in addition,
		\be\label{contra1}\theta_1(x_1)-\beta\theta_1(x_2)-\lambda-s >0\ ,  \ee
		then,
		from Example \ref{exAlambda},  we obtain the proof for small $m_1$ and
		\be\label{www1}V_1\approx \theta_1(x_1)m_1 >m_1(\tilde{\theta}_1(x_2) +s)  \approx \tilde{V}_1+sm_1
		\ee
		\par
		From (\ref{this1}) and since $x_1$ is a maximizer of $\theta_1-\theta_2$:
		$$ \lambda>(\beta-1)(\theta_1(x_1)-\theta_1(x_2))$$
		so  (\ref{contra1}) and (\ref{this1}) are compatible provided
		$$\lambda>(\beta-1)^2\theta_1(x_2)+(\beta-1)\left[\lambda+s\right] \ , $$
		namely
		\be\label{00}\lambda \frac{2-\beta}{\beta-1}>(\beta-1)\theta_1(x_2)+s \ . \ee
		Thus, if we assume further that, say,   $\theta_1(x_2)=0$ (which is consistent with the assumption that $\theta_1, \theta_2\geq 0$) then (\ref{00}) is verified for $s<\lambda(2-\beta)/(\beta-1)$.
	\end{description}
\end{proof}

\chapter{Sharing the individual value}\label{chinprofit}
\noindent\index{individual (surplus) value}
{\small {\it Share it fairly but don't take a slice of  my pie} (Pink-Floyd)} \vskip.3in

The i.v of an agent is the {\it surplus}  she produces for her clients. The question we are going to address  is
\begin{tcolorbox}
	How an agent shares her i.v with her clients?
\end{tcolorbox}
We already now that, under a prescribed  capacity vector $\vM$, \index{capacity vector} the price that agent $i$ charges for her service is determined by $p_i$. Recall
\be\label{newXi+} \Xept(\vpp):= \int_X\max_{i\in\I}(\theta_i(x)-p_i)_+d\mu \ ; \ \ \Sigma^\theta(\vM)=\min_{\vpp\in\R^I} \Xept(\vpp) + \vpp\cdot \vM \ .  \ee
The relation between the capacity and price is given by
\be\label{dualrel} p_i=\frac{\partial \Sigma^\theta}{\partial m_i} \ \ , \ \ m_i=-\frac{\partial\Xept}{\partial p_i} \ , \ee
provided $\Xept$ and $\Sigma^\theta$ are differentiable.
\vskip .3in
\begin{tcolorbox}
	The profit ${\cal P}_i$ of agent $i$ fixing a price $p_i$ is just $p_i m_i$. The residual profit of her consumers is ${\cal C}_i:=V_i-{\cal P}_i$, where $V_i$ is the individual value.\index{individual (surplus) value}
\end{tcolorbox}

Using the duality relation (\ref{dualrel}) we can determine the profit of the agent  in terms of either the  prices $\vpp$ charged by the group of agents  or in terms of  the capacity vector $\vM$:\index{capacity vector}
\be\label{parprofit}{\cal P}_i(\vpp):= -p_i\frac{\partial\Xept}{\partial p_i} \ ; \ \ {\cal P}_i(\vM)=m_i\frac{\partial\Sigma^\theta}{\partial m_i} \ ,\ee
and we use ${\cal P}_i$ for both representations, whenever no confusion is expected.
\par
There is, however, another possibility:  In addition to (or instead of)  the fixed, flat price $p_i$ of her service the agent may charge  a  {\em commission}. \index{commission} This commission is a certain proportion, say $q_i\in[0,1)$, of the {\em gross} profit $\theta_i(x)$ she makes for  consumer $x$. In that case, the   profit of an agent $i$ out of a {\em single consumer} $x$ is just $p_i+q_i\theta_i(x)$, while the {\em net} profit of this consumer is $(1-q_i)\theta_i(x)-p_i$.

Given a price vector $\vpp=(p_1, \ldots p_N)\in \R_+^N$ and a commission vector $\vq=(q_1, \ldots q_N)\in [0,1)^N$,
the part  of the population  not attending {\em any} agent is
$$A^{\theta}_0(\vpp, \vq):=\{x\in X; \max_{j}(1-q_j)\theta_j(x)-p_j\leq 0 \} \ . $$
The population attending agent $i$ is, then
$$ A^{\theta,+}_i(\vpp, \vq):= A^{\theta}_i(\vpp, \vq)-A^{\theta}_0(\vpp, \vq)$$
where
\be\label{AsubiPQ}\ A^{\theta}_i(\vpp, \vq):=\{x\in X; (1-q_i)\theta_i(x)-p_i>\max_{j\not= i}(1-q_j)\theta_j(x)-p_j \} \ee
(compare to (\ref{AsubiP}, \ref{AsubiP+})).
\par
\vskip .3in 
\begin{tcolorbox}
	The profit ${\cal P}_i$ of agent $i$ fixing a price $p_i$ and commission \index{commission} $q_i$  is  $p_i m_i+q_iV_i$, where
	$$m_i(\vpp,\vq):= \mu\left(A^{\theta}_i(\vpp, \vq)\right), \ \ \ V_i(\vpp,\vq):= \int_{A^{\theta}_i(\vpp, \vq)}\theta_id\mu \ . $$
	The residual profit of her consumers is $${\cal C}_i:= (1-q_i)V_i-p_im_i \ . $$
\end{tcolorbox}
Can we express this profit in terms of "potential functions" as in (\ref{parprofit})? For this we generalize (\ref{newXi+}) into
$$ \Xept(\vpp,\vq):= \int_X\max_{i}((1-q_i)\theta_i(x)-p_i)_+d\mu$$
and the dual function
$$ \Sigma^\theta(\vM, \vq):= \inf_{\vpp\in \R^N} \Xept(\vpp,\vq)+ \vpp\cdot \vM \ . $$

The condition for differentiability of $\Xept$  and $\Sigma$ is the following  generalization of Assumption \ref{mainass4}
\begin{assumption}\label{mainass5} .
	For any $i,j\in\{1, \ldots N\}$ and any $r\in \R$, $\alpha>0$   $$\mu\left(x\in X \ ; \ \  \theta_i(x)-\alpha\theta_j(x)=r\right)=0 \ . $$
\end{assumption}
Under Assumption \ref{mainass5} we obtain that $\Xept$ is differentiable in {\em both} variables, provided    $\vq\in [0,1)^N$.
Recalling Corollary \ref{coruniquediff2} we obtain  that $\Sigma^\theta$ is also   differentiable
with respect to $\vM$ for fixed $\vq\in[0,1)^N$ for any under saturated $\vM$,\index{under saturated (US)} (and differentiable in the negative direction  for saturated $\vM$).\footnote{Recall, by Remark \ref{remmaxprice}, that $\vpp(\vM,\vq):= -\nabla^-_{\vM}\Sigma^\theta(\vM,\vq)$ is, in the saturated case, the maximal price vector charged by the agents.}  Moreover, it can be shown that $\Sigma^\theta$ is also differentiable with respect to $\vq$ for any $\vM$ in the simplex $\Delta^N$  (\ref{simplexM}), so    the i.v of agent $i$ is given by either $(\vpp, \vq)$ or $(\vM, \vq)$ representation as
$$ V_i(\vpp, \vq)=-\frac{\partial \Xept}{\partial q_i} \ \ ; \ \ V_i(\vM, \vq)=-\frac{\partial \Sigma^\theta}{\partial q_i} \ . $$
Thus, we obtain the profit of agent $i$  as a function of either $(\vpp, \vq)$ or $(\vM, \vq)$:
\be\label{parprofitM}{\cal P}_i(\vpp, \vq):= -\left(p_i\frac{\partial\Xept}{\partial p_i}+q_i\frac{\partial\Xept}{\partial q_i}\right) \ ; \ \ {\cal P}_i(\vM, \vq)=m_i\frac{\partial\Sigma^\theta}{\partial m_i}-q_i\frac{\partial\Sigma^\theta}{\partial q_i} \ .  \ee
Note that (\ref{parprofitM}) reduces to (\ref{parprofit}) if $\vec{q}=0$.
\section{Maximizing the agent's profit}\label{maxprofitagent}
It is, evidently, more natural for an agent to maximize her  profit rather than her individual value. Let us consider \index{individual (surplus) value} first the case of a single agent which does not collect a commission. \index{commission}If the utility   function for this agent is $\theta$, the flat price she collect is a maximizer of the function $p\rightarrow {\cal P}(p)$, where
$$ {\cal P}(p)=p\mu(x; \theta(x)\geq p) \ . $$
Note that ${\cal P}$ is non-negative for any $p\in \R$. Moreover, it is positive in the domain $0<p<\bar\theta:=\max\theta$. If (as we assume throughout this book) $\theta$ is a bounded continuous function and $X$ is compact then $\bar\theta$ is {\em always} obtained in $X$.   However, the maximizer {\em many not be unique}.
\begin{example}\label{exmaxprofitone} Let $(X,\mu)=([0,1], dx)$ and $\theta$ is a positive on $[0,1)$,  monotone decreasing, $\theta(1)=0$. For $p\in [0, \theta(0)]$ we get $m(p)=\theta^{-1}(p)$ so
	${\cal P}(p)=p\theta^{-1}(p)$.   Non uniqueness  of $\max{\cal P}(p)$ can be visualized easily. 
	(see Fig \ref{figrec}).
	
\end{example}
\begin{figure}
	\centering
	\includegraphics[height=8.cm, width=12.cm]{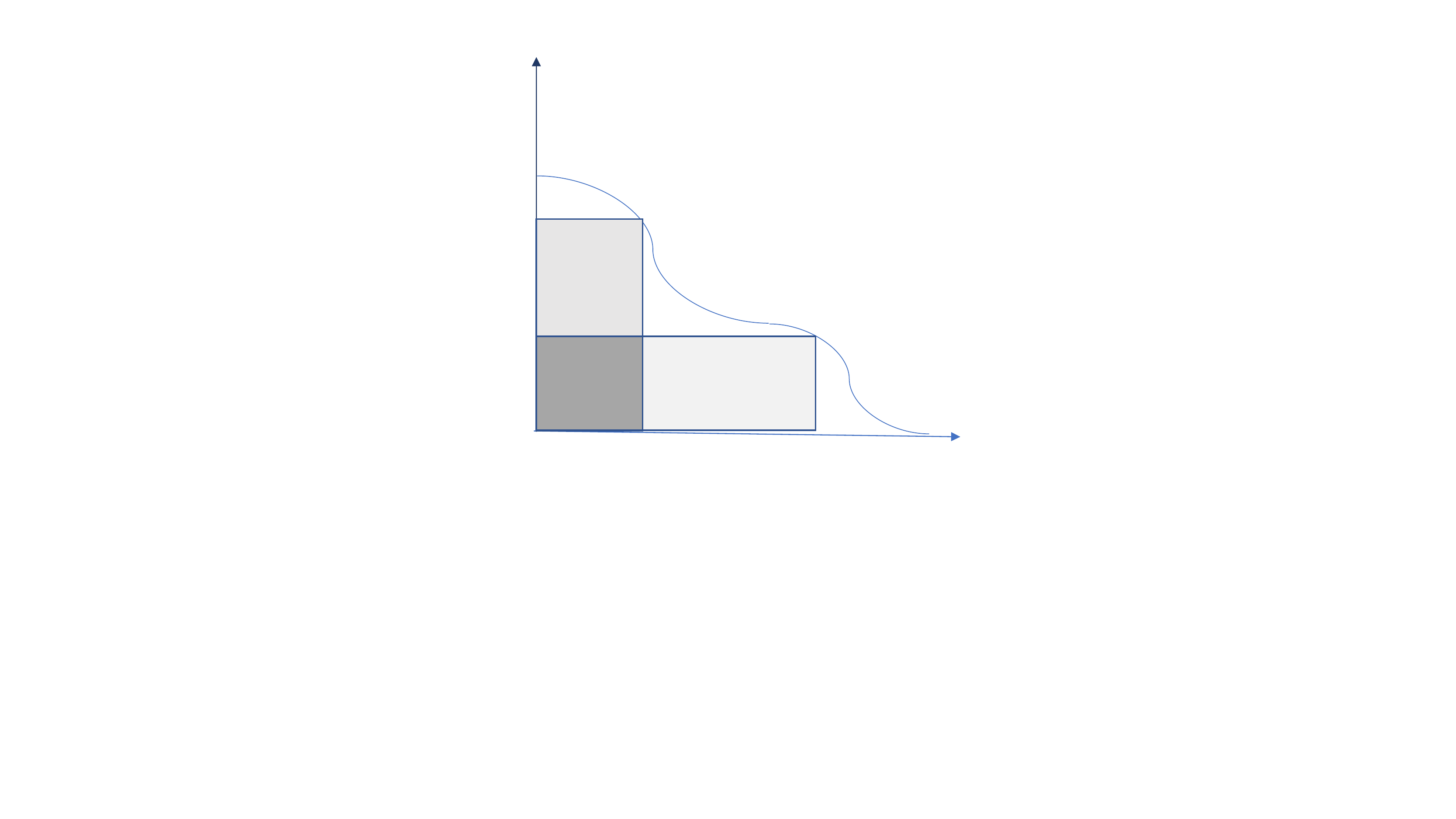}\\
	\caption{The two gray rectangles maximizes the area of all rectangles below the graph of $\theta$ in the positive quadrature whose edges are parallel to the axis. }
	\label{figrec}
\end{figure} 
\begin{tcolorbox}
	If we also allow  a commission $q$ then the situation is changed dramatically. Evidently, ${\cal P}(0,q)$ can approach the i.v $(=\int_0^1[\theta(x)]_+dx)$ arbitrary close as $q\uparrow 1$.
\end{tcolorbox}
\section{Several agents: Nash equilibrium}\index{Nash equilibrium}
The case of several agents is much more complicated. Let ${\cal P}_i={\cal P}_i(\vpp, \vq)$ the profit of the agents $i$ for given price-commission \index{commission} vectors $\vpp, \vq$. A natural definition of an equilibrium is the  {\em Nash equilibrium}. It is satisfied whenever each agent $i$ chooses his strategy (i.e his price-commission value $(p_i, q_i)$) to maximize his profit, {\em assuming that his choice does not affect  the choices of other agents}:
\begin{defi}
	The vectors $\vpp=(p_1, \ldots p_N)\in \R^N$,
	$\vq=(q_1, \ldots q_N)\in [0,1]^N$ are said to be in  {\em Nash equilibrium} if
	$$ {\cal P}_i(\vpp_{-i}, p^{'}_i;  \vq_{-i}, q^{'}_i)\leq {\cal P}_i(\vpp, \vq)$$
	for any $p^{'}_i\in\R$, $q^{'}_i\in [0,1)$ and $i\in\I$. Here $\vpp_{-i}$ is the vector $\vpp$ where the $i-$ coordinate omitted. Same for $\vq_{-i}$. \par
	If no commission is charged, the Nash equilibrium $\vpp$ is defined with respect to flat prices only:
	$$ {\cal P}_i(\vpp_{-i},  p^{'}_i)\leq {\cal P}_i(\vpp) $$
	where ${\cal P}(\vpp):= {\cal P}(\vpp, \vec{0})$.
\end{defi}
An equivalent definition can be given in terms of the dual variable $\vM$ (capacities) and $\vq$. In this sense, the agents may control their capacities (instead of the flat prices) and their commissions. \index{commission}Using this, we may assume the existence of {\em capacity constraints} $\vM\leq \vM_*$, and define the {\em constraint Nash equilibrium}
\begin{defi}
	The capacity vector  $\vM\leq \vM_*$ and commission vector \index{capacity vector}
	$\vec{q}$ are said to be in $\vM_*$-{\em conditioned Nash equilibrium} if
	$$ {\cal P}_i(\vM_{-i}; , \vq_{-i},  q^{'}_i)\leq {\cal P}_i(\vM, \vq)$$
	for any $m^{'}_i\leq m_{*,i}$, $q^{'}_i\in [0,1)$ and $i\in\{1,\ldots N\}$.\par
	Again,
	if no commission is charged, the Nash equilibrium \index{Nash equilibrium} $\vM\leq \vM_*$ conditioned on $\vM_*$ is defined with respect to the capacities which are determined by the flat prices:
	$$ {\cal P}_i(\vM_{-i}, m^{'}_i)\leq {\cal P}_i(\vM) \ \ ;  \ \forall m^{'}_i\leq M_{*,i}$$
	where ${\cal P}(\vM):= {\cal P}(\vM, \vec{0})$.
\end{defi}
If, in addition, the functions ${\cal P}_i$ are differentiable as well, then we obtain the necessary conditions for a Nash equilibrium:
\begin{prop}
	If $(\vpp_0, \vq_0)$ is a Nash equilibrium and  ${\cal P}_i$ are differentiable at $\vpp_0, \vq_0$  then
	$$\partial{\cal P}_i/\partial{p_i}=\partial {\cal P}_i/\partial{q_i}=0 \ \ \text{at} \ \  (\vpp_0, \vq_0) \ . $$
	If $(\vM_0, \vq_0)$ is a $\vM_*$-{\em conditioned Nash equilibrium} then
	$$\partial{\cal P}_i/\partial{m_i}\geq 0 \ \ ; \ \ \partial {\cal P}_i/\partial{q_i}=0 \ \ \text{at}  \ (\vM_0, \vq_0) $$
	and $\partial{\cal P}_i/\partial{m_i}=0$  if $m_{0,i}<m_{*,i}$.
\end{prop}
Evidently, the same condition with respect to $\vpp$ (resp. $\vM$)  holds if no commission is imposed $(\vq_0=0)$.
\section{Existence of Nash equilibrium}\index{Nash equilibrium}
In general, the existence of Nash equilibrium is not guaranteed. There are, however, some cases in which a {\em conditioned} Nash equilibrium exists. For example, if the capacities $\vM_*$ are sufficiently small then we expect that, {\em at least if no commission is imposed}, the  "saturated" capacities $\vM_0=\vM_*$ is an $\vM_*$-conditioned Nash equilibrium. 

In general, however, there always exist a Nash equilibrium if we allow  {\em mixed states} \cite{nash}. 
\section{Efficiency}
A (sub)partition is called {\em efficient} if the sum of all i.v of all agents is maximized. Here we pose no restriction on the capacities. Alternatively, a (sub)partition is efficient if each consumer $x$ attends the agent $i$ which is {\em best for him}, provided the utility of this agent is positive, i.e
\begin{defi}\label{defstronefi}
	A (sub)partition $(A_1, \ldots A_N)$ is efficient iff \\ $A_i:= A_i(\vec{0})\equiv\{x; 0< \theta_i(x)=\bar\theta(x)\}$
	where $\bar\theta(x):=\max_{1\leq j\leq N}\theta_j(x)$.
\end{defi}
We observe that, in the case of no commission,  the efficiency condition is met if all agents set their flat prices to zero. In that case, the sum of all i.v is maximized, and
$$\overline{\Sigma^\theta}:=\max\{  \Sigma^\theta(\vM)\ ; |\vM|\leq \mu(X)\}
=\int_X[\max_{i\in\I}\theta_i(x)]_+d\mu\equiv \Xept(\vec{0}) \ . $$
Evidently, such an efficiency is not in the best interest of the agents (even though it is, of course, in the best interest of the consumers). An alternative definition, which is more realistic from the agent's point of view, is the {\em Weak Efficiency}:
The case of weak efficiency is obtained if all agents make a cartel, i.e.  \index{cartel} whenever all agents agree on a common price $\bar{p}=p_i$ for any  $i\in \{1, \ldots N\}$. In that case the set of {\em inactive consumers} which does not attend any agent is $A_0(\bar{p})=\{x; \theta_i(x)-\bar{p}\leq 0\}$.
\begin{defi}
	A sub-partition $\vA:=(A_1, \ldots A_N)$ is weakly efficient  iff there exists a common flat price $\bar{p}$ such that any {\em active} consumer attends the agent best for himself, i.e.
	\\ $A_i:= A_i(\bar{p}, \ldots \bar{p})\equiv\{x; \bar{p}<\theta_i(x)=\bar\theta(x)\}$.
\end{defi}
It leaves the agents the freedom to choose the common price $\bar{p}$. If they choose $\bar{p}$ in order to maximize {\em the sum of their profits}, then this $\bar{p}$ is determined by the optimal price for a {\em single agent} whose utility function is $\bar\theta$:
$$  \bar{p}=\arg\max_p p\mu\left(x; \overline{\theta}(x)-p\geq 0\right) \ , $$ See Example
\ref{exmaxprofitone}. If, on the other hand, the agents chose their common flat price $\bar{p}$ in order to maximize the sum of their i.v, then, evidently, $$\bar{p}=\min_X[\overline{\theta}(x)]_+$$
which leads to a strong efficiency.

An additional, dual way to characterize a  weakly efficient (sub)partitions is to characterize a given {\em total capacity} $m=|\vM|$:
\begin{theorem}
	For any $m<\mu(X)$ there exists a weakly efficient subpartition $\bar\vA=(\bar A_1, \ldots \bar A_N)$, $\mu(\bar A_i):=\bar m_i$ verifying $\sum_{i} \bar m_i=m$. The capacity vector \index{capacity vector} $\vm$ maximizes $\Sigma^\theta=\sum_{i} V_i(m_i)$ on $\{\vM; \sum_{i} m_i\leq m\}$, and the common price $\bar{p}$  for this subpartition  minimizes
	$$ p \mapsto \int_X[\bar\theta(x)-p]_+d\mu + pm \ . $$
\end{theorem}
\begin{proof}
	Recall  $\bar\theta(x):= \max_i\theta(x)$ and
	$$\Sigma^\theta(\vM)=\min_{\vpp\in \R^N} \Xept(\vpp)+\vpp\cdot\vM\ . $$
	Since $[\bar\theta(x)-p]_+\geq [\theta_i]_+(x)-p$ for any $i$ and any $p\in\R$, it follows from definition of $\Xept$ that for any $\vM$ satisfying $\sum_{i}^N m_i\leq m$:
	$$ \int_X[\bar\theta(x)-p]_+d\mu + pM\geq \Xept(p\vec{1}) +p\vec{1}\cdot\vM  \ . $$
	In particular
	\be\label{sumeq10.6}\min_{p\in\R}\int_X[\bar\theta(x)-p]_+d\mu + pm\geq \min_{\vpp\in\R^N}\Xept(\vpp) +\vpp\cdot\vM=\Sigma^\theta(\vM)\ . \ee
	On the other hand, for the minimizer $\bar p$ we get $m=\mu(x\in X; \bar\theta(x)>\bar p)$. Let
	$\bar{A}_i=\{x\in X; \theta_i(x)>\bar p\}$ and $\bar{m}_i:=\mu(\bar{A}_i)$. Then $\sum_{i}\bar{m}_i=m$ and
	$$\int_X[\bar\theta(x)-\bar{p}]_+d\mu+\bar pm = \sum_{i} \int_{\bar{A}_i}\theta_id\mu =\Sigma^\theta(\bar{m}_1, \ldots \bar{m}_N) \ . $$
	This implies the equality in (\ref{sumeq10.6}) for $\vM=(\bar m_1, \ldots \bar m_N)$.
\end{proof}

To summarize:

\begin{tcolorbox}
	A weakly efficient (sub)partitions is obtained by either a cartel sharing a common flat price, or\index{cartel} 
	by maximizing the sum of the individual values subjected to a  maximal {\em total capacity} $\sum_{i} m_i\leq m$.\index{individual (surplus) value}
\end{tcolorbox}

A natural question is
\begin{tcolorbox}
	Is a weakly efficient (sub)partition guaranteed by maximizing the sum of {\em agent's profit} (rather than the sum of their i.v)?
\end{tcolorbox}

Unfortunately, answer to this question is negative, in general.

\begin{example}\label{disjointsupports}Consider the case where the supports of all agent's utility functions   are disjoint. The best price of agent $i$ is then 
	$$p_i=\arg\max p\mu\{x; \theta_i(x)\geq p\} \ . $$
	Evidently, there is no reason for  all $p_i$ to be the same in that case! 
\end{example}
\subsection{Efficiency for agents of comparable utilities}\label{effcomputi}
The opposite situation for Example \ref{disjointsupports}  is whenever the support of all agent's utilities are the same. A particular case is demonstrated in Example \ref{exlp}, where  $\theta_i=\lambda_i\theta$,  $0\leq \lambda_i<\lambda_{i+1}$. By example \ref{exlp} 
\be\label{dty}\frac{\partial\Sigma^\theta}{\partial m_i}= \sum_{j=1}^i\lambda_j\left({\cal F}_\theta^{'}({\cal M}_j)-{\cal F}_\theta^{'}({\cal M}_{j+1})\right) + \lambda_i{\cal F}_\theta^{'}({\cal M}_{i+1}) \ , \ee
where we used ${\cal M}_j:=\sum_{i=j}^Nm_i$.
By (\ref{parprofit}) we obtain that the sum of the profit of all agents, as a function of $\vM$,  is
\begin{multline}\label{mlines}\overline{\cal P}(\vM):=\sum_{i\in\I}{\cal P}_i(\vM)=\sum_{i\in\I}m_i\frac{\partial {\Sigma^\theta}}{\partial m_i}=\sum_{i\in\I} \sum_{j=1}^im_i\lambda_j
\left({\cal F}_\theta^{'}({\cal M}_j)-{\cal F}_\theta^{'}({\cal M}_{j+1})\right)+ \\ \sum_{i\in\I}m_i\lambda_i{\cal F}_\theta^{'}({\cal M}_{i+1})= \sum_{j=1}^N\left( \sum_{i=j}^Nm_i\right)\lambda_j\left({\cal F}_\theta^{'}({\cal M}_j)-{\cal F}_\theta^{'}({\cal M}_{j+1})\right) \\ + \sum_{i\in\I}m_i\lambda_i{\cal F}_\theta^{'}({\cal M}_{i+1})=
\sum_{j=1}^N {\cal M}_j\lambda_j\left({\cal F}_\theta^{'}({\cal M}_j)-{\cal F}_\theta^{'}({\cal M}_{j+1})\right) + \\ \sum_{i\in\I}\left({\cal M}_i-{\cal M}_{i+1}\right)\lambda_i{\cal F}_\theta^{'}({\cal M}_{i+1})
= \sum_{i\in\I}{\cal M}_i(\lambda_i-\lambda_{i-1}){\cal F}_\theta^{'}({\cal M}_{i}) \end{multline}
where we used $\lambda_0={\cal M}_{N+1}\equiv  0$. The maximum of $\overline{\cal P}(\vM)$ then follows for ${\cal M}_i={\cal M}_0$ for any $i\in\I$,
where ${\cal M}_0$ is the maximizer of $m\mapsto m{\cal F}_\theta^{'}(m)$. It implies that $m_N={\cal M}_0$ and $m_i=0$ for $1\leq i< N$. Thus:
\vskip .2in\noindent

\begin{tcolorbox}
	Under the assumption of Example \ref{exlp}, the agents maximize  the sum of their profits in the weakly effective state where all active consumers attend the leading agent $N$.
\end{tcolorbox}
The cartel state in the last example is {\em not} necessarily a Nash equilibrium. Indeed, if $p_0:= {\cal F}_\theta^{'}({\cal M}_0)$ is the flat price of the leading agent $N$ which maximizes his profit (as a single agent), then  the cartel \index{cartel} state is a Nash equilibrium  \index{Nash equilibrium} iff the "second best" agent $N-1$ cannot attract some consumers if she set her price higher than $p_0$, i.e. iff
\be\label{optimalp0}\lambda_{N-1}\max_X\theta<p_0 \ . \ee
Indeed, if this inequality is reversed then the agent $N-1$ can set a price $p_0<p<\lambda_{N-1}\max_X\theta$, attract the non-empty set of consumers $A_{N-1}=\{x; \lambda_{N-1}\theta(x)>p\}$ and gain a positive profit $p\mu(A_{N-1})$. We obtained
\begin{theorem}Under the conditions of Example \ref{exlp}, let $p_0={\cal F}_\theta^{'}({\cal M}_0)$ where ${\cal M}_0$ is the maximizer of $m\mapsto m{\cal F}_\theta(m)$ (equivalently, $p_0$ is the minimizer of  $p \mapsto p\mu\{x; \theta(x)>p\}$). Then the price vector $p_N=p_0$, $p_i\geq 0$ for $i<N$ is a Nash equilibrium under flat prices strategy iff (\ref{optimalp0}) is satisfied.
\end{theorem}
\subsection {Efficiency  under commission strategy}\label{euc}\index{efficiency}
In general, however, it seems that under flat prices policy  we cannot expect  the cartel strategy  \index{cartel} leading to a maximal sum of the profit of  the agents  to be  a (weakly)  effective state. The situation changes dramatically if the strategy of the agents involves commissions. \index{commission}
Then efficiency  can always be  obtained if all agents makes a cartel \index{cartel} of zero flat prices $\vpp=0$ and a common commission $q_i=Q\in [0,1)$. Indeed, in that case the (sub)partition is given by
$$A_i=\{x; 0<(1-Q)\theta_i(x)=\max_{1\leq j\leq N}(1-Q)\theta_j(x)\}$$
which is identical to Definition \ref{defstronefi}.
\par
It seems that the strategy of a cartel of commissions is a winning strategy for the agents.\index{cartel} 
Fortunately (for the consumers), it is {\em never} a Nash equilibrium. Indeed, if all agents choose to collect a commission $Q\approx 1$, then any agent can lower his commission a little bit and attract all consumers!

What will be a Nash equilibrium in the case of \index{Nash equilibrium} Example  \ref{exlp} under a commission policy?  Suppose the leading agent $N$ set up the commission \index{commission}
$q_N=1-\lambda_{N-1}/\lambda_N$. Then, for {\em any} choice $q_i\in(0,1)$ for the other agents $i\not=N$, the leading agent get all consumers and here profit is $(\lambda_N-\lambda_{N-1})\mu(\theta)$, while the profit of all other agents is zero. If agent $N$ increases her commission even just a little bit, the next agent $N-1$ may charge a sufficiently small (but positive) commission and win all the consumers! Since, in the case $q_N=1-\lambda_{N-1}/\lambda_N$ all agents except  $N$ get a zero profit anyway, they can set their commission arbitrarily at $(0,1)$.
\begin{tcolorbox}
	In case of Example \ref{exlp}, the  Nash equilibrium for the "only commission"  strategy is $q_N=1-\lambda_{N-1}/\lambda_N$ and $q_i\in (0,1)$ for $i=1, \ldots N-1$.
\end{tcolorbox}
It seems, however, that an  equilibrium in this class  is not so safe for the leading agent $N$. Indeed, agent $N-1$ may declare his commission \index{commission} $q_{N-1}=0$. Even though she gains nothing from this choice,  she competes with the leading agent $N$, since each consumer is now indifferent to the choice between  $N-1$ or $N$. \footnote{Note that in that case, however, Assumption \ref{mainass5} is {\em not met}.}
Agent $N-1$ may, then,  try to negotiate with  $N$ for an agreement to share her profit. See  Chapter \ref{Cooperative partitions}. 
\section{Free price strategy}\label{fps}
Let us consider now the strategy by which each agent may choose here price arbitrarily: she is allowed to differentiate the consumers according to their utility functions with respect to {\em all other agents}.
\par
Let $w_i(x)$ the charge of agent $i$ from consumer $x$. The partition is now defined by
$$ A_0(\vec{w}):= \{x\in X; \theta_i(x)-w_i(x)\leq 0, \ 1\leq  i\leq N\} \ , $$
$$ A_i(\vec{w}):= \{x\in X; \theta_i(x)-w_i(x)> \theta_j(x)-w_j(x) \ \forall j\not= i\}-A_0(\vec{w}) \ . $$
Note that if $\theta_i, w_i$ are continuous functions then $A_i(\vec{w})$ are open sets for any $i$.
\par
The notion of Nash equilibrium is naturally generalized to the case of free strategies. However, the functions\index{Nash equilibrium}
$$  \vec{w}\mapsto{\cal P}_i(\vec{w}):= \int_{A_i(\vec{w})}w_id\mu$$
are not continuous with respect to $\vec{w}\in C(X; \R^N_+)$. Indeed, the dichotomy set $\{x; \theta_j(x)-w_j(x)=\theta_i(x)-w_i(x)\}$, $i\not= j$, is not necessarily of measure zero for any admissible strategy $\vec{w}$. This leads us to the following generalization:
\begin{defi}\label{nashgen}
	Let ${\cal P}_i$, $i=1, \ldots N$, be defined and continuous on an open subset $Q\subset C(X; \R^N_+)$. Then $\vec{w}_0\in \overline{Q}$
	is a weak Nash equilibrium if, for any sequence $\vec{w}_n\in Q$ converging uniformly to $\vec{w}_0$ there exists a sequence of positive reals $\eps_n\downarrow 0$ such that
	$$ {\cal P}_i(\vec{w}_{n, -i}, \zeta)\leq {\cal P}(\vec{w}_n)+\eps_n$$
	for any $\zeta\in C(X, \R_+)$, $i\in\I$  such that $(\vec{w}_{n, -i},\zeta)\in Q$ is the price strategy where agent $i$ charges $\zeta(x)$ from a consumer $x$, while all other agents $j\not=i$ retain their prices $w_j$.
	
	Such $\vec{w}_0$ is efficient if, along such a sequence, $\mu(A_i(\vec{w}_n)\Delta A_i)\rightarrow 0$   for $i\in\I$ where $A_i$ as given in Definition \ref{defstronefi}.\footnote{Here $A\Delta B:= (A-B)\cup (B-A)$ is the {\em symmetric difference}.}
\end{defi}
Another formulation of the weak Nash equilibrium is presented in the box below: \index{Nash equilibrium (weak)}
\begin{tcolorbox}
	$\vw_0$ is a weak Nash equilibrium iff for any $\eps>0$ there exists an $\eps-$neighborhood of $\vw_0$ such that for any admissible strategy $\vw$ in this neighborhood, no agent can improve her reward more than $\eps$ by changing the price she collects,  provided all other agents retain their pricing $\vw$.
\end{tcolorbox}

The free price strategy contains, as special cases,    the flat price strategy $w_i(x)=p_i$, the commission strategy $w_i(x)=q_i\theta_i(x)$, $q_i\in (0,1]$, and the mixed strategy $w_i(x)=p_i+q_i\theta_i(x)$.
\par
We recall that the existence of a (pure-strategy) Nash equilibrium is not guaranteed, in the general case, for  either the flat price, commission \index{commission} or mixed strategies. Moreover, even in the case where such a Nash equilibrium exists, it is not efficient, in general. In the case of a {\em free price strategy}, however, we can guarantee the existence of a weak Nash equilibrium which is efficient.\index{Nash equilibrium (weak)}
\subsection{Where Nash equilibrium meets efficiency}
Let
\be\label{freestrategy} w_i(x):= \left\{\begin{array}{cc}
	\theta_i(x)-\max_{j\not= i}\theta_j(x) & \text{if} \ x\in A_i \\
	0 & \text{if} \ x\not\in A_i
\end{array} \right. \ ,\ee
where $A_i$ as given in Definition \ref{defstronefi}.
\begin{tcolorbox}
	Under the strategy (\ref{freestrategy}), any consumer $x$ obtain the utility of his "best next agent", that is
	$$ \max_{j\not=i(x)}\theta_j(x)$$
	where $i(x):=\arg\max_{1\leq j\leq N}\theta_j(x)$.
\end{tcolorbox}
We leave the reader to prove the following Theorem:
\begin{theorem}\label{thfps}
	If $\mu\{x; \theta_i(x)=\theta_j(x)=0\}=0$ for any $i\not=j$, then  the free strategy (\ref{freestrategy}) is an efficient Nash equilibrium (in the sense of Definition \ref{nashgen}).
\end{theorem}
\vskip .3in
The free strategy seems to be  good news for the consumers. At least, it guarantees that each consumer will get the utility of his next best agents, and verifies both the stability under competitive behavior (in the sense that the weak Nash equilibrium \index{Nash equilibrium (weak)}condition is satisfied) and efficiency. In the next chapter we shall see, however, that this strategy does not survive a cooperative behavior of the agents. \index{cooperative game} 
\chapter{Cooperative partitions}\label{Cooperative partitions}
 {\small {\it 
		Competition has been shown to be useful up to a certain point and no further, but cooperation, which is the thing we must strive for today, begins where competition leaves off
}	(F.D.R)}
\section{Free-price strategy}\label{fpp}
Using a free price strategy discussed in section \ref{fps}, we obtained a weak Nash equilibrium \index{Nash equilibrium (weak)}which is efficient via Theorem \ref{thfps}. However, the agents may beat this strategy by forming a coalition. Let us elaborate this point.
\par
Suppose that
some  agents $\J\subset \I:=\{1, \ldots N\}$ decide to establish  a coalition: they offer any client $x$ the {\em maximal} utility of the coalition members
\be\label{thetaJ}\theta_\J(x):= \max_{i\in \J}\theta_i(x)\ \ . \ee
So, the "super-agent" $\J$ is now competing against the other agents $\I-\J$. The efficient partition of $X$ now takes the form
\be\label{defAslashJ} A_{\J}:= \{x\in X; \theta_\J(x)> [\max_{i\in \I-\J}\theta_i(x)]_+\}=\cup_{i\in \J}A_i\ee
where $A_i$ as given in Definition \ref{defstronefi}.
The $\J$ component of the free price strategy 
(\ref{freestrategy}) corresponding to the set of agents $\{\theta_{\J}, \theta_i, i\not\in \J\}$ is, via Theorem \ref{thfps},
\be\label{defw_slaphJ} w_\J(x):= \left\{\begin{array}{cc}
	\theta_\J(x)-\max_{j\not= \J}\theta_j(x) & \text{if} \ x\in A_\J \\
	\text{any positive value} & \text{if} \ x\not\in A_\J
\end{array} \right. \ .\ee
Clearly, $w_\J(x)\geq w_j(x)$ for any $x\in A_\J$ and any $j\in \J$. In particular, the profit of the super-agent $\J$ (denoted as $\nu(\J)$) is not smaller than the combined profits of all agents $j\in \J$ together (under the free price strategy):
\be\label{Jfreestrategy}\nu(\J):=\int_{A_\J}w_\J d\mu\geq  \sum_{j\in \J}\int_{A_j}w_jd\mu \ , \ee
\par
The inequality in (\ref{Jfreestrategy}) can be strong. 
Evidently, this profit is monotone in the coalition, namely $\nu(\J^{'})\geq\nu(\J)$ whenever $\J^{'}\supset \J$. In particular, if $\J=\I$ (the {\em grand coalition}), then $w_\I=\bar{\theta}_+\equiv \max_{i\in \I}[\theta_i]_+$. In that  case the grand coalition of agents wins the whole surplus value \index{individual (surplus) value} $\nu(\I)=\int_X\bar{\theta}_+d\mu$, and, in particular,  we get an efficient partition.
\par\index{grand coalition}
Is the grand coalition, indeed, a stable position for the agents? It depends on how the agents share the surplus value between themselves. A natural way of sharing is as follows: each agent  collects the surplus value in the domain in which she dominates, that is
$${\cal P}_i= \int_{A_i}\bar{\theta}_+d\mu\equiv \int_{A_i}[\theta_i]_+ d\mu$$
(recall Definition \ref{defstronefi}).

Notice that the agents my almost obtain such  a sharing if they act individually, and use the commission  \index{commission}strategy $w_i=q\theta_{i}$, $q\in (0,1)$ for sufficiently small $1-q$. However,  such a sharing  it  is {\em not} a Nash equilibrium  \index{Nash equilibrium} by the argument in section \ref{euc}, as any agent may slightly 
lower her commission and attract the consumers of other agents.

At this point we leave the realm of Nash equilibrium and competitive game theory  and enter into the realm of {\em Cooperative Games}:
\section{Cooperative games- a crash review}\label{gamerev}
A cooperative game is a game where groups of players ("coalitions") may enforce cooperative behavior,  hence the game is a competition between coalitions of players, rather than between individual players.
\par
This section is based on the monograph \cite{Gi}.
\begin{defi}\label{deficoal0}
	A {\it cooperative  game} (CG) in $\I:=\{1, \ldots N\}$ is given  by a {\em reward function} $\nu$  on the subsets of $\I$:\index{cooperative game} 
	$$ \nu: 2^{\I}\rightarrow \R_+ \ \ , \ \ \nu(\emptyset)=0 \ . $$
\end{defi}
The set of {\em imputations}\index{imputations}
is composed of  vectors $\vx:=(x_1, \ldots x_N)\in \R_+^N$ which satisfy the following conditions
\be\label{fs1} \sum_{i\in\I} x_i\leq \nu(\I)  \ .  \ee
\begin{defi}
	The {\it core} of a game $\nu:2^{\I}\rightarrow\R_+$ ($Core(\nu)$)
	is the collection of all imputation vectors which satisfy
	\be\label{fs2} \forall \J\subseteq \I, \ \ \sum_{i\in \J} x_j\geq \nu(\J) \ . \ee

\end{defi}

If the core is not empty then \index{core}
no sub-coalition $\J$ of the grand coalition $\I$ will be formed. Indeed, if such a sub-coalition $\J$ is formed, its reward $\nu(\J)$ is not larger than the sum of the imputations  of its members, guaranteed by the grand coalition.\index{grand coalition}\index{imputations}

\noindent
In many cases, however, the core is empty.


We can easily find a necessary condition for the core to be non-empty. Suppose we divide $\I$ into a  set of coalitions $\J_k\subset \I$ , $k=1, \ldots m$  such that $\J_k\cap \J_{k^{'}}=\emptyset$ for $k\not= k^{'}$ and $\cup_{k=1}^m \J_j=\I$.
\begin{prop}\label{pncI}
	For any such division, the condition
	\be\label{ncI} \sum_{k=1}^m \nu(\J_k)\leq \nu(\I)\ee
	is  necessary  for the grand coalition to be stable.
\end{prop}
\begin{proof}
\index{core}	Suppose $\vx\in Core(\nu)$.  Let $\tilde{\nu}(\J):= \sum_{i\in \J} x_i$. Then $\tilde{\nu}(\J)\geq \nu(\J)$ for any $\J\subseteq \I$. If (\ref{ncI}) is violated for some division $\{\J_1, \ldots \J_m\}$, then $\sum_{k=1}^m\tilde{\nu}(\J_k)\geq \sum_{k=1}^m\nu(\J_k) >\nu(\I)$. On the other hand, $\sum_{k=1}^m\tilde{\nu}(\J_k)=\sum_{i\in\I}x_i\leq \nu(\I)$, so we get a contradiction.
\end{proof}
Note that {\em super-additivity}
\be\label{sad} \nu(\J_1)+\nu(\J_2)\leq \nu(\J_1\cup \J_2) \ \ \ \forall \ \J_1\cap \J_2=\emptyset \ee
is a sufficient condition for (\ref{ncI}). However, (\ref{sad}) by itself is {\em not} a sufficient condition for   the stability of the grand coalition.
\begin{example}
	In case $N=3$ the  game $\nu(1)=\nu(2)=\nu(3)=0$, $\nu(12)=\nu(23)=\nu(13)=3/4$, $\nu(123)=1$ is super-additive but its core is empty.
\end{example}
We may extend condition (\ref{ncI}) as follows: A {\it weak division} is a function $\lambda: 2^{\I}\rightarrow \R$ which satisfies the following:
\begin{description}
	\item {i)} \ For any $\J\subseteq \{1, \ldots N\}$, $\lambda(\J)\geq 0$.
	\item{ii)} \ For any $i\in\I$, $\sum_{\J\subseteq \I; i\in \J} \lambda(\J)=1$.
\end{description}
A collection of such sets $\{\J\subset \I; \lambda(\J)>0\}$ verifying (i,ii) is called  a \underline{\em balanced collection} \cite{Gi}.
\par
We can think about $\lambda(\J)$ as the probability of the coalition $\J$. In particular, (ii) asserts that any individual $i\in \I$ has a probability 1 to belong to {\em some} coalition $\J$. Note that any division $\{\J_1, \ldots \J_m\}$ is, in particular, a weak division where $\lambda(\J)=1$ if $\J\in \{\J_1, \ldots  \J_m\}$, and $\lambda(\J)=0$ otherwise.
\par
It is not difficult to extend the necessary condition (\ref{ncI}) to weak subdivisions  \index{weak subpartition} as follows:
\begin{prop}\label{pnscI}
	For any weak subdivision $\lambda$, the condition
	\be\label{nscI} \sum_{\J\in 2^{\I}} \lambda(\J)\nu(\J)\leq \nu(\I) \  \ee
	is  necessary  for the grand coalition to be stable.
\end{prop}
The proof of Proposition \ref{pnscI} is a slight modification of the proof of Proposition \ref{pncI}.

However, it turns out that (\ref{nscI}) is also a sufficient condition for the stability of the grand coalition $\I$. This is the content of Bondareva-Shapley Theorem 
\begin{theorem}\label{shaply}\cite{bon,shap1}
	The  grand coalition  is stable if and only if it satisfies (\ref{nscI}) for any weak division $\lambda$.\index{grand coalition}
\end{theorem}
The condition of Theorem \ref{shaply} is easily verified for super-additive game in case $N=3$.
\begin{cor}\label{corrrrr}A super additive cooperative \index{super additive game} game of 3 agents ($N=3$) admits a non-empty core iff\index{core}
	\be\label{onlytp} \nu(12)+\nu(13)+\nu(23)<2\nu(123) \ . \ee
\end{cor}
Indeed, it can  be shown that all weak subdivision for  $N=3$ are spanned by
$$\lambda(\J)=1/2 \ \ \text{if} \ \ \J=(12), (13), (23) \ \ \ ; \ \ \ \lambda(\J)=0 \ \ \text{otherwise}  \ , $$
and the trivial ones.
\subsection{Convex games}
A game $\nu$ is said to be convex if a larger coalition gains from joining  a new agent {\em at least as much as } a smaller coalition gains from adding the same agent. That is, if $\J_2\supset \J_1$ and $\{i\}\not\in \J_1\cup \J_2$ then
\be\label{bigcoalbetter} \nu(\J_2\cup \{i\})-\nu(\J_2)\geq \nu(\J_1\cup \{i\})-\nu(\J_1) \ . \ee
The inequality (\ref{bigcoalbetter}) follows if, for {\em any} $ \J_1, \J_2\in 2^{\I}$
\be\label{supermod}  \nu(\J_1)+\nu(\J_2)\leq \nu(\J_1\cup \J_2)+\nu(\J_1\cap \J_2) \ . \ee
In fact, it turns out  that (\ref{bigcoalbetter}) and (\ref{supermod}) {\em are equivalent}. The last condition is called {\em super-modular} (See sec. 7.4 in \cite{CM}).
Note that super-modularity  \index{super-modularity}  is {\em stronger} than super-additivity   \index{super-additivity} (\ref{sad}). However, in contrast to super-additivity,
super-modularity  {\em does} imply the existence of a non-empty core. Moreover, it characterizes the core in a particular, neat way:\index{core} \index{super-modularity}

Let $i_1, \ldots i_N$ be any arrangement of the set $\I$. For each such arrangement, consider the imputations:
\be\label{impu}x_{i_1}=\nu(\{i_1\}), \ \ \ldots x_{i_k}=\nu(\{i_1, \ldots i_k\}) - \nu(\{i_1, \ldots i_{k-1}\}) \ . \ldots \ee
\begin{theorem}(c.f. \cite{Dehez})\label{convexisstable}
	If the game is convex \index{convex game} then any imputation (\ref{impu}) obtained from an arbitrary arrangement of the agents is in the core. Moreover, the core is the convex hull \index{convex hull}  of all such imputations.\index{imputations}\index{core}
\end{theorem}
\begin{example}\label{convexgame}
	Let $(\bar{X}, \bar{\mu})$ be a finite measure space. Let us associate with each agent $i\in \I$ a measurable set $\bar{A}_i\subset \bar{X}$. For any $\J\subset \I$ let
	$$\nu(\J):= \bar{\mu}\left( \bar{X}-\cup_{j\not\in\J}\bar{A}_j\right) \ . $$
\end{example}
\begin{lemma}\label{stable=convex}
	The game defined in Example \ref{convexgame} is convex.
\end{lemma}
\begin{proof}
	By the postulates of measure
	$$ \nu(\J)=\bar{\mu}(\bar{X}) - \bar{\mu}(\bar{A}_{\I-\J}) \ . $$
	where $\bar{A}_{\J}:= \cup_{j\in\J}\bar{A}_j$. Then
	$$ \bar{A}_{\I-(\J_1\cup\J_2)}\subset \bar{A}_{\I-\J_1}\cap \bar{A}_{\I-\J_2} \ . $$
	Indeed, $x\in \bar{A}_{\I-(\J_1\cup\J_2)}$ iff there exists $i\in \I-(\J_1\cup\J_2)$ such that $x\in\bar{A}_i$, which implies that $x\in \bar{A}_{\I-\J_1}\cap \bar{A}_{\I-\J_2}$. This inclusion can be strict since $x\in \bar{A}_{\I-\J_1}\cap \bar{A}_{\I-\J_2}$ implies that there exists $i\in\I-\J_1$ and $j\in\I-\J_2$ such that $x\in A_i\cap A_j$ (but not necessarily $i=j$).
	
	On the other hand
	\be\label{tu1} \bar{A}_{\I-(\J_1\cap\J_2)}=\bar{A}_{\I-\J_1}\cup \bar{A}_{\I-\J_2} \ . \ee
	Hence
	\be\label{tu2} \bar{\mu}\left(\bar{A}_{\I-(\J_1\cup\J_2)}\right)\leq \bar{\mu}\left( \bar{A}_{\I-\J_1}\cap \bar{A}_{\I-\J_2}\right)  \ee
	and
	$$ \bar{\mu}\left(\bar{A}_{\I-(\J_1\cap\J_2)}\right)=\bar{\mu}\left(\bar{A}_{\I-\J_1}\cup \bar{A}_{\I-\J_2}\right) \ . $$
	By the axioms of a measure we also get
	$$ \bar{\mu}\left( \bar{A}_{\I-\J_1}\cup \bar{A}_{\I-\J_2}\right) =\bar{\mu}\left( \bar{A}_{\I-\J_1}\right) +\bar{\mu}\left( \bar{A}_{\I-\J_2}\right)-\bar{\mu}\left( \bar{A}_{\I-\J_1}\cap \bar{A}_{\I-\J_2}\right)  \ . $$
	Since
	$$ \bar{\mu}\left( \bar{A}_{\I-\J_1}\cup \bar{A}_{\I-\J_2}\right)\equiv\bar{\mu}\left( \bar{A}_{\I-(\J_1\cap\J_2)}\right)\equiv \bar{\mu}(\bar{X})-\nu(\J_1\cap\J_2)$$
	and
	$$ \bar{\mu}\left( \bar{A}_{\I-\J_1}\cap \bar{A}_{\I-\J_2}\right) \geq \bar{\mu}\left(\bar{A}_{\I-(\J_1\cup\J_2)}\right)
	\equiv \bar{\mu}(\bar{X})-\nu(\J_1\cup\J_2)$$
	we obtained
	$$ \nu\left(\J_1\cup \J_2\right)+\nu\left( \J_1\cap \J_2\right)  \geq \nu\left( \J_1\right) +\nu\left(\J_2\right) \ . $$
	
\end{proof}

\section{Back to cooperative partition games}\label{sectionBtM}\index{cooperative game} 
Let us re-examine  the game described in Section \ref{fpp}. Here we defined
\be\label{gamefreeprice}\nu(\J):= \int_{A_\J}w_\J d\mu  \ , \ee
see (\ref{Jfreestrategy}), where $A_\J, w_\J$ as in (\ref{defAslashJ}, \ref{defw_slaphJ}). Let us extend the space $X$ to the graph below the maximal utility function $\bar{\theta}_+$, that is:
$$\bar{X}:= \{(x,s);\ x\in X, \ 0\leq  s\leq\bar{\theta}_+(x):=\max_{i\in\I}[\theta_i(x)]_+\} $$

Let us further define
$$ \bar{A}_j:= \{(x,s)\in\bar{X}; \ 0\leq  s\leq [\theta_{j}(x)]_+\}\ . $$
It follows that the game (\ref{gamefreeprice}) is equivalent, under this setting, to the game described in Example \ref{convexgame}. From Lemma \ref{stable=convex} and Theorem \ref{convexisstable} we obtain:
\begin{theorem}
	Under condition of Theorem \ref{thfps}, the cooperative game of free price  (\ref{gamefreeprice}) is stable.
\end{theorem}
These are   good news for the agents but very bad for the consumers! Indeed, the stable grad coalition of the agents collects all the surplus to themselves (as $\nu(\I)=\int_X\bar{\theta}_+dx$) and leave nothing to the consumers.
and the measure $\mu$ on $X$ to a measure $\bar{\mu}(dxds):=\mu(dx)ds$ on $\bar{X}$.
In order to defend the consumers we have to impose some regulation on the agents:
\vskip .3in
\begin{tcolorbox} {\bf Consumer's based pricing is forbidden! }
\end{tcolorbox}

\subsection{Flat prices strategy: Regulation by capacity}
Let us assume now that each agent has a limited capacity. So, $\mu(A_i)\leq m^0_i$ where $A_i\subset X$ is the set of consumers of agent $i$. The agents may still form a coalition $\J\subset \I$, and the capacity of $\J$ is just
$$m^0_\J:=\sum_{i\in\J}m^0_i \ \ . $$
The utility of the coalition $\J$ is given by maximizing the utilities of its members, i.e. $\theta_\J$ as defined (\ref{thetaJ}).

\begin{tcolorbox}
	We assume that for any coalition $\J\subset \I$, the rest of the agents form the complement coalition $\J^{-}:= \I-\J$.
\end{tcolorbox}

Let us consider a  cooperative game \index{cooperative game}  $\nu$ where the utility of a coalition $\nu(\J)$ is the {\em surplus} value  of this coalition, where competing against the complement coalition $\J^{-}$.
For this we consider\index{$\XicJ$}
\be\label{XiJ}\XicJ(p_\J, p_{\J^{-}}):=\int_X\max\left[(\theta_J(x)-p_J),(\theta_{J^{-}}(x)-p_{J^{-}}),0\right]d\mu(x):\R^2\rightarrow\R  \ , \ee
and 
\be\label{maxMJ}\Sigma(m^0_\J, m^0_{\J^{-}}) = \max_{m_\J\leq m^0_\J, m_{\J^{-}}\leq m^0_{\J^{-}}}\left[\min_{(p_\J, p_{\J^{-}})\in \R^2}\XicJ(p_J, p_{\J^{-}})+p_Jm_\J +p_{J^{-}}m_{\J^{-}}\right] . \ee

\begin{prop}\label{propgamecap}
	Under the assumption of Theorem \ref{old},
	there exists  unique vectors $(m_\J, m_{\J^{-}})$  which maximizes (\ref{maxMJ}) and a unique  $(p_\J^0, p_{\J^{-}}^0)$ which minimize $(p_\J, p_{\J^{-}})\mapsto \XicJ(p_J, p_{\J^{-}})+p_Jm_\J +p_{J^{-}}m_{\J^{-}}$.
\end{prop}
\begin{proof}
	First note that
	$$ \{ x\in X; \theta_{J}(x)-\theta_{J^{-}}(x)=r \}\subset \cup_{j\in J, i\in J^{-}} \{ x\in X; \theta_j(x)-\theta_i(x)=r\}$$
	so
	Assumption \ref{mainass4} implies, for any $r\in\R$,
	$$\mu(x; \theta_\J(x)-\theta_{\J^{-}}(x)=0)=0 \ \ ; \ \ \ \mu(x; \theta_\J(x)=0)=\mu(x; \theta_{\J^{-}}(x)=0)=0  \ . $$
	Hence, the conditions of Theorem \ref{old} hold for this modified setting.
\end{proof}
The partition $(A^0_\J, A^0_{\J^{-}})$ is also given by
$$ A^0_\J:=\cup_{i\in\J}A_i(\vpp_0)$$
where $A_i(\vpp)$ as defined in (\ref{AsubiP}) and  $p_{0,i}=p^0_\J$ if $i\in \J$ and $p_{0,j}=p^0_{\J^{-}}$ if $j\not\in \J$. Indeed,
$$\Xept(\vpp_0)\equiv \Xi_{\J}(p^0_\J, p^0_{\J^{-}})$$
where $\Xept$ given by (\ref{Xiwithplus}). Thus, we may characterize the coalitions $\J$ as  a {\em cartel}:\index{cartel} 

\begin{tcolorbox}
	The coalitions $\J$ is obtained as a cartel \index{cartel} where all members of this coalition (and, simultaneously, all members of the complementary coalition $\J^{-}$)
	agree on equal  flat prices.
\end{tcolorbox}
\begin{defi}\label{gamesurp}
	Let
	$A_\J^0:= A_\J(p^0_J, p^0_{\J^{-}})$, where
	$$ A_J(p_\J, p_{\J^{-}}):=\left\{ x\in\ X; \theta_{\J}(x)-p_{\J}\geq (\theta_{\J^{-}}(x)-p_{\J^{-}})_+\right\} $$
	and $(p^0_J, p^0_{\J^{-}})$ the unique minimizer as defined in Proposition \ref{propgamecap}.
	
	The surplus-based coalition game $\nu$  subjected to a given capacity vector $\vM^0\in \R^N_+$ is given by\index{capacity vector}
	$$ \boxed{\nu(\J):= \int_{A^0_\J}\theta_\J d\mu \ .} $$
\end{defi}

Note that this game satisfies the following condition: For each $J\subset \I$,
\be\label{nc1}\nu(\J) + \nu(\J^-)\leq \nu(\I) \ \ \forall \ \ \J\subset \I \ , \ee
which is  a necessary condition for super-additivity (\ref{sad}).

In general, however, thus game is {\em not} super-modular.
\begin{example}
	Let us consider 3 agents corresponding to $\theta_1\geq\theta_2\geq\theta_3$. Assume also $m_1, m_2<<1$. Let $x_0:=\arg\max\left(\theta_1-\theta_3\right)$ and $x_1:=\arg\max\left(\theta_1-\theta_2\right)$.  Since $\nu(\{1\})$ is the surplus of agent 1 competing against agents 2+3, and $\theta_{2,3}:=\theta_2\vee\theta_3\equiv \theta_2$, it follows by Example \ref{exAlambda} (dealing in the case of two agents-low capacity) that  $\nu(\{1\})\approx m_1\theta(x_1)$.  On the other hand $\theta_{1,2}=\theta_1$ is competing against $\theta_3$  so, by the same example,
	$\nu(\{1,2\})\approx(m_1+m_2)\theta_1(x_0)$. Thus, if
	$$\theta_1(x_1)>\frac{m_1+m_2}{m_1} \theta_1(x_0)$$
	then $$\nu(\{1,2\}) < \nu(\{1\})\leq \nu(\{1\}) +\nu(\{2\}) \  .  $$
\end{example}

An alternative definition of a coalition game is based on the agent's profit. In that case there is an upper limit to the capacity of all agents, and each coalition $\J$ maximizes its  profit against the complement coalition $\J^{-}$:

\begin{defi}\label{gameselfprof}
	Let $m>0$. Given  a coalition $\J\subset\I$, define the {\em self profit} coalition game as
	$$ \nu^{\cal P}(\J):=m_\J\frac{\partial}{\partial m_{\J}}\Sigma(m_\J, m_{\J^{-}})$$
	where $\Sigma$ as defined in (\ref{maxMJ}).
\end{defi}
Recall that $\nu^{\cal P}(\J)/m_{\J}$ stands for the flat price of the first (super)agent $\J$.

Surely we cannot expect the self-profit game to be super additive, \index{super additive game} in general. Even the inequality (\ref{nc1}) is not necessarily valid for such a game, {\em even in the case of only two agents} (see Example \ref{disjointsupports}).

\subsection{Coalition games under comparable utilities}
We obtained that both coalitions games given by Definitions \ref{gamesurp}, \ref{gameselfprof} are not super-additive in general.

However, there is a special case, introduced in Example \ref{exlp}  for which we can guarantee super-additivity and, moreover, even stability under certain additional conditions (c.f Example \ref{exlp}).

\begin{assumption}\label{asspsi} 
	There exists non-negative $\theta:X\in C(X)$ satisfying $\mu(x; \theta(x)=r)=0$ for any $r\in \R$. The utilities $\theta_i$ are given by $\theta_i=\lambda_i\theta$ where  $\vec{\lambda}:=(\lambda_1, \ldots \lambda_N)\in \R^N_+$ such that $0<\lambda_1<\ldots <\lambda_N$.
\end{assumption}
\begin{prop}\label{propmulti}
	Under assumption \ref{asspsi}, for any $\vM:=(m_1, \ldots m_N)\in\R^N_+$, the surplus-based game $\nu$   is super-additive. 
	\par
	If, in addition,  $m\mapsto m({\cal F}_\theta(m))^{'}$ is monotone non decreasing on $[0, M]$  (see Example \ref{singajex}) then the profit-based game  is super-additive  as well, provided 
	 $\sum_{i\in\I}m_i\leq M$ 
\end{prop}
\begin{proof}.  \\
	Surplus-based game:
	\\
	From Example \ref{exlp}
	(in particular from (\ref{unidi})) we obtain that the surplus  value of agent $i$ under optimal partition is \index{optimal partition}
	\be\label{1Div}V_i\equiv \lambda_i\left({\cal F}_\theta({\cal M}_{i})-{\cal F}_\theta({\cal M}_{i+1})\right) \ . \ee
	It follows that, if $\{N\}\in \J$,
	\be\label{onenu}\nu(\J)=\lambda_N {\cal F}_\theta( m_{\J}) \ , \ee
	while if $\{N\}\not\in \J$:
	\be\label{nunu}\nu(\J)=\lambda_\J \left( {\cal F}_\theta(M)-{\cal F}_\theta(M-m_\J)\right)\ . \ee
	where $\lambda_\J=\max_{i\in \J}\lambda_i<\lambda_N$ and $M= m_\J+m_{\J^{-}}\equiv \sum_{i\in\I}m_i$.

	Let now  $\J_1, \J_2\subset \I$ such that    $\J_1\cap \J_2=\emptyset$ (in particular, $m_{\J_1}+m_{\J_2}\leq M$.

	\par
	Assume first $\{N\}\not\in \J_1\cup \J_2$. Then from (\ref{nunu})
	\begin{multline}\label{zzb} \nu(\J_1\cup \J_2)=\lambda_{\J_1}\vee\lambda_{\J_2}\left({\cal F}_\theta(M)-{\cal F}_\theta(M-m_{\J_1\cup \J_2})\right)= \\
	\lambda_{\J_1}\vee\lambda_{\J_2}\left({\cal F}_\theta(M^*)-{\cal F}_\theta(M-m_{\J_1}-m_{\J_2})\right) \ . \end{multline}
	Now,
	$$ {\cal F}_\theta(M)-{\cal F}_\theta(M-m_{\J_1}-m_{\J_2})\geq 2{\cal F}_\theta(M)-{\cal F}_\theta(M-m_{\J_1})-{\cal F}_\theta(M-m_{\J_2})$$
	since
	$$ {\cal F}_\theta(M)-{\cal F}_\theta(M-m_{\J_1})\leq {\cal F}_\theta(M-m_{\J_2})-{\cal F}_\theta(M-m_{\J_1}-m_{\J_2})$$
	by concavity of ${\cal F}_\theta$.  It follows form (\ref{zzb})
	$$ \nu(\J_1\cup \J_2)\geq \lambda_{\J_1}\vee\lambda_{\J_2}\left[ \left({\cal F}_\theta(M)-{\cal F}_\theta(M-m_{\J_1})\right) +
	\left({\cal F}_\theta(M)-{\cal F}_\theta(M-m_{\J_2})\right)\right]$$
	$$ \geq \lambda_{\\J_1}\left({\cal F}_\theta(M)-{\cal F}_\theta(M-m_{\\J_1})\right) +
	\lambda_{\\J_2}\left({\cal F}_\theta(M)-{\cal F}_\theta(M-m_{\J_2})\right)=\nu(\J_1)+\nu(\J_2) \ . $$

	Next, if, say,  $\{N\}\in \J_1$ then, using (\ref{onenu}, \ref{nunu})
	$$ \nu(\J_1\cup \J_2)= \lambda_N{\cal F}_\theta(m_{\J_1}+m_{\J_2}) \ \ , \ \ \nu(\J_1)=\lambda_N{\cal F}_\theta(m_{\J_1})\ \ ,$$
	$$ \nu(\J_2)=\lambda_{\J_2}\left({\cal F}_\theta(M)-{\cal F}_\theta(M-m_{\J_2})\right)\ , $$
	so
	$\nu(\J_1\cup \J_2)-\nu(\J_1)-\nu(\J_2)\geq$
	$$\lambda_N\left[{\cal F}_\theta(m_{\J_1}+m_{\J_2})-{\cal F}_\theta(m_{\J_1})  -{\cal F}_\theta(M) +{\cal F}_\theta(M-m_{\J_2})\right]\geq 0  \ , $$
	again, by concavity of ${\cal F}_\theta$ and since  $M\geq m_{\J_1}+m_{\J_2}$.
	\par\noindent
	Case of Profit-based game: \\
	From (\ref{unidi}) with the two agents $(\lambda_{1}\theta,m_{1})$, $(\lambda_{2}\theta, m_{2})$ where $\lambda_2>\lambda_2$  we get
	$$\Sigma^\theta(m_{1}, m_{2})= \lambda_{1}\left({\cal F}_\theta(m_1+m_2)-{\cal F}_\theta(m_2)\right) + \lambda_{2}{\cal F}_\theta(m_2) \ . $$
	Assume first $N\not\in \J_1\cup \J_2$. Then, we substitute $(m_1, m_2)$ for either $(m_{\J_1},M-m_{\J_1})$, $(m_{\J_2}, M-m_{\J_2})$ and $(m_{\J_1\cup \J_2}, 1-m_{(\J_1\cup \J_2)})$ we get
	$$\nu^{\cal P}(\J_1)=m_{\J_1}\lambda_{\J_1}\left({\cal F}_\theta\right)^{'}(M) ,  \ \nu^{\cal P}(\J_2)=m_{\J_2}\lambda_{\J_2}\left({\cal F}_\theta\right)^{'}(M)$$ and  $$\nu^{\cal P}(\J_1\cup \J_2)=m_{\J_1\cup \J_2}\lambda_{\J_1\cup \J_2}\left({\cal F}_\theta\right)^{'}(M)\equiv \lambda_{\J_1}\vee\lambda_{\J_2}(m_{\J_1}+m_{\J_2})({\cal F}_\theta)^{'}(M) \ , $$
	In particular, we obtain
	$$\nu^{\cal P}(\J_1 \cup \J_2)-\nu{\cal P}(\J_1)-\nu^{\cal P}(\J_2)=\left(\lambda_{\J_1}\vee\lambda_{\J_2}(m_{\J_1}+m_{\J_2})-\lambda_{\J_1}m_{\J_1}-\lambda_{\J_2}m_{\J_2}\right)
	\left({\cal F}_\theta\right)^{'}(M) >0$$
	(unconditionally!).
	
	Assume now that $N\in \J_2$. In particular $\lambda_N>\lambda_{\J_1}$.
	Thus, under the same setting:
	$$\nu^{\cal P}(\J_1)=m_{1}\partial\Sigma^\theta/\partial m_{1}= m_{\J_1}\lambda_{\J_1}\left({\cal F}_\theta\right)^{'}(M) \ ,  $$
	$$\nu^{\cal P}(\J_2)=m_{2}\partial\Sigma^\theta/\partial m_{2}= m_{\J_2}\left(\lambda_{\J_1}\left({\cal F}_\theta\right)^{'}(M)+(\lambda_{N}-
	\lambda_{\J_1})\left({\cal F}_\theta\right)^{'}(m_{\J_2})\right) \ ,$$
	$$ \nu^{\cal P}(\J_1\cup \J_2)= m_{\J_1\cup \J_2}\left(\lambda_{\J_1}\left({\cal F}_\theta\right)^{'}(M)+ (\lambda_{\J_2}-\lambda_{\J_1})\left({\cal F}_\theta\right)^{'}(m_{\J_1\cup \J_2})\right)
	$$
	$$
	= (m_{\J_1}+m_{\J_2})\left(\lambda_{\J_1}\left({\cal F}_\theta\right)^{'}(M)+ (\lambda_N-\lambda_{\J_1})\left({\cal F}_\theta\right)^{'}(m_{\J_1}+m_{\J_2})\right) \ . $$
	It follows that
	$\nu^{\cal P}(\J_1\cup \J_2)-\nu^{\cal P}(\J_1)-\nu^{\cal P}(\J_2)=$
	$$(\lambda_{\J_2}-\lambda_{\J_1})\left( (m_{\J_1}+m_{\J_2})\left({\cal F}_\theta\right)^{'}(m_{\J_1}+m_{\J_2})- m_{\J_2}\left({\cal F}_\theta\right)^{'}(m_{\J_2})\right)\geq 0$$
	by assumption of monotonicity of $m\mapsto m\left(F_\theta^{*}\right)^{'}(m)$ on $[0,M]$, and $m_\J, m_{\J^{-}}\in[0,M]$.
\end{proof}

Under the  assumption of Proposition \ref{propmulti} we may guess, intuitively,  that the \index{grand coalition} grand coalition is stable if  the  gap between the utilities of the agents is   sufficiently large (so the other agents are motivated to join the smartest one), and the capacity of the wisest agent ($N$) is sufficiently small (so she is motivated to join the others as well).  Below we prove this intuition in the case  $N=3$:
\begin{prop}\label{example1}
	Under the assumption of Proposition \ref{propmulti} and $N=3$,
	$$  \frac{\lambda_3}{\lambda_2}> \frac{{\cal F}_\theta(m_1+m_2)}{{\cal F}_\theta(m_2)+{\cal F}_\theta(m_1)}$$
	is a necessary and sufficient for the stability of the grand coalition in the surplus game. Here ${\cal F}_\theta$ is as defined in Example \ref{exlp}. 
\end{prop}

\begin{proof}
	From  Corollary \ref{corrrrr} and Proposition \ref{propmulti} we have only to prove  (\ref{onlytp}). Now,
	$\nu(123)=\lambda_3{\cal F}_\theta(\mu(X))$,  $\nu(13)=\lambda_3({\cal F}_\theta(\mu(X))-{\cal F}_\theta(m_2))$, $\nu(23)=\lambda_3({\cal F}_\theta(\mu(X))-{\cal F}_\theta(m_1))$ and  $\nu(12)=\lambda_2{\cal F}_\theta(m_1+m_2)$.
	The result follows from substituting the above in (\ref{onlytp}).
	
\end{proof}

\begin{theorem}\label{gameof3}
	Assume $m\mapsto m\left({\cal F}_\theta\right)^{'}(m)$ is non-decreasing on $[0, M]$ where $M=m_1+m_2+m_3$. Assume further that
	\be\label{nesufcore} \alpha \left({\cal F}_\theta\right)^{'}(m_2)+\beta \left({\cal F}_\theta\right)^{'}(m_1) < \left({\cal F}_\theta\right)^{'}(m_1+m_2+m_3)  \ee
	where\\
	$\alpha:= \frac{(m_1+m_3)(\lambda_3-\lambda_2)}{2m_1(\lambda_3-\lambda_2)+ (m_2+m_3)(2\lambda_3-\lambda_2)}$, \\
	$\beta:=
	\frac{(m_2+m_3)(\lambda_3-\lambda_1)}{2m_1(\lambda_3-\lambda_2)+ (m_2+m_3)(2\lambda_3-\lambda_2)} $. \\
	Then the  self-profit game $\nu$ as given in Definition \ref{gameselfprof} is stable.
\end{theorem}
Recall that ${\cal F}_\theta$ is a concave function, hence $\left({\cal F}_\theta\right)^{'}(m_1+m_2+m_3)$ is smaller than both $\left({\cal F}_\theta\right)^{'}(m_1)$, $\left({\cal F}_\theta\right)^{'}(m_2)$. Hence $0<\alpha+\beta<1$ is a necessary condition for (\ref{nesufcore}). Check that this condition is {\em always} satisfied (since $\lambda_3>\lambda_2$).

\begin{proof}
	Again, the super-additivity is given by Proposition \ref{propmulti}.
	
	$$ \nu^{\cal P}(123)=(m_1+m_2+m_3)\lambda_3\left({\cal F}_\theta\right)^{'}(m_1+m_2+m_3) \ . $$
	$$ \nu(13)=(m_1+m_3)\left[ \lambda_2 \left({\cal F}_\theta\right)^{'}(m_1+m_2+m_3)+ (\lambda_3-\lambda_2)\left({\cal F}_\theta\right)^{'}(m_2)\right] \ , $$
	$$ \nu(23)=(m_2+m_3)\left[ \lambda_1 \left({\cal F}_\theta\right)^{'}(m_1+m_2+m_3)+ (\lambda_3-\lambda_1)\left({\cal F}_\theta\right)^{'}(m_1)\right] \ , $$
	while
	$$ \nu^{\cal P}(12)=(m_1+m_2)\lambda_2 \left({\cal F}_\theta\right)^{'}(m_1+m_2+m_3) \ . $$
	Thus
	$$ 2\nu^{\cal P}(123)-\nu(12)-\nu(13)-\nu(23)=$$
	$$\left({\cal F}_\theta\right)^{'}(m_1+m_2+m_3)\left[2m_1(\lambda_3-\lambda_2)+ (m_2+m_3)(2\lambda_3-\lambda_2)\right] - $$ $$(m_1+m_3)(\lambda_3-\lambda_2)\left({\cal F}_\theta\right)^{'}(m_2)-(m_2+m_3)(\lambda_3-\lambda_1)\left({\cal F}_\theta\right)^{'}(m_1) \  $$
	and the result follows by (\ref{onlytp}) as well.
\end{proof}

\begin{appendices}
	\appendixpage
	\noappendicestocpagenum
	\addappheadtotoc
	
	\chapter{Convexity}
	For the completeness of exposition we introduce basic notion from the theory of convexity.
	We only consider linear spaces $\Dt$ over the reals $\R$ of {\it finite dimension}. This restriction, which is sufficient for our purpose, will render the reference to any topology. In fact, topology enters only trough the definition of the dual space of $\Dt$,  $\D$, that is, the space of all {\it continuous} linear functionals on $\Dt$, and denote the duality pairing by
	$$(\cP:\cM) \ :\  \Dt\times \D\rightarrow\R \ . $$ Since, as we know, all norms are equivalent on a linear space of finite dimension, it follows that the notion of a continuous functional is norm-independent. Even though we distinguish between the space $\Dt$ and its dual $\D$ (which are isomorphic), we do not distinguish weak, weak* and strong (norm) convergence \index{weak* convergence}of sequences in the spaces $\Dt$ and $\D$, respectively . The notion  of open, closed sets and interior, cluster points of sets are defined naturally in terms of a generic norm.
	\par
	\section{Convex sets}\label{AppA1}
	The notion of a {\it convex set} is pretty  natural:
\begin{figure} 
	\centering
	\includegraphics[height=7.cm, width=10.cm]{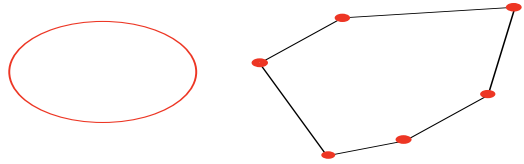}\\
	\caption{ Left: Strictly convex set. All boundary points are exposed. Right: Convex  (not strictly). Exposed points marked in red}
\end{figure}
		A set $C\subset \Dt$ is convex iff for any $\cP_1, \cP_2\in A$, the interval connection $\cP_1, \cP_2$ is contained in $C$. Namely $s\cP_1+(1-s)\cP_2\in C$ for any $s\in[0,1]$. 
		
Note that a convex set may be open, closed or neither. 

A convex set is called  {\em strictly convex} if for any two points   $\cP_1, \cP_2\in C$, the {\em open interval} $s\cP_1+(1-s)\cP_2\  \ ,  s\in(0,1)$ is contained in the interior of $C$. In particular, convex set whose  interior is empty are not strictly convex.

For example, if $C$ is contained in a subspace of $L\subset\D$, $L\not= D$, are not strictly convex. More generally, if the boundary of a convex set contains an open set in the relative topology of a subspace than it is not strictly convex.

		$\cP\in C$ is an {\em extreme point} iff it is not contained in the interior of any interval contained in $A$, i.e. there exists no $\cP_1\not=\cP_2$, both in $A$ and $\alpha\in(0,1)$ such that $\cP=\alpha\cP_1+(1-\alpha)\cP_2$. Examples of extreme points \index{extreme point}  are the boundary of an ellipsoid, or the vertices of a 

A stronger notion is of {\em exposed points}. \index{exposed point} A point is an exposed point of $C$ there exists a linear functional such that is point is the {\em unique maximizer} of the functional on $C$. Alternatively, there exists a co-dimensional 1 hyperplane whose intersection with $C$ is this single point. 

Some properties of Convex sets:
\begin{prop}\label{propconvex} \ . 
\begin{itemize}
	\item The closure and the interior of a convex set is convex.
	\item The intersection of any number of convex sets is convex.
	\item If the interior of a convex set  $C$ is not empty, then the closure of the interior of $C$ is the closure of $C$.
\end{itemize}
\end{prop}
\begin{defi}
	The {\em convex hull}  of a set $A$ ($Conv(A)$)  is the intersection of all convex set containing $A$. In particular, it is the {\em minimal} convex set containing $A$. 
	\end{defi}
An equivalent definition of a the convex hull \index{convex hull} is obtained in terms of the {\em convex combinations}: A convex combination of points $x_1, \ldots x_k$, $k\in \mathbb{N}$ is a point $x=\sum_{i=1}^k \lambda_i x_i$ where $\lambda_i\geq 0$ and $\sum_{i=1}^k \lambda_i=1$.  
\begin{lemma}
	The convex hull of a set $A$ is the set of all convex combinations of its points. 
\end{lemma}

A fundamental Theorem is the Krein-Milman theorem \index{Krein-Milman theorem}
\begin{theorem}\cite{Kmil}
	Any convex set is the convex hull of its extreme points. \index{extreme point} 
\end{theorem} \index{Borel measure}
The Krein-Milman Theorem is valid in a much wider cases, namely for any set a  Hausdorff locally convex topological vector space. In particular, it is valid for the set of Borel measures in compact metric space. 
\section{Convex functions}\label{Convexfunctins}

The basic notion we consider is that of a convex function $$\Xi:\Dt\rightarrow \R\cup \{\infty\}:= \hat{\R}\ . $$
The fundamental definition is
\begin{defi}\label{defA1}
	$\Xi$ is a convex function on $\Dt$ if for any $\cP_1, \cP_2\in \Dt$ and any $s\in [0,1]$:
	$$ \Xi(s\cP_1+(1-s)\cP_2)\leq s\Xi(\cP_1)+ (1-s)\Xi(\cP_2) \ . $$ 
	$\Xi$ is {\em strictly convex} at $\cP_0$ if for any $\cP_1\not=\cP_2$ and $s\in (0,1)$ such that $\cP_0:=s\cP_1+(1-s)\cP_2$
	$$ \Xi(\cP_0)< s\Xi(\cP_1)+ (1-s)\Xi(\cP_2) \ . $$
	Note that we allow $\Xi$ to obtain the value $\{\infty\}$ (but not the value $\{-\infty\}$), and we use, of course, the rule $r+\infty=\infty$ for any $r\in \R$.

	The {\it essential domain} of $\Xi$ ($ED(\Xi)$) is the set on which $\Xi$ admits finite values:
	$$ED(\Xi):= \{ \cP\in\Dt; \ \ \Xi(\cP)\in\R \ \} . $$
\end{defi}
\begin{remark}
	In this book we  are usually assuming that $\Xi$ is real valued for any $\cP\in \Dt$ (i.e. $ED(\Xi)=\Dt$). This, however, is not true for the Legendre transform \index{Legendre transform}  of $\Xi$ defined below on the dual space $\D$. Since we treat $(\Dt, \Xi)$ and $(\D, \Xi^*)$ on the same footing, we allow $\Xi$ to take infinite values as well.
\end{remark}
There are two natural connections between convex functions and convex set, as well as between points of strict convexity and extreme points. \index{extreme point} The first corresponds to the definition of a {\it characteristic function} of a set:\index{characteristic function}
\begin{defi}
	A characteristics function corresponding to a set $A\subset \Dt$ is
	$$ 1_A(\cP):= \left\{\begin{array}{cc}
	0 & \ \text{if} \ \cP\in A \\
	\infty & \text{otherwise}
	\end{array}\right. $$
\end{defi}
The second corresponds to the definition of a {\it supergraph}\index{supergraph} 
\begin{defi}\label{defA5}
	The supergraph of a function $\Xi:\Dt\rightarrow \hat{\R}$ is the set
	$$ SG(\Xi):=\{ (\cP,r)\in \Dt\times\R ; \ \ \Xi(\cP)\geq r \} \ . $$
	In particular, $GR(\Xi)$ does not contain the line $\cP\times\R$ whenever $\Xi(\cP)=\infty$.
\end{defi}
From these definitions we can easily obtain:
\begin{prop}\label{propA1} \ .
	\begin{itemize} 
		\item $A\subset \cP$ is a convex set iff $1_A$ is a convex function. 
		\item $\cP\in A$ is an extreme point iff is a strictly convex point of $1_A$.
		\item $\Xi$ is a convex function on $\Dt$ iff $SG(\Xi)$ is a convex set in $\Dt\times \R$. 
		\item $\cP$ is a strictly convex point of $\Xi$ iff $(\cP, \Xi(\cP))$ is an extreme point of $SG(\Xi)$.\index{extreme point} 
	\end{itemize}
\end{prop} By the first point of Proposition \ref{propA1} and second point in Proposition \ref{propconvex} we obtain
We recall that both convex and closed sets enjoy the property of being preserved under intersections:
\begin{prop}\label{propA3}
	If $\{\Xi_\alpha\}$ is a collection of  convex functions, then $\bigvee_\alpha \Xi_\alpha$ is a  convex function as well.
\end{prop} 

Another nice property of  convex sets are the preservation under projection. Let $\Dt=\Dt_1\times\Dt_2$ and the projection $Proj_1:\Dt\rightarrow \Dt_1$ is defined as $Proj_1(\cP_1, \cP_2)= \cP_1$. One can easily verify that, if $C\subset \Dt$ is a convex set in $\Dt$, then $Proj_1(C)$ is convex in $\Dt_1$ as well (note that the same statement {\em does not hold} for closed sets!).
\begin{prop}
	Let $\Xi:\Dt_1\times\Dt_1\rightarrow \hat{\R}$ be a convex function. Then
	$$ \underline{\Xi}(\cP_1):=\bigwedge_{\cP_2\in\Dt_2}\Xi(\cP_1, \cP_2) $$
	is convex on $\Dt_1$.
\end{prop}
Indeed, we observe that $SG(\underline{\Xi})$ is the projection from $\Dt_1\times\Dt_2\times\R$ of $SG(\Xi)$ into $\Dt_1\times\R$, and apply Proposition \ref{propA1}.

\section{Lower-semi-continuity}
Another closely related notion is Lower-Semicontinuity:
\begin{defi}\label{USC}
	$\Xi$ is Lower-Semi-Continuous (LST) at a point $\cP_0\in\Dt$ iff for any sequence $\cP_n$ converging to $\cP_0$:
	$$ \liminf_{n\rightarrow\infty} \Xi(\cP_n)\geq \Xi(\cP_0) \ . $$
	$\Xi$ is said to be LST if it is LSC at any $\cP\in\Dt$. \index{Lower-semi-continuous (LSC)} 
\end{defi}
In particular, if $\Xi(\cP_0)=\infty$, then $\Xi$ is LSC at $\cP_0$ iff $\lim_{n\rightarrow\infty} \Xi(\cP_n)=\infty$ for any sequence $\cP_n\rightarrow\cP_0$.
\par

From  Definitions \ref{USC} and \ref{defA5} we also get the connection between LSC and closed sets: \index{Lower-semi-continuous (LSC)} 
\begin{prop}\label{propA2}
	A function $\Xi$ on $\Dt$  is LSC at any point $\cP\in\Dt$ iff $SG(\Xi)$ is closed  on $\Dt\times\R$.
\end{prop} 
Warning: Not any convex function is LSC at any point of its essential domain. For example, consider a convex and open set $\A\subset \Dt$ such that any point on the boundary of its closure $\A^c$ is an extreme point \index{extreme point} of $\A^c$ (e.g, $\A$ is the open  ball in $\R^n$). Let $\Xi=0$ on $\A$, $\Xi=\infty$ on $\sim \A^c$ and $\Xi$ takes arbitrary real values on the boundary of $\A$. Then $\Xi$ is convex on $\Dt$ and its essential domain is $\A^c$. Still,  $\Xi$ is not LSC, in general, at points on the boundary of $\A$. However:
\begin{prop}\label{propinnercon}
	If $\Xi$ is convex on $\Dt$ then it is continuous at any {\em inner point}  of its essential domain.
\end{prop}

Recall that the intersection of a family of closed set is closed as well. Using Propositions \ref{propA1}, \ref{propA2}, \ref{propA3} we obtain
\begin{prop}\label{propA4}
	If  \ $\{\Xi_\beta\}$ is a collection of LSC (resp. convex) functions on $\Dt$, then $\bar{\Xi}(\cP):=\bigvee_\beta \Xi_\beta(\cP)$ is a LSC (resp. convex) function as well. \index{Lower-semi-continuous (LSC)} 
\end{prop}

\par
\section{Legendre transformation}\index{Legendre transform}
Let now $\{\Xi_\beta\}$ be a collection of {\em affine functions} on $\Dt$, i.e.
$\Xi_\beta(\cP):= \gamma(\beta)+\cP:\vec{\Gamma}(\beta)$, where $\gamma(\beta)\in\R$ and $\vec{\Gamma}(\beta)\in \D$. By Proposition \ref{propA4}
$$ \Xi_\beta(\cP): \cP\mapsto\bigvee_\beta \left[ \gamma(\beta)+\cP:\vec{\Gamma}(\beta)\right] \in\R\cup\{\infty\}$$ is a convex function.

 In particular, if the set of elements $\beta$ are points in the dual space $\D$ and $\Gamma(\cM):=-\Xi(\cM)$ is {\em any} function on $\D$, then
\be\label{Ltransform} \Sigma(\cM):=\bigvee_{\cP\in\Dt}\left[ \cP:\cM -\Xi(\cP)\right] \ . \ee
is a convex function on $\D$. Thus, the operation (\ref{Ltransform}) defines a transformation from the functions on the space $\D$ to a {\em convex} functions on the dual space $\Dt$.

 In addition, if we consider only LSC-convex functions $\Xi$ in (\ref{Ltransform}), it defines the {\em Legendre Transform} \index{Legendre transform} from {\em LST convex functions} on $\D$ to LSC convex functions on its dual space $\Dt$. Since a finite dimensional linear space is reflexive (i.e. $\Dt$ is the dual of $\D$ as well as $\D$ is the dual of $\Dt$), we can represent the Legendre transform as a transformation from LSC convex functions $\Xi$ on $\Dt$ to LSC convex functions $\Xi^*$ on $\D$ as well: \index{Lower-semi-continuous (LSC)} 
\begin{defi}
	The Legendre Transform (LT) of a LSC convex function $\Xi$ on $\Dt$ is the LSC convex function $\Xi^*$ on $\D$ given by
	$$\Xi^*(\cM):=\bigvee_{\cP\in\Dt} \cP:\cM-\Xi(\cP)  \ . $$
\end{defi}

In particular we obtain that
\be\label{cf4} \Xi(\cP)+\Xi^*(\cM)\geq \cP:\cM \ \ee
for any $\cP\in\Dt$, $\cM\in \D$.
\par
The two-way duality relation between $\Dt$ and $\D$ implies the possibility to define $\Xi^{**}:= \left(\Xi^*\right)^*$ as a LSC convex function on $\Dt$. It is an elementary exercise to prove that
\be\label{Legendreineq} \Xi^{**}(\cP)\leq \Xi(\cP)\ee
for any $\cP\in\Dt$. Note that  (\ref{Legendreineq}) holds for any function $\Xi:\Dt\rightarrow\hat{\R}$ (not necessarily convex or LSC). In fact that  for a general function $\Xi$, $\Xi^{**}$ is the {\em maximal convex LSC envelop} of $\Xi$, that is, the maximal convex and L.S.C function dominated by $\Xi$.

However, if $\Xi$ {\em is} both convex and LSC on $\Dt$ then we get an equality in (\ref{Legendreineq}): \index{Lower-semi-continuous (LSC)} 
\begin{prop}\label{propA6}
	If $\Xi:\Dt\rightarrow\hat{\R}$ is convex and LSC on $\Dt$ then $\Xi^{**}=\Xi$.
\end{prop}
\begin{cor}
	If $\Xi$ is the support function of a convex closed set $A\subset \D$ then its  Legendre transform  \index{Legendre transform} is the characteristic function of $A$. \index{characteristic function}
\end{cor} 
For the proof of Proposition \ref{propA6} see, e.g. \cite{Rockafllar}.
\section{Subgradients}
\begin{defi}\label{cf5}\index{subgradient}
	The {\em subgradient} of a function $\Xi:\Dt\rightarrow\hat{\R}$  is defined for any $\cP$ in the {\em essential domain} of $\Xi$ as
	$$ \partial_{\cP}\Xi:=\{ \cM\in\D; \Xi(\cP_1)-\Xi(\cP)\geq  (\cP_1-\cP):\cM \ \ , \ \forall \ \cP_1\in \Dt \  \}\subset \D \ .  $$
\end{defi}

Note that we only defined $\partial_{\cP}\Xi$ for $\cP$ in the essential domain \index{essential domain}  of $\Xi$. The reason is to avoid the ambiguous expression $\infty-\infty$ in case both $\Xi(\cP)=\Xi(\cP_1)=\infty$.

It can easily be shown that $\partial_{\cP}\Xi$ is a closed and convex set for {\em any} LSC function $\Xi$. However, it can certainly be the case that the subgradient is an empty set. If, however,  $\Xi$ is also convex then we can guarantee that $\partial_{\cP}\Xi$ is non-empty:
\begin{prop}\label{propA7}  \cite{Rockafllar}
	If $\Xi$ is LSC and convex  \index{Lower-semi-continuous (LSC)} then the subgradient $\partial_{\cP}\Xi$ is non-empty \index{subgradient} for any $\cP\in \text{Int}\left( ED(\Xi)\right)$.    If $\cM\in \text{Int}\left( ED(\Xi^*)\right)$ then there exists $\cP\in \text{Int}\left( ED(\Xi)\right)$
	such that   $\cP\in \partial_{\cM}\Xi^*$ and $\cM\in \partial_{\cP}\Xi$. In particular,  the equality
	\be\label{cf6} \Xi(\cP)+\Xi^*(\cM)=\cP:\cM \  \ee
	holds iff {\em both} $\cM\in \partial_{\cP}\Xi$ {\em and} $\cP\in \partial_{\cM}\Xi^*$.
\end{prop}

In particular, $\cP$ is a minimizer of $\Xi$ if and only if $0\in\partial_{\cP}\Xi$ (and, of course, $\cM$ is a minimizer of $\Xi^*$ if and only if $0\in\partial_{\cM}\Xi^*$).

There is a relation between differentiability of a convex function and the {\em strict convexity} of its Legendre transform: \index{Legendre transform}
\begin{prop}\label{cf7}
	A LSC convex  \index{Lower-semi-continuous (LSC)} function $\Xi$ is differentiable at $\cP\in\Dt$ if and only if $\partial_{\cP}\Xi$ is a singleton,  if and only if its directional  derivatives  exist on a spanning set of directions. In that case $\partial_{\cP}\Xi$ is identified with the gradient of $\Xi$:   $\partial_{\cP}\Xi=\{\nabla\Xi(\cP)\}$. Moreover, in that case  $\Xi^*$ is {\em strictly convex} at $\cM_0=\nabla\Xi(\cP)$, namely, for any $\alpha\in (0,1)$ and any $\cM_1\not= \cM_2$ verifying $\cM_0=\alpha\cM_1+(1-\alpha)\cM_2$:
	$$ \Xi^*(\cM_0)< \alpha\Xi^*(\cM_1)+(1-\alpha)\Xi^*(\cM_2) \ . $$
\end{prop}
Let us see the proof of the last statement.

Let $\{\cM_0\}=\partial_{\cP_0}\Xi$. Assume there exists $\alpha\in (0,1)$ and   $\cM_1\not=\cM_2$ in the essential domain \index{essential domain} of $\Xi^*$ such that $\cM_0=\alpha\cM_1+(1-\alpha)\cM_2$ and
\be\label{XiXi*}\Xi^*(\cM_0)=\alpha\Xi^*(\cM_1)+(1-\alpha)\Xi^*(\cM_2) \ . \ee
Then from (\ref{cf6})
\be\label{p3m1} \Xi(\cP_0)-\cP_0:\cM_0= \Xi(\cP_0)-\cP_0:(\alpha\cM_1+(1-\alpha)\cM_2)=-\Xi^*(\cM_0) \ . \ee
and from (\ref{cf4}):
\be\label{p3m12} \Xi(\cP_0)-\cP_0:\cM_1\geq -\Xi^*(\cM_1) \ \ \ ; \ \ \ \Xi(\cP_0)-\cP_0:\cM_2\geq -\Xi^*(\cM_2) \ .  \ee
Summing $\alpha$ times the first inequality and $(1-\alpha)$ times the second inequality of (\ref{p3m12}) we get
$$ \Xi(\cP_0)-\cP_0:(\alpha\cM_1+(1-\alpha)\cM_2))\geq \alpha\Xi^*(\cM_1)+(1-\alpha)\Xi^*(\cM_2)=\Xi^*(\cM_0) \ .  $$
From (\ref{p3m1}) we get that the two inequalities in (\ref{p3m12}) are, in fact, equalities:
$$ \Xi(\cP_0)-\cP_0:\cM_1= -\Xi^*(\cM_1) \ \ \ ; \ \ \ \Xi(\cP_0)-\cP_0:\cM_2=- \Xi^*(\cM_2) \ .  $$
Then Proposition \ref{propA7} implies that  $\cM_1, \cM_2\in\partial_{\cP_0}\Xi$. In particular $\Xi$ is not differentiable at $\cP_0$, in contradiction. Hence (\ref{XiXi*}) is violated.

Another property of closed convex functions is the following:
\begin{prop}\label{cf8}
	If $\{\Xi_n\}$ is a sequence of LSC convex functions on $\Dt$ and the limit \index{Lower-semi-continuous (LSC)} 
	$\lim_{n\rightarrow\infty}\Xi_n:=\Xi$ holds pointwise on $\Dt$, then $\Xi$ is convex  and for any interior point $\cP$ of the essential domain \index{essential domain}of $\Xi$,
	$$\partial_{\cP}\Xi\subset \liminf_{n\rightarrow\infty}\partial_{\cP} \Xi_n \ . $$
	It means that for any $\cM\in \partial_{\cP}\Xi$ there exists a subsequence $\cM_n\in \partial_{{\cP}}\Xi_n$ converging, as $n\rightarrow\infty$, to $\cM$.
\end{prop}
\section{Support functions}\label{AppA6}\index{Support function}
\begin{defi}
	The support function of a set $A\subset \D$ is defined  on the dual space $\Dt$ as 
	$$ Supp_A(\cP):= \bigvee_{\cM\in A }\cP:\cM \ . $$
	In particular, if $A$ is convex and closed then $Supp_A$ is the Legendre transform of the characteristic function\index{characteristic function} of $A$. \index{Legendre transform}
\end{defi}

Note that the support function is finite everywhere if and only if $A$ is a compact set. 
A support function is also positively homogeneous of order 1: 
\begin{defi}\label{defPH} A function $\Xi$ on $\Dt$ is positively homogeneous of order 1 if for  any real $r\geq 0$ and $\cP\in\Dt$: 
	\be\label{phom} \Xi(r\cP)=r\Xi(\cP) \ .  \ee
\end{defi}
From Proposition \ref{propA6} we obtain 
\begin{prop}\label{propA11}
	If $\Xi$ is convex, LSC  \index{Lower-semi-continuous (LSC)} and positively homogeneous of order one on $\Dt$  then there exists a closed convex set $\K\subset \D$ such that $\Xi^*=1_\K$ on $\D$. In particular,
	$$ \Xi=Supp_{\K}\ . $$
\end{prop}
Let us sketch the proof of Proposition \ref{propA11}. 
	Since, in particular, $\Xi(0)=0$ then  $\Xi^*(\cM)\equiv \sup_{\cP\in\Dt}\cP:\cM-\Xi(\cP)\geq 0$ for any $\cM\in\D$. . Moreover, we observe by (\ref{phom}) that if there exists $\cM\in\D$ for which $\cP:\cM-\Xi(\cP)>0$ then $\Xi^*(\cM)=\infty$. Indeed
	$\Xi^*(\cM)\geq \sup_{r\geq 0} r[\cP:\cM-\Xi(\cP)]$. It follows that $\Xi^*$ is the characteristic function  \index{characteristic function}of some $\K\subset\D$. Since it is, in addition, a convex and LSC  \index{Lower-semi-continuous (LSC)}  function, it follows from Proposition \ref{propA1} that $\K$ is convex and closed. By Proposition \ref{propA6}
	$$ \Xi(\cP)=\chi^*_\K(\cP)\equiv \bigvee_{\cM\in \D} \cP:\cM-\chi(\cP)=  \bigvee_{\cM\in \K} \cP:\cM$$
	by definition of the characteristic function $1_\K$.\index{characteristic function}
\par
From Propositions \ref{propA1}, \ref{cf7} we also obtain
\begin{prop}\label{propA12}
	$\cP_0\not=0$ is a differentiable point of a LSC, \index{Lower-semi-continuous (LSC)}  convex and positively homogeneous of order 1 
	function $\Xi$ iff $\cM_0=\nabla\Xi(\cP_0)$ is an extreme point  \index{extreme point} of the corresponding closed and convex set $\K$ satisfying $1_\K^*=\Xi$.
\end{prop}
\chapter{Convergence of measures}\label{weakconv}
\section{Total variation}
A strong notion of convergence of Borel  measures on a compact space $(X, {\cal B})$  is the convergence in total variations. 
The total variation (TV) norm is defined by
\be\label{totvar1}\| \mu_1-\mu_2\|_{TV}= \sup_{\phi\in C(X); |f|\leq 1} \int \phi(d\mu_1-d\mu_2) \ . \ee
In fact, the TV norm  \index{Total variation (TV) norm} is taken, in general, as the supremum  with respect to the measurable functions bounded by 1. However, in the case of a compact space (or, more generally, in the case of Polish space\index{Polish space}\footnote{separable, completely metrizable topological space}, the two definitions coincide. 

In general, this norm is not restricted to probability (or even positive) measures. In particular, the total variation distance between a positive measure $\mu$ to the zero measure is $\mu(X)$. If $\mu$ is not a positive measure then by  the Hahn-Jordan decomposition \cite{Bi}
$$ \mu=\mu_+-\mu_-$$ 
where $\mu_\pm$ are both non-negative measures and 
$$\|\mu-0\|_{TV}:= \|\mu\|_{TV} = \mu_+(X)+\mu_-(X) \ . $$

In the special case of probability measures, there is another, equivalent definition as follows:
\be\label{totvar2} \|\mu_1-\mu_2\|_{TV}=\sup_{A\in{\cal B}}\mu_1(A)-\mu_2(A) \ . \ee
In particular, the TV distance between two probability measures is between 0 and 2. 

The equivalence between the two definition (\ref{totvar1}, \ref{totvar2}) for probability measures is a non trivial result, based on duality theory (in the spirit of Kantorovich duality \index{Kantorovich duality} mentioned in section \ref{dualitysection} ).  

The TV norm \index{Total variation (TV) norm} also induces a notion of distance between measurable sets. Given a positive measure $\mu$ on $X$ (e.g. the Lebesgue measure), the TV distance between $A,B\in{\cal B}$ is the TV norm  \index{Total variation (TV) norm} between the measure $\mu$ restricted to $A$ and $B$:
$$ \|A-B\|_{TV,\mu}:= \|\mu\lfloor A - \mu\lfloor B\|_{TV}= \mu(A\Delta B) $$
where $A\Delta B$ is the {\em symmetric difference} between $A$ and $B$, namely  \\ 
$A\Delta B= (A-B)\cap (B-A)$.  The reader may compare it with the Hausdorff distance between sets in a metric space $(X,d)$:
$$ d_H(A, B):= \{ \sup_{x\in A}\inf_{y\in B} d(x,y)\}\vee \{\sup_{x\in B}\inf_{y\in A}d(x,y)\} \ . $$

If $\mu_1, \mu_2$ are both absolutely continuous with respect to another measure $\mu$, then an equivalent definition (independent of the choice of $\mu$ satisfying this condition) is
\be\label{difermu12} \|\mu_1-\mu_2\|_{TX}= \int_X\left|\frac{d\mu_1}{d\mu}-\frac{d\mu_2}{d\mu}\right|d\mu \ . \ee
The TV norm  \index{Total variation (TV) norm} is, indeed, a strong norm in the sense that it demands a lot from a sequence of measures to converge. Let us consider, for example, the measure  $\mu=\delta_{x}$ where $x\in X$ , i.e. the measure defined as 
$$\delta_x(A)= \left\{\begin{array}{cc}
1 & \text{if} \ x\in A \\
0 & \text{if} \ x\not\in A
\end{array} \right. \ \  \forall A\in{\cal B}  \ . $$
Let now a sequence $x_n\rightarrow x$ with respect to the topology of $X$ (e.g. $\lim_{n\rightarrow\infty}d(x_n,x)=0$ if $(X,d)$ is a metric space). 
Then $\mu_n:=\delta_{x_n}$ {\em does not} converge to $\delta_x$ in the TV norm, unless $x_n=x$ for all $n$ large enough. Indeed, one can easily obtain that 
$$ \|\delta_x-\delta_y\|_{TV}=2$$
for any $x\not= y$. 
\section{Strong convergence}
The TV norm can be weaken by the following definition\index{Total variation (TV) norm}
\begin{defi}\label{defB.1.2} A sequence $\mu_n$ converges strongly to $\mu$ if for any $A\in {\cal B}$
	$$ \lim \mu_n(A)=\mu(A) \ . $$
	\end{defi} 
The notion of strong convergence is evidently weaker than TV convergence. Consider, for example, $X=[0,1]$ and $\mu_n(dx)=f_n(x)dx$ where 
$$f_n(x)= \left\{\begin{array}{cc}
1 & \text{if} \exists k \  \text{even} , \ \  x\in [k/n, (k+1)/n) , k \ \\
0 & \text{otherwise} 
\end{array} \right. \ \   \ . $$
Then we can easily verify that $\mu_n$ converges strongly to the uniform measure\\  $\mu=(1/2)dx$ on the interval $X$. 
However, by (\ref{difermu12})
$$ \|\mu_n-\mu\|_{TV}=\int_0^1\left|f_n-\frac{1}{2}\right| dx\rightarrow 1 \ . $$
An equivalent definition of strong convergence is the following: 
$\mu_n$ strongly converge to $\mu$ if for any bounded measurable $f$ on $X$
\be\label{defstrongcon} \lim_{n\rightarrow\infty} \int _Xf d\mu_n = \int_X fd\mu \ . \ee
Indeed, Definition \ref{defB.1.2} implies this for any characteristic function \index{characteristic function} on ${\cal B}$, hence for any simple function, and from here we can extend to any Borel measurable function by a limiting argument. 

Even though strong convergence is weaker than TV convergence, it is not weak enough. In particular, the sequence $\delta_{x_n}$ does not strongly converges, in general,  to $\delta_x$ if   $x=\lim_{n\rightarrow \infty}x_n$.  Indeed, if $A=\cup_n\{x_n\}$ and $x\not\in A$ then evidently $\delta_{x_n}(A)=1$ for any $n$ but $\delta_x(A)=0$. 

In particular, if, in the above example, $x_n\not= x_j$ for $n\not= j$ then there is no strongly convergence subsequence of $\delta_{x_n}$ which, in other words, implies that the strong convergence is not sequentially compact on the set of probability measures. 
\section{Weak*  convergence}\label{wconofme}\index{weak* convergence}
There are many notions of weak* -convergence \index{weak* convergence}in the literature, which depends on the underlying spaces. Since we concentrate in this book on continuous functions on a compact space, we only need one definition. 

Let us start with the following observation: 
 Any continuous function is Borel measurable and bounded (due to compactness of $X$). Therefore,  we can integrate any function in $C(X)$ with respect to a given, bounded Borel measure $\nu\in {\cal M}(X)$.  By the property of integration , this integration we may be viewed as a {\em linear functional} on $C(X)$:
$$ \nu(\phi):= \int_x\phi d\nu \ . $$
\begin{defi}\label{defBweakconv} \index{Borel measure}
	A sequence of of Borel measures $\{\nu_n\}$ on a compact set $X$  is said to converge weakly-* to $\nu$ ($\nu_n \rightharpoonup \nu$) if
	$$ \lim_{n\rightarrow\infty} \nu_n(\phi)=\nu(\phi) \  \ \ \forall \phi\in C(X) \ . $$
\end{defi}
In spite of the apparent similarity between this Definition and  (\ref{defstrongcon}),  we may observe that this notion of weak* convergence\index{weak* convergence} is, indeed, weaker than the  strong (and, certainly, TV) convergence.  In particular, if $\nu_n=\delta_{x_n}$ and $\lim_{n\rightarrow\infty} x_n=x$ in $X$, then $\nu_n$ converges weakly-* to $\delta_x$. Indeed, the continuity of $\phi$ (in particular, its continuity at the point $x\in X$), implies
$$ \delta_{x_n}(\phi):= \phi(x_n)\rightarrow \phi(x):=\delta_x(\phi) \ . $$\index{weak* convergence}
This is in contrast to strong convergence, as indicated above.

   The space of continuous functions on a compact set is a Banach space with respect to the supremum norm 
   $$ \|\phi\|_\infty = \sup_{x\in X} |\phi(x)| \  , \ \phi\in C(X) \ . $$
  
   If we consider $C(X), \|\cdot\|_\infty$ as a Banach space, then any such functional is also {\em continuous}
   $$ |\nu(\phi)|\leq \nu(X)\|\phi\|_\infty \ . $$
   Recall that the set ${\cal M}(X)$ of bounded Borel measures is also a linear space. We may invert our point of view, and consider any $\phi\in C(X)$ as a {\em linear functional }  on ${\cal M}(X)$:
   \be\label{actiononM}\phi(\nu):= \nu(\phi) \ \ \ \forall \nu\in {\cal M}(X) \ . \ee
   Then, Definition \ref{defBweakconv} can be understood in the sense that any $\phi\in C(X)$ is a {\em continuous} linear  functional on ${\cal M}(X)$, {\em taken with respect to the weak* convergence}. Indeed, 
   $$ \lim_{n\rightarrow\infty} \phi(\nu_n) = \phi(\nu) \ \ \ \text{if and only if} \ \ \nu_n\rightharpoonup\nu \ . $$
   Stated differently, 
   \begin{tcolorbox}
   The weak*convergence \index{weak* convergence}of measures is the {\em weakest topology} by which the action (\ref{actiononM}) of any $\phi\in C(X)$ on ${\cal M}(X)$ is continuous. 
   \end{tcolorbox}
There is more to say about weak* convergence. 
The set of all continuous linear functionals on a Banach space $B$ is its dual space, usually denoted by $B^*$, is a Banach space as well with respect to the norm induced by $\|\cdot\|_B$. 
 Since $(C(X), \|\cdot\|_\infty)$ is a Banach space, its dual $C^*(X)$ {\em contains} the space of bounded Borel measures ${\cal M}(X)$. 
By the Riesz-Markov-Kakutani representation theorem \cite{kak}, {\em any} continuous functional on $(C(K), \|\cdot\|_\infty)$ is represented  by  finite Borel measure. Thus, 
\be\label{C=M*} C^*(X)={\cal M}(X) \ . \ee


Here comes  the Banach-Alaoglu theorem \cite{Rud}: 
\begin{theorem}
The closed unit ball of  the dual $B^*$ of a Banach space $B$ (with respect to the norm topology)  is compact with respect to the weak* topology. 
\end{theorem}
\begin{remark}
	In the case of $C^*(X)$, the norm topology is just the TV norm defined in (\ref{totvar1}). 
\end{remark}
Together with (\ref{C=M*}) we obtain the local compactness of ${\cal M}(X)$ with respect to the weak* -topology.\index{weak* convergence}

There is much more to say about the weak* topology. In particular
 the set of probability measures  ${\cal M}_1(X)$ under  the weak* -topology is metrizable, i.e. there exists a metric  on ${\cal M}_1$  compatible with the weak* topology. This, in fact, is a special case of a general theorem which states that the unit ball of the dual space $B^*$ of a separable Banach space is metrizable. The interesting part which we stress here is:
 
 \begin{theorem}
  The metric Monge  distance, described in Example \ref{9.4.2} is a metrization of the weak* topology on ${\cal M}_1(X)$. 
\end{theorem}
We finish this very fast and dense introduction to weak* convergence by proving this last  Theorem. Recall (c.f. example \ref{9.4.2}) that the metric Monge distance \index {Monge distance} on ${\cal M}_1$ is given by (\ref{wasser1}): 
\be\label{secrepW1} d(\mu,\nu)=\sup_{\phi\in Lip(1)} \int_X \phi d(\nu-\mu) \ ,  \ \ \ \mu,\nu\in {\cal M}_1(X) \ . \ee
Curiously, this is very similar to the definition of the TV norm (\ref{totvar1}), \index{Total variation (TV) norm}which is just the norm topology on ${\cal M}_1$ induced by the supremum norm $\|\cdot\|_\infty$ on $C(X)$. The only difference is that here we consider the supremum on the set of $1-$Lipschitz functions, instead of the whole unit ball of $(C(X), \|\cdot\|_\infty)$.

First, we show that a  convergence of a sequence  $\nu_n$  in the metric Monge distance \index {Monge distance}  to $\nu$ implies $\nu_n\rightharpoonup\nu$.   This follows  from the density of Lipschitz functions in $(C(X), \|\cdot\|_\infty)$. Given $\phi\in C(X)$ and $\eps>0$, let $\tilde{\phi}\in C(X)$ be a Lipschitz function such that $\|\phi-\tilde\phi\|_\infty <\eps$.  By the definition of the metric Monge distance, 
$$\int_X\phi(d\nu_n-d\nu)\leq \eps+\int_X\tilde\phi(d\nu_n-d\nu)\leq \eps + |\tilde\phi|_1d(\nu_n,\nu)$$
where $|\tilde\phi|_1:= \sup_{x\not=y} \frac{|\tilde\phi(x)-\tilde\phi(y)|}{|x-y|}$ is the Lipschitz norm of $\tilde\phi$. 

For the other direction we take advantage of the compactness of the $1-$Lipschitz functions in $C(X)$. This implies, in particular, the existence of a maximizer $\phi_{(\nu,\mu)}$ in (\ref{secrepW1}):
$$ d(\mu,\nu)=\int_X \phi_{(\mu,\nu)}d(\nu-\mu) \ \ . $$
Let now $\phi_{(\nu_n,\nu)}$ be the sequence of the maximizers realizing $d(\nu_n,\nu)$. By the above mentioned compactness, there is a subsequence of the series $\phi_{(\nu_{n_k},\nu)}$ which converges in the supremum norm to a function $\psi\in C(X)$. Then 
$$ \lim_{k\rightarrow\infty} \int_X\psi(d\nu_{n_k}-d\nu)=0$$
by assumption. It follows that
$$ d(\nu_{n_k}, \nu)=
\int_X\phi_{(\nu_{n_k},\nu)}
(d\nu_{n_k}-d\nu)
=  \int_X\psi(d\nu_{n_k}-d\nu)
+\int_X(\phi_{(\nu_{n_k},\nu)}-\psi)(d\nu_{n_k}-d\nu) \ . $$
Since 
$$ \left| \int_X(\phi_{(\nu_{n_k},\nu)}-\psi)(d\nu_{n_k}-d\nu)  \right|\leq\|\phi_{(\nu_{n_k},\nu)}-\psi\|_\infty \rightarrow 0$$
we obtain the convergence of this subsequence to $\nu$  in the Monge metric.  \index {Monge distance}Finally, the same argument implies that any converging subsequence has the same limit $\nu$, thus the whole sequence converges to $\nu$. 
\end{appendices}

%
%





\printindex

\end{document}